\documentclass[11pt]{report}
\usepackage[utf8]{inputenc}
\usepackage[T1]{fontenc}
\usepackage{graphicx}
\graphicspath{ {images/} }
\usepackage{caption}
\usepackage{subcaption}
\usepackage{float}
\usepackage[width=150mm,top=35mm,bottom=25mm,bindingoffset=6mm]{geometry}
\usepackage{fancyhdr}
\usepackage{setspace}
\renewcommand{\headrulewidth}{0pt}
\usepackage[utf8]{inputenc}
\usepackage{tikz}
\usetikzlibrary{matrix}
\usepackage{amsfonts, amssymb, amsmath, amsthm}
\usepackage{graphicx}
\usepackage{float, comment}
\usepackage[hidelinks]{hyperref}
\usepackage{verbatim}
\usepackage{mathtools, nccmath}

\usepackage{enumitem}
\usepackage{verbatim}

\newtheorem{teor}{Theorem}[section]
\newtheorem{lema}[teor]{Lemma}
\newtheorem{prop}[teor]{Proposition}
\newtheorem{fact}[teor]{Fact}

\newtheorem{defin}[teor]{Definition}
\newtheorem*{claim}{Claim}
\newtheorem{claimnum}{Claim}
\usepackage{etoolbox}
\AtEndEnvironment{proof}{\setcounter{claimnum}{0}}

\newtheorem*{prop*}{Proposition}
\newtheorem{cor}[teor]{Corollary}
\newtheorem{remark}[teor]{Remark}
\newtheorem{problem}[teor]{Problem}
\newtheorem{question}[teor]{Question}
\newtheorem{external claim}[teor]{Claim}
\newtheorem{notation}[teor]{Notation}

\newtheorem{example}[teor]{Example}
\newtheorem{counterexample}[teor]{Counterexample}

\newcommand{\bigslant}[2]{{\raisebox{.2em}{$#1$}\left/\raisebox{-.2em}{$#2$}\right.}}
\newcommand{\bslant}[2]{{\raisebox{.2em}{$#1$}/\raisebox{-.2em}{$#2$}}}

\newcommand{\C}{{\mathfrak C}}
\newcommand{\R}{\mathbb{R}}

\DeclareMathOperator{\tp}{{tp}}
\DeclareMathOperator{\qftp}{{tp^{qf}}}

\DeclareMathOperator{\Th}{{Th}}
\DeclareMathOperator{\aut}{{Aut}}
\DeclareMathOperator\dcl{dcl}
\DeclareMathOperator\auto{Aut}
\DeclareMathOperator\stab{Stab}
\DeclareMathOperator\dom{dom}
\DeclareMathOperator{\med}{{med}}
\DeclareMathOperator{\EM}{{EM}}
\DeclareMathOperator{\Lin}{{Lin}}
\newcommand{\mL}{\mathcal{L}}
\newcommand{\I}{\mathcal{I}}
\newcommand{\II}{\mathbf{I}}
\newcommand{\fwapp}{F'_{\textrm{WAP}}}
\newcommand{\fwap}{F_{\textrm{WAP}}}
\newcommand{\ftaa}{F'_{\textrm{Tame}}}
\newcommand{\fta}{F_{\textrm{Tame}}}

\DeclareMathOperator{\emtp}{{EMtp}}
\DeclareMathOperator{\age}{{Age}}
\DeclareMathOperator{\Sym}{{Sym}}
\DeclareMathOperator{\End}{{End}}
\DeclareMathOperator{\Image}{{Im}}
\DeclareMathOperator{\Ind}{{Ind}}
\DeclareMathOperator{\cl}{{cl}}

\newcommand{\notiff}{%
		\mathrel{{\ooalign{\hidewidth$\not\phantom{"}$\hidewidth\cr$\iff$}}}}

\usepackage[
backend=bibtex,
style=alphabetic,
]{biblatex}
\usepackage{stackengine}
\stackMath
\DeclareMathOperator*\dcap{{\stackinset{r}{-0.35ex}{c}{-1.9pt}{\downarrow}
		{\bigcap}\mkern2mu}}
\DeclareMathOperator*\acup{{\stackinset{r}{-0.35ex}{c}{1.9pt}{\uparrow}
		{\bigcup}\mkern2mu}}

\title{Canonical quotients in Model Theory}
\author{Adrián Portillo Fernández}
\date{Day Month Year}

\begin{document}
\pagenumbering{roman} 

\begin{titlepage}
    \begin{center}
        \vspace*{1cm}
        
        \Huge
        \textbf{Kanoniczne ilorazy w teorii modeli}
        
        \vspace{1.5cm}
        \LARGE
        Rozprawa doktorska\\
        \textbf{Adrián Portillo Fernández}

        \vspace{2cm}

        Instytut Matematyczny\\
        Uniwersytet Wroc\l awski

        \vspace{2cm}
        Promotor: Prof. dr hab. Krzysztof Krupi\'nski
        \vfill
        Wroc\l aw, 2024
        
    \end{center}
    
\end{titlepage}
\begin{titlepage}
    \begin{center}
        \vspace*{1cm}
        
        \Huge
        \textbf{Canonical quotients in model theory}
        
        \vspace{1.5cm}
        \LARGE
        Doctoral Thesis\\
    
        \textbf{Adrián Portillo Fernández}
        
        \vspace{2cm}
                Institute of Mathematics\\
        University of Wroc\l aw
        
        \vspace{2cm}
Supervised by: Prof. Krzysztof Krupi\'nski

        \vfill
        Wroc\l aw, 2024
        
        
    \end{center}
    
\end{titlepage}





\begin{center}
    
    
    
    \textbf{Streszczenie}
\end{center}

Badamy pewne kanoniczne ilorazy w teorii modeli, głównie stabilne ilorazy grup typowo definiowalnych oraz typów niezmienniczych w teoriach z własnością NIP.

Główne wyniki rozprawy są następujące:

\begin{itemize}
    \item Rozwiązujemy dwa problemy z pracy \cite{MR3796277} dotyczące maksymalnych stabilnych ilorazów grup typowo definiowalnych w teoriach z własnością NIP. Pierwszy wynik mówi, że jeśli $G$ jest typowo definiowalną grupą w teorii dystalnej, to $G^{st}=G^{00}$ (gdzie $G^{st}$ jest najmniejszą typowo definiowalną podgrupą o stabilnym ilorazie $G/G^{st}$, a $G^{00}$ jest najmniejszą typowo definiowalną podgrupą o ograniczonym indeksie). Aby go uzyskać, dowodzimy, że dystalność jest zachowana przy przejściu od teorii $T$ do hiperurojonego rozszerzenia $T^{heq}$. Drugim wynikiem jest przykład grupy $G$ definiowalnej w niedystalnej teorii z własnością NIP, dla której $G=G^{00}$, ale $G^{st}$ nie jest przekrojem grup definiowalnych. Naszym przykładem jest nasycone rozszerzenie $(\mathbb{R},+,[0,1])$. Ponadto poczyniliśmy pewne obserwacje dotyczące pytania, czy istnieje taki przykład, który jest grupą o skończonym wykładniku. Podajemy też kilka charakteryzacji stabilności zbiorów hiperdefiniowalnych, m.in. w terminach logiki ciągłej.

    \item Dla teorii $T$ z własnością NIP, dostatecznie nasyconego modelu $\C$ teorii $T$ (tzw. modelu monstrum) oraz niezmienniczego (nad pewnym małym podzbiorem $\C$) globalnego typu $p$ dowodzimy, że istnieje najdrobniejsza relatywnie typowo definiowalna nad małym zbiorem parametrów z $\C$ relacja równoważności na zbiorze realizacji typu $p$, która ma stabilny iloraz. Jest to odpowiednik w kontekście relacji równoważności głównego wyniku z pracy \cite{MR3796277} o istnieniu maksymalnych stabilnych ilorazów grup typowo definiowalnych w teoriach z własnością NIP. Nasz dowód adaptuje idee dowodu tego wyniku, używając relatywnie typowo definiowalnych podzbiorów grupy automorfizmów modelu monstrum w sensie \cite{hrushovski2021order}.

    \item Definiujemy ciągłą własność modelowania dla struktur pierwszego rzędu i pokazujemy, że struktura pierwszego rzędu ma tę własność wtedy i tylko wtedy, gdy jej wiek ma własność Ramseya. Używamy uogólnionych ciągów nieodróżnialnych w logice ciągłej do badania i charakteryzowania $n$-zależności dla teorii ciągłych oraz dla zbiorów hiperdefiniowalnych (w logice pierwszego rzędu) w terminach kolapsu ciągów nieodróżnialnych.

    \item Niech $T$ będzie teorią zupełną, $\C$ jej modelem monstrum, a $X$ zbiorem typowo definiowalnym nad $\emptyset$. Badamy maksymalne ilorazy $\aut(\C)$-potoku $S_X(\C)$, które są WAP lub oswojone (ang. {\em tame}) w sensie dynamiki topologicznej. Mianowicie, niech $\fwap\subset S_X(\C)\times S_X(\C)$ będzie najmniejszą domkniętą $\aut(\C)$-niezmienniczą relacją równoważności na $S_X(\C)$ taką, że potok $(\aut(\C), S_X(\C)/\fwap)$ jest WAP, i niech $\fta\subset S_X(\C)\times S_X(\C)$ będzie najmniejszą domkniętą $\aut(\C)$-niezmienniczą relacją równoważności na $S_X(\C)$ taką, że potok $(\aut(\C), S_X(\C)/\fwap)$ jest oswojony. Wykazujemy dobre zachowanie $\fwap$ i $\fta$ przy zmianie modelu monstrum $\C$. Mianowicie, dowodzimy, że jeśli $\C'\succ \C$ jest większym modelem monstrum, a $\fwapp$ i $\ftaa$ są odpowiednikami $\fwap$ i $\fta$ obliczonymi dla $\C'$ i $r: S_X(\C')\to S_X(\C)$ jest funkcją obcięcia, to $r[\fwapp]=\fwap$ i $r[\ftaa]=\fta$. Korzystając z tych wyników, pokazujemy, że grupy Ellisa potoków $( \aut(\C), S_X(\C)/\fwap )$ i $( \aut(\C), S_X(\C)/\fta )$ nie zależą od wyboru modelu monstrum $\C$.
\end{itemize}

Wyniki zawarte w pierwszym, drugim i czwartym punkcie zostały uzyskane wspólnie z Krzysztofem Krupińskim, a w trzecim punkcie przeze mnie. Rezultaty z pierwszego punktu pochodzą z pracy \cite{10.1215/00294527-2022-0023}, z drugiego z pracy \cite{KRUPIŃSKI_PORTILLO_2023}, z trzeciego będą zawarte w mojej samodzielnej pracy, a z czwartego w przyszłej wspólnej pracy z K. Krupińskim.

\newpage

\fancyhf{} 
\fancyhead[RO,R]{\thepage} 
\renewcommand{\headrulewidth}{0pt}

\begin{center}
%
%
%
    \textbf{Abstract}
\end{center}
We study canonical quotients in model theory, mainly stable quotients of type-definable groups and invariant types in NIP theories.

The main results of the thesis are the following:
\begin{itemize}
\item We solve two problems from \cite{MR3796277} concerning maximal stable quotients of groups type-definable in NIP theories. The first result says that if $G$ is a type-definable group in a distal theory, then $G^{st}=G^{00}$ (where $G^{st}$ is the smallest type-definable subgroup with $G/G^{st}$ stable, and $G^{00}$ is the smallest type-definable subgroup of bounded index). In order to get it, we prove that distality is preserved under passing from a theory $T$ to the hyperimaginary expansion $T^{heq}$. The second result is an example of a group $G$ definable in a non-distal, NIP theory for which $G=G^{00}$ but $G^{st}$ is not an intersection of definable groups. Our example is a saturated extension of $(\mathbb{R},+,[0,1])$. Moreover, we make some observations on the question whether there is such an example which is a group of finite exponent. We also take the opportunity and give several characterizations of stability of hyperdefinable sets, involving continuous logic.
\item 	For a NIP theory $T$, a sufficiently saturated model $\C$ of $T$ (so-called monster model), and an invariant (over some small subset of $\C$) global type $p$, we prove that there exists a finest relatively type-definable over a small set of parameters from $\C$ equivalence relation on the set of realizations of $p$ which has stable quotient. This is a counterpart for equivalence relations of the main result of \cite{MR3796277} on the existence of maximal stable quotients of type-definable groups in NIP theories. Our proof adapts the ideas of the proof of that result, working with relatively type-definable subsets of the group of automorphisms of the monster model as defined in \cite{hrushovski2021order}.
\item We define the continuous modelling property for first-order structures and show that a first-order structure has the continuous modelling property if and only if its age has the embedding Ramsey property. We use generalized indiscernible sequences in continuous logic to study and characterize $n$-dependence  for continuous theories and first-order hyperdefinable sets in terms of the collapse of indiscernible sequences.
\item We study maximal WAP and tame (in the sense of topological dynamics) quotients of $S_X(\C)$, where $\C$ is a monster model of a complete theory $T$ and $X$ is an $\emptyset$-type-definable set. Namely, let $\fwap\subset S_X(\C)\times S_X(\C)$ be the finest  closed $\aut(\C)$-invariant equivalence relation on $S_X(\C)$ such that the flow $( \aut(\C), S_X(\C)/\fwap )$ is WAP, and let $\fta\subset S_X(\C)\times S_X(\C)$ be the finest closed $\aut(\C)$-invariant equivalence relation on $S_X(\C)$ such that the flow $( \aut(\C), S_X(\C)/\fta )$ is tame. We show good behaviour of $\fwap$ and $\fta$ under changing the monster model $\C$. Namely, we prove that if $\C'\succ \C$ is a bigger monster model, $\fwapp$ and $\ftaa$ are the counterparts for $\fwap$ and $\fta$ computed for $\C'$, and $r:S_X(\C')\to S_X(\C)$ is the restriction map, then $r[\fwapp]=\fwap$ and $r[\ftaa]=\fta$ Using these results, we show that the Ellis groups of $( \aut(\C), S_X(\C)/\fwap )$ and  $( \aut(\C), S_X(\C)/\fta )$ do not depend on the choice of the monster model $\C$.
\end{itemize}
The results contained in the first, second and fourth bullets are joint with Krzysztof Krupi\'nski and the ones contained in the third bullet are mine alone. The results in the first bullet come from  \cite{10.1215/00294527-2022-0023}, in the second from \cite{KRUPIŃSKI_PORTILLO_2023}, the results in the third bullet will be contained in a future paper by myself, and the results in the last bullet will be contained in a future joint paper with Krzysztof Krupi\'nski.





\chapter*{Acknowledgements}

To my advisor, Krzysztof Krupi\'nski, for his support and guidance throughout these years, as well as his indispensable help in polishing this thesis and our papers.\\

\noindent To my colleagues in the Wroc\l aw model theory seminar, for all their knowledge and insights, as well as the attentive care and interest that they showed towards my work.\\

\noindent To the anonymous referees for helpful comments about my papers.\\

\noindent To my family and friends, for their support and encouragement through the challenges.\\

\noindent To Eli, I could not have undertaken this journey without you, and no words could convey how deeply thankful I am for everything you did.\\

\vfill
\noindent This research was supported by the Narodowe Centrum Nauki grant no. 2016/22/E/ST1/00450 from where I was receiving Ph.D. scholarship.

\tableofcontents



\chapter{Introduction}
\pagenumbering{arabic} 
	
The core of model-theory is stability theory, developed in the 70's and 80's of the previous century. In the past three decades, one of the main goals of model theory has become finding extensions of stability theory to various unstable contexts, covering many mathematically interesting examples. One either tries to impose some general global assumptions on the theory in question (such as NIP or simplicity) or some local ones (e.g work with a stable definable set or generically stable type), and prove some structural results. 
It is also natural and ubiquitous in model theory to look at a ``global-local'' situation, namely quotients by type-definable equivalence relations and assume (or prove) their good properties (e.g. boundedness) to get some further conclusions. Recall that a {\em hyperimaginary} is a class of a type-definable equivalence relation, and a {\em hyperdefinable set} is a quotient of a type-definable set by such a relation. While bounded quotients have played an important role in model theory and its applications (e.g. to approximate subgroups) for many years, stable quotients have not been studied thoroughly. This project originates with a talk by Anand Pillay in Lyon in 2009 on finest stable hyperdefinable quotients in NIP theories and was continued with the paper \cite{MR3796277}, where stable quotients of group first appeared. 
It is folklore that hyperimaginaries can be treated as imaginaries in continuous logic via a definable pseudometric (see \cite{2010,conant2021separation} and \cite[Chapter 3]{hanson2020thesis} in the language of continuous logic and \cite{10.2178/jsl/1122038916} in the language of CATs), so in a sense stability of hyperdefinable sets is equivalent to stability (of imaginary sorts) in continuous logic developed in \cite{MR2657678,MR2723787}. This is an additional motivation to consider stable quotients. So in Section \ref{section: characterizations of stability} we take the opportunity and give several characterizations of stability of hyperdefinable sets in various terms, involving continuous logic, including generically stable types. 

In the main parts of the thesis, however, we will study stability of hyperdefinable sets 
without referring to continuous logic, just using the definition from \cite{MR3796277} which we recall below, or a characterization via bounds on the number of types observed in Section \ref{section: characterizations of stability} (see Theorem \ref{equiv of stab intro} below). 

Let $T$ be a complete theory, $\C \models T$ a monster (i.e., $\kappa$-saturated and strongly $\kappa$-homogeneous for a strong limit cardinal $> |T|$) model in which we are working, and $A \subset \C$ a small set of parameters (i.e., $|A| < \kappa$); a cardinal $\gamma$ is {\em bounded} if $\gamma < \kappa$. Recall that for a hyperdefinable set $X/E$, the complete type over $A$ of an element of $X/E$ can be defined as the $\aut(\C/A)$-orbit of that element, or the preimage of this orbit under the quotient map, or the partial type defining this preimage.

		\begin{defin}\label{def stability intro}
		A hyperdefinable (over $A$) set $X/E$ is \emph{stable} if for every $A$-indiscernible sequence $(a_i,b_i)_{i<\omega}$ with $a_i\in X/E$ for all (equivalently, some) $i<\omega$, we have  
		$$\tp(a_i,b_j/A) = \tp(a_j,b_i/A)$$ for all (some) $i\neq j < \omega$.
	\end{defin}
	
We introduce $\mathcal{F}_{X/E}$, a special family of functions related to the hyperdefinable set $X/E$ (see Section \ref{section: characterizations of stability}). The family $\mathcal{F}_{X/E}$ allows us to study properties (stability, NIP) of hyperdefinable sets using continuous logic. In particular, it allows us to prove the following characterizations of stability for hyperdefinable sets with NIP (see Definitions \ref{gen stable type}, \ref{def weakly stable} and Section \ref{section: characterizations of stability} for the definitions of the terms used in the formulation): 
\begin{teor}\label{equiv of stab intro}
	Assume $X/E$ has NIP. The following conditions are equivalent:
	\begin{enumerate}
		\item $X/E$ is stable.
		\item $\forall$ $M\models T$ $\forall f\in \mathcal{F}_{X/E}$ $\forall p\in S_f(M)$ $($$p$ is definable$)$.
		\item $\exists \lambda\geq \lvert T\rvert$ $\forall M\models T$ $(\lvert M\rvert\leq \lambda \implies \lvert S_{X/E}(M)\rvert\leq \lambda)$.
		\item Any indiscernible sequence of elements of $X/E$ is totally indiscernible.
		\item Any global invariant $($over some $A$$)$ type $p \in S_{X/E}(\C)$ is generically stable.
              \item $X/E$ is weakly stable.
	\end{enumerate}
\end{teor}

Let $G$ be a $\emptyset$-type-definable group. There is always a smallest $A$-type-definable subgroup of $G$ of bounded index, which is denoted by $G^{00}_A$. Under NIP, the group $G_A^{00}$ does not depend on the choice of $A$ (see \cite{MR2361885}) and is denoted by $G^{00}$. So $G^{00}$ is the smallest type-definable (over parameters) subgroup of $G$ of bounded index, and it is in fact $\emptyset$-type-definable and normal. 
Staying in the NIP context, $G^0$ is defined as the intersection of all relatively definable subgroups of bounded index, and it turns out to be $\emptyset$-type-definable and normal. Regarding stable quotients, since stability of hyperdefinable sets is closed under taking products and type-definable subsets (see \cite[Remark 1.4]{MR3796277}; for the proof see Proposition \ref{stable preserved product} in the appendix), it is clear that there always exists a smallest $A$-type-definable subgroup $G^{\textrm{st}}_A$ such that the quotient $G/G^{\textrm{st}}_A$ is stable. The main result of \cite{MR3796277} says that under NIP, $G^{\textrm{st}}_A$ does not depend on $A$, and so it is the smallest type-definable (over parameters) subgroup with stable quotient $G/G^{\textrm{st}}$, and it is in fact $\emptyset$-type-definable and normal. 
Under NIP, there is also a $\emptyset$-type-definable subgroup $G^{\textrm{st},0}$ which is defined as the intersection of all relatively definable (with parameters) subgroups $H$ of $G$ such that $G/H$ is stable. It is interesting to study those canonical ``components'' as well as quotients by them. To give an unstable example, consider a monster model $K$ of ACVF, and $G:=(V,+)$, where $V$ is the valuation ring of $K$. Then $G^{\textrm{st}}=G^{\textrm{st},0}$ is precisely the additive group of the maximal ideal of $V$, and $G/G^{\textrm{st}}$ is the additive group of the residue field.

In \cite{MR3796277}, the authors suggested that it should be true that for groups definable in o-minimal theories, and, more generally, in distal theories (see Definition \ref{definition: distality}), $G^{\textrm{st}}=G^{00}$. This agrees with the intuition that distality should be thought of as something at the opposite pole from stability. As an illustration, consider the unit circle in the monster model of RCF: then $G^{\textrm{st}}=G^{00}$ is the group of infinitesimals and $G^0=G$.  In Section \ref{section: distal theories}, we prove this conjecture in the following more general form (see Corollary \ref{corollary: stable iff bounded}).

\begin{prop}\label{proposition: distal implies bdd}
If $T$ is distal, then every stable hyperdefinable set is bounded.
\end{prop}

This is deduced from the following result (see Theorem \ref{teordistal}), where $T^{\textrm{heq}}$ denotes the ``expansion'' of $T$ by all hyperimaginary sorts which are quotients by $\emptyset$-type-definable equivalence relations (we just mean here that in the definition of distality one also allows indiscernible sequences of hyperimaginaries).

\begin{teor}\label{proposition: preservation of distality}
If $(a_i)_{i\in \I}$ is a (dense) distal sequence of tuples from $\C^\lambda$, then $(a_i/E)_{i\in \I}$ is a distal sequence of hyperimaginaries. Thus, if $T$ is distal, then $T^{\textrm{heq}}$ is distal (by which we mean that all dense indiscernible sequences of hyperimaginaries are distal).
\end{teor}

We prove the theorem above by elaborating on some arguments from \cite{MR3001548}. 

By Hrushovski's theorem (i.e. \cite[Ch. 1, Lemma 6.18]{pillay1996geometric}), we know that a type-definable group in a stable theory is an intersection of definable groups. 
%
However, although $G/G^{\textrm{st}}$ is stable, it may happen that $G^{\textrm{st}}$ is not an intersection of relatively definable subgroups of $G$, e.g. in the above example with the unite circle, $G^{\textrm{st}}=G^{00}$ is not an intersection of definable groups. 
In \cite{MR3796277}, the authors stated as a problem to find an example of a definable group $G$ where $G^{00}=G$ but $G^{\textrm{st}} \ne G^{\textrm{st},0}$ (i.e. $G^{\textrm{st}}$ is not an intersection of definable groups). In Section \ref{section: main example}, we give such an example: it is  the monster model of $\Th((\mathbb{R},+,[0,1])$.

When we lack the group structure, a natural counterpart of taking the quotient by a subgroup is to take 
	 the quotient by an equivalence relation. 
	Thus, it is natural to ask whether results similar to those appearing in \cite{MR3796277} hold outside of the context of type-definable groups. 
However, the naive counterpart of \cite[Theorem 1.1]{MR3796277} is easily seen to be false. 
Namely, in general, for any non-stable type-definable set $X$ (e.g. the home sort of a non-stable theory), a finest type-definable (over an arbitrary small set of parameters)  equivalence relation on $X$ with stable quotient does not exist. The reason is that given any type-definable equivalence relation $E$ on $X$ with stable quotient, $E$ is not the relation of  equality, so we can find an $E$-class which contains at least two distinct elements $a$ and $b$. Then, the equivalence relation on $X$ being the intersection of $E$ and the relation $\equiv_a$ of having the same type over $a$ is strictly finer that $E$ and has stable quotient by \cite[Remark 1.4]{MR3796277} (as both $X/E$ and $X/\!\!\equiv_a$ are stable).

Let $\C\prec \C'$ be two monster models of a NIP theory $T$ 
	such that $\C$ is small in $\C'$.
	Recall that a {\em relatively type-definable over a (small) set of parameters $B$} subset of a set $Y$ is the intersection of $Y$ with a set which is type-definable over $B$. 
	The main result of this thesis is the following theorem, which will be proved in Section \ref{section: main theorem}.

	\begin{teor}\label{main theorem intro}
		Assume NIP. Let $p(x)\in S(\mathfrak{C})$ be an $A$-invariant type. Assume that $\C$ is at least $\beth_{(\beth_{2}(\lvert x \rvert+\lvert T \rvert + \lvert A\rvert))^+}$-saturated. Then, there exists a finest equivalence relation $E^{\textrm{st}}$ on $p(\C')$ 
		relatively type-definable over a small 
		(relative to $\C$) set of parameters of $\C$ and with stable quotient $p(\C')/E^{\textrm{st}}$.
	\end{teor}

	Our proof is via a non-trivial adaptation of the ideas from the proof of the main theorem of \cite{MR3796277}, using relatively type-definable subsets of the group of automorphisms of the monster model (as defined in \cite{hrushovski2021order}).

	We do not know whether $E^{\textrm{st}}$ is relatively type-definable over $A$. At the end of Section \ref{section: main theorem}, we will observe that if it was true, then the specific (large) saturation degree assumption  in the above theorem could be removed. Another question is whether one could drop the invariance of $p$ hypothesis from the above theorem. If such a strengthening is true, a proof would probably require some new tricks.


  We devote Chapter \ref{Chapter 5} to continuous model theory. Continuous model theory is a growing area of model theory that has been developing very fast in recent years. Many of the most important dividing lines for first-order theories have been also defined for continuous theories (Stability \cite{MR2657678, MR2723787}, NIP \cite{ContVCclasses}, Distality \cite{anderson2023fuzzy}). One invaluable tool for the characterization of dividing lines in first-order theories is (generalized) indiscernible sequences (See \cite[Chaper VII]{Shelah1982-SHECTA-5}, \cite{SCOW20121624}, \cite{Chernikov2014OnN}, \cite{Guingona2019}). We present natural continuous counterparts of generalized indiscernibles and the modeling property (where the index structures are still first-order) and show that a first-order structure has the continuous modeling property (see Definition \ref{definition: CMP}) if and only if its age has the embedding Ramsey property (Theorem \ref{CMP iff Ramsey}). Several notions around this topic have been also defined in positive logic. Dobrowolski and Kamsma (see \cite{Dobrowolski2021KimindependenceIP}) proved that $s$-trees have the (positive logic version of) modeling property in thick theories, later in  \cite[Theorems 1.2, 1.2 and 1.3]{kamsma2023positive} it was shown that $str$-trees, $str_0$-trees (the reduct of $str$-trees that forgets the length comparison relation) and arrays also have the modeling property in positive thick theories.
 
The notion of a dependent theory was first introduced by Shelah in \cite{Shelah1982-SHECTA-5}. 
In later work \cite{Shelah2005StronglyDT, Shelah2007DefinableGF}, Shelah introduced the more general notion of $n$-dependence. This notion was studied in depth in \cite{Chernikov2014OnN}, where the authors give a characterization of $n$-dependent theories in terms of the collapse of indiscernible sequences (See \cite[Theorem 5.4]{Chernikov2014OnN}). In the continuous context, the definition of $n$-dependence was introduced in \cite{Chernikov2020HypergraphRA} using a generalization of the $VC_n$ dimension. Section 10 of the aforementioned paper is dedicated to several operations which preserve $n$-dependence.  The proof of \cite[Theorem 5.4]{Chernikov2014OnN} $(3)\implies (2)$ is not fully correct\footnote{After sending the manuscript to the authors of \cite{Chernikov2014OnN}, they acknowledged that there is a mistake and proposed a short correction that we discuss at the end of Section \ref{section:n-dep}}; we provide a counterexample to the key claim (see Counterexample \ref{counterexample}). Using  the tools developed in Section \ref{section: modelling property} we give an alternative proof of the theorem, obtaining generalizations of \cite[Theorem 5.4]{Chernikov2014OnN} to continuous logic theories (Theorem \ref{n-dep and collapsing intro}) and hyperdefinable sets (Theorem \ref{ndep an collapsing hyperim intro}).

Let $G_{n+1,p}$ be the Fraïssé limit of the class of ordered $(n+1)$-partite $(n+1)$-uniform hypergraphs and $G_{n+1}$ be the Fraïssé limit of the class of ordered $(n+1)$-uniform hypergraphs. By a $G_{n+1,p}$-indiscernible sequence $(a_g)_{g_\in G_{n+1,p}}$ we mean that for any $W,W'\subseteq G_{n+1,p}$ if the quantifier free types of $W$ and $W'$ coincide (in the language $\mL_{opg}$ defined in Section \ref{section:n-dep}), then the types of the tuples $(a_g)_{g\in W}$ and $(a_g)_{g\in W'}$ also coincide (similarly for $G_{n+1}$-indiscernibility, see Definition \ref{defin: gen. indisc.}). We say that the sequence $(a_g)_{g_\in G_{n+1,p}}$ is $\mL_{op}$-indiscernible if for any $W,W'\subseteq G_{n+1,p}$ with the same quantifier free type in the language $\mL_{op}$ (see Section \ref{section:n-dep}), the types of the tuples $(a_g)_{g\in W}$ and $(a_g)_{g\in W'}$ also coincide. 
\begin{teor} \label{n-dep and collapsing intro}
    Let $T$ be a complete continuous logic theory. The following are equivalent:
    \begin{enumerate}
        \item $T$ is $n$-dependent.
        \item Every $G_{n+1,p}$-indiscernible is $\mL_{op}$-indiscernible.
        \item Every $G_{n+1}$-indiscernible is order indiscernible.
    \end{enumerate}
\end{teor}
  

We introduce the following definition of $n$-dependent hyperdefinable sets:
\begin{defin} The hyperdefinable set $X/E$ has the $n$-independence property, $IP_n$ for short, if for some $m<\omega$ there exist two distinct complete types $p,q\in S_{X/E\times \C^m}(\emptyset)$ and a sequence $(a_{0,i},\dots, a_{n-1,i})_{i< \omega}$ from anywhere such that for every finite $w\subset \omega^n$ there exists $b_w\in X/E$ such that $$\tp(b_w, a_{0,i_0},\dots, a_{n-1,i_{n-1}}  )=p \iff (i_0,\dots, i_{n-1})\in w$$ $$\tp(b_w, a_{0,i_0},\dots, a_{n-1,i_{n-1}}  )=q \iff (i_0,\dots, i_{n-1})\notin w. $$ 
\end{defin}
Using the tools developed in Sections \ref{section: modelling property} and \ref{section:n-dep}, we prove the theorem below, 
 which is a counterpart for hyperdefinable sets of Theorem \ref{n-dep and collapsing intro}. Here, $\Psi^{n+1}_{\mathcal{F}_{X/E}}$  is a special family of functions (see Notation \ref{notation PSI} in Section \ref{section: hyperdef n-dep}).
\begin{teor}\label{ndep an collapsing hyperim intro}
 The following are equivalent:
    \begin{enumerate}
        \item $X/E$ is $n$-dependent.
        \item Every $G_{n+1,p}$-indiscernible $(a_g)_{g\in G_{n+1,p}}$, where for every $g\in P_0(G_{n+1,p})$ we have $a_g\in X/E$, is $\mL_{op}$-indiscernible.
        \item For every $m\in \mathbb{N}$, every $G_{n+1}$-indiscernible with respect to $\Psi^{n+1}_{\mathcal{F}_{X/E}}$ sequence of elements of $\C^m\times X$  is order indiscernible with respect to $ \Psi^{n+1}_{\mathcal{F}_{X/E}}$.
    \end{enumerate}
\end{teor}
Here $P_0(G_{n+1,p})$ is the first part of the partition of $G_{n+1,p}$, and by a $G_{n+1}$-indiscernible sequence $(a_g)_{g_\in G_{n+1}}$ with respect to $\Psi^{n+1}_{\mathcal{F}_{X/E}}$  we mean that for any $W,W'\subseteq G_{n+1,p}$ if the quantifier free types of $W$ and $W'$ coincide (in the language $\mL_{og}$ defined in Section \ref{section:n-dep}), then the $\Psi^{n+1}_{\mathcal{F}_{X/E}}$-types of the tuples $(a_g)_{g\in W}$ and $(a_g)_{g\in W'}$ also coincide.

Item $(3)$ of Theorem \ref{ndep an collapsing hyperim intro} cannot be improved by replacing indiscernibility with respect to $\Psi^{n+1}_{\mathcal{F}_{X/E}}$ by indiscernibility with respect to more general families of functions from $\mathcal{F}_{(\C^m \times X/E)^{n+1}}$ or by replacing $\mathcal{F}_{(\C^m \times X/E)^{n+1}}$ by $\Psi^{n'+1}_{\mathcal{F}_{X/E}}$ with $n'>n$ as showed by Example \ref{example optimal}.

 In the last chapter, we apply topological dynamics methods to study maximal WAP and tame quotients of flows of the form $(\aut(\C),S_X(\C))$, where $X$ is an $\emptyset$-type-definable set. Topological dynamics methods were introduced in model theory by Newelski \cite{10.2178/jsl/1231082302, 10.1112/jlms/jdr075}, with the goal to extend results from stable group theory to the unstable context. Since then, the topic has been broadened by a multitude of authors: Chernikov, Hrushovski, Krupi\'nski, Newelski, Pillay, Rzepecki, Simon and others (e.g. \cite{definably_Amenable_nip, doi:10.1142/S0219061319500120, KPR18, 959fd4d1-73ad-3b7b-bc69-e998be32d2cf}). It is well known that, in various contexts, stability corresponds to WAP and NIP to tameness. We check that stable and NIP quotients by type-definable equivalence relations indeed yield respectively WAP and tame quotients of the space of types by the corresponding closed equivalence relations. However, we show that for an arbitrary $\emptyset$-type-definable set $X$, the finest closed $\aut(\C)$-invariant equivalence relation on $S_X(\C)$ with WAP quotient might not be induced by an $\emptyset$-type-definable equivalence relation on $X$ with stable quotient. Similarly,  for an arbitrary $\emptyset$-type-definable set, the finest closed $\aut(\C)$-invariant equivalence relation on $S_X(\C)$ with tame quotient is might not be induced by an $\emptyset$-type-definable equivalence relation on $X$ with NIP quotient (see Proposition \ref{fwap strictly finer}). 

 The Ellis groups of flows play a very important role both in abstract topological dynamics and in model theory, e.g., to get new information about model-theoretic invariants such as $G/G^{000}$ or $\textrm{GAL}_L(T)$ (see \cite{KRUPIŃSKI_PILLAY_2017, KPR18}) and in recent applications to additive combinatorics (see \cite{krupiński2023generalized}). The main result of \cite{doi:10.1142/S0219061319500120} shows that the Ellis group of a given theory is absolute (i.e., does not depend on $\C$). We show that the same is true for the Ellis groups of the maximal WAP and tame quotients. 

Let $\C \prec \C'$ be models of a complete theory $T$ which are $\kappa$-saturated and strongly $\kappa$-homogeneous, where $\kappa$ is specified in the statements below. 
Let $X$ be an $\emptyset$-type-definable subset of $\C^\lambda$. Let $F'$ be a closed, $\aut(\C')$-invariant equivalence relation defined on $S_X(\C')$, and $F$ a closed, $\aut(\C)$-invariant equivalence relation defined on $S_X(\C)$. We say that $F'$ and $F$ are compatible if $r[F']=F$, where $r: S_X(\C')\to S_X(\C)$ is the restriction map.

  Let $\fwap\subset S_X(\C)\times S_X(\C)$ be the finest closed $\aut(\C)$-invariant equivalence relation on $S_X(\C)$ such that the flow $( \aut(\C), S_X(\C)/\fwap )$ is WAP, and let $\fta\subset S_X(\C)\times S_X(\C)$ be the finest closed $\aut(\C)$-invariant equivalence relation on $S_X(\C)$ such that the flow $( \aut(\C), S_X(\C)/\fta )$ is tame. 

\begin{cor}\label{intro: compatible fwap and fta}
    The Ellis groups of $S_X(\C)/\fwap$ and  $( \aut(\C), S_X(\C)/\fta )$ (treated as topological groups with the $\tau$-topology) do not depend on the choice of $\C$ as long as $\C$ is $(\aleph_0+\lambda)^+$-saturated and strongly $(\aleph_0+\lambda)^+$-homogeneous.
\end{cor}

This follows from the following more general result:

\begin{teor}\label{compatible relations teor intro}
    Assume that $\C$ and $\C'$ are $\aleph_0$-saturated and strongly-$\aleph_0$ homogeneous. If $F'$ and $F$ are compatible equivalence relations respectively on $S_X(\C')$ and $S_X(\C)$, then the Ellis group of the flow $(\aut(\C'),S_X(\C')/F')$ is topologically isomorphic to the Ellis group of the flow  $(\aut(\C),S_X(\C)/F)$.
\end{teor}

In order to deduce Corollary \ref{intro: compatible fwap and fta}, we prove the following
\begin{teor}
    Assume $\C$ and $\C'$ are $(\aleph_0+\lambda)^+$-saturated and strongly $(\aleph_0+\lambda)^+$-homogeneous. Then $\fwapp$ is compatible with $\fwap$ and $\ftaa$ is compatible with $\fta$ (where $\fwapp$ and $\ftaa$ are the counterparts of $\fwap$ and $\fta$ computed for $\C')$.
\end{teor}



\section*{Structure of the thesis}

Chapter \ref{Chapter 2} contains the preliminaries. It is divided in the following parts:
\begin{itemize}
    \item Model theory.
    \item Continuous model theory.
    \item Hyperdefinable sets.
    \item Ramsey theory.
    \item Topological dynamics.
\end{itemize}

In Chapter \ref{Chapter 3} (based mostly on \cite{10.1215/00294527-2022-0023} and slightly in \cite{KRUPIŃSKI_PORTILLO_2023}), we study stability of hyperdefinable sets and stable quotients of type-definable groups. In Section \ref{section: characterizations of stability} we give several characterizations for stability and NIP of hyperdefinable sets in various terms, involving continuous logic. Section \ref{section: distal theories} confirms the conjecture stated by Haskell and Pillay in \cite{MR3796277} that in a distal theory $T$ the groups $G^{\textrm{st}}$ and $G^{00}$ coincide. In the same paper, the authors stated as a problem to find a definable group $G$ where $G^{00}=G$ but $G^{\textrm{st}}$ is not an intersection of definable groups. Such an example is presented in Section \ref{section: main example}. However, it is not clear to us how to find an example of a torsion (equivalently, finite exponent) group $G$ with those properties, or just satisfying $G^{00}\neq G^{\textrm{st}}\neq G^{\textrm{st},0}$. In Section \ref{section 5}, we make some observations on this problem, 
describing what should be constructed in order to find such an example. Dropping the requirement that $G$ is a torsion group, we give a  large class of examples where $G^0\neq G^{00}\neq G^{\textrm{st}}\neq G^{\textrm{st},0}$; this does not include an example of finite exponent, as $G$ being of  finite exponent implies that  $G^0 =G^{00}$ by general topological reasons (i.e. compact torsion groups are profinite).

Chapter \ref{Chapter 4} (based fully on \cite{KRUPIŃSKI_PORTILLO_2023}) is dedicated to study the existence of finest relatively type-definable equivalence relations on invariant types with stable quotients. In Section \ref{section: basics}, we prove several basic results concerning the existence of finest relatively type-definable equivalence relations on invariant types with stable quotients, some of which are used in Section \ref{section: main theorem}, and we discuss the transfer of the existence of finest relatively type-definable equivalence relations with stable quotients between models. Section \ref{section: main theorem} contains a proof of Theorem \ref{main theorem intro}, the main theorem of the thesis. In the last section of Chapter \ref{Chapter 4}, we compute $E^{\textrm{st}}$ in two concrete examples which are expansions of local orders. In fact, in these examples, we give full classifications of all relatively type-definable over a small subset of $\C$ equivalence relations on $p(\C')$ for a suitable invariant type $p \in S(\C)$.

Chapter \ref{Chapter 5} is dedicated to the study of generalized indiscernibles in continuous logic and its applications to $n$-dependence. In Section \ref{section: modelling property} we prove that the first-order structure $\I$ has the continuous modeling property if and only if $\age(\I)$ has the embedding Ramsey property.  In Section \ref{section:n-dep}, we use this result to give a characterization of $n$-dependence in continuous logic through the collapse of indiscernible sequences analogous to \cite[Theorem 5.4]{Chernikov2014OnN}. Finally, in Section \ref{section: hyperdef n-dep}, we define and characterize $n$-dependent hyperdefinable sets.

In Chapter \ref{Chapter 6} we use topological dynamics methods to study the maximal WAP and tame quotients of the flow $(\aut(\C),S_X(\C))$, where $\C$ is a monster model of a theory $T$ and $X$ is an $\emptyset$-type-definable set. Section \ref{section: prelim top dyn} contains the necessary known results about topological dynamics needed for the chapter. In Section \ref{section: compatible quotients} we prove Theorem \ref{compatible relations teor intro}, the main result of the chapter. In Section \ref{section: nip stable tame wap} we study the finest closed $\aut(\C)$-invariant equivalence relation with WAP and tame quotients, as well as the finest $\emptyset$-type-definable equivalence relations on $X$ with stable and NIP quotient and show that they fall under the hypothesis of Theorem \ref{compatible relations teor intro}. As a conclusion we get Corollary \ref{intro: compatible fwap and fta}. In the last section of the chapter we compare finest closed $\aut(\C)$-invariant equivalence relations with WAP quotient with the one induced by the finest $\emptyset$-type-definable equivalence relations on $X$ with stable quotient and show that the former is always strictly finer than the latter for any unstable $\emptyset$-type-definable set $X$. Similarly for the tame and NIP case.

The appendix contains proofs that stability and NIP are preserved under taking (possibly infinite) Cartesian products.


\chapter{Background} \label{Chapter 2}
\section{Model theory}
In this section, we recall some model-theoretic basic facts, definitions and conventions. This is not a fully comprehensive introduction, for more in-depth explanations, see e.g. \cite{Tent_Ziegler_2012, Hodges_1993}.

Let $\mL$ be a first-order language and $T$ be a complete first-order theory. By an abuse of notation, we write $\varphi\in \mL$ to indicate that $\varphi$ is an $\mL$-formula.

\begin{defin}
    A \emph{partial $\mL$-type} of $T$ is a consistent relative to $T$ collection of $\mL$-formulas. We say that an $\mL$-type $\Sigma$ is complete if it is maximal with respect to the inclusion.
\end{defin}

\begin{defin}
\begin{itemize}
    \item We say that an $\mL$-formula \emph{$\varphi(x)$ is satisfied by the tuple $a\in\mathcal{M}$} if $\mathcal{M}\models \varphi(a)$, where $\mathcal{M}$ is a model of $T$. We say $\varphi$ is \emph{satisfiable} if there exists some tuple $a$ from some model $\mathcal{M}$ of $T$ satisfying $\varphi$.
       \item  We say that an $\mL$-type \emph{$\Sigma(x)$ is satisfied by the tuple $a\in\mathcal{M}$} if all formulas in $\Sigma$ are satisfied by $a$, where $\mathcal{M}$ is a model of $T$. We say $\Sigma$ is \emph{satisfiable} if there exists some tuple $a$  from some model $\mathcal{M}$ of $T$ satisfying $\Sigma$.
       \item  We say that an $\mL$-type $\Sigma(x)$ is finitely satisfied if for every finite $\Sigma_0\subseteq \Sigma$ there exists some tuple $a$ from some model $\mathcal{M}$ of $T$ satisfying $\Sigma_0$.
\end{itemize}
\end{defin}

\begin{defin}
    Let $\mathcal{M}\models T$ be a model, $A\subseteq \mathcal{M}$ a set and $x$ a possibly infinite tuple of variables. We denote by $S_x(A)$ the space of complete types over $A$ in variables $x$.
\end{defin}
We can endow the space of types with a topology where the basic clopen sets are of the form $$ [\varphi]:=\{ p: p\in S_x(A), \varphi\in p \}$$ for $\varphi$ an $\mL(A)$-formula. This topology is Hausdorff and zero-dimensional. By the next theorem, it is also compact.
\begin{teor}[Compactness theorem]
     Let $\Sigma$ be a partial type. Then $\Sigma$ is satisfiable if and only if it is finitely satisfiable.
\end{teor}

Fix a strong limit cardinal $\kappa$ larger than $\lvert T \rvert$.

\begin{defin}
    A \emph{monster model} is a model $\C\models T$ which is 
    \begin{itemize}
        \item $\kappa$-saturated: Every type over an arbitrary set of parameters from $\C$ of size less than $\kappa$ is realized in $\C$
        \item strongly-$\kappa$-homogeneous: Every elementary map between subsets of $\C$ of cardinality less than $\kappa$ extends to an automorphism of $\C$.
    \end{itemize}
    We call $\kappa$ the \emph{degree of saturation} of $\C$.
    \end{defin}

\begin{fact}
    Monster models exist for every $\kappa$ as above.
\end{fact}

We will consider $\kappa$ to be bigger than the cardinality of all the objects we are interested in. We call an object \emph{small} or \emph{$\C$-small} if its cardinality is smaller than the degree of saturation of $\C$. By $\kappa$-saturation, we may assume that all small models that appear are elementary substructures of the monster model in which we are working. The next fact remarks the utility of working with monster models.

\begin{notation}
Let $A\subseteq \C$ be a small set. We denote by $\equiv_A$ the equivalence relation of having the same type over $A$.
\end{notation}

\begin{fact}
    If $\C$ is a monster model, then for every small $A\subseteq \C$ and any small tuples $a,b$ of elements of $\C$ we have $\tp(a/A)=\tp(b/A)$ if and only if there is $\sigma \in \aut(\C/A)$ such that $\sigma(a)=b$.
\end{fact}

In light of the previous fact, we can also interpret $\equiv_A$ as the equivalence relation of lying in the same $\aut(\C/A)$ orbit.

We define three families of sets which we are going to use extensively.

\begin{defin}
    Let $A\subseteq \C$ be small and let $X$ be a subset of a fixed product of sorts.
    \begin{itemize}
        \item We say that $X$ is \emph{$A$-definable} if there is $\varphi\in \mL(A)$ such that $X=\varphi(\C)$. If we do not specify a set of parameters $A$ and simply say $X$ is definable, we mean that it is $A$-definable for some small $A\subseteq \C$.
        \item We say that $X$ is \emph{$A$-type-definable} if there is a partial type $\Sigma\subseteq \mL(A)$ such that $X=\Sigma(\C)$. If we do not specify a set of parameters $A$ and simply say $X$ is type-definable, we mean that it is $A$-type-definable for some small $A\subseteq \C$.
        \item We say that $X$ is \emph{$A$-invariant} if it is fixed setwise under the action of $\aut(\C/A)$. If we simply say invariant, we mean that the set $X$ is fixed setwise under the action of $\aut(\C)$.
    \end{itemize}
\end{defin}
\begin{remark}
    Let $A\subseteq \C$ be small and let $X$ be a subset of a fixed product of sorts. 
    \begin{itemize}
        \item If a set $X$ is definable and $A$-invariant then it is $A$-definable.
        \item If a set $X$ is type-definable and $A$-invariant then it is $A$-type-definable.
    \end{itemize}
\end{remark}

\begin{defin}Let $p\in S(\mathcal{M})$ where $\mathcal{M}\models T$, and $q\in S(B)$ an extension of $p$ to $B\supseteq \mathcal{M}$.
\begin{itemize}
    \item $q$ is a \emph{coheir} of $p$ if it is finitely satisfiable in $\mathcal{M}$.
    \item $q$ is a \emph{heir} of $p$ if for every $\varphi(x,y)\in \mL(\mathcal{M)}$ such that $\varphi(x,b)\in q$ for some $b\in B$ there is some $m\in \mathcal{M}$ with $\varphi(x,m)\in p$.
\end{itemize}
\end{defin}

\subsection{Indiscernible sequences}
\begin{defin}
    Let $A\subseteq \C$ be small and $(I,\leq)$ be a totally ordered set.
    \begin{itemize}
        \item We say that a sequence $(a_i)_{i\in I}$ of tuples of elements from $\C$ is \emph{(order) indiscernible} over $A$ if for each $n<\omega$ and for all $i_1<\cdots<i_n$, $j_1<\cdots< j_n$ increasing subsequences of $I$ we have $$ a_{i_1}\dots a_{i_n} \equiv_A a_{j_1}\dots a_{j_n} .$$
        \item  We say that a sequence $(a_i)_{i\in I}$  of tuples of elements from $\C$ is \emph{totally indiscernible} over $A$ if for each $n<\omega$ and for all $i_1,\cdots,i_n$, $j_1,\cdots, j_n$ arbitrary elements of $I$ we have $$ a_{i_1}\dots a_{i_n} \equiv_A a_{j_1}\dots a_{j_n} .$$
    \end{itemize}
    In both cases, if we do not specify the set $A$ we mean (totally) indiscernible over $\emptyset$.
\end{defin}

\begin{defin}  Let $A\subseteq \C$ be small, $(I,\leq)$ be a totally ordered set and $\mathbf{I}=(a_i)_{i\in I}$ a sequence of tuples from $\C$. The \emph{Ehrenfeucht-Mostowski type of $\mathbf{I}$ over $A$}, denoted by $\EM(\mathbf{I}/A)$, is the set of all formulas $\varphi(x_1,\dots,x_n)$ such that $\C\models \varphi(a_{i_1},\dots,a_{i_n})$ for all $i_1<\cdots<i_n\in I$, where $n$ ranges over $\omega$.
\end{defin}

The following result shows that indiscernible sequences exist. It follows from the fact that finite linear orders form a Ramsey class (i.e. classical Ramsey theorem) and the compactness theorem (see \cite[Lemma 5.1.3]{Tent_Ziegler_2012} for a detailed proof).

\begin{fact}
    Let $(I,\leq)$ and $(J,\leq)$ be two infinite linear orders and $ (a_i )_{i\in I}$ a sequence of elements of $\C$. Then there an indiscernible sequence $(b_j )_{j\in J}$ realizing the Ehrenfeucht–Mostowski type of $(a_i )_{i\in I}$.
\end{fact}

There are situations when a stronger version is needed. We will refer to the following result as ``extracting indiscernibles''. A proof can be found in \cite[Lemma 1.2]{doi:10.1142/S0219061303000297}.

\begin{fact}\label{extracting indiscernibles}
    Let $A\subseteq \C$ be small. Fix a cardinal $\lambda <\kappa$ and let $\nu \geq \lvert S_{\lambda}(A) \rvert.$ Set $\mu=\beth_{\nu^+}$. Then for any sequence $(a_i)_{i\in \mu}$ of tuples from $\C$ of length $\lambda$ there is an $A$-indiscernible sequence $(b_i)_{i\in \omega}$ such that for all $n<\omega$ there are $i_0<\dots<i_{n-1}\in \mu$ for which $$ \tp(b_0,\dots,b_{n-1}/A)=\tp (a_{i_0},\dots,a_{i_{n-1}}/A ).$$
\end{fact}

To finish the section, we recall two of the most important properties of formulas.

\begin{defin}
We say that a formula $\varphi(x,y)$ is \emph{independent} (or has IP) if there is an indiscernible sequence $(a_i)_{i\in\omega}$ and a tuple $b$ such that $$\models \varphi(a_i,b)\iff \text{i is even}.$$
Otherwise, we say that $\varphi$ is \emph{dependent} or NIP. We say that the $\mL$-theory $T$ is NIP if every $\varphi(x,y)\in \mL$ is NIP.
\end{defin}
\begin{defin}
    We say that a formula $\varphi(x,y)$ is \emph{stable} if there no sequence $(a_i,b_i)_{i\in\omega}$  such that for all $i,j\in \omega$ $$\models \varphi(a_i,b_j)\iff i<j.$$
    We say that the $\mL$-theory $T$ is stable if every $\varphi(x,y)\in \mL$ is stable.
\end{defin}

\subsection{Type-definable group components}
The following families of subgroups have played a crucial role in model theory and have been studied thoroughly, specially the connected components.
\begin{defin}[Connected components] Let $A$ be a small set of parameters and $G$ an $\emptyset$-type-definable group. The following subgroups of $G$ always exist and are known as the (model-theoretic) \emph{connected components} of $G$:
\begin{itemize}
    \item $G^0_A$ is the intersection of all relatively $A$-definable subgroups of $G$ of finite index.
    \item $G^{00}_A$ is the smallest $A$-type-definable subgroup of $G$ of bounded index (i.e. $<\kappa$).
\end{itemize}
\end{defin}
The following is due to Shelah(see \cite{MR2361885, Shelah2007DefinableGF}).
\begin{fact}
    If $T$ is a NIP theory, the subgroups above do not depend on the choice of the set of parameters and are denoted by $G^0$ and $G^{00}$, respectively.
\end{fact}

\begin{defin}[Stable components]
Let $A$ be a small set of parameters and $G$ an $\emptyset$-type-definable group. The following subgroups of $G$ always exist and are known as the \emph{stable components} of $G$:
\begin{itemize}
    \item $G^{\textrm{st},0}_A$  the intersection of all relatively $A$-definable subgroups $H$ of $G$ for which the quotient $G/H$ is stable as in Definition \ref{def stability intro}.
    \item $G^{\textrm{st},00}_A$ or just $G^{\textrm{st}}_A$ is the smallest $A$-type-definable subgroup of $G$ for which the quotient $G/H$ is stable as in Definition \ref{def stability intro}.
\end{itemize}
\end{defin}
The following is due to Haskell and Pillay (see \cite[Theorem 1.1]{MR3796277}).
\begin{fact}
    If $T$ is an NIP theory, the subgroups above do not depend on the choice of the set of parameters and are denoted by $G^{0,st}$ and $G^{\textrm{st}}$, respectively.
\end{fact}

\section{Continuous model theory}

We present some basic definitions and facts about continuous logic, we refer the reader to \cite{MR2657678}  and \cite{MR2723787} for a more detailed exposition.

\begin{defin}
    A \emph{continuous (metric) signature} consists of:
    \begin{itemize}
        \item A collection of function symbols f, together with their arity $n_f<\omega$.
        \item A collection of predicate symbols P, together with their arity $n_P<\omega$.
        \item A binary predicate symbol, denoted by $d$, specified as the distinguished \emph{distance symbol}.
        \item For each $n$-ary symbol $s$ and $i<n$ a continuity modulus $\delta_{s,i}$, called the \emph{uniform continuity modulus of } $s$ with respect to the $i$th argument.
    \end{itemize}
\end{defin}

Given a continuous signature $\mL$, the collection of $\mL$-terms and atomic $\mL$-formulas are constructed as usual. In the continuous context, the quantifiers $\sup_x$ and $\inf_x$ play the roles of $\forall$ and $\exists$, respectively. The issue of connectives is a bit more delicate and we refer the reader to \cite{MR2657678} for an in depth treatment. Depending of the context, we would like to consider all uniformly continuous functions $u:[0,1]^n\to [0,1]$ for all $n<\omega$ as connectives or just some finite subset of such functions.

A \emph{condition} is an expression of the form $\varphi=0$ where $\varphi$ is a formula. Note that expressions of the form $\varphi\geq r$ and $\varphi\leq r$ can be expressed as conditions. 

We now introduce the semantics of continuous logic.

\begin{defin}
    Let $\mL$ be a continuous signature. An $\mL$-structure is a set $M$ equipped with:
    \begin{itemize}
        \item A complete metric $d^M:M^2\to[0,1]$;
        \item A mapping $f^M: M^n\to M$ for every $n$-ary function symbol;
        \item A mapping $P^M: M^n\to [0,1]$ for every $n$-ary predicate symbol
    \end{itemize}
    satisfying the following sets of axioms,
    
    Pseudometric axioms:
    \begin{align*}
       &\sup_x d(x,x)=0 \\
        &\sup_{xy} d(x,y)\dot{-}d(y,x)=0 \\
        &\sup_{xyz} ( d(x,z)\dot{-}d(y,z) )\dot{-} d(x,y)=0
    \end{align*}
    
    Uniform continuity axioms:
    \begin{align*} 
        \sup_{x_{<i},y_{<n-i-1},z,w}( \delta_{f,i}(\varepsilon ) \dot{-} d(z,w) ) \wedge ( d(f(\overline{x},z,\overline{y}), f(\overline{x},w,\overline{y})) \dot{-}\varepsilon )=0\\
        \sup_{x_{<i},y_{<n-i-1},z,w}( \delta_{P,i}(\varepsilon ) \dot{-} d(z,w) ) \wedge ( (P(\overline{x},z,\overline{y})\dot{-} P(\overline{x},w,\overline{y})) \dot{-}\varepsilon )=0
    \end{align*}
\end{defin}

\begin{defin}
    \begin{itemize}
        \item Let $\tau(x_{<n})$ be a term, $\mathcal{M}$ a $\mL$-structure. The \emph{interpretation of $\tau$ in $\mathcal{M}$} is a function $\tau^M:M^n\to M$ defined inductively as:
        \begin{itemize}
            \item If $\tau=x_i$ then $\tau^M(\overline{a})=a_i$.
            \item If $\tau=f(\sigma_0,\dots,\sigma_{m-1})$ then $\tau^M(\overline{a})=f^M(\sigma_0^M(\overline{a}),\dots\sigma_{m-1}^M(\overline{a}))$.
        \end{itemize}
        \item Let $\varphi(x_{<n})$ be a term, $\mathcal{M}$ a $\mL$-structure. The \emph{interpretation of $\varphi$ in $\mathcal{M}$} is a function $\varphi^M:M^n\to [0,1]$ defined inductively as:
         \begin{itemize}
            \item If $\varphi=P(\tau_0,\dots,\tau_{m-1})$ is atomic then $\varphi^M(\overline{a})=P^M(\tau_0^M(\overline{a}),\dots\tau_{m-1}^M(\overline{a}))$.
            \item If $\varphi=\lambda(\psi_0,\dots,\psi_{m-1})$ where $\lambda$ is a continuous logic connective then $\varphi^M(\overline{a})=\lambda(\psi_0^M(\overline{a}),\dots\psi_{m-1}^M(\overline{a}))$.
            \item If $\varphi=\inf_y\psi(x,y)$ then $\varphi^M(\overline{a})=\inf_{b\in M}\psi^M(\overline{a},b)$ and similarly for the supremum.
        \end{itemize}
    \end{itemize}
\end{defin}

A \emph{continuous $\mL$-theory} $T$ is a consistent (i.e. it has a model) set of $\mL$-conditions $\varphi=0$ where $\varphi$ is a sentence. A continuous theory $T$ is complete if its set of logical consequences is maximal with respect to the inclusion.
From now on let $\mL$ be a continuous logic signature and $T$ a complete $\mL$-theory.

\begin{defin}
    An \emph{$\mL$-type} is a consistent relative to $T$ collection of $\mL$-conditionseith a distinguised set of free variables. We say that an $\mL$-type $\Sigma$ is complete if it is maximal with respect to the inclusion.
\end{defin}

\begin{defin}
\begin{itemize}
    \item We say that an $\mL$-condition \emph{$\varphi(x)=0$ is satisfied by the tuple $a\in \mathcal{M}$} if $\mathcal{M}\models \varphi(a)=0$, where $\mathcal{M}$ is a model of $T$. We say $\varphi=0$ is \emph{satisfiable} if there exists some tuple $a$ from some model $\mathcal{M}$ of $T$ satisfying $\varphi=0$.
       \item  We say that an $\mL$-type \emph{$\Sigma(x)$ is satisfied by the tuple $a \in \mathcal{M}$} if all conditions in $\Sigma$ are satisfied by $a$, where $\mathcal{M}$ is a model of $T$. We say $\Sigma$ is \emph{satisfiable} if there exists some tuple $a$ from some model $\mathcal{M}$ of $T$ satisfying $\Sigma$.
       \item  We say that an $\mL$-type $\Sigma(x)$ is finitely satisfiable if for every finite $\Sigma_0\subseteq \Sigma^+$ there exist some tuple $a$ from some model $\mathcal{M}$ of $T$ satisfying $\Sigma_0$, where $$\Sigma^+:=\{ \varphi\leq \frac{1}{n}: n\in\omega, (\varphi=0)\in \Sigma \}  $$
\end{itemize}
\end{defin}

    \begin{defin}
    Let $\mathcal{M}\models T$ be a model and $A\subseteq \mathcal{M}$. A \emph{complete type over $A$} in variables $x$ is a maximal satisfiable set of $\mL(A)$-conditions with free variables contained in $x$. The space of all types in variables $x$ is denoted by $S_{x}(A)$. If $x=(x_1,\dots,x_n)$, we denote $S_{x}$ by $S_{n}(A)$.
    \end{defin}

    \begin{fact}
        The space $S_{x}(A)$ is a compact Hausdorff space when equipped with the finest topology for which all continuous formulas $\varphi\in \mL(A)$ are continuous functions $\varphi:S_{x}(A)\to[0,1]$. In this topology, the sets of the form $$ [\varphi \leq r]:=\{ p: p\in S_{x}(A), (\varphi\leq r) \in p \}$$ where $r$ ranges in $[0,1]$ are the basic closed sets.
    \end{fact}

Fix a strong limit cardinal $\kappa$ larger than $\lvert T \rvert$. As in the first-order case,  monster models exist for every $\kappa$ as above. We will consider $\kappa$ to be bigger than the cardinality of all the objects we are interested in. We call an object \emph{small} if its cardinality is smaller than the degree of saturation of $\C$. By $\kappa$-saturation, we may assume that all small models that appear are elementary substructures of the monster model we are working with.

\begin{defin}
    An \emph{$A$-definable predicate} $f$ in variables $x$ is a continuous function $f:S_{x}(A)\to[0,1]$.
\end{defin}

By \cite[Proposition 3.10]{MR2657678}, when dealing with a finite tuple of variables $x$, this is equivalent to the forced limit construction.

The following was proven in \cite[Proposition 3.4]{MR2657678} for a finite number of variables but the proof also applies for an infinite tuple $x$.

\begin{fact}
    Definable predicates in variables $x$ can be uniformly approximated by continuous logic formulas in variables contained in $x$. That is, for every $\varepsilon>0$ we can find a formula $\varphi(x)$ such that $\lVert \varphi -f\rVert_\infty\leq \varepsilon$.
\end{fact}

This result relates this approach to continuous logic to the formalism from \cite[Subsection 3.1]{hrushovski2021amenability} and \cite[Section 3]{hrushovski2021order} (see Section \ref{section: characterizations of stability} for more details). Using this formalism, we will usually also refer to definable predicates (even in infinite variables) as formulas. We need to allow the domain of definable predicate to be an infinite Cartesian power of $\C$ to deal with hyperdefinable sets $X/E$ for which $X$ is contained in an infinite product of sorts.

We end this section with a short discussion on indiscernible sequences in continuous logic. Indiscernible and totally indiscernible sequences are defined exactly as in first-order logic.


\begin{defin}
Let $A\subseteq \C$ be small, $(I,\leq)$ be a totally ordered set and $\mathbf{I}=(a_i)_{i\in I}$ a sequence of tuples from $\C$. The \emph{Ehrenfeucht-Mostowski type of $\mathbf{I}$ over $A$}, denoted by $\EM(\mathbf{I}/A)$, is the set of all conditions $\varphi(x_1,\dots,x_n)=0$ such that $\C\models \varphi(a_{i_1},\dots,a_{i_n})=0$ for all $i_1<\cdots<i_n\in I$, where $n$ ranges over $\omega$.
\end{defin}

As in the first order case, indiscernible sequences exist.

\begin{fact}
    Let $(I,\leq)$ and $(J,\leq)$ be two infinite linear orders and $ (a_i )_{i\in I}$ a sequence of elements of $\C$. Then there an indiscernible sequence $(b_j )_{j\in J}$ realizing the Ehrenfeucht–Mostowski type of $(a_i )_{i\in I}$.
\end{fact}



\section{Hyperdefinable sets}
We dedicate this section to basic facts about hyperdefinable sets. Hyperimaginaries and their types were first introduced in \cite{coordinatisation}. For a more in-depth exposition see \cite[Chapters 15 and 16]{casanovas2011simple} and \cite[Chapter 3]{Wagner2002-WAGST-2}.

Let $T$ be a complete, first order theory, and $\C \models T$ a monster model. In this section 
we consider $\emptyset$-type-definable equivalence relations defined on $\emptyset$-type-definable subsets of $\C^\lambda$ (or a product of sorts), where $\lambda< \kappa$ (where $\kappa$ is the degree of saturation of $\C$). 

\begin{defin}
    A {\em hyperdefinable set} $X/E$ is a quotient of a type-definable set by a type-definable equivalence relation. If both $X$ and $E$ are $A$-type-definable we say that $X/E$ is hyperdefinable over $A$. A {\em hyperimaginary} is an element of a hyperdefinable set.
\end{defin}

Whenever the equivalence relation $E$ plays an important role, we write $a/E$ or $[a]_E$ to emphasize it. Sometimes, the domain of the equivalence relation is not very important and we assume that $E$ is an equivalence relation in the full product of sorts. The following result allows us to consider $X/E$ as simply some type-definable subset of $\C^\lambda/E$ for the appropriate $\lambda$. 

\begin{remark}
    Let $\Sigma(x)$ be a partial type over $A$. If $E$ is an $A$-type-definable equivalence relation on the set of realizations of $\Sigma$, then the $A$-type-definable equivalence relation $E'(x,y):=(\Sigma(x)\wedge \Sigma(y) \wedge E(x,y))\vee (x=y)$ is defined for all sequences of length $\lvert x\rvert$ and agrees with $E$ on $\Sigma(\C)$.
\end{remark}

We now introduce complete types of hyperimaginary elements.

\begin{defin}
    Let $A \subset \C$ be small. The \emph{complete types over $A$ of elements of $X/E$} can be defined as the $\aut(\C/A)$-orbits on $X/E$,  or the preimages of these orbits under the quotient map, or the partial types defining these preimages. The space of all such types is denoted by $S_{X/E}(A)$. This space is naturally a quotient of $S_X(A)$ and the topology on $S_{X/E}(A)$ is the quotient topology (See Remark \ref{2.7} for more details).
\end{defin}

We can also define the complete type of a hyperimaginary element syntactically.

\begin{defin}
Let $a/E$ and $b/F$ be hyperimaginaries. For each formula $\varphi(x,y)\in \mL$ let $$\Delta_{\varphi}(x,y):= \exists x',y'( E(x,x') \wedge F(y,y') \wedge \varphi (x',y') ).$$    
We define $\tp([a]_E /[b]_F)$ as the union of all partial types $\Delta_{\varphi}(x,b)$ such that $\models \varphi(a,b)$.
\end{defin}

We say that an automorphism $\sigma\in \aut(\C)$ \emph{fixes} a hyperimaginary $a/E$ if $\sigma(a/E)=a/E$, that is, $\models E(a,\sigma(a))$. As in the first section of this chapter, one of the advantages of working inside a monster model is the following:

\begin{notation}
Let $A\subseteq \C$ be a small set (possibly containing hyperrimaginaries). We denote by $\equiv_A$ the equivalence relation of having the same type over $A$.
\end{notation}

\begin{fact}
    Let $d$ be a hyperimaginary. If $\tp([a]_E/d)=tp([b]_E/d)$ then there is $\sigma\in \aut(\C/d)$ such that $\models E(\sigma(a),b)$.
\end{fact}

We say that $a/E$ is a \emph{countable hyperimaginary} if $a$ is a tuple of countable length. The following result allows us to consider sequences of hyperimaginaries as a single hyperimaginary and vice versa.

\begin{defin}
    Let $a$ and $b$ be (possibly infinite) tuples of hyperimaginaries. We say that $a$ and $b$ are interdefinable if any automorphism fixing $a$ fixes $b$ and vice versa.
\end{defin}

The following result can be found in \cite[Lemmas 15.3 and 15.4]{casanovas2011simple}:
\begin{fact}\label{hyperim as count hyperim}
    Any hyperimaginary is interdefinable with a sequence of countable hyperimaginaries. Any sequence of hyperimaginaries is interdefinable with a hyperimaginary.
\end{fact}

We present indiscernibility of hyperimaginary elements, which will be used extensively throughout the thesis.

\begin{defin}
    Let $A$ be a small set (possibly containing hyperimaginaries) and $(I,\leq)$ be a totally ordered set. We say that a sequence $(a_i)_{i\in I}$ of hyperimaginaries is \emph{(order) indiscernible} over $A$ if for each $n<\omega$ and for all $i_1<\cdots<i_n$, $j_1<\cdots< j_n$ increasing sequences of $I$ we have $$ a_{i_1}\dots a_{i_n} \equiv_A a_{j_1}\dots a_{j_n} .$$ 
\end{defin}
Note that indiscernibility implies that for each $i\in I$ the hyperimaginary $a_i$ is of the form $a'_i/E$ for a fixed type-definable equivalence relation $E$ and $a_i'$ a tuple of real elements.

A proof of the next two facts can be found in \cite[Lemma 16.2]{casanovas2011simple} and \cite[Proposition 16.3]{casanovas2011simple} respectively.

\begin{fact}
    Let $d$ be a hyperimaginary, and let $(I,\leq)$ and $(J,\leq)$ be two infinite linear orders. If $ (a_i )_{i\in I}$ is a indiscernible over $d$ sequence of hyperimaginaries, then there is a indiscernible over $d$ sequence $(b_j )_{j\in J}$ such that for each $n<\omega$ and for all increasing sequences $i_1<\cdots<i_n\in I$, $j_1<\cdots< j_n\in J$ we have $$ a_{i_1}\dots a_{i_n} \equiv_d b_{j_1}\dots b_{j_n} .$$
\end{fact}

\begin{fact}\label{fact: indisc representatives}
    If $(a_i)_{i\in I}$ is a sequence of hyperimaginaries indiscernible over a hyperimaginary $b$, then there is a sequence $(a'_i)_{i\in I}$ consisting of representatives of elements of the sequence $(a_i)_{i\in I}$  and a representative and $b'$ of $b$ such that $(a'_i)_{i\in I}$ is indiscernible over $b'$.
\end{fact}
\subsection{Hyperimaginaries as continuous logic imaginaries}\label{section: hyperim as cl im}
We make explicit the connection between hyperimaginaries and continuous logic imaginaries mentioned in the introduction. 

Let $T$ be a first-order theory, and let $T'$ be its continuous logic counterpart (i.e., any model of $T'$  is a model of $T$ with the discrete metric). Recall that, by Fact \ref{hyperim as count hyperim}, a hyperimaginary is always interdefinable with some sequence of countable hyperimaginaries. Moreover, any hyperdefinable set $X/E$ can be identified with the diagonal of some product of hyperdefinable sets $\prod_{i\in I} X/E_i$, where each $E_i$ is an equivalence relation on $X$ relatively type-definable by a countable type $\pi_i(x,y)$.

Each of the hyperdefinable sets $X/E_i$ of the product $\prod_{i\in I} X/E_i$ can be interpreted as a type-definable set of a continuous logic product of imaginary sorts by the following fact (see \cite[Lemma 3.4.4]{hanson2020thesis}):
\begin{fact}
    Let $x$ and $y$ be countable tuples of variables. If $E(x,y)$ is an equivalence relation defined by a countable type $\pi(x,y)$, then there is a continuous logic formula $\rho(x,y)$ for which $E(x,y)\cong_{T'} (\rho(x,y)=0)$ and moreover $\rho$ defines a pseudo-metric in all models $\mathcal{M}\models T$.
\end{fact}

Therefore, by combining these facts, the hyperdefinable set $X/E$ can be interpreted as a type-definable set of tuples (maybe of infinite length) of continuous logic imaginary elements. Note that, as in \cite{hanson2020thesis}, we allow quotients of a countable product of sorts as continuous logic imaginaries.
 \section{The embedding Ramsey property}
Let $\mL'$ be a first-order language. Given $\mL'$-structures $A\subseteq B$, we write $\binom{B}{A}$ for the set of all embeddings from $A$ into $B$.   
\begin{defin}\label{defin: arrow notation}
    Let $A\subseteq B \subseteq C$ be $\mL'$-structures and let $r\in\mathbb{N}$. We write $$ C\to (B)^A_r$$ if for each coloring $\chi: \binom{C}{A}\to r$ there exists some $f\in \binom{C}{B}$ such that $\chi\upharpoonright_{f\circ \binom{B}{A}}$ is constant.
\end{defin}

\begin{defin}\label{defin ERP}
    Let $\mL'$ be a first-order language and $\mathcal{C}$ be a class of finite $\mL'$-structures. We say that $\mathcal{C}$ has the \emph{embedding Ramsey property}, \emph{ERP} for short, if for every $A\subseteq B \in \mathcal{C}$ and $r<\omega$ there is $C\in \mathcal{C}$ such that $C\to (B)^A_r$.
\end{defin}

Note that the following holds (for a reference, see \cite[Lemma 2.9, Corollary 2.10]{Bodirsky_2015}):
\begin{fact}
    If a class of finite $\mL'$-structures has $ERP$, then all the structures of $\mathcal{C}$ are rigid (i.e. they have no nontrivial automorphisms).
\end{fact}

\begin{remark}
    Sometimes, the symbol $\binom{B}{A}$ is used to denote the set of all isomorphic copies of $A$ in $B$. If the class $\mathcal{C}$ consists of finite $\mL'$-structures which are rigid, then coloring embeddings from $A$ into $B$ is equivalent to coloring substructures $A\subseteq B$. 
\end{remark}
\begin{defin}\label{defin: locally finite and age}
 Given a $\mL'$-structure $\mathcal{M}$ and a subset $A\subset \mathcal{M}$
\begin{itemize}
    \item We denote by $\langle A \rangle_\mathcal{M}$ the substructure of $\mathcal{M}$ generated by $A$.
    \item We say that $\mathcal{M}$ is \emph{locally finite} if for any finite $A\subseteq \mathcal{M}$ the substructre $\langle A \rangle_\mathcal{M}$ is finite.
    \item The \emph{age} of $\mathcal{M}$, denoted by $\age(\mathcal{M})$, is the class of isomorphism types of finitely generated substructures of $\mathcal{M}$
\end{itemize}
\end{defin}

\section{Topological dynamics}
We introduce some basic definitions and state some facts about topological dynamics. For a more in depth study of the topic see e.g. \cite{auslander1988minimal} and \cite{Glasner1976}.
\begin{defin}\label{defin: ellis semigroup}
\begin{itemize}
    \item A \emph{$G$-flow} is a pair $(G,X)$ consisting of a topological group $G$ that acts continuously on a compact Hausdorff space $X$.
    \item If $(G,X)$ is a $G$-flow, then its \emph{Ellis semigroup} $E(X)$ is the pointwise closure in $X^X$ of the set of functions $\pi_g: x \to g\cdot x$ for $g\in G$.
\end{itemize}
\end{defin}

\begin{fact}
    The Ellis semigroup of a $G$-flow $(G,X)$ is a compact left topological semigroup with composition as its semigroup operation. Moreover, $E(X)$ is itself a $G$-flow equipped with the action $g\eta:=\pi_g \circ \eta$ for $g\in G$ and $\eta \in E(X)$.
\end{fact}

Recall that a left ideal $I$ of a semigroup $S$, written as $I\unlhd S$, is a set such that $SI\subseteq I$.

The following is due to Ellis:
\begin{fact}\label{Ellis theorem}
 Minimal left ideals of $E(X)$ exist and coincide with the minimal subflows of $(G,E(X))$. If $\mathcal{M}\unlhd E(X)$ is a minimal left ideal then:
\begin{itemize}
    \item The ideal $\mathcal{M}$ is closed and for every $u\in \mathcal{M}$ we have $\mathcal{M}=E(X)u$.
    \item The set of idempotents of $\mathcal{M}$, $\mathcal{J}(\mathcal{M})\subseteq \mathcal{M}$, is nonempty. Moreover, $\mathcal{M}=\bigsqcup_{u\in\mathcal{J}(\mathcal{M})} u\mathcal{M}$.
    \item For every $u\in\mathcal{J}(\mathcal{M})$, $u\mathcal{M}$ is a group with identity element $u$. Moreover, the isomorphism type of this group does not depend on the choice of $u$ and $\mathcal{M}$. It is called the $\emph{Ellis group}$ of $X$ (We note that this is the convention of model theorists, mathematicians working on topological dynamics usually mean something different by "Ellis group").
    \item For every $u\in\mathcal{J}(\mathcal{M})$ and $s\in \mathcal{M}$, $su=s$.
    \item For every minimal left ideal $\mathcal{N}\unlhd E(X)$ and for every $u\in \mathcal{J}(\mathcal{M})$, $v\in \mathcal{J}(\mathcal{N})$ we have $u\mathcal{M}\cong v\mathcal{N}$.
\end{itemize} 
\end{fact}
    
The Ellis group has an inherited topology from $E(X)$. However, there exists another important topology on $E(X)$, usually called the \emph{$\tau$-topology}. We define it now. First, for any $a\in E(X)   $ and $B\subseteq E(X)$ let $a\circ B$ be the set of all limits of the nets $(g_ib_i)_{i\in \I}$ such that $g_i\in G$, $b_i\in B$ and $\lim g_i=a$. We define a closure operator $cl_\tau$ given by $cl_\tau(B)=u\mathcal{M}\cap (u\circ B)$ where $B\subseteq u\mathcal{M}$. The $\tau$-topology is the topology induced on $u\mathcal{M}$ by the closure operator $cl_\tau$.

The following fact is \cite[Proposition 5.41]{rzepecki2018bounded}.

\begin{fact}\label{epimorphism semigroup to group}
    Let $(G,X)$ and $(G,Y)$ be two $G$-flows, and let $\Phi: X\to Y$ be a $G$-flow epimorphism. Then $\Phi_*:E(X)\to E(Y)$ given by $$\Phi_*(\eta)(\Phi(x)):=\Phi(\eta(x))$$ is a continuous epimorphism.  
    
    If $\mathcal{M}$ is a minimal left ideal of $E(X)$ and $u\in \mathcal{J}(\mathcal{M})$ then:
    \begin{itemize}
        \item $\mathcal{M}':=\Phi_*[\mathcal{M}]$ is a minimal left ideal of $E(Y)$ and $u'=\Phi_*(u)\in \mathcal{J}(\mathcal{M}')$.
        \item $\Phi_*\!\!\upharpoonright_{u\mathcal{M}}: u\mathcal{M}\to u'\mathcal{M}'$ is a group epimorphism and a quotient map in the $\tau$-topologies.
    \end{itemize}
    Moreover, if $\Phi: X\to Y$ is a $G$-flow isomorphism then $\Phi_*\!\!\upharpoonright_{u\mathcal{M}}: u\mathcal{M}\to \Phi(u)\Phi[\mathcal{M}]$ is a group isomorphism and a homeomorphism in the $\tau$-topologies.
\end{fact}

From Ellis' theorem we easily deduce the following:
\begin{remark}\label{restrictio to image is monomorphism}
    If $X$ is a $G$-flow, $\mathcal{M}$ a minimal left ideal in $E(X)$, and $u \in \mathcal{M}$ an idempotent, then the map $f \colon u\mathcal{M} \to \Sym(\Image(u))$ given by $f(\eta):=\eta\!\!\upharpoonright_{\Image(u)}$ is a group monomorphism.
\end{remark}

We now recall the definition of \emph{content}. This definition was originally introduced in \cite[Definition 3.1]{Krupiski2017BoundednessAA}.

\begin{defin}\label{defin: content} Fix $A\subseteq B$.
\begin{itemize}
    \item For $p(x)\in S(B)$, the \emph{content} of $p$ over $A$ is the following set:
    $$ c_A(p):= \{ (\varphi(x,y),q(y))\in \mL(A)\times S_y(A) : \varphi(x,b)\in p(x) \text{ for some } b\models q  \} .$$
    \item Similarly, the content of a sequence $p_0(x),\dots,p_n(x)\in S(B)$ over $A$, $c_A( p_0,\dots,p_n )$, is defined as the set of all $ (\varphi_0(x,y),\dots ,\varphi_n(x,y),q(y))\in \mL(A)^n\times S_y(A)$ such that for some $b\models q$ and for every $i\leq n$ we have $\varphi_i(x,b)\in p_i$.
\end{itemize}
    If $A=\emptyset$ we simply omit it.
\end{defin}
The fundamental connection between contents and the Ellis semigroup of the flow $(\aut(\C),S_\pi(\C))$ is the following.
\begin{fact}\label{content and ellis semigroup}
    Let $\pi(x)$ be a partial type over $\emptyset$, $S_\pi(\C)$ the set of complete types over $\C$ extending $\pi$, and $(p_0,\dots,p_n)$ and $(q_0,\dots,q_n)$ sequences from $S_{\pi}(\C)$. Then $c(q_0,\dots,q_n)\subseteq c(p_0,\dots, p_n)$ if and only if there exists $\eta\in E(S_\pi(\C))$ such that $\eta(p_i)=q_i$ for every $i\leq n$.
\end{fact}
The proof of this fact can be found in \cite[Proposition 3.5]{Krupiski2017BoundednessAA}.

\begin{defin}\label{defin: content order}
    Let $\pi(x)$ be a partial type over $\emptyset$ and $(p_0,\dots,p_n)$ and $(q_0,\dots,q_n)$ sequences from $S_{\pi}(\C)$. We write $(q_0,\dots,q_n)\leq^c (p_0,\dots, p_n)$ if $c(q_0,\dots,q_n)\subseteq c(p_0,\dots, p_n)$.
\end{defin}

For the rest of the section we fix a $G$-flow $(G,X)$. Let $C(X)$ denote the space of all continuous real-valued maps on $X$. Given $f\in C(X)$ and $g\in G$, we denote $gf:=f(g^{-1} x)$.

We recall two important classes of flows: \emph{weakly almost periodic flows} and \emph{tame flows}. For a more in depth treatment of the topic we recommend \cite{55bad191-0781-3b5e-a214-06bd2e8fc797} for weakly almost periodic flows and  \cite{Glasner2018} for tame ones.

Recall that the \emph{weak topology} on $C(X)$ is defined as the coarsest topology such that for all $\ell$ continuous (with respect to the original topology) linear functional $\ell:C(X)\to \R$, the map $$\ell:C(X)\to \R$$
is continuous.

\begin{defin}
    We say that a function $f\in C(X)$ is \emph{weakly almost periodic} (WAP) if $(gf:g\in G)$ is relatively compact in the weak topology on $C(X)$. A flow $(G,X)$ is WAP if every $f\in C(X)$ is WAP.
\end{defin}

The following fact is due to Grothendieck \cite{f935e45f-f66d-33ab-9bf8-2d48a5e00aee}.
\begin{fact} \label{double limit}
    Let $X_0$ be any dense subset of $X$. Let $f\in C(X)$. The following are equivalent:
    \begin{itemize}
        \item $f$ is WAP.
        \item  $(gf:g\in G)$ is relatively compact in the topology of pointwise convergence on $C(X)$.
        \item For any sequences $(g_nf)_{n<\omega}\subset (gf:g\in G)$ and $(x_n)_{n<\omega}\subset X_0$ we have $$\lim_n\lim_m g_nf(x_m)=\lim_m\lim_n g_nf(x_m) $$ whenever both limits exits.
    \end{itemize}
\end{fact}

The next two facts will be useful throughout Chapter \ref{Chapter 6}:
\begin{fact}\label{wap closed unital}
    For any flow $(G,X)$, the WAP functions form a closed (with the supremum norm) unital subalgebra of $C(X)$. 
\end{fact}

 Details about this result can be found in \cite{ibarlucia2016dynamical} after Fact 2.1. Combining it with the Stone-Weierstrass theorem, we obtain the second fact:

\begin{fact} \label{separating points: WAP}
    If $\mathcal{A}\subseteq C(X)$ is a family of functions that separate points, then $(G,X)$ is WAP if and only if every $f\in \mathcal{A}$ is WAP.
\end{fact}

\begin{defin}\label{defin: indep functons}
    We say that a sequence of functions $(f_n)_{n<\omega} \in C(X)$ is \emph{independent} if there are $r<s\in \R$ such that $$ \bigcap_{n\in P}f_n^{-1}(-\infty,r) \cap \bigcap_{n\in M}f_n^{-1}(s,\infty)\neq \emptyset$$ for all finite disjoint $P,M\subset\omega$. Given a dense $X_0\subseteq X$, we can equivalently require $$ \bigcap_{n\in P}f_n^{-1}(-\infty,r) \cap \bigcap_{n\in M}f_n^{-1}(s,\infty)\cap X_0\neq \emptyset$$ for all finite disjoint $P,M\subset\omega$.
\end{defin}

\begin{defin}
    Let $\{ f_n: X\to \R \}_{n\in \mathbb{N}}$ be a uniformly bounded sequence of functions. We say that this sequence is an $\ell_1$-sequence on $X$ if there exists a real constant $a>0$ such that for all $n\in \mathbb{N}$ and choices of real numbers $c_1,\dots,c_n$ we have $$ a\cdot \sum_{i=1}^n \lvert c_i\lvert \leq \left\lVert \sum_{i=1}^n c_i f_i \right\rVert_\infty$$
\end{defin}

The following equivalences can be found in \cite[Theorem 2.4]{Glasner2018}

\begin{fact}\label{l_1 and indep seq}
    The following are equivalent for a bounded $F\subseteq C(X)$:
\begin{itemize}
    \item $F$ does not contain an independent sequence.
    \item  $F$ does not contain an $\ell_1$-sequence.
    \item Each sequence in $F$ has a pointwise convergent subsequence in $\R^X$.
\end{itemize}
\end{fact}

\begin{defin}
    We say that a function $f\in C(X)$ is \emph{tame} if $(gf:g\in G)$ does not contain an independent subsequence. A flow $(G,X)$ is tame if every $f\in C(X)$ is tame.
\end{defin}

The next two facts will be useful throughout Chapter \ref{Chapter 6}:

\begin{fact}\label{tame closed unital}
    For any flow $(G,X)$, the tame functions form a closed (with the supremum norm) unital subalgebra of $C(X)$. 
\end{fact}
A proof can be found in \cite[Fact 2.72]{rzepecki2018bounded}.

From the previous fact and Stone-Weierstrass theorem we obtain the following:

\begin{fact} \label{separating points: tame}
    If $\mathcal{A}\subseteq C(X)$ is a family of functions that separate points, then $(G,X)$ is tame if and only if every $f\in \mathcal{A}$ is tame.
\end{fact}

Let $T$ be a complete, first-order theory, and $\C \models T$ a monster model. Let $E$ be a $\emptyset$-type-definable equivalence relation on a $\emptyset$-type-definable subset $X$ of $\C^\lambda$ (or a product of sorts), where $\lambda< \kappa$ (from the definition of $\C$). We can define a closed $\aut(\C)$-invariant equivalence relation $\Tilde{E}$ on $S_X(\C)$ given by $$ p\Tilde{E}q\iff \exists a\models p, b\models q \left( aEb \right),$$ or equivalently, $$ p\Tilde{E}q\iff \exists a\models p, b\models q \left( \tp(\bslant{a/E}{\C})=\tp(\bslant{b/E}{\C})  \right) .$$ The equivalence relation $\Tilde{E}$ satisfies that $S_{X/E}(\C)\cong S_X(\C)/\Tilde{E}$ as $\aut(\C)$-flows. Note that $a$ and $b$ above are from a bigger monster model, so by $E$ we mean the interpretation of $E$ on this bigger monster model.

\chapter{Maximal stable quotients of type-definable groups in NIP theories}\label{Chapter 3}


\section{Characterizations of stability of hyperdefinable sets}\label{section: characterizations of stability}


Let $T$ be a complete, first-order theory, and $\C \models T$ a monster model. Let $E$ be a $\emptyset$-type-definable equivalence relation on a $\emptyset$-type-definable subset $X$ of $\C^\lambda$ (or a product of sorts), where $\lambda< \kappa$ (from the definition of $\C$). Recall that $E$ is said to be {\em bounded} if $|X/E| < \kappa$.

In this section, we give some characterizations of stability of the hyperdefinable set $X/E$, analogous to classical characterizations of stability of a first-order theory. This involves continuous logic (CL). Assuming NIP, we also give a characterization using generically stable types, which we introduce in the context of $X/E$ (which in fact could be also done in a general CL context). 

Since finite models are trivially stable, we will assume that $T$ has infinite models.

It is folklore that $E$ yields a pseudometric (or a set of pseudometrics) on $X$ (see \cite{2010,conant2021separation} and \cite[Chapter 3]{hanson2020thesis} in the language of continuous logic and \cite{10.2178/jsl/1122038916} in the language of CATs), which in turn leads to a presentation of $X/E$ as a type-definable set of imaginaries in the sense of continuous logic as presented in Section \ref{section: hyperim as cl im}.
Note that in this translation hyperdefinable sets do not translate to continuous logic definable sets. However, for our purposes, it is more convenient to look at the connection with continuous logic in a different way.

We will focus on the first-order theory $T$ and treat it as a continuous logic theory, as the aim of this thesis is to talk about $X/E$ rather than develop continuous logic in general. We will use the formalism from \cite[Subsection 3.1]{hrushovski2021amenability} and \cite[Section 3]{hrushovski2021order} and adapt some results from \cite{MR2657678}. In particular, by a {\em CL-formula over $A$} we mean a continuous function $\varphi \colon S_n(A)\to \mathbb{R}$.  If $\varphi$ is such a CL-formula, then for any $\bar b\in M^n$ (where $M \models T$) by $\varphi(\bar b)$ we mean $\varphi(\tp(\bar b/A))$; note that the range of every CL-formula is compact. So a CL-formula can be thought of as a function from $\C^n$ to $\mathbb{R}$ which factors through $S_n(A)$ via a continuous map $S_n(A) \to \mathbb{R}$.
What are called {\em definable predicates}, in finitely many variables and without parameters,  in \cite{MR2657678} are precisely CL-formulas over $\emptyset$, but where the range is contained in $[0,1]$. In any case, a CL-formula can be added as a new CL-predicate and then it becomes a legitimate formula in the sense of continuous logic. It is not so if we allow the domain of a CL-formula to be an infinite Cartesian power of $\C$ (which is necessary to deal with $X/E$ in the case when $\lambda$ is infinite), but still the results from \cite{MR2657678} 
which we will be using are valid for such generalized continuous logic formulas.

Let $M$ be a model, and $\varphi(x,y)$ a CL-formula over $M$. Let $a\in \C^{|x|}$. 
Then $\tp_{\varphi}(a/M)$ is the function taking $b\in M^{|y|}$ (or $\varphi(x,b)$) to $\varphi(a,b)$, and is called a {\em complete $\varphi(x,y)$-type over $M$}. The space of all complete $\varphi$-types over $M$ is denoted by $S_\varphi(M)$ (it is naturally a quotient of $S(M)$, and the topology on $S_\varphi(M)$ is the quotient topology).
The type $\tp_{\varphi}(a/M)$ is {\em definable} if it is the restriction to $M^{|y|}$ of a CL-formula $\psi(y)$ over $M$, i.e. $\varphi(a,b)=\psi(b)$ for $b \in M^{|y|}$.





From now on, let $\mathcal{F}_{X/E}$ be the family of all functions $f : X \times \mathfrak{C}^m \to \mathbb{R}$ which factor through $X/E\times\mathfrak{C}^m$ and can be extended to a CL-formula $\C^{\lambda} \times \C^m \to \mathbb{R}$ over $\emptyset$, where $m$ ranges over $\omega$. 
(Note that, by Tietze's extension theorem, a function $f : X \times \C^m \to \mathbb{R}$ extends to a CL-formula over $\emptyset$ if and only if it factors through the type space $S_{X \times \C^m}(\emptyset)$ via a continuous function $S_{X \times \C^m}(\emptyset) \to \mathbb{R}$.)
For $f\in \mathcal{F}_{X/E}$, a {\em complete $f$-type over $M$} is the function taking  $f(x,b)$ (for $b\in M^{|y|}$) to $f(a,b)$ for some fixed $a \in X$, and is denoted by $\tp_f(a/M)$. We get the space $S_f(M)$ of all complete $f$-types over $M$. A {\em complete $\mathcal{F}_{X/E}$-type over $M$} is 
the union $\bigcup_{f \in \mathcal{F}_{X/E}} \tp_f(a/M)$ for some $a \in X$, and  $S_{\mathcal{F}_{X/E}}(M)$ is the space of all complete  $\mathcal{F}_{X/E}$-types over $M$. The definition of $\tp_f(a/M)$ being definable is the same as in the previous paragraph; a type in $S_{\mathcal{F}_{X/E}}(M)$ is {\em definable} if its restriction to any $f \in \mathcal{F}_{X/E}$ is  definable.

Let $A \subset \C$ (be small). Recall that the complete types over $A$ of elements of $X/E$ can be defined as the $\aut(\C/A)$-orbits on $X/E$,  or the preimages of these orbits under the quotient map, or the partial types defining these preimages, or the classes of the equivalence relation on $S_X(A)$ given in the proof of Remark \ref{2.7}. The space of all such types is denoted by $S_{X/E}(A)$.

\begin{prop}\label{proposition: F_X/E separates points}
	For any $a_1=a'_1/E$, $a_2=a'_2/E$ in $X/E$ and $b_1,b_2\in \mathfrak{C}^m$ 
	$$\tp(a_1,b_1)\neq \tp(a_2,b_2)\iff (\exists f\in \mathcal{F}_{X/E})(f(a'_1,b_1)\neq f(a'_2,b_2)) $$
\end{prop}

\begin{proof}
	Let us define an equivalence relation $E'$ on $X\times \mathfrak{C}^{m}$ by $$ (x_1,y_1) {E'} (x_2,y_2)\iff (x_1/E,y_1)\equiv(x_2/E,y_2). $$
	Note that $E'$ is a $\emptyset$-type-definable, bounded equivalence relation.
	
	($\Leftarrow$) Assume $r_1:=f(a_1',b_1)\neq f(a'_2,b_2)=:r_2$ for some $f\in \mathcal{F}_{X/E}$. Since the sets $f^{-1}(r_1)$ and $f^{-1}(r_2)$ are  $\emptyset$-type-definable and they are unions of ($E \times \{=\}$)-classes, they are unions of $E'$-classes. But they are also disjoint. Hence, $ (a_1',b_1)$ is not $E'$-related to $(a_2',b_2)$, i.e. $\tp(a_1,b_1) \ne \tp(a_2,b_2)$. 
	
	($\Rightarrow$)  Since $E'$ is $\emptyset$-type-definable and bounded, $\bigslant{(X\times \mathfrak{C}^{m})}{E'}$ 
	is a compact  (Hausdorff) topological space (with the {\em logic topology}, in which closed sets are those whose preimages by the quotient map are type-definable). Since we assume that $\tp(a_1,b_1) \ne \tp(a_2,b_2)$, we have $[(a_1',b_1)]_{E'} \ne [(a_2',b_2)]_{E'}$ in $\bigslant{(X\times \mathfrak{C}^{m})}{E'}$. The space $\bigslant{(X\times \mathfrak{C}^{m})}{E'}$ is $T_{3+\frac{1}{2}}$, so the above two distinct points can be separated by a continuous function $$ h:\bigslant{(X\times \mathfrak{C}^{m})}{E'}\to \mathbb{R} $$ such that $h([(a_1',b_1)]_{E'})=0$ and $h([(a_2',b_2)]_{E'})=1$. Let $\pi_{E'} : X \times \C^m \to  \bigslant{(X\times \mathfrak{C}^{m})}{E'}$ be the quotient map. We conclude that the function $$f:=h\circ \pi_{E'}: X\times \mathfrak{C}^{m} \to \mathbb{R}  $$ satisfies the required conditions.
\end{proof}

We say that $f \in \mathcal{F}_{X/E}$ is {\em stable} if for all $\varepsilon > 0$ there do not exist $a_{i}, b_{i}$ for $i<\omega$ with $a_i \in X$ for each $i$, such that for all $i<j$, $|f(a_{i},b_{j}) - f(a_{j},b_{i})| \geq \varepsilon$ (see \cite[Definition 7.1]{MR2657678} and \cite[Definition 3.8]{hrushovski2021amenability}). By Ramsey's theorem and compactness, $f$ is stable if and only if whenever $(a_{i},b_{i})_{i<\omega}$ is indiscernible (with $a_i \in X$), then $f(a_{i},b_{j}) = f(a_{j},b_{i})$ for all (some) $i< j$. 

\begin{cor}\label{2.4}
	$X/E$ is stable as a hyperdefinable set if and only if every  $f\in \mathcal{F}_{X/E}$ is stable.
\end{cor}
\begin{proof}
	
	
	($\Rightarrow$) Suppose that there is an unstable $f\in \mathcal{F}_{X/E}$. Then there is an indiscernible sequence $(a_i,b_i)_{i<\omega}$ with $a_i \in X$ such that $$f(a_i,b_j)\neq f(a_j,b_i)$$ for all $i<j$.
	Hence, by Proposition \ref{proposition: F_X/E separates points}, $\tp(a_i/E,b_j)\neq \tp(a_j/E,b_i)$ for all $i<j$. Since the sequence $(a_i/E,b_i)_{i<\omega}$ is indiscernible, we conclude that $X/E$ is not stable.
	
	($\Leftarrow$) Suppose that $X/E$ is not stable. Then, there is an indiscernible sequence $(a_i/E,b_i)_{i<\omega}$ with $a_i \in X$ such that $$\tp(a_i/E,b_j)\neq \tp(a_j/E,b_i).$$
	for all $i<j$.
	By Fact \ref{fact: indisc representatives}, we can assume that the sequence  $(a_i,b_i)_{i<\omega}$ is indiscernible.
	
	By Proposition \ref{proposition: F_X/E separates points} and indiscernibility of the sequence $(a_i,b_i)_{i<\omega}$, we conclude that there is $f\in \mathcal{F}_{X/E}$ such that $f(a_i,b_j)\neq f(a_j,b_i)$ for all $i<j$. Hence, $f$ is not stable.
\end{proof}

The next result follows from \cite[Proposition 7.7]{MR2657678} and its proof. However, one should be a bit careful here. In the case when $\lambda$ is finite and $X=\C^\lambda$, one just applies \cite[Proposition 7.7]{MR2657678}, but in general one should say that the proof of \cite[Proposition 7.7]{MR2657678} goes through working with $f \in \mathcal{F}_{X/E}$ in place of a legitimate continuous logic formula $\varphi$. Also, since we are working in the first-order theory $T$ treated as a continuous logic theory, models are discrete spaces and the density characters 
of models are just cardinalities. 
The density character of $S_f(M)$ (denoted by $||S_f(M)||$) is computed with respect to a certain metric on $S_f(M)$ defined after Definition 6.1 in \cite{MR2657678}.

\begin{fact}[\cite{MR2657678}, Proposition 7.7]\label{equiv1}
	Let $f\in \mathcal{F}_{X/E}$. The following conditions are equivalent:
	\begin{enumerate}
		\item $f$ is stable.	
		\item For every $M\models T$, every $p\in S_f(M)$ is definable.
		\item For every $M\models T$, $||S_f(M)||\leq | M |$.
		\item For every $M\models T$, $|S_f(M)|\leq | M |^{\aleph_0}$.
		\item There is $\mu \geq |T|$ such that when $M\models T$ and $|M| \leq \mu$, then $||S_f(M)|| \leq \mu$.
		\item For every $\mu =\mu^{\aleph_0} \geq |T|$, when $M\models T$ and $|M| \leq \mu$, then $|S_f(M)| \leq \mu$.
	\end{enumerate}
\end{fact}


\begin{cor}\label{corollary: 2.4}
	The following conditions are equivalent:
	\begin{enumerate}
		\item $\forall f \in \mathcal{F}_{X/E}$, $f$ is stable.
		\item $\forall M\models T$ $\forall f\in \mathcal{F}_{X/E}$ $\forall p\in S_f(M)$, $p$ is definable.
		\item $\exists \mu\geq \lvert T \rvert$ s.t. $\forall M\models T$ if  $M\models T$ and $|M| \leq \mu$, then $| S_{\mathcal{F}_{X/E}}(M)| \leq \mu$.
		\item $\forall \mu = \mu ^{|T|+\lambda} \geq \lvert T \rvert$ $\forall M\models T$ if  $M\models T$ and $|M| \leq \mu$, then $| S_{\mathcal{F}_{X/E}}(M)| \leq \mu$.
	\end{enumerate}
\end{cor}

\begin{proof}
	This follows easily from Fact \ref{equiv1}. Only (1) $\Rightarrow$ (4) is a bit more delicate, which we will explain. So assume (1). Then we have (6) from Fact \ref{equiv1}. 
	
	Fix $m < \omega$. By the Stone-Weierstrass theorem, the first-order formulas restricted to $X \times \C^m$ generate a dense subalgebra $\mathcal{A}_m$ of cardinality at most $|T|+ \lambda$ of the Banach algebra $\mathcal{B}_m$ of all functions $f: X \times \C^m \to \R$ which extend to a CL-formula from $\C^\lambda \times \C^m$ to $\R$. As the family $\mathcal{F}^m_{X/E}$ of those functions from $\mathcal{B}_m$ which factor through  $X/E \times \C^m$ is a subspace of $\mathcal{B}_m$, it also has a dense subset $\mathcal{D}_m$ of cardinality at most $|T|+\lambda$. Since clearly $\mathcal{F}_{X/E} = \bigcup_{m<\omega} \mathcal{F}^m_{X/E}$, we get that the complete $\mathcal{F}_{X/E}$-type over $M$ of an element $a \in \C^\lambda$ is determined by $\bigcup_{m<\omega} \bigcup_{f \in \mathcal{D}_m} \tp_f(a/M)$. Using this and (6) from Fact \ref{equiv1}, one easily gets (4) in Corollary \ref{corollary: 2.4}.
\end{proof}


\begin{remark}\label{2.7}
	For any model $M$ of $T$ there is a natural bijection $$ S_{X/E}(M)\to S_{\mathcal{F}_{X/E}}(M) $$ that sends $$\tp([a]_E/M) \to \bigcup_{f \in \mathcal{F}_{X/E}} \tp_f(a/M).$$
\end{remark}
\begin{proof}
	$S_{\mathcal{F}_{X/E}}(M)$ can be seen as $\bigslant{S_X(M)}{\sim_{\mathcal{F}_{X/E}}}$,
	where for every $p,q \in S_X(M)$ and some (equivalently, any) $a_1' \models p$ and $a_2' \models q$:
	$$p\sim_{\mathcal{F}_{X/E}} q \iff (\forall f(x,y)\in \mathcal{F}_{X/E}) (\forall b\in M^{|y|})( f(a_1',b)=f(a_2',b)).$$
	On the other hand, $S_{X/E}(M)=\bigslant{S_X(M)}{\sim_E}$, where for every $p,q \in S_X(M)$ and some (equivalently, any) $a_1' \models p$ and $a_2' \models q$:
	$$p\sim_Eq\iff a_1'/E\equiv_M a_2'/E \iff (\forall m<\omega)(\forall b\in M^m)( (a_1'/E,b)\equiv (a_2'/E,b) ).$$
	By Proposition \ref{proposition: F_X/E separates points}, $p\sim_{\mathcal{F}_{X/E}} q$ if and only if $p\sim_Eq$. Hence, the conclusion follows.
\end{proof}

From the previous results,
we get some characterizations of stability of $X/E$.
\begin{cor}\label{2.8}
	The following conditions are equivalent:\begin{enumerate}
		\item $X/E$ is stable.
		\item $\forall f\in \mathcal{F}_{X/E}$ $(f$ is stable$)$.
		\item $\forall$ $M\models T$ $\forall f\in \mathcal{F}_{X/E}$ $\forall p\in S_f(M)$ $(p$ is definable$)$.
		\item $\exists \mu\geq \lvert T\rvert$ $\forall M\models T$ $(\lvert M\rvert\leq \mu \implies \lvert S_{\mathcal{F}_{X/E}}(M)\rvert\leq \mu)$.  
		\item $\exists \mu\geq \lvert T\rvert$ $\forall M\models T$ $(\lvert M\rvert\leq \mu \implies \lvert S_{X/E}(M)\rvert\leq \mu)$. 
		\item $\forall \mu = \mu ^{|T|+\lambda} \geq \lvert T \rvert$ $\forall M\models T$ $(\lvert M\rvert\leq \mu \implies \lvert S_{X/E}(M)\rvert\leq \mu)$.
	\end{enumerate}
\end{cor}

\begin{proof}
	The equivalence between (1), (2), (3), and (4) follows from Corollaries \ref{2.4} and \ref{corollary: 2.4}. The equivalence of (4) and (5) follows from Remark \ref{2.7}. The equivalence of (2) and (6) follows from Corollary \ref{corollary: 2.4} and Remark \ref{2.7}.
\end{proof}

As an application of the characterization from Corollary \ref{2.8}(6), we give a quick proof of Remark 2.5(iii) from \cite{MR3796277} that $G^{\textrm{st}}$ does not have proper hyperdefinable, stable quotients (which was left to the reader in  \cite{MR3796277}). 
Namely, suppose $H < G^{\textrm{st}}$ is a proper $A$-type-definable subgroup for some $A$; add all elements of $A$ as new constants.
We need to show that $G^{\textrm{st}}/H$ is unstable. By minimality of $G^{\textrm{st}}$, $G/H$ is unstable. So, by Corollary \ref{2.8}(6), there is $\mu = \mu ^{|T|+\lambda} \geq \lvert T \rvert$, a model $M$ of $T$ of cardinality $\mu$, and a sequence $(g_i)_{i<\mu^+}$ in $G$ such that $\tp(g_iH/M) \ne \tp(g_jH/M)$ for all $i \ne j$. Since $G/G^{\textrm{st}}$ is stable, by Corollary \ref{2.8}(6), there is a subset $I$ of $\mu^+$ of cardinality $\mu^+$ such that $g_iG^{\textrm{st}}\equiv_M g_jG^{\textrm{st}}$ for all $i,j \in I$. Fix $i_0 \in I$ and put $I_0:=I \setminus \{i_0\}$. Mapping all $g_i$, $i \in I_0$, by automorphisms over $M$, we can assume that they are all in the coset $g_{i_0}G^{\textrm{st}}$. Then $g_i':= g_{i_0}^{-1}g_{i} \in G^{\textrm{st}}$ for all $i \in I_0$. Moreover, take any $N \succ M$ containing $g_{i_0}$ and with $|N|= \mu$. Then $\tp(g_i'H/N) \ne \tp(g_j'H/N)$ for every distinct $i,j \in I_0$. Hence,  by Corollary \ref{2.8}, $G^{\textrm{st}}/H$ is unstable.

Next, we recall the definition of NIP for a hyperdefinable set, given in \cite[Remark 2.3]{MR3796277}, and we introduce the notion of generic stability 
for hyperimaginary types.
\begin{defin}\label{nip hyperdef}
	A hyperdefinable set $X/E$ has {\em NIP} if there	do not exist an indiscernible sequence $(b_i)_{i<\omega}$ and $d\in X/E$ such that
	$((d, b_{2i}, b_{2i+1}))_{i<\omega}$ is indiscernible and $tp(d, b_0) \neq tp(d, b_1)$. $($Note that
	the $b_i$ can be anywhere, not necessarily in $X/E$.$)$
\end{defin}

Generically stable types in NIP theories were introduced by Shelah \cite{Shelah1982-SHECTA-5}, and then thoroughly studied by Hrushovski and Pillay \cite{Hrushovski2011} and Usvyatsov \cite{10.2178/jsl/1231082310}. We extend the notion of a generically stable type to our hyperdefinable context (and more generally to continuous logic).

Let $p \in S_{X/E}(\C)$ be invariant over $A$. A {\em Morley sequence in $p$ over $A$} is a sequence $(a_i)_i$ of elements of $X/E$ such that $a_i \models p |_{Aa_{<i}}$. As in the home sort, by a standard argument, one can check that Morley sequences (of a given length) in $p$ over $A$  are $A$-indiscernible and have the same type over $A$.

\begin{defin}\label{gen stable type}
	An $A$-invariant type $p \in S_{X/E}(\C)$ is {\em generically stable} if every  Morley sequence $(a_i/E)_{i<\omega+\omega}$ in $p$ over $A$ satisfies $(\forall \varepsilon >0)$ $(\forall r\in \mathbb{R})$ $(\forall s\leq r-\varepsilon)$ $(\forall f(x,y)\in \mathcal{F}_{X/E})$ $(\forall b\in \mathfrak{C}^{\lvert y \rvert})$
	\useshortskip
	\begin{equation*} 
		\begin{gathered}
			\{i<\omega+\omega: f(a_i,b)\leq s\} \text{ is finite}\\
			\text{or}\\
			\{i<\omega+\omega: f(a_i,b)\geq r\}\text{ is finite.}
		\end{gathered}
	\end{equation*}
\end{defin}

Generic stability of $p$ does not depend on the choice of $A$ over which $p$ is invariant. Using the compactness theorem, one can show the following characterization.

\begin{prop}\label{proposition: characterization of gen. stab.}
	An $A$-invariant type  $p\in S_{X/E}(\C)$ is generically stable if and only if for every $\varepsilon>0$ and $f(x,y)\in\mathcal{F}_{X/E}$ there exists $N(f,\varepsilon)\in \mathbb{N}$ for which there is no Morley sequence $(a_i/E)_{i<\omega}$ in $p$ over $A$, disjoint subsequences $R,S\subseteq \omega$ each of which is of length at least $N(f,\varepsilon)$, and $b\in \C^{|y|}$ such that  $\lvert f(a_i,b)-f(a_j,b)\rvert \geq \varepsilon$ for all $a_i/E\in R$ and $a_j/E\in S$.
\end{prop}

The following definition is the hyperimaginary analog of \cite[Definition 1.2]{https://doi.org/10.1002/malq.200610046}.
\begin{defin}\label{def weakly stable}
	A hyperdefinable (over $A$) set $X/E$ is \emph{weakly stable} if for every $A$-indiscernible sequence $(a_i,b_i,c)_{i<\omega}$ with $a_i,b_i\in X/E$ for all (equivalently, some) $i<\omega$, we have  
	$$\tp(a_i,b_j,c/A) = \tp(a_j,b_i,c/A)$$ for all (some) $i\neq j < \omega$.
\end{defin}

Our next goal is to extend Corollary \ref{2.8} to:

\begin{teor}\label{equivteor}
	Assume $X/E$ has NIP. The following conditions are equivalent:
	\begin{enumerate}
		\item $X/E$ is stable.
		\item $\forall$ $M\models T$ $\forall f\in \mathcal{F}_{X/E}$ $\forall p\in S_f(M)$ $($$p$ is definable$)$.
		\item $\exists \lambda\geq \lvert T\rvert$ $\forall M\models T$ $(\lvert M\rvert\leq \lambda \implies \lvert S_{X/E}(M)\rvert\leq \lambda)$.
		\item Any indiscernible sequence of elements of $X/E$ is totally indiscernible.
		\item Any global invariant $($over some $A$$)$ type $p \in S_{X/E}(\C)$ is generically stable.
            \item $X/E$ is weakly stable.
	\end{enumerate}
\end{teor}
From the proof of this theorem, it will be clear that $(1)$, $(2)$, $(3)$, and $(5)$ are equivalent and imply $(4)$ without the NIP assumption; NIP is used to prove the implication from $(4)$ to $(1)$.

In order to prove Theorem \ref{equivteor}, we will first prove some results about hyperdefinable sets with NIP and about generically stable types.

From now on, EM will stand for Ehrenfeucht-Mostowski.  By the {\em EM-type of a sequence $I=(a_i/E)_{i \in \mathcal{I}}$} (symbolically, $\EM(I)$) we mean the set of all formulas $\varphi(x_1,\dots,x_n)$ such that for every $i_1 < \dots<i_n \in \mathcal{I}$, $\varphi(a_{i_1}',\dots,a_{i_n}')$ holds for all $a_{i_1}' \in [a_{i_1}]_E, \dots, a_{i_n}' \in [a_{i_n}]_E$, where $n$ ranges over $\omega$. We say that an indiscernible sequence $J$ {\em satisfies the EM-type of $I$} if $\EM(I) \subseteq \EM(J)$. By Ramsey's theorem and compactness, for every sequence $I$ there is an indiscernible sequence $J$ satisfying $\EM(I)$; we can even require that $J$ has an indiscernible sequence of representatives.

In the next three lemmas and in Corollary \ref{corollary: switched role of X/E and anything}, $X/E$ and $Y/F$ are arbitrary $\emptyset$-hyperdefinable sets.

\begin{lema}\label{lemma: 2.11} The following conditions are equivalent:
	\begin{enumerate}
		\item       There	exists an indiscernible sequence $(b_i)_{i<\omega}$ in $Y/F$ and $d\in X/E$ such that
		$((d, b_{2i}, b_{2i+1}))_{i<\omega}$ is indiscernible and $\tp(d, b_0) \neq \tp(d, b_1)$.
		
		\item 	There exists an  indiscernible sequence $(b_i)_{i<\omega}$ in $Y/F$ and $d\in X/E$ such that 
		\useshortskip
		\begin{align*}
			\tp(d,b_i)=\tp(d,b_0)&\iff i \text{ even}\\
			\tp(d,b_i)=\tp(d,b_1)&\iff i \text{ odd}.
		\end{align*}
	\end{enumerate}
\end{lema}

\begin{proof}
	$(1) \Rightarrow (2)$ is clear.
	
	$(2) \Rightarrow (1)$ Assume (2). 
	Then $\tp(d,b_0)\neq \tp(d,b_1)$.
	As $(b_i)_{i<\omega}$ is indiscernible, so is $(b_{2i},b_{2i+1})_{i<\omega}$. Choose  an indiscernible sequence $(\tilde{d},b_{2i},b_{2i+1} )_{i<\omega}$ satisfying the EM-type of $(d,b_{2i},b_{2i+1})_{i<\omega}$. Then, $\tilde{d}\in X/E$ and $(b_i)_{i<\omega}$ witness (1).
\end{proof}

\begin{lema}\label{lemma: equivalences NIP}
	Let $p, q \in S_{X/E \times Y/F}(\emptyset)$ be distinct types. 
	Then the following conditions are equivalent:
	\begin{enumerate}
		\item There exists an indiscernible sequence $(b_i)_{i<\omega}$ in $Y/F$ and an element $d\in X/E$ with: 
		\useshortskip
		\begin{align*}
			\tp(d,b_i)=p&\iff i \text{ even}\\
			\tp(d,b_i)=q&\iff i \text{ odd}.
		\end{align*}
		\item There is a sequence $(b_i)_{i<\omega}$ in $Y/F$ (not necessarily
		indiscernible) which is shattered by $(p,q)$ in the sense that for every
		$I \subseteq \omega$ there is $d_I \in X/E$ with 
		\useshortskip
		\begin{align*} \tp(d_I,b_i) =p \iff i
			\in I\\ \tp(d_I,b_i)=q \iff i \notin I.
		\end{align*}
	\end{enumerate}
\end{lema}

\begin{proof}
	$(1)\Rightarrow (2)$ Let $(b_i)_{i<\omega}$ in $Y/F$ and $d\in X/E$ witness $(1)$. Let $I\subseteq \omega$. We can find an increasing one to one map $\tau:\omega \to \omega$ such that for all $i\in \omega$, $\tau(i)$ is even if and only if $i\in I$. By indiscernibility, the map sending $b_i$ to $b_{\tau(i)}$ for all $i\in \omega$ can be extended to an automorphism $\sigma$. The element $d_I:=\sigma^{-1}(d)$ satisfies the conditions in $(2)$.
	
	$(2)\Rightarrow (1)$ Let $(b_i)_{i<\omega}$ witness $(2)$. We can find an indiscernible sequence $(c_i)_{i<\omega}$ in $Y/F$ satisfying the EM-type of $(b_i)_{i<\omega}$. It follows that for any two disjoint finite sets $I_0,I_1\subseteq \omega$, the partial type 
	$$\{ p(x;c_i): i\in I_0\}\cup \{q(x;c_i): i\in I_1\}$$ 
	is consistent. By compactness, the sequence $(c_i)_{i<\omega}$ is shattered by $(p,q)$. In particular, there is $d\in X/E$ such that $\tp(d,c_i)=p$ if and only if $i$ is even and $\tp(d,c_i)=q$ if and only if $i$ is odd.
\end{proof}

\begin{lema}
	Let $p,q \in S_{X/E \times Y/F}(\emptyset)$ be distinct. Then there exists an infinite sequence in $Y/F$ shattered by $(p,q)$ if and only if there exists an infinite sequence in $X/E$ shattered by $(p^{opp},q^{opp})$, where $p^{opp}(x,y):=p(y,x)$.
\end{lema}
\begin{proof}
	Let $(b_i)_{i<\omega}$ be a sequence in $Y/F$ shattered by $(p,q)$. By compactness, we can find a sequence $(c_i)_{i\in \mathcal{P}(\omega)}$ in $Y/F$ which is shattered by $(p,q)$ as witnessed by the family $\{ d_I: I\subseteq \mathcal{P}(\omega) \} \subseteq X/E$. Consider the sequence $(d_j)_{j< \omega}$ in $X/E$, where $d_j:=d_{I_j}$ and $I_j:=\{X\subseteq \omega: j\in X\}$. Then for any $J\subseteq \omega$ we have: 
	\begin{align*} \tp(d_j,c_J) =p \iff j
		\in J\\ \tp(d_j,c_J)=q \iff j \notin J,
	\end{align*}
	because $j\in J$ if and only if $J\in I_j$. 
	Thus, $(d_j)_{j<\omega}$ is shattered by $(p^{opp},q^{opp})$.
	
	The converse follows by symmetry.
\end{proof}

From the last three lemmas, one easily deduces the following

\begin{cor}\label{corollary: switched role of X/E and anything}
	$X/E$ has NIP if and only if there do not exist an indiscernible sequence $(b_i)_{i<\omega}$ of elements of $X/E$ and $d$ (from anywhere) such that the sequence $(d,b_{2i},b_{2i+1})_{i<\omega}$ is indiscernible and $\tp(d,b_0)\neq \tp(d,b_1)$.
\end{cor}

The next lemma is analogous to the finite-cofinite lemma for NIP theories.

\begin{lema}\label{fincofin}
	Suppose that $X/E$ has NIP. Let $(a_i)_{i\in I}$ be an infinite, totally indiscernible sequence of elements of $X/E$ and $b$ any tuple from $\C$. Then, for any $j_0,j_1\in I$, whenever $\tp(a_{j_0},b)\neq \tp(a_{j_1},b)$, either
	
	\begin{equation*} 
		\begin{gathered}
			I_0:=\{i\in I: \tp(a_i,b)=\tp(a_{j_0},b)\} \text{ is finite}\\
			\text{or}\\
			I_1:=\{i\in I:\tp(a_i,b)=\tp(a_{j_1},b)\} \text{ is finite.}
		\end{gathered}
	\end{equation*}
\end{lema}
\begin{proof}
	Otherwise, we can build a sequence $(i_k)_{k<\omega}$ of pairwise distinct elements of $I$ so that $i_k\in I_0$ if and only if $k$ is even, and $i_k\in I_1$ if and only if $k$ is odd. Then, since $(a_i)_{i\in I}$ was totally indiscernible, the sequence $(a_{i_k})_{k<\omega}$ is indiscernible and $$\tp(a_{i_k},b)=\tp(a_{j_0},b)\iff k \text{ even}$$ $$\tp(a_{i_k},b)=\tp(a_{j_1},b)\iff k \text{ odd},$$
	which by Lemma \ref{lemma: 2.11} and Corollary \ref{corollary: switched role of X/E and anything} imply that $X/E$ does not have NIP, a contradiction. 
\end{proof}

\begin{defin}
	The median value connective $\med_n:[0,1]^{2n-1}\to[0,1]$ is defined by $$ \med_n(t_{<2n-1})=\max_{\substack{w\subseteq 2n-1 \\ \lvert w\rvert=n}}\min_{i\in w} t_i= \min_{\substack{w\subseteq 2n-1 \\ \lvert w\rvert=n}}\max_{i\in w} t_i. $$
\end{defin}

This connective, as its name indicates, literally computes the median value of the list of arguments.

For  $f(x,y) \in \mathcal{F}_{X/E}$ and $N \in \mathbb{N}^+$, put  
$$d^N f(y,x_{<2N-1}):=\med_{N}(f(x_i,y): i<2N-1).$$	

As in classical model theory, generically stable types are definable, which follows from the next proposition. For  $p \in S_{X/E}(\C)$, $f(x,y) \in \mathcal{F}_{X/E}$, and $b \in \C^{|y|}$, by $f(x,b)^p$ we mean the value of $p$ at $f(x,b)$ for $p$ treated as an element of $S_{\mathcal{F}_{X/E}}(\C)$ as explained in Remark \ref{2.7} (in other words, it is $f(a,b)$ for $a/E \models p$). We say that $p$ is {\em definable} if it is so as a type in $S_{\mathcal{F}_{X/E}}(\C)$.

\begin{prop}\label{2.21}
	If $p \in S_{X/E}(\C)$ is generically stable over $A$, then for any $f(x,y)\in \mathcal{F}_{X/E}$, $\varepsilon>0$, $b\in \mathfrak{C}^{\lvert y\rvert}$, and $(a_i/E)_{i<\omega}$ a Morley sequence in $p$ over $A$, $$ \lvert f(x,b)^p- d^{N(f,\varepsilon)} f(b,a_{<2N(f,\varepsilon)-1})\rvert \leq \varepsilon,$$
	where $N(f,\varepsilon)$ is a number as in Proposition \ref{proposition: characterization of gen. stab.}. 
\end{prop}

\begin{proof}
	Suppose this is not true. Write $N$ for $N(f,\varepsilon)$. Then, we have two cases:
	
	1) $d^N f(b,a_{<2N-1})-\varepsilon>f(x,b)^p$. This implies $$\max_{\substack{w\subseteq 2N -1 \\ \lvert w\rvert=N}}\min_{i\in w} f(a_i,b)>f(x,b)^p+\varepsilon.$$
	Hence, there is $w$ of size $N$ such that for all $i\in w$ we have $f(a_i,b)>f(x,b)^p + \varepsilon$. Taking a Morley sequence in $p$ over $Ab(a_i)_{i\in w}$, we get a contradiction with the choice of $N(f,\varepsilon)$.
	
	2)	$d^N f(b,a_{<2N-1})+\varepsilon<f(x,b)^p$. This case is analogous to the previous one, using the other definition of the median value connective.
\end{proof}

\begin{cor}\label{corollary: gen stab implies definable}
	All generically stable types in $S_{X/E}(\C)$ are definable.
\end{cor}

\begin{proof}
	Consider any generically stable (over $A$) $p \in S_{X/E}(\C)$ and $f(x,y)\in \mathcal{F}_{X/E}$. Let $(a_i/E)_{i<\omega}$ be  a Morley sequence in $p$ over $A$.
	Define a CL-formula $df(y,z)$ to be the forced limit of the sequence $(d^{N(f,2^{-n})} f(y,x_{<2N(f,2^{-n})-1}))_{n<\omega}$ (see Definitions 3.6 and 3.8 in \cite{MR2657678}). By the last proposition and \cite[Lemma 3.7]{MR2657678}, we get that $df(y,(a_{i})_{i<\omega})$ 
	is the $f$-definition of $p$
\end{proof}

Let $M \prec \C$ (small), $f \in \mathcal{F}_{X/E}$, $p \in S_f(M)$, and $q \in S_X(\C)$. It is clear what it means that $q$ extends $p$, namely: for $d \models q$ and for all $b \in M^{|y|}$, $f(x,b)^p = f(d,b)$ (here $f$ is understood as a function $f:X(\C')\times \C'^m\to \R$  for $\C'$ a bigger monster model to which $d$ belongs; note that there is a unique function that can be extended to a CL-formula over $\emptyset$). 
This is equivalent to saying that the partial type over $M$ defining $\{d \in \C^{|x|}: f(d,b) =f(x,b)^p \text{ for all } b \in M^{|y|}\}$ is contained in $q$. 
Thus, using the well-known fact that each partial type over $M$ (even in infinitely many variables) extends to a global coheir over $M$ (i.e. global type finitely satisfiable in $M$), we have:

\begin{fact}\label{fact: instead of BY}
	Every type $p \in S_f(M)$ has an extension to a global type $q \in S_X(\C)$ finitely satisfiable in $M$.
\end{fact}

Finally, we present the proof of the main result of this section.

\begin{proof}[Proof of Theorem \ref{equivteor}] $(1)\Leftrightarrow (2)\Leftrightarrow (3)$ is a part of Corollary \ref{2.8}. The implication $(1)\implies (6)$ follows by definition.
	
	$(1)\Rightarrow (4)$ Suppose that the sequence $(a_i)_{i<\omega}$ in $X/E$ is a counter-example. Without loss of generality we can replace $\omega$ by $\mathbb{Q}$. Then, there are rational numbers $i_0<\dots <i_{n-1}$ and a natural number $j<n-1$ such that $$\tp(a_{i_j},a_{i_{j+1}}/A)\neq \tp(a_{i_{j+1}},a_{i_j}/A),$$ where $A$ is the set of all $a_{i_k}$ for $k<n$ such that $j\neq k \neq j+1$. Choose any rationals $l_0<l_1<\dots$ in the interval $(i_j,i_{j+1})$. Let $b_i=a_{l_i}$ for $i<\omega$. Then, the sequence $(b_i)_{i<\omega}$ is $A$-indiscernible and $\tp(b_i,b_j/A)\neq \tp(b_j,b_i/A)$ for all $i<j<\omega$. This contradicts the stability of $X/E$.
	
	$(4)\Rightarrow (1)$ Assume that $X/E$ is unstable and $(4)$ holds. Since $X/E$ is unstable, there exists an indiscernible sequence $(a_i/E,b_i)_{i\in\mathbb{Z}}$ such that $\tp(a_i/E,b_j)\neq \tp(a_j/E,b_i)$ for $i<j\in \mathbb{Z}$. The sequence $(a_i/E)_{i\in\mathbb{Z}}$ is totally indiscernible and the sets $\{i\in\mathbb{Z}: \tp(a_i/E,b_0)=\tp(a_1/E,b_0)\}$ and $\{i\in\mathbb{Z}: \tp(a_i/E,b_0)=\tp(a_{-1}/E,b_0)\}$ are both infinite, contradicting Lemma \ref{fincofin} and the assumption that $X/E$ has NIP.
	
	$(1)\Rightarrow (5)$ Let $p\in S_{X/E}(\C)$ be invariant (over some A) but not generically stable. Then, there exists a Morley sequence $(a_i/E)_{i\in \omega + \omega}$ in $p$ over $A$, $\varepsilon>0$, $r\in \mathbb{R}$, $s\leq r-\varepsilon$, $f(x,y)\in \mathcal{F}_{X/E}$, and $b \in \C^{|y|}$ such that the sets $ \{i\in \omega + \omega: f(a_i,b)\leq s\} $ and $ \{i\in \omega + \omega: f(a_i,b)\geq r\}$ are both infinite. There are two possible cases: either there are infinitely many alternations, or after removing a finite number of elements $a_i$: 
	$$ (f(a_i,b)\leq s \iff i<\omega) \; \text{ or } \; (f(a_i,b)\geq r \iff i<\omega).$$
	By the indiscernibility of $(a_i/E)_{i \in \omega +\omega}$, Ramsey's theorem and compactness, in each case we get a contradiction with the stability of $f$.
	
	$(5)\Rightarrow (2)$ Consider any $f\in \mathcal{F}_{X/E}$ and $p\in S_{f}(M)$. By Fact \ref{fact: instead of BY}, choose $q'\in S_{X}(\mathfrak{C})$ extending $p$ which is a coheir over $M$; let $q \in S_{X/E}(\C)$ be induced by $q'$. Being a coheir over $M$, $q'$ is $M$-invariant; so $q$ is $M$-invariant, hence generically stable by (5). By Corollary \ref{corollary: gen stab implies definable}, $q$ is definable. Denote the $f$-definition of $q$ by $\psi$. Since $q$ is $M$-invariant, so is the CL-formula $\psi$. Therefore, $\psi$ is definable over $M$ (i.e. a CL-formula over $M$). Hence, $p$ is definable by the CL-formula $\psi$.

	$(6)\implies (4)$ 
	Suppose that the sequence $(a_i)_{i<\omega}$ in $X/E$ is indiscernible but not totally indiscernible. Let us, without loss of generality, replace $\omega$ by $\mathbb{Q}$. Then, there exist a natural number $n$ and 
	$j<n-1$ such that $$\tp(a_{j},a_{{j+1}}/A)\neq \tp(a_{{j+1}},a_{j}/A),$$ where $A$ is the set of all $a_{k}$ for $k<n$ distinct from $j$ and $j+1$. Choose any rationals $l_0<l_1<\dots$ in the interval $(j,{j+1})$. Let $b_i:=a_{l_i}$ for $i<\omega$. Then, the sequence $(b_i)_{i<\omega}$ is $A$-indiscernible and $\tp(b_i,b_j/A)\neq \tp(b_j,b_i/A)$ for all $i<j<\omega$.
	Let $a$ be an enumeration of $A$.
	We conclude that the sequence $(b_i,b_i,a)_{i<\omega}$ contradicts the weak stability of $X/E$.
\end{proof}
The equivalence under $NIP$ of stability and weak stability extends \cite[Proposition 4.2]{https://doi.org/10.1002/malq.200610046}  to the  hyperdefinable context.

\section{Distal hyperimaginary sequences}\label{section: distal theories}

The goal of this section is to prove Theorem \ref{proposition: preservation of distality} and deduce Proposition \ref{proposition: distal implies bdd}, which in turn confirms the prediction from \cite{MR3796277} that in a distal theory $G^{\textrm{st}}=G^{00}$.


We work in a monster model $\C$ of a complete, classical first-order theory $T$ with NIP.	 This section is based on \cite{MR3001548}, in particular the next two definitions are from there.

\begin{defin}
	For any indiscernible sequence $I$, if $I = I_1 + I_2$ (the concatenation of $I_1$ and $I_2$), we say that $\mathfrak{c}=(I_1, I_2)$ is a {\em cut} of $I$.
	
	We write $(I_1',I_2') \unlhd (I_1,I_2)$ if $I_1'$ is an end segment of $I_1$ and $I_2'$ an initial
	segment of $I_2$.
	
	If $J \subset I$ is a convex subsequence, a cut $\mathfrak{c} = (I_1, I_2)$ is said to be {\em interior to $J$} if $I_1 \cap J$ and $I_2 \cap  J$ are infinite.
	
	A cut is {\em Dedekind} if both $I_1$ and $I^*_2$  ($I_2$ with the reversed order) have infinite cofinality.
	
	A {\em polarized cut} is a pair $(\mathfrak{c},\varepsilon)$, where $\mathfrak{c}$ is a cut $(I_1,I_2)$ and
	$\varepsilon \in \{1,2\}$ is such that $I_\varepsilon$ is infinite. We say that the cut is \emph{left polarized} if $I_1$ is infinite and \emph{right polarized} if $I_2$ is infinite.
	
	If $\mathfrak{c} = (I_1, I_2)$ is a cut, we say that a tuple $b$ {\em fills} $\mathfrak{c}$ if $I_1 + b + I_2$ is indiscernible.
\end{defin}

Sometimes, if it is clear that the tuple $b$ fills some cut $\mathfrak{c}=(I_1,I_2)$ of $I$, we will write $I\cup \{b\}$ instead of $I_1+b+I_2$. And similarly, in the case of two elements $a,b$ filling respectively distinct cuts $\mathfrak{c}_1, \mathfrak{c}_2$, abusing notation, we will  write $I \cup \{a\} \cup \{b\}$ for the associated concatenation $I_1+a+I_2+b+I_3$.

\begin{defin}\label{definition: distality}
	A dense indiscernible sequence $I$ is {\em distal} if for any distinct Dedekind cuts $\mathfrak{c}_1,\mathfrak{c}_2$, if $a$ fills $\mathfrak{c}_1$ and $b$ fills $\mathfrak{c}_2$, then $I \cup\{a\} \cup \{b\}$ is indiscernible.
	
	The theory $T$ is {\em distal} if all dense indiscernible sequences (of tuples from the home sort) are distal.
	
	We say that {\em $T^{\textrm{heq}}$ is distal} if all dense indiscernible sequences $(a_i/E)_{i \in \mathcal{I}}$ of hyperimaginaries (where $E$ is $\emptyset$-type-definable) are distal. 
\end{defin}

Let $E$ be a $\emptyset$-type-definable equivalence relation on $\C^\lambda$, and let $\pi_E: \C^\lambda \to \C^\lambda /E$ be the quotient map. 

The next lemma is a variant of \cite[Lemma 2.8]{MR3001548} for hyperimaginaries.

\begin{lema}\label{2.8'}
	Let $I=(a_i/E)_{i\in \mathcal{I}}$ be a dense indiscernible sequence and $A\subset \C$ a (small) set of parameters. Let $(\mathfrak{c}_i)_{i<\alpha}$ be a sequence of pairwise distinct Dedekind cuts in $I$. For each $i<\alpha$ let $d_i$ fill the cut $\mathfrak{c_i}$. Fix a polarization of each $\mathfrak{c_i}$, $i < \alpha$. Then there are $(d_i')_{i<\alpha}$ satisfying $(d_i')_{i<\alpha}\equiv_I (d_i)_{i<\alpha}$  such that for every formula  $\theta$ with parameters from $A$ and $i<\alpha$: 
	if $$\pi_E^{-1}(d_i')\subseteq 
	\theta(\mathfrak{C}),$$
	then $$ \pi_E^{-1}(a_j/E)\not\subseteq \neg \theta(\mathfrak{C}) $$
	for $a_j/E$ from a co-final fragment of  the left part of $\mathfrak{c}_i$ if $\mathfrak{c}_i$ is left-polarized, or from a co-initial fragment of the right part of $\mathfrak{c}_i$ if $\mathfrak{c}_i$ is right-polarized.
\end{lema}
\begin{proof}
	
	To simplify notation, let all the cuts $\mathfrak{c}_i$ be left-polarized. 
	The negation of the conclusion says that for every $(d_i')_{i<\alpha} \equiv_{I} (d_i)$ there exists  $i< \alpha$ and a formula $\theta(x)$ over $A$ such that $$\pi_E^{-1}(d_i') \subseteq \theta(\mathfrak{C})$$ and $$ \pi_E^{-1}(a_j/E)\subseteq \neg \theta(\mathfrak{C}) $$ for all $a_j/E$ in some end segment of $\mathfrak{c}_i$. For $i<\alpha$ put $C_i(x_i):=\{\varphi(x_i) \in L(A) :  \pi_E^{-1}(a_j/E)\subseteq \varphi(\C) \text{ for all $a_j/E$ in some end segment of } \mathfrak{c}_i  \}$. Note that these sets are closed under conjunction. 
	The negation of the conclusion is equivalent to
	$$(*) \;\;\;\;\;\;\;\;\; \tp((d_i)_{i<\alpha})/I)\cup \bigcup_{i<\alpha}C_i(x_i) \text{ is inconsistent,}$$ 
	which in turn is equivalent to the existence of a finite subset $J$ of $\alpha$ such that $\tp((d_i)_{i\in J})/I)\cup \bigcup_{i \in J}C_i(x_i)$ inconsistent. Therefore, without loss of generality, $\alpha<\omega$. We will show a detailed proof for $\alpha=2$; the proof for an arbitrary $\alpha \in \omega$ is the same.

	Suppose the conclusion fails. By $(*)$, choose a finite subsequence $I_0$ of $I$ so that $\tp(d_0,d_1/I_0)\cup C_0(x_0)\cup C_1(x_1)$ is inconsistent, and let $\varphi_0(x_0)\in C_0(x_0)$ and  $\varphi_1(x_1) \in C_1(x_1)$ be formulas witnessing it.  Now, for $i \in \{0,1\}$ take $(J_i,J_i')\unlhd \mathfrak{c}_i$ such that $\pi_E^{-1}(a_j/E)\subseteq \varphi_i(\mathfrak{C})$ for all $a_j/E\in J_i$, $J_i\cup J_i'$ contains no element of $I_0$, and $(J_0 \cup J_0') \cap (J_1 \cup J_1') =\emptyset$.
	
	\begin{claim}
		For every two cuts $\mathfrak{d}_0,\mathfrak{d}_1$ in $I$ respectively interior to $J_0,J_1$ we can find hyperimaginaries $e_0$ filling $\mathfrak{d}_0$ and $e_1$ filling $\mathfrak{d}_1$ such that $\tp(e_0,e_1/I_0)=\tp(d_0,d_1/I_0)$.
	\end{claim}
	
	\begin{proof}[Proof of claim]
		Consider any finite  $K\subseteq I$. For $i \in \{0,1\}$, the cut $\mathfrak{d}_i$ decomposes $K$ into $L_i^- +L_i^+$. It is enough to find $e_0,e_1$ such that \begin{align*}
			\tp(L_1^-,e_0,L_1^+)&\subset \EM(I),\\
			\tp(L_2^-,e_1,L_2^+)&\subset \EM(I),\\
			\tp(e_0,e_1/I_0)&=\tp(d_0,d_1/I_0),
		\end{align*}
		where $\EM(I)$ denotes the Erenfeucht-Mostowski type of $I$.
		
		We can decompose $K$ into sequences $K_0\subseteq J_0+J_0'$, $K_1\subseteq J_1+J_1'$, and $K_2\subseteq I \setminus ((J_0+J_0') \cup (J_1+J_1'))$. Next, we construct new finite sequences $K_0'$, $K_1'$, $K_2'$, and $K'$ in the following way: $K_2'=K_2$; for every element $a\in K_0$ we take $a'\in J_0 \cup J_0'$ such that $a'$ is in the same relative position to $\mathfrak{c}_0$ as $a$ was to $\mathfrak{d}_0$ and also preserving the order between elements, and we define $K_0'$ to be the constructed sequence of the $a'$'s; for $K_1'$ we proceed in an analogous manner; finally, $K':=K_0'\cup K_1'\cup K_2'$ written as a sequence in an obvious order provided by the construction. By the indiscernibility of the sequence $I$, there is $\sigma\in \aut(\C)$ such that $\sigma(KI_0)=K'I_0$. The elements $e_0:=\sigma^{-1}(d_0)$ and $e_1:=\sigma^{-1}(d_1)$ satisfy the desired conditions.
	\end{proof}

	Fix $\mathfrak{d}_0,\mathfrak{d}_1$ as in the claim, and choose $e_0=:e^0_0$ and $e_1=:e^0_1$ provided by the claim. By the choice of $\varphi_i(x)$, $\pi_E^{-1}(e^0_0)\subseteq \neg \varphi_0(\mathfrak{C})$ or $\pi_E^{-1}(e^0_1)\subseteq \neg \varphi_1(\mathfrak{C})$. For example, $\pi_E^{-1}(e^0_0)\subseteq \neg \varphi_0(\mathfrak{C})$. Set $I^0:=I\cup \{e^0_0\}$; this is an indiscernible sequence. 
	Let $J_0^0$ be an end segment of $J_0$ not containing  $\mathfrak{d}_0$, and $J_1^0:=J_1$. 
	By the same argument as in the above claim, we get
	
	\begin{claim}
		For every two cuts $\mathfrak{d}^1_0,\mathfrak{d}^1_1$ in $I^0$ respectively interior to $J_0^0,J_1^0$ we can find hyperimaginaries $e^1_0$ filling $\mathfrak{d}^1_0$ and $e_1^1$ filling $\mathfrak{d}^1_1$ (seen as cuts in $I^0$) such that $\tp(e^1_0,e^1_1/I_0)=\tp(d_0,d_1/I_0)$.
	\end{claim} 

	Fix $\mathfrak{d}^1_0,\mathfrak{d}^1_1$ as in the claim, and choose $e^1_0, e^1_1$ provided by the claim. Then $\pi_E^{-1}(e^1_0)\subseteq \neg \varphi_0(\mathfrak{C})$ or $\pi_E^{-1}(e^1_1)\subseteq \neg \varphi_1(\mathfrak{C})$. For example,  $\pi_E^{-1}(e^1_1)\subseteq \neg \varphi_1(\mathfrak{C})$. Set $I^1 := I^0 \cup \{e^1_1\}$; this is again an indiscernible sequence. 
	Let $J_0^1: =J_0^0$, and $J_1^1$ be an end segment of $J_1^0$ not containing $\mathfrak{d}_1^1$. 
	
	Iterating this process $\omega$ times, we get a sequence $(\varepsilon_k)_{k<\omega}$ of 0's and 1's and a sequence of hyperimaginaries $(e^k_{\varepsilon_k})_{k<\omega}$ such that $I\cup \{ e^k_{\varepsilon_k}\}_{k<\omega}$ is indiscernible, the  $e^k_{\varepsilon_k}$'s with $\varepsilon_k=0$ fill pairwise distinct cuts in $J_0$, the  $e^k_{\varepsilon_k}$'s with $\varepsilon_k=1$ fill pairwise distinct cuts in $J_1$, and $\pi_E^{-1}(e_{\varepsilon_k}^k)\subseteq \neg \varphi_{\varepsilon_k}(\mathfrak{C})$ for all $k<\omega$. W.l.o.g. $\varepsilon_k=0$ for all $k<\omega$.
	
	Finally, by Ramsey's theorem and compactness, we can find an indiscernible sequence of representatives of the hyperimaginaries from the indiscernible sequence $J_0\cup \{ e_0^k\}_{k<\omega}$. In this way, we have produced an indiscernible sequence for which $\varphi_0(x_0)$ has infinite alternation rank, which contradicts NIP.
\end{proof}

\begin{teor}\label{teordistal}
	If $(a_i)_{i \in \mathcal{I}}$ is a (dense) distal sequence of tuples from $\C^\lambda$, then $(a_i/E)_{i \in \mathcal{I}}$ is a distal sequence of hyperimaginaries.	Thus, if $T$ is distal, then $T^{\textrm{heq}}$ is distal.
\end{teor}

\begin{proof}
	The fact that the first part implies the second follows from the observation that for any indiscernible sequence of hyperimaginaries we can find an indiscernible sequence of representatives. So let us prove the first part.
	
	Assume $I:=(a_i)_{i \in \mathcal{I}}$ is a (dense) distal sequence of tuples from $\C^\lambda$, and let $I'=(a_i/E)_{i \in \mathcal{I}}$. Present $E$ as $\dcap_{t\in\mathcal{T}}R_t$ for some $\emptyset$-definable (not necessarily equivalence) relations $R_t$. Consider any distinct Dedekind cuts $\mathfrak{c}'_1$ and $\mathfrak{c}'_2$ of $I'$, 
	say $\mathfrak{c}_1'$ is on the left from $\mathfrak{c}_2'$. They partition $I'$ into $I_1'$, $I_2'$, and $I_3'$. The cuts $\mathfrak{c}'_1$ and $\mathfrak{c}'_2$ induce Dedekind cuts $\mathfrak{c}_1$ and $\mathfrak{c}_2$ of $I$ which partition $I$ into $I_1$, $I_2$, and $I_3$. Take any $d_1$ and $d_2$ filling the cuts  $\mathfrak{c}'_1$ and $\mathfrak{c}'_2$, respectively. Apply Lemma \ref{2.8'} to this data  (taking left-polarization of $\mathfrak{c}_1'$ and $\mathfrak{c}_2'$) and $A$ being the set of all coordinates of all tuples $a_i$, $i \in \mathcal{I}$. This yields $d_1'=e_1'/E$ and $d_2'=e_2'/E$ satisfying the conclusion of Lemma \ref{2.8'}.
	
	\begin{claim}
		For every $i \in \{1,2\}$ there is  $b_i$ filling the cut $\mathfrak{c}_i$ such that $\pi_{E}(b_i)=d_i'$.
	\end{claim}
	\begin{proof}[Proof of the claim] It is enough to consider $i=1$.
		By compactness, it suffices to show that for any formula $\varphi(x_1,\dots, x_n)\in \EM(I)$, $i_1<\dots <i_{k-1}\in \mathcal{I}_1$, and $i_{k+1}<\dots <i_{n}\in \mathcal{I}_2+\mathcal{I}_3$ there is $b_1$ such that $\models \varphi(a_{i_1},\dots, a_{i_{k-1}},b_1,a_{i_{k+1}},\dots,a_{i_n})$ and $\pi_{E}(b_1)=d_1'$. By compactness, it is enough to show that for every $t\in \mathcal{T}$ 
		$$\models \exists y(yR_te_1' \wedge \varphi(a_{i_1},\dots, a_{i_{k-1}},y,a_{i_{k+1}},\dots,a_{i_n}) ).$$ 
		Assume this fails. Choosing $t'\in \mathcal{T}$ such that $R_{t'}\circ R_{t'}\subseteq R_t$, we get $$\pi_{E}^{-1}(d_1')\subseteq\neg(\exists y)( yR_{t'}x\wedge \varphi(a_{i_1},\dots, a_{i_{k-1}},y,a_{i_{k+1}},\dots,a_{i_n}) )(\mathfrak{C})=:\theta(\mathfrak{C}).$$
		By the choice of $d_1'$, for $a_j/E$ from a co-final fragment of the left part of $\mathfrak{c}_1$ we have $\pi_{E}^{-1}(a_j/E)\not\subseteq \neg\theta(\mathfrak{C})$. So, for any of these indices $j$, there is $a'E a_j$ such that $$ \models \neg\exists y ( y R_{t'} a' \wedge \varphi(a_{i_1},\dots, a_{i_{k-1}},y,a_{i_{k+1}},\dots,a_{i_n}) ).$$ 
		As the set of such indices $j$ is co-final in $\mathcal{I}_1$, we can find such an index $j \in \mathcal{I}_1$ with $j>i_{k-1}$. Then $y:=a_j$ contradicts the last formula (by the indiscernibility of $I$). 	
	\end{proof}
	
	By the distality of $I$, the sequence $I_1+b_1+I_2+b_2+I_3$ is indiscernible. Hence, the sequence $I_1'+d_1'+I_2'+d_2'+I_3'$  is also indiscernible. On the other hand, by our choice of $d_1',d_2'$, we know that $d_1d_2\equiv_{I'}d_1'd_2'$. Thus, the sequence $I_1'+d_1+I_2'+d_2+I_3'$ is indiscernible, too. As $d_1,d_2$ were arbitrary, we conclude that $I'$ is a distal sequence.
\end{proof}

\begin{cor}\label{corollary: stable iff bounded}
	For a distal theory $T$, a hyperdefinable set $X/E$ is stable if and only if $E$ is a bounded equivalence relation. In particular, for a group $G$ type-definable in a distal theory, $G^{\textrm{st}}=G^{00}$.
\end{cor}

\begin{proof}
	If $E$ is bounded, then $X/E$ is stable (as each indiscernible sequence in $X/E$ is constant). To prove the other implication, assume that $X/E$ is stable. 
	Since distality is preserved under naming parameters, w.l.o.g. both $X$ and $E$ are $\emptyset$-type-definable. If $X/E$ is not bounded, taking a very long sequence of pairwise distinct elements of $X/E$, by extracting indiscernibles, there exists a dense indiscernible sequence of pairwise distinct elements of $X/E$. By stability and Theorem \ref{equivteor}, this sequence is totally indiscernible.  Since non-constant, totally indiscernible sequences are not distal (see \cite[Corollary 2.15]{MR3001548}), we get a contradiction with the distality of $T^{\textrm{heq}}$ (which we have by Theorem \ref{teordistal}).
\end{proof}

We find very natural to see Corollary \ref{corollary: stable iff bounded} as an easy consequence of  Theorem \ref{teordistal} which in turn is a fundamental result concerning distality and hyperimaginaries. However, there is a short direct (i.e. not using Theorem \ref{teordistal}) proof of Corollary \ref{corollary: stable iff bounded}.

\begin{proof} (Of Corollary \ref{corollary: stable iff bounded})
If $E$ is bounded, then $X/E$ is stable (as each indiscernible sequence in $X/E$ is constant). 

Assume now that $X/E$ is stable and $E$ is not bounded. Since $X/E$ is not bounded, we can find an indiscernible sequence $(a_i)_{i\in\mathbb{Q}}$ such that $([a_i]_E)_{i\in \mathbb{Q}}$ consists of pairwise elements (using Fact \ref{extracting indiscernibles}). Let $$I_0:=(a_i)_{i\in \mathbb{Q}\cap (-\infty,\sqrt{2}]},$$ $$I_1:=(a_i)_{i\in \mathbb{Q}\cap [\sqrt{2},\sqrt{3}]},$$ $$I_1:=(a_i)_{i\in \mathbb{Q}\cap [\sqrt{3},+\infty)},$$ and take $a_{\sqrt{2}},a_{\sqrt{3}}\in X$ such that $I_0+a_{\sqrt{2}}+I_1+a_{\sqrt{3}}+I_2$ is indiscernible.

From the definition of stability of $X/E$, applying it to the sequence $([a_i]_E,a_i)_{i\in\mathbb{Q}\cup\{\sqrt{2},\sqrt{3} \} }$, it is easy to see that $$ \tp([a_{\sqrt{2}}]_E/I)=\tp([a_{\sqrt{3}}]_E/I).$$ Hence, there is $\sigma\in\aut(\C/I)$ such that $f([a_{\sqrt{2}}]_E)=f([a_{\sqrt{3}}]_E)$. Taking $a'_{\sqrt{2}}=f(a_{\sqrt{3}})$ we obtain that the sequences $$ I_0+I_1+a'_{\sqrt{2}}+I_2 $$ and $$I_0+a_{\sqrt{2}}+I_1+I_2$$ are both indiscernible. However, the sequence $$ I_0+a_{\sqrt{2}}+I_1+a'_{\sqrt{2}}+I_2 $$ is not indiscernible since $[a_{\sqrt{2}}]_E=[a'_{\sqrt{2}}]_E$ but for any rationals $i\neq j<\sqrt{2}$ and $a_i,a_j\in I_0$ we have $[a_i]_E\neq  [a_j]_E$. This contradicts the distality of the theory $T$.
\end{proof}

\section{An example of $G^{\textrm{st},0}\neq G^{\textrm{st}}\neq G^{00}=G$}\label{section: main example}

Our objective is to find a definable group $G$ in a NIP theory $T$ satisfying $G^{\textrm{st},0}\neq G^{\textrm{st}}\neq G^{00}=G$. In this section, we will change the notation: the group interpreted in the monster model will be denoted by $G^*$ instead of $G$.

Consider the structure $\mathcal{M}:=(\mathbb{R},+,I)$, where $I:=[0,1]$. Let $T:=\Th(\mathcal{M})$ and $G:=(\mathbb{R},+)$. 
Let $\mathcal{M}^*=(\mathbb{R}^*,+,I^*) \succ \mathcal{M}$ be a monster model 
($\kappa$-saturated and strongly $\kappa$-homogeneous for large $\kappa$)
which expands to a monster model  $(\mathbb{R}^*,+,\leq ,1) \succ (\mathbb{R},+,\leq,1)$ (with the same $\kappa$), and $G^*:=(\R^*,+)$. Denote by $\mu$ the subgroup of infinitesimals, i.e. $\bigcap_{n \in \mathbb{N}^+} [-1/n,1/n]^*$. 


Some observations below may follow from more general statements in the literature, but we want to be self-contained and as elementary as possible in the analysis of this example.

\begin{prop}\label{proposition: NIP and unstable}
	T has NIP and is unstable.
\end{prop} 
\begin{proof}
	The structure $\mathcal{M}$ is the reduct of the o-minimal structure $(\mathbb{R},+,\leq ,1)$, hence $T$ has NIP. 
	
	Note that for $\varepsilon\in(0,\frac{1}{2})$ we can write the interval $[-\varepsilon,\varepsilon]$ as $(I-\varepsilon)\cap (I+(\varepsilon -1))$. Hence, the formula $\varphi(x,y):= \text{``}x\in I-y \wedge x\in I+(y-1)\text{''}$ has SOP.
\end{proof}

\begin{remark}\label{1 in dcl}
	$0,1 \in \dcl^{\mathcal{M}}(\emptyset)$.
\end{remark}
\begin{proof}
	$0 \in \dcl^{\mathcal{M}}(\emptyset)$ as the neutral element of $(\mathbb{R},+)$. To see that $1 \in \dcl^{\mathcal{M}}(\emptyset)$, note that $1$ is defined by the formula $\varphi(x):=\text{``}x\in I \wedge \forall y (y\in I \implies (x-y)\in I)\text{''}$.
\end{proof}

\begin{lema} \label{lemma: subgroups of G}
	\begin{enumerate}
		\item The only invariant subgroups of $G^*$ are: $\{0\}$, the subgroups of $\mathbb{Q}$, the subgroups of the form $\mu+R$ where $R$ is a subgroup of $\mathbb{R}$, and $G^*$.
		\item The only $\emptyset$-type-definable subgroups of $G^*$ are $\{0\}$, $\mu$, and $G^*$.
		\item The only definable (over parameters) subgroups of $G^*$ are $\{0\}$ and $G^*$.
	\end{enumerate}
\end{lema}

\begin{proof}
	$(1)$ 
	By q.e. in $(\R,+, \leq,1)^*$, Remark \ref{1 in dcl}, and the fact that the order restricted to any interval $[-r,r]$ (where $r \in \mathbb{N}^+$) is $\emptyset$-definable in $\mathcal{M}$ (see Lemma \ref{proposition: interdefinability}), the following holds in $\mathcal{M}^*$: 
	\begin{enumerate}[label=(\roman*)]
		\item All elements $a>\mathbb{R}$ have the same type over $\emptyset$.
		\item 
		$\dcl^{\mathcal{M}}(\emptyset)=\mathbb{Q}$.
		\item 
		For any $a\in \mathbb{R}\setminus\mathbb{Q}$, $a+\mu$ is the set of all realizations of a type in $S_1(\emptyset)$.
		\item For any $a \in \mathbb{Q}$,  all the elements $a+h$, where $h$ ranges over positive infinitesimals, form the set of realizations of a type in $S_1(\emptyset)$; and the same is true for all elements $a-h$. 
	\end{enumerate}
	This easily implies that the groups in the lemma are indeed invariant. 
	
	For the converse, let $H$ be an invariant subgroup. If it contains some $a>\mathbb{R}$ or $a < \mathbb{R}$, then $H=G^*$ by (i). So suppose that $H \subseteq \mathbb{R} + \mu$. 
	If $H$ contains an element from $a + \mu$ for some $a \in \mathbb{R} \setminus \mathbb{Q}$, then, by (iii), it contains $a + \mu$ and so $\mu$ as well; thus, $H$ is of the form $\mu + R$, where $R$ is a subgroup of $\mathbb{R}$.
	If $H$ contains some element $a+h$ with $a\in \mathbb{Q}$ and $h$ a positive infinitesimal, then, by (iv),  $H$ contains all the elements of that form. Since $H$ is a group, it contains $\mu$ (because we can subtract any two elements $a+h$, $a+h'$); thus, $H$ is again of the form $\mu + R$, where $R$ is a subgroup of $\mathbb{R}$.  If $H$ contains some element of the form $a-h$ with $a\in \mathbb{Q}$ and  $h$ a positive infinitesimal, we proceed in an analogous manner. The only remaining case is that $H$ is a subgroup of $\mathbb{Q}$.
	
	$(2)$ 
	A $\emptyset$-type-definable subgroup $H$ either contains some $a>\mathbb{R}$, in which case $H=G^*$, or the type defining $H$ implies the formula $x\leq n$ for some $n\in\mathbb{N}$. This implies that $H$ is contained in $\mu$, so, by (1), either $H=\{0\}$ or $H=\mu$.
	
	$(3)$ 
	By o-minimality of $(\R,+, \leq,1)^*$, any definable subgroup (over parameters) $H$ of $G^*$ is a finite union of points and intervals, so the conclusion easily follows. 	
\end{proof}


\begin{cor}
	$G^*=G^{*0}=G^{*00}=G^{*000}$
\end{cor}

\begin{proof}
	Since $G^*\geq G^{*0}\geq G^{*00}\geq G^{*000}$, it is enough to show that $G^{*000}=G^*$. But this follows from Lemma \ref{lemma: subgroups of G}(1), as the index of $\mathbb{R} + \mu$ in $G^*$ is unbounded.
\end{proof}

\begin{cor}
	${G^*}^{\textrm{st},0}=G^*$
\end{cor}

\begin{proof}
	It follows directly from Lemma \ref{lemma: subgroups of G}(3) and Proposition \ref{proposition: NIP and unstable}.
\end{proof}

Let $\mathcal{N}:=(\mathbb{R}, +,-, R_r)_{r\in \mathbb{N}^+}$, where $R_r(x,y)$ holds if and only if $0\leq y-x\leq r$. Let $T':=\Th(\mathcal{N})$. 

\begin{remark}\label{R} The family $(R_r)_{r\in \mathbb{N}^+}$ satisfy the following conditions:
	\begin{enumerate}
		\item $R_r(x,y)\iff R_r(0, y-x)$;
		\item $R_r(x,y)\iff R_r(-y,-x)$ ;
		\item $R_r(x,y)\iff R_{nr}(nx,ny)$ (where $n \in \mathbb{N}^+$);
		\item $R_r(x,y)\iff R_1(x,y) \vee R_1(x,y-1)\vee \dots \vee R_1(x,y-(r-1)) $.
	\end{enumerate}
\end{remark}

\begin{lema}\label{proposition: interdefinability}
	The structures $\mathcal{M}$ and $\mathcal{N}$ are interdefinable over $\emptyset$.
\end{lema}

\begin{proof}
	$\mathcal{M}$ is definable over $\emptyset$ in $\mathcal{N}$, because $x\in[0,1]$ if and only if $R_1(0,x)$ if and only if $R_1(x,2x)$.
	
	To see that $\mathcal{N}$ is definable over $\{1\}$ in $\mathcal{M}$, note that the function $-$ can be defined using $+$ as usual, $R_1(x,y)\iff y-x\in [0,1]$, and then use the last property in Remark \ref{R} to conclude that  all $R_r$, $r\in\mathbb{N}^+$, are definable over $\{1\}$ in $\mathcal{M}$.
	Since $1\in\dcl^{\mathcal{M}}(\emptyset)$, we conclude that $\mathcal{N}$ is definable over $\emptyset$ in $\mathcal{M}$.
\end{proof}

The result above shows that the theory $T'$ also has NIP and is unstable.

\begin{prop}\label{proposition: qe}
	$T'$ has quantifier elimination after expansion by the constant $1$.
\end{prop}

\begin{proof}
	We argue by induction on the length of the formula. So the proof boils down to showing that a primitive formula $(\exists y) \varphi(y,\bar x)$ is $T'$-equivalent to a quantifier free formula, assuming that all shorter formulas are $T'$-equivalent to qf-formulas. Recall that $(\exists y) \varphi(y,\bar x)$ being primitive means that $\varphi(y,\bar x)$ is a conjunction of atomic formulas and negations of such, i.e. $\varphi(y,\bar x) = \bigwedge_{j=1}^m R^{\varepsilon_j}_{r_j}(t^l_j(y,\bar x), t^r_j(y,\bar x))$, where $\varepsilon_j \in \{\pm 1,\pm 2\}$, $r_j \in \mathbb{N}^+$, $t^l_j(y,\bar x)$ and $t^r_j(y,\bar x)$ are terms, and: $R_{r_j}^{-2}(t,z):= \neg R_{r_j}(t,z)$, $R_{r_j}^2(t,z):= R_{r_j}(t,z)$, $R_{r_j}^{-1}(t,z): = \neg (t=z)$, $R_{r_j}^1(t,z): = (t=z)$.  
	
	Using $+$, $-$, multiplying by suitable integers, Remark \ref{R}, and induction hypothesis, we can assume  that there is an integer $n \ne 0$  such that for every $j$: either $t^l_j(y,\bar x) =0$ and $t^r_j(y,\bar x) = ny - t_j(\bar x)$, or $t^l_j(y,\bar x) = ny - t_j(\bar x)$ and  $t^r_j(y,\bar x)=0$.
	
	If some $\varepsilon_j=1$, one gets $ny= t_j(\bar x)$ and the quantifier $\exists y$ can be eliminated. So assume that all $\varepsilon_j \ne 1$. If additionally all $\varepsilon_j \ne 2$, then $(\exists y) \varphi(y,\bar x)$ is $T'$-equivalent to $\top$. So assume that some $\varepsilon_j= 2$, e.g. $\varepsilon_1=2$. Then $R^{\varepsilon_1}_{r_1}(t^l_1(y,\bar x), t^r_1(y,\bar x))$ either says (in $\mathcal{N}$) that $ny \in [t_1(\bar x) -r_1, t_1(\bar x)]$, or that $ny \in [t_1(\bar x), t_1(\bar x) +r_1]$. Suppose the latter case holds.  Consider all (finitely many) possibilities taking into account: 
	\begin{itemize}
		\item which terms $t_j(\bar x)-r_j, t_j(\bar x), t_j(\bar x) +r_j$ (for $j\in \{2,\dots,m\}$) belong to $[t_1(\bar x) -r_1, t_1(\bar x)]$; 
		\item for those which belong to this interval, how they are ordered by $R_{r_1}$;
		\item for $\varepsilon_j \ne -1$, writing  $R^{\varepsilon_j}_{r_j}(t^l_j(y,\bar x), t^r_j(y,\bar x))$ as ``$ny \in I$'' or as ``$ny \notin I$'', where $I:= [t_j(\bar x) -r_j, t_j(\bar x)]$ or $I:=[t_j(\bar x), t_j(\bar x) +r_j]$, we should specify which of the terms $t_1(\bar x)$ and $t_1(\bar x) +r_1$ belong to $I$.
	\end{itemize}
	
	Each of these possibilities is clearly a qf-definable condition on $\bar x$ (using finitely many integers, but they are terms, as 1 was added to the language). On the other hand, $(\exists y) \varphi(y,\bar x)$ is $T'$-equivalent to the disjunction of some subfamily of these conditions (by a simple combinatorics on intervals). Therefore, $(\exists y) \varphi(y,\bar x)$ is $T'$-equivalent to a qf-formula.
\end{proof}

\begin{prop}\label{proposition: stable quotient}
	The quotient $G^*/\mu$ is stable (in $\mathcal{M}$, equivalently in $\mathcal{N}$).
\end{prop}

\begin{proof}
	
	By Theorem \ref{equivteor}, it is enough to show that for every $A \subseteq \mathcal{M}^*$ with $|A| \leq \mathfrak{c}$ we have $|S_{G^*/\mu}(A)| \leq \mathfrak{c}$. 
	
	By Lemma \ref{proposition: interdefinability}, $\mathcal{M}^*$ can be treated as an elementary extension of $\mathcal{N}$.
	
	Consider an arbitrary set $A$ as above. Put $V_A:=\Lin_{\mathbb{Q}}(A\cup \mathbb{R})$ and $\tilde{V}_A:=V_A+\mu$. First, note that for any $a,a'\in G^*$, if $a-a'\in \mu$, then trivially $\tp((a+\mu)/A)=\tp((a'+\mu)/A$). Now, consider any $a\in G^*\setminus \tilde{V}_A$. Then $a$ satisfies the formulas $$kx\neq t$$ and $$\neg R_r(t,kx)$$ for all $k \in \mathbb{Z} \setminus \{0\}$, $r\in \mathbb{N}^+$, and $t\in V_A$. By Remark \ref{R}, Lemma \ref{proposition: interdefinability} and Proposition \ref{proposition: qe}, these formulas completely determine $\tp(a/V_A)$. Hence, any $a,a'\in G^*\setminus \tilde{V}_A$ have the same type over $A$. Therefore, $|S_{G^*/\mu}(A)| \leq |\tilde{V}_A/\mu| +1 \leq|V_A| +1 = \mathfrak{c}$.	
\end{proof}

\begin{cor}
	$\mu= G^{*st}$.
\end{cor}

\begin{proof}
	It follows from Lemma \ref{lemma: subgroups of G}(2), Proposition \ref{proposition: stable quotient}, and Proposition \ref{proposition: NIP and unstable}.
\end{proof}
To summarize, we have proved that $G^*$ is a $\emptyset$-definable group in a monster model of a NIP theory such that ${G^*}^{\textrm{st}}\neq {G^*}^{00}$ and ${G^*}^{\textrm{st},0}={G^*}^{00}=G^*$.

\section{How to construct examples with $G^{\textrm{st}} \ne G^{\textrm{st},0}$?}\label{section 5}

Our context is that $G$ is a $\emptyset$-definable group in a monster model $\C$ of a complete theory $T$ with NIP. We use the symbol $\dcap$ to indicate that it is a directed intersection.

\begin{prop}\label{proposition: based on Hrushovski}
	If $G^{\textrm{st}}=G^{\textrm{st},0}$, then any type-definable subgroup $H$ of $G$ with $G/H$ stable is an intersection of definable subgroups.
\end{prop}

\begin{proof}
	Assume that $G^{\textrm{st}}=G^{\textrm{st},0}=\dcap_{j\in J}G_j$, where $G_j$ are definable groups. Since $G/G^{\textrm{st},0}$ is a stable hyperdefinable group, by \cite[Remark 2.5(iv)]{MR3796277}, the intersection of all the conjugates of each $G_j$ is a bounded subintersection. By compactness, there is a chain of unbounded length $H_{1}\supset H_{2}\supset \cdots$ consisting of conjugates of $G_j$ such that $H_i\subset \bigcap_{j<i} H_j$ if and only if there is such a chain of length $\omega$. Thus, the intersection of all the conjugates of each $G_j$ is a finite subintersection.  Replacing each $G_j$ by such a finite subintersection, we can assume that all the $G_j$'s are normal subgroups of $G$. 
	Hence, for every $j\in J$ we have  $G_j\leq H\cdot G_j\leq G$ and  
	$$\bigslant{H\cdot G_j}{G_j}\leq \bigslant{G}{G_j}.$$
	Using Hrushovski's theorem (see \cite[Ch. 1, Lemma 6.18]{pillay1996geometric}) inside the (definable) stable group $G/G_j$, we get $$\bigslant{H\cdot G_j}{G_j} = \dcap_{i\in I_j} \bigslant{ K^i_j}{G_j},$$ for some definable subgroups $K_j^i$ of $G$ such that $K_j^i\cdot G_j=K^i_j$ for all $i \in I_j$.
	
	Since $G/H$ is assumed to be stable, $G^{\textrm{st}} \leq H$. Thus,
	$$H=H \cdot \dcap_{j\in J}G_j= \dcap_{j\in J}H\cdot G_j=\dcap_{j\in J}\dcap_{i\in I_j}K^i_j.$$
\end{proof}

If we do not require that $G^{00} = G^0$, then it is easy to find
examples where $G^{00} \ne G^{\textrm{st}} \ne G^{\textrm{st},0}$; that is why  \cite{MR3796277}
required in this problem also $G^{00}=G$. However, the requirement $G^{00} =
G^0$ seems sufficiently interesting. The next proposition yields a whole class of examples where $G^{00} \ne G^{\textrm{st}} \ne G^{\textrm{st},0}$ (but without the requirement that $G^{00}=G^0$).

\begin{prop}\label{4.10}
	Let $G$ be definably isomorphic to a definable semidirect
	product of definable groups $H$ and $K$ (symbolically, $G
	\cong_{def} H \ltimes K$) such that $H^{00}\ne H^0$ and $K^{\textrm{st}} \ne
	K^{00}$. Then $G^0 \ne G^{00} \ne G^{\textrm{st}} \ne G^{\textrm{st},0}$.
\end{prop}

\begin{proof}
	W.l.o.g. $G = H \ltimes K$ and $H$, $K$, and the action of $H$ on $K$ are $\emptyset$-definable. Recall that by NIP, the 00-components exist (i.e. do not depend on the choice of parameters over which they are computed). Hence, $K^{00}$ is invariant under all definable automorphisms, in  particular under the action of $H$. So $G^{00} = H^{00} \ltimes K^{00}$ (e.g. by Corollary 4.11 in \cite{gismatullin2020bohr}).
	But $G^{\textrm{st}} \leq H^{00} \ltimes K^{\textrm{st}}$, because the map $(h,k)/ (H^{00} \ltimes K^{\textrm{st}}) \mapsto h/H^{00} \times k/K^{\textrm{st}}$ is an invariant bijection from 
	$G / (H^{00} \ltimes K^{\textrm{st}})$ to $H/H^{00} \times K/K^{\textrm{st}}$ and the last set is stable as a product of stable sets.
	Thus, since $K^{\textrm{st}}$
	is a proper subgroup of $K^{00}$, we get that $G^{00} \ne G^{\textrm{st}}$.
	
	To see that $G^{\textrm{st}} \ne G^{\textrm{st},0}$,  it is enough to note that $H^{00} \ltimes K \leq
	G$ and that  $H^{00} \ltimes K$ is not an
	intersection of definable groups (because $G/(H^{00} \ltimes K) \cong_{def}
	H/H^{00}$ is not profinite in the logic topology as $H^{00} \ne H^0$). Indeed, having this, since $G/(H^{00} \ltimes K)$  is bounded and so stable, by Proposition \ref{proposition: based on Hrushovski}, we conclude that  $G^{\textrm{st}} \ne G^{\textrm{st},0}$.
	
	The fact that $G^{00} \ne G^{0}$ follows  from $H^{00}\ne H^0$, as $G^{0}= H^0 \ltimes K^0$.
\end{proof}

\begin{remark}
	The assumption of Proposition \ref{4.10} is equivalent to saying
	that $G$ has a definable, normal  subgroup $K$ with $K^{\textrm{st}} \ne
	K^{00}$
	and $(G/K)^{00} \ne (G/K)^{0}$ such that the quotient map $G \to
	G/K$ has a section which is a definable homomorphism.
\end{remark}

The proof of Proposition \ref{4.10} can be easily modified to get the following variant.

\begin{remark}
	The conclusion of Proposition \ref{4.10} remains true with the assumption ``$H^{00}\ne H^0$ and $K^{\textrm{st}} \ne
	K^{00}$'' replaced by ``$H^{\textrm{st}}\ne H^{00}$ and $K^{00} \ne
	K^{0}$''.
\end{remark}

One can find many examples satisfying the assumptions of Proposition \ref{4.10}.
For instance, take any group $H$ (definable in a monster model of a NIP
theory $T_1$) with $H^{00} \ne H^0$ (e.g. the circle group in the
theory of real closed fields) and any group $K$  (definable in a monster
model of a NIP theory $T_2$; where $T_1$ and $T_2$ are in disjoint
languages) with $K^{\textrm{st}} \ne K^{00}$ (e.g. $T_2$ is stable and $K$ is
infinite). Consider $T$ being the union of $T_1$ and $T_2$ living on two
disjoint sorts. Then $G:=H \times K$ satisfies the assumptions of Proposition \ref{4.10}
as a group definable in $T$.

One could still ask if it is possible to find examples satisfying the condition $G^0 =
G^{00} \ne G^{\textrm{st}} \ne G^{\textrm{st},0}$ by finding a definable, normal subgroup
$K$ satisfying $K^{00}\ne K^0$ where $G \to G/K$ does not have a definable
section. 
However, there is no chance for this potential method to work for groups
of finite exponent, as for any torsion (equivalently finite exponent)
group $K$ definable in a monster model, we have $K^{00}=K^0$. This is
because $K/K^{00}$ is a compact torsion group, and such groups are known
to be profinite (see \cite[Theorem 8.20]{hewitt2013abstract}). 


\begin{problem}
	Construct $G$ of finite exponent with $G^{00} \ne G^{\textrm{st}} \ne G^{\textrm{st},0}$. (The equality $G^{0}=G^{00}$ always holds by the fact at the end of the last paragraph.)
\end{problem}

In the final part of this section, we describe how one could try to construct examples where  $G^{00} \ne G^{\textrm{st}} \ne G^{\textrm{st},0}$. In fact, originally we used this approach to find the example in Section \ref{section: main example}. We will also point out a difference between the situation in the example from Section \ref{section: main example} and the finite exponent case.

\begin{prop}\label{4.1}
	The conditions $G^{\textrm{st},0}\neq G^{\textrm{st}}$ and $G^{\textrm{st}}\neq G^{00}$ are equivalent to the existence of a type-definable subgroup $H$ of  $G$ such that:
	\begin{enumerate}
		\item $H$ is a countable  intersection $\dcap_{n < \omega} D_n$ of definable subsets of $G$ satisfying $D_{n+1}D_{n+1}\subseteq D_{n}$ and symmetric (i.e. $D_{n}^{-1}=D_{n}$ and $e \in D_n$);
		\item $[G:H]$ is unbounded;
		\item $H$ is not an intersection of definable subgroups;
		\item $G/H$ is stable.
	\end{enumerate}
\end{prop}
\begin{proof}
	If $G^{\textrm{st}}\neq G^{00}$, then, by compactness, $G^{\textrm{st}}=\dcap \{H: H \text{ satisfies (1), (2), (4)}\}.$ Hence, assuming additionally that $G^{\textrm{st},0}\neq G^{\textrm{st}}$, at least one of those groups $H$ has to also satisfy condition $(3)$.
	
	Assume now that $H\leq G$ satisfies conditions $(1)$, $(2)$, $(3)$, and $(4)$. Since $G^{\textrm{st}} \leq H$, we get $G^{\textrm{st}}\neq G^{00}$. The fact that $G^{\textrm{st},0}\neq G^{\textrm{st}}$ follows from Proposition \ref{proposition: based on Hrushovski}. 
\end{proof}

Note that assuming (1), the negation of (3) is equivalent to saying that for every $n<\omega$ there is $m>n$ and a definable subgroup $K$ of $G$ such that $D_m \subseteq H \subseteq D_n$.

\begin{remark}
	If we have a situation as in the last proposition, then the same holds for $G$ treated as a group definable in $(G,\cdot, (D_n)_{n<\omega})$.
\end{remark}

So an idea is to look for a group $G$ and a decreasing sequence $(D_n)_{n<\omega}$ of symmetric subsets of $G$ with $D_{n+1}D_{n+1} \subseteq D_n$ for all $i<\omega$, such that for $M:=(G,\cdot, (D_n)_{n<\omega}))$ and $G^*:= G(\C)$ (where $\C =(G^*,+, (D_n^*)_{n<\omega}) \succ M$ is a monster model), the group $H:=\dcap_{n < \omega} D_n^*$ satisfies (1)-(4) from the last proposition (with $*$ added everywhere). In the example from Section \ref{section: main example}, $G := (\R,+)$ and as $D_n$ we can take $[-1/2^n,1/2^n]$. Then the $D_n$'s are definable in $(\R,+,[0,1])$,
hence $M$ is interdefinable with $(\R,+,[0,1])$, and so we focused on the latter structure. In the proof of stability of $G^*/\mu$ (see Proposition \ref{proposition: stable quotient}), for the counting argument to work it was important that $D_{n+1}$ is generic in $D_{n}$ (i.e. finitely many translates of $D_{n+1}$ cover $D_n$), 
as this guarantees that $\mu = \bigcap D_n^*$ has bounded index in the subgroup generated by $D_1^*$.  The next proposition shows that for abelian groups of finite exponent this genericity condition always fails.


\begin{prop}
	If $G$ is abelian of finite exponent, then there is no sequence $(D_n)_{n<\omega}$ of definable sets such that: \begin{enumerate}
		\item $D_n$ is symmetric and  $D_{n+1}+D_{n+1}\subseteq D_n$ for all $n<\omega$;
		\item $D_{n+1}$ is generic in $D_n$ for all $n<\omega$;
		\item $\dcap_{n<\omega} D_n$ is not an intersection of definable groups. 
	\end{enumerate}
\end{prop}
\begin{proof}
	Assume that there is such a sequence  $(D_n)_{n<\omega}$ of definable sets. Replacing $D_n$ by $D_{n+1}$ if necessary, we can assume that $D_0$ is an approximate subgroup (i.e. finitely many translates of $D_0$ cover $D_0+D_0$), because $D_1$ is an approximate subgroup by (1) and (2).
	We denote $D_0^{+n}:=D_0+\overset{n}{\dots}+D_0$. Then, $\langle D_0\rangle =\acup\limits_{n<\omega}D_0^{+n}$ is a $\bigvee$-definable group and, by (1), (2), and the assumption that $D_0$ is an approximate subgroup, we see that $\dcap_{k<\omega}D_k\leq \langle D_0\rangle$ is a type-definable subgroup of bounded index. Hence, $$H:=\bigslant{\langle D_0\rangle}{\dcap_{k<\omega}D_k}$$ is a locally compact group with the logic topology (in which closed sets are defined as those whose preimages under the quotient map have type-definable intersections with all sets $D_0^{+n}$, $n<\omega$; see \cite[Lemma 7.5]{MR2373360}). 
	Since $H$ is a torsion group, it follows from \cite[Theorem 3.5]{MR637201} that $H$ has a basis  $(H_i)_{i\in I}$ of neighbourhoods of the identity consisting of clopen subgroups. 
	Since each $\bigslant{D_n}{\dcap_{k<\omega}D_k}$ is a neighborhood of the identity, there is $H_n\subseteq \bigslant{D_n}{\dcap_{k<\omega}D_k}$ which is a clopen subgroup of $H$. Let $\pi: \langle D_0\rangle \to H$ be the quotient map. Then, $$\pi^{-1}[H_n]\subseteq D_n+\dcap_{k<\omega}D_k\subseteq D_n+D_n\subseteq D_{n-1}$$ is a type-definable group. Since $\pi^{-1}[H_n]^c\cap D_{n-1}$ is also type-definable, we deduce that $\pi^{-1}[H_n]$ is a definable group laying between $\dcap_{k<\omega} D_k$ and $D_{n-1}$. 
	Since this is true for any $n>0$, we get a contradiction with $(3)$.
\end{proof}

The following corollary yields some hints on how an example of finite exponent could be constructed.

\begin{cor}
	If $G$ is abelian of finite exponent, then the condition $G^{\textrm{st},0}\neq G^{\textrm{st}}\neq G^{00}$ is equivalent to the existence of a sequence $(D_n)_{n<\omega}$ of definable sets such that:
	\begin{enumerate}
		\item $D_n$ is symmetric and  $D_{n+1}+D_{n+1}\subseteq D_n$, for all $n<\omega$;
		\item $D_{n+1}$ is not generic in $D_n$ for all $n<\omega$;
		\item $\dcap_{n<\omega} D_n$ is not an intersection of definable groups;
		\item $[G: \dcap_{n<\omega} D_n]$ is unbounded;
		\item $\bigslant{G}{\dcap_{n<\omega} D_n}$ is stable.
	\end{enumerate}
\end{cor}
\begin{proof}
	From Proposition \ref{4.1}, we obtain that the condition $G^{\textrm{st},0}\neq G^{\textrm{st}}\neq G^{00}$ is equivalent to the existence of a sequence  $(D_n)_{n<\omega}$ satisfying  $(1)$, $(3)$, $(4)$, and $(5)$. Furthermore, by the previous proposition, such a sequence $(D_n)_{n<\omega}$ must contain an (infinite) subsequence satisfying $(2)$. 
\end{proof}


\begin{remark}
	This section could be naturally generalized to the context of a type-definable group $G$. This would require checking a few things, mainly that Hrushovski's theorem (i.e. \cite[Ch. 1, Lemma 6.18]{pillay1996geometric}) is valid for a stable type-definable group (not necessarily living in a stable theory). We leave it to the reader.
\end{remark}

\chapter{Maximal stable quotients of invariant types in NIP theories} \label{Chapter 4}

We present the general framework of this chapter.	
	Let $T$ be a complete first-order theory 
	of infinite models in a language $L$. Let $\C\prec\C'$ be models of  $T$ such that $\C$ is $\kappa$-saturated 
with a strong limit cardinal $\kappa>\lvert T\rvert$,  and $\C'$ is $\kappa'$-saturated and strongly $\kappa'$-homogeneous with a strong limit cardinal $\kappa'>\lvert \C\rvert$. We say that $\kappa$ is the {\em degree of saturation of $\C$} and $\kappa'$ is the {\em degree of saturation of $\C'$}. We say that a set is {\em $\C$-small} if its cardinality is smaller than $\kappa$ and {\em $\C'$-small} if its cardinality is smaller than $\kappa'$.
	Note that $|T|$ is the cardinality of the set of all formulas in $L$.
	Unless stated otherwise, $p(x)$ will always be a type in $S_x(\C)$ invariant over some $\C$-small $A\subseteq \C$, where $x$ is a $\C$-small tuple of variables. (In fact, instead of assuming that $\kappa$ is a strong limit cardinal, in Section \ref{section: basics} it is enough to assume that $\kappa >2^{|T|+|A|}$ and in Section \ref{section: main theorem} that $\kappa \geq \beth_{(2^{2^{|T| + |A|} +\lvert x \rvert})^+}$.) Whenever $B \subseteq \C'$, by $p \!\upharpoonright_B$ we mean the restriction to $B$ of the unique extension of $p$ to an $A$-invariant type in $S(\C')$. If $E$ is a type-definable equivalence relation and $a$ is an element of its domain, $[a]_E$ denotes the $E$-class of $a$.

\section{Basic results and transfers between models}\label{section: basics}

The goal of this section is to present a useful criterion that allows us to check whether a relatively type-definable over a $\C$-small $B\subseteq \C$ equivalence  relation $E$  on $p(\C')$ with stable quotient is, in fact, the finest one (see Lemma \ref{equivalence for being the finest}). As a corollary, we get the transfer to elementary extensions of $\C$ of the property of being the finest relatively type-definable equivalence relation on $p(\C')$ (see Corollary \ref{corollary: from C to C_1}). 

We present a definition that we use throughout the whole section. This definition first appeared in \cite[Definition 3.2]{doi:10.1142/S0219061319500120}.
\begin{defin}\label{strong heirs}
	Let $A\subseteq \mathcal{M}\subseteq B$ and $q(x)\in S(B)$.
	We say that $q(x)$ is a \emph{strong heir extension over A} of $q\!\upharpoonright_\mathcal{M}(x)$ if   for all finite $m\subseteq \mathcal{M}$ 
	$$(\forall \varphi(x,y)\in L)(\forall b\subseteq B)[ \varphi(x,b)\in q(x) \implies (\exists b'\subseteq \mathcal{M})(\varphi(x,b')\in q(x)\wedge b\underset{Am}{\equiv}b') ].$$
\end{defin}

Note that if $q \in S(\C)$ is a strong heir extension over $A$ of $q\!\upharpoonright_\mathcal{M}(x)$, then $\mathcal{M}$ is an $\aleph_0$-saturated model in the language $L_A$ (i.e., $L$ expanded by constants from $A$). The converse is also true (see \cite[Lemma 3.3]{doi:10.1142/S0219061319500120} for a proof):
\begin{fact}\label{strong heir exists}
    If $\mathcal{M}$ is an $\aleph_0$-saturated model in $L_A$ and $q(x)\in S(\mathcal{M})$, there always exists $q'(x)\in S(B)$ which is a strong heir over $A$ of $q$  
\end{fact}

\begin{lema}\label{strong heir invariance}
	Assume that $q(x)\in S(\mathcal{M})$ is $A$-invariant (for some $A\subseteq \mathcal{M}$) and $q'(x)\in S(\mathfrak{C})$ is a strong heir extension over $A$ of $q(x)$. 
	Then $q'(x)$ is the unique global $A$-invariant extension of $q(x)$.
\end{lema}
\begin{proof}
	
	To show $A$-invariance, suppose for a contradiction that there are $a,b$ and $\varphi(x,a)\in q'(x)$ with $a \underset{A}\equiv b$ and $\neg \varphi(x,b)\in q'(x)$. Then, there exist $a',b'\in \mathcal{M}$ such that $a'\underset{A}{\equiv} a$ and $b' \underset{A}{\equiv} b$ for which $\varphi(x,a') \in q(x)$ while $\neg \varphi(x,b')\in q(x)$. Then $a' \underset{A}{\equiv} b'$, which contradicts the $A$-invariance of $q(x)$.
	
	Uniqueness follows from the fact that $\mathcal{M}$ is $\aleph_0$-saturated in $L_A$.
\end{proof}

Given a partial type (possibly over a non-small set of parameters from $\C$) $\pi(x,y)$, we say that $\pi(x,y)$ \emph{relatively defines an equivalence relation} on a type-definable (over an arbitrary set of parameters from $\C$) set $X$ if $\pi(\C',\C')\cap X(\C')^2$ is an equivalence relation. 
Given a type-definable equivalence relation $E$ on a type-definable set $X$, a \emph{partial type relatively defining} $E$ is any partial type $\pi(x,y)$ such that $\pi(\C',\C')\cap X(\C')^2 = E$. 
We say that a type-definable equivalence relation $E$ on a type-definable set $X$ is \emph{countably relatively defined} if some partial type $\pi(x,y)$ relatively defining it consists of countably many formulas, and we say that $E$ is {\em relatively type-definable over $B$} (or {\em $B$-relatively type-definable}) if it is relatively defined by a partial type over $B$.

Lemma \ref{reduction to a small model} gives us a useful stability criterion when an equivalence relation on $p(\C')$ is relatively type-definable over a sufficiently saturated model.

\begin{lema}\label{reduction to a small model}
	Let $\mathcal{M}\prec \mathfrak{C}$ be $\aleph_0$-saturated in $L_A$, and $\pi(x,y)$ a partial type over $\mathcal{M}$ relatively defining an equivalence relation on $p\!\upharpoonright_{\mathcal{M}}(\mathfrak{C}')$. Then, $\pi$ relatively defines an equivalence relation on $p(\mathfrak{C}')$ with stable quotient if and only if it relatively defines an equivalence relation on $p\!\upharpoonright_{\mathcal{M}}(\mathfrak{C}')$ with stable quotient.
\end{lema}

\begin{proof}
	Firstly, note that since $\pi$ relatively defines an equivalence relation on the set $p\!\upharpoonright_{\mathcal{M}}\!\!(\C')$,  it relatively defines an equivalence relation on $p(\C')$. Let $E$ be the equivalence relation relatively defined by $\pi(x,y)$ on $p(\C')$ and let $E'$ be the equivalence relation relatively defined by $\pi(x,y)$ on $p\!\upharpoonright_{\mathcal{M}}\!\!(\C')$.
	
	Assume first that $p(\C')/E$ is unstable. Then, there exists a $\C$-indiscernible sequence $(c_i,b_i)_{i<\omega}$ such that $c_i\in p(\C')$ for all $i<\omega$ and for all $i\neq j$ 
	$$\tp\left(\bigslant{[c_i]_{E},b_j}{\C}\right)\neq \tp\left(\bigslant{[c_j]_{E},b_i}{\C}\right).$$ 
	This implies that for all $i\neq j$ we have $$ \tp\left(\bigslant{[c_i]_{E'},b_j}{\C}\right)\neq \tp\left(\bigslant{[c_j]_{E'},b_i}{\C}\right),$$
	and so  $p\!\upharpoonright_{\mathcal{M}}\!\!(\mathfrak{C}')/E'$ is unstable.
	
	Assume now that $p\!\upharpoonright_{\mathcal{M}}\!\!(\mathfrak{C}')/E'$ is unstable. This is witnessed by an $\mathcal{M}$-indiscernible sequence $(c_i,b_i)_{i<\omega}$ such that $c_i\in p\!\upharpoonright_{\mathcal{M}}\!\!(\C')$ for all $i<\omega$ and for all $i\neq j$ $$ \tp\left(\bigslant{[c_i]_{E'},b_j}{\mathcal{M}}\right)\neq \tp\left(\bigslant{[c_j]_{E'},b_i}{\mathcal{M}}\right).$$ 
	Consider $q:=\tp\left(\bigslant{(c_i,b_i)_{i<\omega}}{\mathcal{M}}\right)$ and let $q'\in S(\C)$ be a strong heir extension over $A$ of $q$, which exists by Fact \ref{strong heir exists}. Let $(c_i',b_i')_{i<\omega}$ be a realization of $q'$. Then, 
	\begin{enumerate}
		\item $(c_i',b_i')_{i<\omega}$ is $\C$-indiscernible;
		\item $\tp(c_i'/\C)=p(x)$ for all $i<\omega$;
		\item $\tp\left(\bigslant{[c'_i]_{E},b'_j}{\C}\right)\neq \tp\left(\bigslant{[c_j']_{E},b'_i}{\C}\right)$ for all $i \ne j$.
	\end{enumerate}
	
	$(1)$ Suppose for a contradiction that $(1)$ does not hold. Then, it is witnessed by a formula (with parameters $d$ from $\C$) of the form $\varphi(x_{i_1},y_{i_1},\dots,x_{i_n},y_{i_n},d)\wedge \neg \varphi(x_{j_1},y_{j_1},\dots,x_{j_n},y_{j_n},d)$, for some $i_1<\dots<i_n$ and $j_1<\dots<j_n$. 
	Now, using that $q'$ is a heir extension (we do not need that it is in fact strong heir extension) over $A$ of $q$, we can find $d'\subseteq \mathcal{M}$ such that $$\varphi(x_{i_1},y_{i_1},\dots,x_{i_n},y_{i_n},d')\wedge \neg \varphi(x_{j_1},y_{j_1},\dots,x_{j_n},y_{j_n},d')\in q,$$ contradicting the $\mathcal{M}$-indiscernibility of  $(c_i,b_i)_{i<\omega}$. 
	
	$(2)$ follows from the fact that $\tp(c_i'/\C)$ is a strong heir extension over $A$ of $p\!\upharpoonright_{\mathcal{M}}\!\!(x)$, which has to be $p(x)$ by Lemma \ref{strong heir invariance}. 
	
	$(3)$ Suppose $(3)$ fails for some $i \ne j$. As $E$ is the restriction to $p(\C)$ of the equivalence relation $E'$, we see that $\tp\left(\bigslant{[c'_i]_{E},b'_j}{\C}\right)= \tp\left(\bigslant{[c_j']_{E},b'_i}{\C}\right)$ implies $\tp\left(\bigslant{[c'_i]_{E'},b'_j}{\C}\right)= \tp\left(\bigslant{[c_j']_{E'},b'_i}{\C}\right)$. So $\tp\left(\bigslant{[c'_i]_{E'},b'_j}{\mathcal{M}}\right)= \tp\left(\bigslant{[c_j']_{E'},b'_i}{\mathcal{M}}\right)$,
	and hence $\tp\left(\bigslant{[c_i]_{E'},b_j}{\mathcal{M}}\right)= \tp\left(\bigslant{[c_j]_{E'},b_i}{\mathcal{M}}\right)$ because $q \subseteq q'$ and $\pi(x,y)$ is over $\mathcal{M}$. This is a contradiction.
	
	By (1), (2), and (3),  $p(\C')/E$ is unstable.
\end{proof}

Even though at first glance the requirement that $\pi(x,y)$ relatively defines an equivalence relation on $p\!\upharpoonright_{\mathcal{M}}\!\!(\mathfrak{C}')$ might not seem very natural, the following result shows that this can always be assumed.


\begin{prop}\label{decomposition of equivalence relations}
	Let $E$ be a $B$-relatively type-definable equivalence relation on $p(\C')$, for some $B \subseteq \C$. 
	Then, 
	$E=\bigcap_{i\in I}E_i \cap p(\C')^2$, where $|I| \leq |B| +|x| + |T|$, and for
	each $i\in I$ there is a countable $B_i\subseteq \C$ such that $E_i$ is a countably $B_i$-relatively defined equivalence relation on $p\!\upharpoonright_{B_i}\!\!(\C')$. 
	Thus, $E$ is the restriction to $p(\C')$ of a $B'$-type-definable equivalence relation $F$ on $p\!\upharpoonright_{B'}\!\!(\C')$ for some $B'\subseteq \C$ with $\lvert B'\rvert \leq |B| +|x| + |T|$.
	
	Moreover, if we start from a given partial type $\pi(x,y)$ over $B$ relatively defining $E$, 
	then $B'$ and $F$ in the previous sentence can be taken so that $|B'|\leq |\pi|+\aleph_0$ and $F$ is $B'$-type-definable on $p\!\upharpoonright_{B'}\!\!(\C')$ by $\pi(x,y)$.
\end{prop}

\begin{proof}
	Fix a partial type $\pi(x,y)$ relatively defining $E$ on $p(\C')$. It clearly consists of reflexive formulas and without loss of generality it is closed under conjunction. Let $\psi_0(x)$ be any formula in $p(x)$ and $\varphi_0(x,y)$ any formula in $\pi(x,y)$. 
	Then the partial type $$p(x)\wedge p(y)\wedge p(z)\wedge \pi(x,y)\wedge \pi(y,z)$$ implies 
	$\varphi_0(x,z)\wedge \varphi_0(z,x)$. By compactness, there are $\varphi_1(x,y)$ in $\pi(x,y)$ and $\psi_1(x)$ in $p(x)$ such that the formula $$\psi_1(x)\wedge\psi_1(y)\wedge \psi_1(z) \wedge \varphi_1(x,y)\wedge \varphi_1(y,z)$$ implies $\varphi_0(x,z) \wedge \varphi_0(z,x).$ Proceeding by induction, we construct a partial type 
	$$\{\varphi_i(x,y): i<\omega\}$$ relatively defining an equivalence relation on $\bigcap_{i<\omega}\psi_i(\C')$. 
	Let $B_{\varphi_0,\psi_0}$ be a countable set containing the parameters of all the constructed formulas $\varphi_i(x,y)$ and $\psi_i(x)$, $i < \omega$.
	Then, the partial type $\{\varphi_i(x,y): i<\omega\}$ clearly relatively defines over $B_{\varphi_0,\psi_0}$ an equivalence relation on $p\!\upharpoonright_{ B_{\varphi_0,\psi_0}}\!\!(\C')$. 
	Applying this process separately to every $\varphi(x,y) \in \pi(x,y)$ yields the desired family of equivalence relations.	
\end{proof}

\begin{cor}\label{existence of a small model set of parameters}
	Let $E$ and $\pi(x,y)$ be as in Proposition \ref{decomposition of equivalence relations}, where $B$ is $\C$-small. Then there is $\mathcal{M}\prec \C$ containing $B$ with $\lvert \mathcal{M} \rvert \leq 2^{\lvert T \rvert + \lvert A \rvert}+ \lvert B \rvert + \lvert \pi \rvert$ which is $\aleph_0$-saturated in $L_A$ and such that $\pi(x,y)$ relatively defines an equivalence relation on $p\!\upharpoonright_{\mathcal{M}}\!\!(\C')$.
\end{cor}


The following result is a criterion for when an equivalence relation on $p(\C')$ relatively type-definable over a sufficiently saturated $\C$-small model is the finest relatively type-definable equivalence relation over a $\C$-small $B\subseteq \C$ on $p(\C')$ with stable quotient.

\begin{lema}\label{equivalence for being the finest}
	Let $\mathcal{M}$ and $\pi(x,y)$ be as in Lemma \ref{reduction to a small model}, and assume that $\mathcal{M}$ is $\mathfrak{C}$-small. Then $\pi$ relatively defines the finest relatively type-definable over a $\C$-small subset of $\C$ equivalence relation on $p(\C')$ with stable quotient if and only if it relatively defines the finest $\mathcal{M}'$-type-definable equivalence relation on $p\!\upharpoonright_{\mathcal{M}'}\!\!(\mathfrak{C}')$ with stable quotient for every $\mathcal{M}'\prec \C'$ with $\lvert \mathcal{M}'\rvert \leq 2^{\lvert T \rvert + \lvert A \rvert}+ \lvert \mathcal{M} \rvert $ that is $\aleph_0$-saturated in $L_A$ and contains $\mathcal{M}$.
\end{lema}

\begin{proof}
	Let $E$ be the equivalence relation relatively defined by $\pi$ on $p(\C')$ and $E'$ be the equivalence relation relatively defined by $\pi$ on $p\!\upharpoonright_{\mathcal{M}}\!\!(\mathfrak{C}')$.

	($\Leftarrow$) By Lemma \ref{reduction to a small model}, the right hand side implies that $E$ has stable quotient. Assume that there exists $E_B$, a relatively type-definable equivalence relation on $p(\C')$ over some $\C$-small set of parameters $B\subseteq \C$ such that the quotient $p(\C')/E_B$ is stable and $E_B\subsetneq E$. 
	Take a presentation of $E_B$ as $\bigcap_{i\in I}E_i \cap p(\C')^2$ satisfying the conclusion of Proposition \ref{decomposition of equivalence relations}. Abusing notation, write $E_i$ for $E_i \cap p(\C')^2$. As $E_B\subsetneq E$, there exists some $i\in I$ such that  $$E\cap E_i\subsetneq E.$$ 
	Since $p(\C')/E_B$ is stable and $E_B\subseteq E\cap E_i$, we have that $p(\C')/E\cap E_i$ is stable. 
	Pick $B_i$ as in  Proposition \ref{decomposition of equivalence relations} and choose any $\mathcal{M}'\supseteq \mathcal{M} \cup B_i$
	$\aleph_0$-saturated in $L_A$, contained in $\C$ and of size at most 
	$2^{|T| + \lvert A \rvert}+ \lvert \mathcal{M}\rvert$. 
	By the choice of $B_i$ and $E_i$, there is  a partial type $\delta(x,y)$ over $\mathcal{M}'$ relatively defining $E_i$ which also relatively defines an equivalence relation on $p \!\upharpoonright_{\mathcal{M}'}\!\!(\C')$. Let $\rho(x,y)$ be $\pi(x,y) \wedge \delta(x,y)$. Then $\rho(x,y)$ relatively defines an equivalence relation on $p\!\upharpoonright_{\mathcal{M}'}\!\!(\C')$ and $\bigslant{p(\C')}{ \rho(\C',\C') \cap p(\C')^2}$ is stable. Hence, applying Lemma \ref{reduction to a small model}, we obtain that the quotient  $$\bigslant{p\!\upharpoonright_{\mathcal{M}'}\!\!(\C')}{ \rho(\C',\C') \cap p\!\upharpoonright_{\mathcal{M}'}\!\!(\C')^2}.$$ is stable. Moreover, 
	$$\rho(\C',\C')\cap p\!\upharpoonright_{\mathcal{M}'}\!\!(\C')^2\subsetneq\pi(\C',\C')\cap p\!\upharpoonright_{\mathcal{M}'}\!\!(\C')^2.$$
	Thus, we have proved that the right hand side of the lemma fails.
	
	
	($\Rightarrow$) 
	By Lemma \ref{reduction to a small model}, the left hand side implies that $E'$ has stable quotient. 
	Assume that the right hand side does not hold, witnessed by a model $\mathcal{M}'$ of size at most $ 2^{\lvert T \rvert + \lvert A \rvert}+ \lvert \mathcal{M} \rvert $ that is $\aleph_0$-saturated in $L_A$ and contains $\mathcal{M}$   
	and a partial type $\rho(x,y)$ over $\mathcal{M}'$. 
	By saturation of $\C$, we can assume that $\mathcal{M}'\subseteq \C$. 
	Hence,  by Lemma \ref{reduction to a small model}, the fact that the quotient $\bigslant{p\!\upharpoonright_{\mathcal{M}'}\!\!(\mathfrak{C}')}{\rho(\mathfrak{C}',\mathfrak{C}')\cap p\!\upharpoonright_{\mathcal{M}'}\!\!(\mathfrak{C}')^2}$ is stable implies that the quotient  $\bigslant{p(\mathfrak{C}')}{\rho(\mathfrak{C}',\mathfrak{C}')\cap p(\mathfrak{C}')^2}$ is stable. Let $b_1,b_2\in p\!\upharpoonright_{\mathcal{M}'}\!\!(\C')$ be elements witnessing $$\rho(\mathfrak{C}',\mathfrak{C}')\cap p\!\upharpoonright_{\mathcal{M}'}\!\!(\mathfrak{C}')^2 \subsetneq \pi(\mathfrak{C}',\mathfrak{C}')\cap p\!\upharpoonright_{\mathcal{M}'}\!\!(\mathfrak{C}')^2,$$
	that is, $(b_1,b_2) \in \pi(\mathfrak{C}',\mathfrak{C}') \setminus \rho(\mathfrak{C}',\mathfrak{C}')$.
	Let $q:=\tp\left(\bigslant{b_1,b_2}{\mathcal{M}'}\right)$ and let $q'\in S(\C)$ be a strong heir extension  over $A$ of $q$. 
	By Lemma \ref{strong heir invariance}, any realization $(b_1',b_2')\in q'(\C')$ satisfies $b_1',b_2'\in p(\C')$, $(b_1',b_2')\in\pi(\mathfrak{C}',\mathfrak{C}')$, and $(b_1',b_2')\not \in\rho(\mathfrak{C}',\mathfrak{C}')$.
	Therefore,
	$$\rho(\mathfrak{C}',\mathfrak{C}')\cap p(\mathfrak{C}')^2 \subsetneq \pi(\mathfrak{C}',\mathfrak{C}')\cap p(\mathfrak{C}')^2,$$
	which contradicts the minimality of $E$.
\end{proof}

\begin{remark}
    Let $\mathcal{M}\prec \C$ be a small model containing $A$ and $\pi(x,y)$ a partial type over $\mathcal{M}$ relatively defining an equivalence relation on $p\!\upharpoonright_{\mathcal{M}}(\mathfrak{C}')$. Then, the quotient $$p\!\upharpoonright_{\mathcal{M}}(\mathfrak{C})/\pi(\C,\C)$$ is stable if and only if the quotient $$p\!\upharpoonright_{\mathcal{M}}(\mathfrak{C}')/\pi(\C',\C')$$ is stable.
\end{remark}
Let $\C\prec \C_1\prec \C'$ be such that $\C_1$ is $\C'$-small and $\kappa_1$-saturated 
with $\kappa_1\geq \kappa$. A set is {\em $\C_1$-small} if its cardinality is smaller than $\kappa_1$. Let $p_1(x)\in S(\C_1)$ be the unique $A$-invariant extension of $p(x)$.

\begin{cor}\label{corollary: from C to C_1}
	Assume that $E$ is  the finest relatively type-definable over a $\C$-small subset of $\C$ equivalence relation on $p(\C')$ with stable quotient. Then $E\cap p_1(\C')^2$ is the finest relatively type-definable over a $\C_1$-small subset of $\C_1$ equivalence relation on $p_1(\C')$ with stable quotient.
\end{cor}

\begin{proof}
	Using Corollary 
	\ref{existence of a small model set of parameters}, we can find a $\mathfrak{C}$-small  $\mathcal{M} \prec \C$ which is
	$\aleph_0$-saturated in $L_A$ and a partial type $\pi(x,y)$ over $\mathcal{M}$ relatively defining $E$ and relatively defining an equivalence relation on $p\!\upharpoonright_{\mathcal{M}}\!\!(\C')$. By Lemma \ref{equivalence for being the finest}, the right hand side of the equivalence in Lemma \ref{equivalence for being the finest} holds. But this right hand side does not depend on the choice of $\C$, and so, again by Lemma \ref{equivalence for being the finest}, $E\cap p_1(\C')^2$ is the finest relatively type-definable over a $\C_1$-small subset of $\C_1$ equivalence relation on $p_1(\C')$ with stable quotient.
\end{proof}


However, there is no obvious transfer going in the opposite direction (i.e., from $\C_1$ to $\C$), as an application of 
Corollary \ref{existence of a small model set of parameters} for $p_1$ may produce a model  $\mathcal{M} \prec \C_1$ whose cardinality is bigger than the degree of saturation of $\C$, and then we cannot embed it into $\C$ via an automorphism. We have only the following corollary.

\begin{cor}\label{corollary: from C_1 to C}
	Assume that $E$ is the finest relatively type-definable over a $\C_1$-small subset of $\C_1$ equivalence relation on $p_1(\C')$ with stable quotient, and suppose that $E$ is relatively defined by a type $\pi(x,y,B)$ over a $\C$-small set $B$. 
	Let $\sigma \in \aut(\C_1/A)$ be such that $\sigma[B] \subseteq \mathfrak{C}$. Then $\pi(\C',\C',\sigma[B]) \cap p(\C')^2$ is the finest relatively type-definable over a $\C$-small subset of $\C$ equivalence relation on $p(\C')$ with stable quotient.
\end{cor}

\begin{proof}
	By Corollary \ref{existence of a small model set of parameters} applied to $\C_1$ and $p_1$ in place of $\C$ and $p$, there is $\mathcal{M}\prec \C_1$ containing $B$ with $\lvert \mathcal{M} \rvert \leq 2^{\lvert T \rvert + \lvert A \rvert}+ \lvert B \rvert + \lvert x \rvert$ which is $\aleph_0$-saturated in $L_A$ and such that $\pi(x,y,B)$ relatively defines an equivalence relation on $p_1\!\upharpoonright_{\mathcal{M}}\!\!(\C')$. Since $\kappa >2^{\lvert T \rvert + \lvert A \rvert}+ \lvert B \rvert + \lvert x \rvert$, we can modify $\sigma $ outside $A \cup B$ so that $\sigma[\mathcal{M}] \subseteq \C$.
	
	By assumption and Lemma \ref{equivalence for being the finest}, the right hand side of that lemma holds for $p_1$ in place of $p$. Since $\sigma(p_1)=p_1$, it still holds for $p_1$ and $\sigma[\mathcal{M}]$ in place of $\mathcal{M}$. 
	Since this right hand side does not depend on $\C_1$ and we have $\sigma[\mathcal{M}] \subseteq \C$, it holds for $p$ and $\sigma[\mathcal{M}]$, so by Lemma \ref{equivalence for being the finest}, we get that  $\pi(\C',\C',\sigma[B]) \cap p(\C')^2$ is the finest relatively type-definable over a $\C$-small subset of $\C$ equivalence relation on $p(\C')$ with stable quotient.
\end{proof}

The following proposition and its proof was proposed by the referee.

\begin{prop}
	The finest relatively type-definable over a $\C$-small subset of $\C$ equivalence relation on $p(\C')$ with stable quotient exists if and only if the finest $\C$-type-definable equivalence relation on $p(\C')$ with stable quotient is relatively type-definable over a $\C$-small subset of $\C$, and, in that case, both equivalence relations coincide.
\end{prop}

\begin{proof}
	Let $F$ be the finest $\C$-type-definable equivalence relation on $p(\C')$ with stable quotient. 
	Every relatively type-definable over a $\C$-small subset of $\C$ equivalence relation on $p(\C')$ with stable quotient is coarser than $F$. Thus, if $F$ is relatively type-definable over a $\C$-small subset of $\C$, then it is the finest one with stable quotient. Conversely, suppose that the finest relatively type-definable over a $\C$-small subset of $\C$ equivalence relation on $p(\C')$ with stable quotient exists and denote it by $E$. As we have already pointed out, we have $F\subseteq E$. On the other hand, let $\pi(x,y)$ be a partial type over a $\C$-small subset of $\C$ relatively defining $E$ and $\rho(x,y)$ a partial type over $\C$ defining $F$. 
	Pick $\C\prec \C_1\prec \C'$ such that $\C_1$ is $\kappa_1$-saturated 
	with $\kappa_1>\lvert \C\rvert$. 
	By Corollary \ref{corollary: from C to C_1}, $\pi(x,y)$ relatively defines the finest relatively type-definable over a $\C_1$-small subset of $\C_1$ equivalence relation on $p_1(\C')$ with stable quotient. 
	Since $p_1(\C')\subseteq p(\C')$ and $\kappa_1>\lvert \C\rvert$, we have that $\rho(x,y)$ relatively defines, over a $\C_1$-small subset of $\C_1$, an equivalence relation on $p_1(\C')$ with stable quotient. Hence, $\pi(x,y)\cup p_1(x)\cup p_1(y)\models \rho(x,y)$. 
	Consider any formula $\phi(x,y,c_0)$ implied by $\rho(x,y)$, where $c_0 \in \C$. 
	By compactness, there is a formula $\psi(x,c_1) \in p_1(x)$ and a formula $\Delta(x,y,c_2)$ implied $\pi(x,y)$ with $c_2 \in \C$ and such that $$\Delta(x,y,c_2)\wedge \psi(x,c_1)\wedge\psi(y,c_1)\models \varphi(x,y,c_0).$$ Now, take $c\in\C$ such that $\tp(c,c_0,c_2/A)=\tp(c_1,c_0,c_2/A)$. Then,  $\Delta(x,y,c_2)\wedge \psi(x,c)\wedge\psi(y,c)\models \varphi(x,y,c_0).$ On the other hand, by $A$-invariance of $p_1(x)$, we get $\psi(x,c)\in p(x)=p_1\!\upharpoonright_{\C}\!\!(x)$. Therefore, $\pi(x,y)\cup p(x)\cup p(y)\models \varphi(x,y,c_0)$. As $\varphi$ was arbitrary, we get  $\pi(x,y)\cup p(x)\cup p(y)\models \rho(x,y)$, so $E\subseteq F$, concluding $E=F$.
\end{proof}

\section{The main theorem}\label{section: main theorem}

The goal of this section is to prove the theorem stated in the introduction (see Theorem \ref{Est exists}). 

We use results on relatively type-definable subsets of the group of automorphisms of $\C'$ extracted from \cite{hrushovski2021order}. 
The following is Definition 2.14 of \cite{hrushovski2021order}, which extends the notion of a relatively definable subset of the automorphism group of the monster model from \cite[Appendix A]{KPR18}.

\begin{defin}\label{relatively-type-def}
	By a  {\em relatively type-definable} subset of $\auto(\C')$, we mean a subset of the form $\{ \sigma \in \aut(\C') : \C' \models \pi(\sigma(a), b))\}$ for some partial type $\pi(x, y)$ without parameters, 
	where $x$ and $y$ are $\C'$-small 
	tuples of variables, and some $a$, $b$ corresponding tuples from $\C'$.
\end{defin}

In particular, given a partial type $\pi(x,y,z)$ over the empty set, a ($\C'$-small) set of parameters $A$ and ($\C'$-small) tuples $a,b,c$ in $\C'$ corresponding to $x,y,z$, respectively, we have a relatively type-definable subset of $\auto(\C')$ of the form $$A_{\pi(x;y,z);a;b,c}(\C'/A):=\{\sigma\in \auto(\C'/A): \C'\models \pi(\sigma(a);b,c)\}.$$
In this section, if $A=\emptyset$, we will omit $(\C'/A)$, and when it is clear 
how the variables are arranged, we will denote sets of the form $A_{\pi(x;y,z);a;a,c}(\C'/A)$ as $A_{\pi;a;c}(\C'/A)$.

We use relatively type-definable sets of the group $\auto(\C')$ to prove the following:

\begin{lema}\label{IP condition}
	Let $a\in \mathfrak{C}'$ and a sequence $(a_i)_{i<\omega} \subseteq \mathfrak{C}'$ (of $\C'$-small tuples $a_i$) be such that $a_0\underset{a}{\equiv}a_i$ for all $i<\omega $ and $a \models
	p\!\upharpoonright_{a_{<\omega}}$.
	Let $\pi(x,y,z)$ be a partial type over the empty set such that for every $i<\omega$ the partial type $\pi(x,y,a_i)$ relatively defines an equivalence relation on $p\!\upharpoonright_{a_i}\!\!(\mathfrak{C'})$. Assume that there is a formula $\varphi(x,y,z)$ implied by $\pi(x,y,z)$ such that for every $i<\omega$ 
	$$ \bigcap_{j\neq i}\pi(\C',\C',a_j)\cap (p\!\upharpoonright_{a_{<\omega}}\!\!(\C'))^2 \not\subseteq \varphi(\C',\C',a_i).$$ Then $T$ has IP.
\end{lema}

To prove this result, we need the following three observations on relatively type-definable subsets of $\auto(\C')$ of special kind.

\begin{external claim}\label{stabilizer} 
	Let $a$, $(a_i)_{i<\omega}$, and  $\pi(x,y,z)$ be as in Lemma \ref{IP condition}, and let $E_{a_i}$ be the equivalence relation on $p\!\upharpoonright_{a_i}\!\!(\C')$ relatively defined by $\pi(x,y,a_i)$. Then, for all $i<\omega$, $ A_{\pi;a;a_i}(\C'/a_{i})$ is the stabilizer of the class $[a]_{E_{a_i}}$ under the action of $\auto(\C'/a_{i})$, and 
	$ A_{\pi;a;a_i}(\C'/a_{<\omega})$ is the stabilizer of the class $[a]_{E_{a_i}}$ under the action of $\auto(\C'/a_{<\omega})$.
\end{external claim}

\begin{proof}
	It is clear that $\auto(\C'/a_i)$ preserves both $p\!\upharpoonright_{a_i}\!\!(\C')$ and $E_{a_i}$.
	
	Let $\sigma\in A_{\pi;a;a_i}(\C'/a_i)$. By the definition of $A_{\pi;a;a_i}$, we have $\models \pi(\sigma(a),a,a_i)$. Hence, $\sigma(a)\in [a]_{E_{a_i}}$, and so $\sigma([a]_{E_{a_i}})= [a]_{E_{a_i}}$. Thus, we have proved that 
	$$ A_{\pi;a;a_i}(\C'/a_{i})\subseteq \stab_{\auto(\C'/a_{i})}([a]_{E_{a_i}}).$$
	
	Conversely, let $\sigma\in\stab_{\auto(\C'/a_{i})}([a]_{E_{a_i}})$. This implies $\sigma(a)E_{a_i}a$. Hence, $\models \pi(\sigma(a),a,a_i)$, and so $\sigma\in  A_{\pi;a;a_i}(\C'/a_i)$. Thus, $$\stab_{\auto(\C'/a_{i})}([a]_{E_{a_i}}) \subseteq A_{\pi;a;a_i}(\C'/a_{i}).$$
	
	The same proof works for $\stab_{\auto(\C'/a_{<\omega})}([a]_{E_{a_i}})$.
\end{proof}	



\begin{external claim}\label{compactness}
	Let  $a$, $a_0$, and $\pi(x,y,z)$ be as in Lemma \ref{IP condition}. Then, for each formula $\varphi(x,y,z)$ implied by $\pi(x,y,z)$ there is a formula $\theta(x,y,z)$ implied by $\pi(x,y,z)$ such that 
	$$  A_{\pi;a;a_0}(\C'/a_0)\cdot A_{\theta;a;a_0}(\C'/a_0)\cdot A_{\pi;a;a_0}(\C'/a_0) \subseteq A_{\varphi;a;a_0}(\C'/a_0).$$
\end{external claim}

\begin{proof}
	Let us consider the type $\pi'(x_1,x_2;y,z):=\pi(x_1,y,z)\cup \{x_2=z\}$. Then,  $$ A_{\pi;a;a_0}(\C'/a_0)=A_{\pi'(x_1,x_2;y,z); aa_0;a,a_0} .$$ Hence, by the previous claim, $A_{\pi'(x_1,x_2;y,z); aa_0;a,a_0}$ is a group, so it satisfies $$A_{\pi'(x_1,x_2;y,z); aa_0;a,a_0}^3=A_{\pi'(x_1,x_2;y,z); aa_0;a,a_0}.$$ For any formula $\varphi(x,y,z)$ implied by $\pi(x,y,z)$ we have $$A_{\pi'(x_1,x_2;y,z); aa_0;a,a_0}^3\subseteq A_{\varphi(x;y,z); a;a,a_0}.$$ 
	Applying compactness (\cite[Corollary 4.8]{hrushovski2021order}), for each $\varphi(x,y,z)$ implied by $\pi(x,y,z)$ there is some $\theta(x,y,z)$ implied by $\pi(x,y,z)$ such that $$ A_{\pi'(x_1,x_2;y,z); aa_0;a,a_0} \cdot A_{\{x_2=z\}\wedge\theta(x_1;y,z); aa_0;a,a_0} \cdot A_{\pi'(x_1,x_2;y,z); aa_0;a,a_0} \subseteq  A_{\varphi(x;y,z); a;a,a_0}.$$
	Finally, since every automorphism on the left hand side belongs to $\auto(\C'/a_0)$, we conclude that $$  A_{\pi;a;a_0}(\C'/a_0)\cdot A_{\theta;a;a_0}(\C'/a_0)\cdot A_{\pi;a;a_0}(\C'/a_0) \subseteq A_{\varphi;a;a_0}(\C'/a_0).$$
\end{proof}

\begin{external claim}\label{independence of the index}
	Let $a$, $(a_i)_{i<\omega}$, and $\pi(x,y,z)$ be as in Lemma \ref{IP condition}. Then, for any formulas $\varphi(x,y,z)$ and $\theta(x,y,z)$ implied by $\pi(x,y,z)$, for every $i<\omega$: $$  A_{\pi;a;a_0}(\C'/a_0)\cdot A_{\theta;a;a_0}(\C'/a_0)\cdot A_{\pi;a;a_0}(\C'/a_0) \subseteq A_{\varphi;a;a_0}(\C'/a_0).$$
	if and only if 
	$$  A_{\pi;a;a_i}(\C'/a_i)\cdot A_{\theta;a;a_i}(\C'/a_i)\cdot A_{\pi;a;a_i}(\C'/a_i) \subseteq A_{\varphi;a;a_i}(\C'/a_i).$$
\end{external claim}

\begin{proof}
	Let $\tau\in\auto(\C'/a)$ be such that $\tau(a_0)=a_i$. The conjugation by $\tau$ \begin{align*}
		\auto(\C'/a_0)&\to \auto(\C'/a_i)\\
		\sigma\hspace{20pt} &\mapsto \hspace{10pt} \tau\sigma\tau^{-1}
	\end{align*}
	is a bijection whose inverse is the conjugation by $\tau^{-1}$. Moreover, $$ \models \pi(\tau\sigma\tau^{-1}(a),a,a_i)\iff \models \pi(\sigma\tau^{-1}(a),a,a_0)\iff \models\pi(\sigma(a),a,a_0).$$
	Analogous equivalences also hold for $\varphi$ and for $\theta$ in place of $\pi$. Hence, the desired equivalence follows by applying the conjugation by $\tau$.
\end{proof}

We are now ready to prove Lemma \ref{IP condition}.

\begin{proof} [{\bf Proof of Lemma \ref{IP condition}}]
	Note that for all $i<\omega$, using automorphisms of $\C'$ fixing $(a_i)_{i<\omega}$, we can reduce the condition $$ \bigcap_{j\neq i}\pi(\C',\C',a_j)\cap (p\!\upharpoonright_{a_{<\omega}}\!\!(\C'))^2\not\subseteq \varphi(\C',\C',a_i)$$ to $$ \bigcap_{j\neq i}\pi(\C',a,a_j)\cap p\!\upharpoonright_{a_{<\omega}}\!\!(\C')\not\subseteq \varphi(\C',a,a_i),$$
	because, given a pair $(c,d)$ witnessing the former condition, there exists some \\ $\sigma\in\auto(\C'/a_{<\omega})$ such that $\sigma(d)=a$, and then the pair $(\sigma(c),a)$ witnesses the latter condition. Moreover, using the same approach, one can see that the latter condition can be expressed using relatively type-definable 
	subsets of $\auto(\C')$ as $$A_{\bigwedge_{j\neq i}\pi(x;y,z_j);a;a,(a_j)_{j\neq i}}(\C'/a_{<\omega}) \not\subseteq A_{\varphi(x;y,z_i);a;a,a_i}.$$
	
	For every $i<\omega$, choose some $$\sigma_i\in A_{\bigwedge_{j\neq i}\pi(x;y,z_j);a;a,(a_j)_{j\neq i}}(\C'/a_{<\omega}) \setminus A_{\varphi(x;y,z_i),a,a,a_i},$$ and let $\sigma_{I}$ denote the composition $\prod_{i\in I}\sigma_i$, for any finite $I\subseteq \omega$.
	
	By Claims \ref{compactness} and \ref{independence of the index}, there is a formula $\theta(x,y,z)$ implied by $\pi(x,y,z)$ such that for all $i<\omega$
	$$ A_{\pi;a;a_i}(\C'/a_i)\cdot A_{\theta;a;a_i}(\C'/a_i)\cdot A_{\pi;a;a_i}(\C'/a_i) \subseteq A_{\varphi;a;a_i}(\C'/a_i). $$ 

	\begin{claim}
		For any finite $I\subseteq \omega$ $$\models \theta(\sigma_I(a),a,a_i)\iff i\notin I.$$
	\end{claim}
	\begin{proof}[{Proof of claim}]
		Firstly, take $i\not\in I$. Then, for every $j\in I$, $\sigma_j$ belongs to the set $ A_{\pi;a;a_i}(\C'/a_{<\omega})$. By Claim \ref{stabilizer}, the set $ A_{\pi;a;a_i}(\C'/a_{<\omega})$ is a group, and so we get  $\sigma_{I}\in  A_{\pi;a;a_i}(\C'/a_{<\omega})$. Hence, $\theta(\sigma_I(a),a,a_i)$ holds.
		
		Now take $i\in I$ and write $I:=I_0\sqcup \{i\} \sqcup I_1$, where $I_0=\{j\in I : j<i\}$ and $I_1=\{j\in I : j>i\}$. For each $j\in I_0\cup I_1$ we have $\sigma_j\in \ A_{\pi;a;a_i}(\C'/a_{<\omega})$. Then, $\theta(\sigma_{I}(a),a,a_i)$ does not hold. Otherwise, $$ \sigma_I = \sigma_{I_0} \sigma_i \sigma_{I_1}\in  A_{\theta;a;a_i}(\C'/a_{<\omega}),$$
		which, by 
		Claim \ref{stabilizer} and	the choice of $\theta$, implies $$\sigma_i\in  A_{\varphi;a;a_i}(\C'/a_{<\omega}),$$ 
		a contradiction with our choice of $\sigma_i$. 
	\end{proof}
	The formula $\theta$ witnesses that $T$ has IP.
\end{proof}

When we write (NIP) in the statement of a result, it means that we assume that the theory $T$ has NIP.

\begin{lema}[NIP]\label{existence of an special index}
	Let $\pi(x,y,z)$ be a partial type over the empty set (with a $\C'$-small $z$), and let $a_0\subseteq \mathfrak{C}'$ be such
	that $\pi(x,y,a_0)$ relatively defines an equivalence relation on $p\!\upharpoonright_{a_0}\!\!(\mathfrak{C}')$. Then, for any $(a_i)_{i<\lambda}$, where $\lambda \geq \beth_{(2^{(\lvert a_0\rvert +\lvert x \rvert + \lvert T\rvert + \lvert A\rvert)})^+}$ and 
	$a_i\underset{A}{\equiv}a_0$ for all $i<\lambda$,  there exists $i<\lambda$ such that $$\bigcap_{j\neq i} \pi(\C',\C',a_j)\cap(p\!\upharpoonright_{a_{<\lambda}}\!\!(\C'))^2\subseteq\pi(\C',\C',a_i).$$
\end{lema}
\begin{proof}
	Assume the conclusion does not hold. 
	Then, for every $i<\lambda$ $$\bigcap_{j\neq i} \pi(\C',\C',a_j)\cap(p\upharpoonright_{a_{<\lambda}}\!\!(\C'))^2\not\subseteq\pi(\C',\C',a_i).$$ Take pairs $(b_i,c_i)_{i<\lambda}$ witnessing it. 
	Let $(a_i',b_i',c_i')_{i<\omega} \subseteq \C'$ be an $A$-indiscernible sequence obtained by extracting indiscernibles (see Fact \ref{extracting indiscernibles}) from the sequence $(a_i,b_i,c_i)_{i<\lambda}$.
	Then, since $p$ is $A$-invariant, for all $i<\omega$ the elements $(a_i',b_i',c_i')$ satisfy:
	\begin{align*}
		(b_i',c_i')&\in\bigcap_{j\neq i}\pi(\C',\C',a_j')\cap (p\!\upharpoonright_{a'_{<\omega}}\!\!(\C'))^2;\\
		(b_i',c_i')&\not\in\pi(\C',\C',a_i');\\
		a_i'&\equiv_{A}a_0' \equiv_A a_0.
	\end{align*} 
	(Note that the $A$-invariance of $p$, together with the property of being an extracted sequence, is used to ensure that $(b_i',c_i')$ belongs to $p\!\upharpoonright_{a'_{<n}}\!\!(\C')^2$ for each $n \in \omega$.) By the indiscernibility of the sequence $(a_i',b_i',c_i')_{i<\omega}$, there exists a formula $\varphi(x,y,z)$ implied by $\pi(x,y,z)$ such that for all $i<\omega$ $$(b_i',c_i')\not\in\varphi(\C',\C',a_i').$$
	Take any $a \models p\!\upharpoonright_{a'_{<\omega}}$. Since $p$ is $A$-invariant, $a_i'\equiv_A a_j'$ implies  $a_i'\equiv_a a_j'$.
	Moreover, since $a_i' \equiv_A a_0' \equiv_A a_0$, $\pi(x,y,a_0)$ relatively defines an equivalence relation on $p\!\upharpoonright_{a_0}\!\!(\mathfrak{C}')$, and $p$ is $A$-invariant, we get that $\pi(x,y,a_i')$ relatively defines an equivalence relation on $p\!\upharpoonright_{a_i'}\!\!(\mathfrak{C}')$ for all $i<\omega$.
	
	Hence, the sequence $(a_i')_{i<\omega}$ together with $a$, $\pi(x,y,z)$, and $\varphi(x,y,z)$ satisfies the assumptions of Lemma \ref{IP condition}, and so we get IP, which is a contradiction.
\end{proof}



The next theorem is the main result of this chapter and the thesis.

\begin{teor}[NIP]\label{Est exists}
	Let $p(x)\in S_x(\mathfrak{C})$ be an $A$-invariant type with a $\C$-small $x$. Assume that the degree of saturation of $\C$ is at least $\beth_{(\beth_{2}(\lvert x \rvert+\lvert T \rvert + \lvert A\rvert))^+}$. Then, there exists a finest equivalence relation $E^{\textrm{st}}$ on $p(\C')$ relatively type-definable over a $\C$-small set of parameters from $\C$ and with stable quotient $p(\C')/E^{\textrm{st}}$. 
\end{teor}

\begin{proof}
	Let $\nu:=\beth_{(\beth_{2}(\lvert x \rvert+\lvert T \rvert + \lvert A\rvert))^+}$.

	\begin{claim}
		If for every countable partial type $\pi(x,y,z)$ over the empty set and countable tuple $a_0$ from $\mathfrak{C}$ such that $\pi(x,y,a_0)$ relatively defines an equivalence relation $E_{a_0}$ on $p(\mathfrak{C}')$ with stable quotient there is no sequence $(a_i)_{i<\nu}$ of (countable) tuples $a_i$ in $\C$ such that for all $i<\nu$ we have $a_i\underset{A}{\equiv}a_0$ and  $\bigcap_{j<i}E_{a_j}\not\subseteq E_{a_i}$, then the theorem holds.
	\end{claim}

	\begin{proof}[Proof of claim]
		Consider an arbitrary collection $(E_i)_{i\in I}$ of 
	 equivalence relations on $p(\C')$ relatively type-definable over a $\C$-small subset of $\C$ and with stable quotients. Our goal is to prove that the intersection $\bigcap_{i\in I}E_i$ is a relatively type-definable over a $\C$-small subset of $\C$ equivalence relation on $p(\C')$ with stable quotient. 
		
		Using Proposition \ref{decomposition of equivalence relations}, we can write each $E_j$ as $\bigcap_{i\in I_j} F^i_j $, where each $F^i_j$ is 
		a type-definable equivalence relation on $p(\C')$ countably relatively definable over a countable subset of $\C$. Since the $F^i_j$'s are coarser than the corresponding $E_j$, each $F^i_j$ also has stable quotient. We can now write $$ \bigcap_{j\in I}E_j=\bigcap_{j\in I}\bigcap_{i\in I_j} F^i_j.$$
		Note that the number of possible countable types  
		over $\emptyset$ whose instances relatively define the $F^i_j$'s is bounded by $2^{\lvert x\rvert+\lvert T\rvert}$, and the set of types over $A$ of the countable tuples of parameters used in the relative definitions of the $F^i_j$'s is bounded by $2^{\lvert T \rvert + \lvert A \rvert}$.
		Hence, by the assumptions of the claim, the intersection $\bigcap_{j\in I}E_j$ coincides with an intersection $\bigcap_{k\in K}F^{i_k}_{j_k}$, where $\lvert K\rvert \leq 2^{\lvert T \rvert + \lvert A \rvert} \times  2^{\lvert T \rvert + \lvert x \rvert}\times \nu = \nu$. In fact, since $2^{\lvert T \rvert + |A|+ \lvert x \rvert}$ is strictly smaller than the cofinality of $\nu$, we can even get $|K| < \nu$.
		Finally, by \cite[Remark 1.4]{MR3796277}, $\bigcap_{k\in K}F^{i_k}_{j_k}$ is a relatively type-definable over a $\C$-small subset of $\C$ (as $\C$ is $\nu$-saturated) equivalence relation on $p(\C')$ with stable quotient.
	\end{proof}
	
	Suppose the theorem fails. 
	By the claim, there exists a countable type $\pi(x,y,z)$ over $\emptyset$ and a countable tuple $a_0$ in $\C$ such that $\pi(x,y,a_0)$ relatively defines an equivalence relation on $p(\C')$ with $\bigslant{p(\C')}{\pi(\C',\C',a_0)\cap p(\C')^2}$ stable and there is $(a_i)_{i<\nu}\subseteq \C$ such that for all $i<\nu$, $a_i\equiv_{A}a_0$ and $\bigcap_{j<i}\pi(\C',\C',a_j)\cap p(\C')^2\not\subseteq \pi(\C',\C',a_i)$. 
	By Corollary \ref{existence of a small model set of parameters}, enlarging $a_0$, we can assume that $a_0$ enumerates an $\aleph_0$-saturated model in $L_A$ of size at most 
	$2^{|T| + |A|}$ and $\pi(x,y,a_0)$ relatively defines an equivalence relation on $p\!\upharpoonright_{a_0}\!\!(\C')$; by Lemma \ref{reduction to a small model}, 
	this relation also yields a stable quotient 
	on $p\!\upharpoonright_{a_0}\!\!(\C')$.  
	
	Let $(b_i,c_i)_{i<\nu}$ be a sequence witnessing that $\bigcap_{j<i}\pi(\C',\C',a_j)\cap p(\C')^2\not\subseteq \pi(\C',\C',a_i)$. 
	Let $(a_i',b_i',c_i')_{i< \nu} \subseteq \C'$ be an $A$-indiscernible sequence extracted from  $(a_i,b_i,c_i)_{i<\nu}$. Then, since $p$ is $A$-invariant, we get that for all $i<\nu$
	$$(b_i',c_i') \in \left(\bigcap_{j<i}\pi(\C',\C',a'_j)\cap (p\!\upharpoonright_{ a'_{<\nu}}\!\!(\C'))^2\right)\setminus \pi(\C',\C',a'_i).$$
	Moreover, since $a_0' \equiv_A a_0$ 
	and $p$ is $A$-invariant, we get that $\pi(x,y,a_0')$ relatively defines an equivalence relation on $p\!\upharpoonright_{a_0'}(\C')$, and we also have $a_i' \equiv_A a_0'$ for all $i<\nu$. Therefore, by Lemma \ref{existence of an special index},
	there exists some $\beta <\nu$ such that  
	$$(*)\;\;\;\;\;\; \bigcap_{\alpha \neq \beta} \pi(\C',\C',a_{\alpha}')\cap(p\!\upharpoonright_{a_{<\nu}'}\!\!(\C'))^2\subseteq\pi(\C',\C',a_{\beta}').$$ 
	
	In the sequence $(a_i',b_i',c_i')_{i< \nu}$, let us insert a sequence $(d'_i,e'_i,f'_i)_{i<\omega}$ from $\C'$ in place of the element
	$(a_\beta',b_{\beta}',c_{\beta}')$ so that the resulting sequence is still $A$-indiscernible. Then, since $p$ is $A$-invariant, for all $i<\omega$
	$$(**)\;\;\;\;\;\; (e_i',f_i') \in \left(\bigcap_{j< i}\pi(\C',\C',d_j')\cap (p\!\upharpoonright_{a'_{\substack{\alpha<\nu\\\alpha \neq \beta}}, d'_{<\omega}}\!\!(\C'))^2\right) \setminus\pi(\C',\C',d_i').$$
	Hence, due to the $A$-indiscernibility of the sequence $(d_i',e_i',f_i')_{i<\omega}$, there exists some formula $\varphi$ implied by $\pi$ such that for all $i<\omega$ we have $(e_i',f_i')\not\in\varphi(\C',\C',d_i')$. 
	
	Moreover, since $d_i' \equiv_A a_0'$ and using the $A$-invariance of $p$, we get that $\pi(x,y,d_i')$ relatively defines an equivalence relation on $p\!\upharpoonright_{d_i'}\!\!(\C')$.
	
	Let us consider the set 
	$$X:=p\! \upharpoonright _{a'_{\substack{\alpha<\nu\\\alpha\neq \beta} }}\!\!(\C').$$ 
	By the above choices and $A$-invariance of $p$, the type $\bigcup_{\substack{\alpha<\nu\\\alpha\neq \beta} }\pi(x,y,a'_\alpha)$ relatively defines an equivalence relation $E$ on $X$ with stable quotient, and the sequence $(d_i',[e_i']_E,[f_i']_E)$  is indiscernible over $$B:=A\cup \{a'_\alpha: \alpha <\nu; \alpha \neq \beta\}.$$
	%
	Hence, 
	$$\tp\left(\bigslant{(d_j',[e_i']_E,[f_i']_E)}{B}\right)=\tp\left(\bigslant{(d_i',[e_j']_E,[f_j']_E)}{B}\right)$$
	for all $j<i$. 
	
	Let $E_i$ be the equivalence relation relatively defined by the partial type $\pi(x,y,d_i')$ on $p \!\upharpoonright _{a'_{\substack{\alpha<\nu\\\alpha\neq \beta} },d_i'}\!\!(\C')$. By $(**)$, $e_i' E_j f_i'$ for all $j<i$. Using this and the previous paragraph, we will deduce that $e_j' E_i f_j'$ for all $j<i$.
	
	Indeed, take any $j<i$. Since $\tp\left(\bigslant{(d_j',[e_i']_E,[f_i']_E)}{B}\right)=\tp\left(\bigslant{(d_i',[e_j']_E,[f_j']_E)}{B}\right)$, there is $\sigma \in \aut(\C'/B)$ such that $\sigma(d_j',[e_i']_E,[f_i']_E)=(d_i',[e_j']_E,[f_j']_E)$. Then, by the $A$-invariance of $p$, $\sigma[E_j]=E_i$. Thus, since $e_i' E_j f_i'$, we conclude that $\sigma(e_i') E_i \sigma(f_i')$.  On the other hand, by $(*)$ and $A$-invariance of $p$, we have $E \upharpoonright_{\dom(E_i)} \subseteq E_i$, which together with the fact that $e_j' E \sigma(e_i')$, $f_j' E \sigma(f_i')$, and $e_j',f_j',\sigma(e_i'),\sigma(f_i') \in \dom(E_i)$ gives us $e_j'E_i\sigma(e_i')$ and $f_j' E_i \sigma(f_i')$. Therefore, $e_j' E_i f_j'$, as required.

	We have shown that the sequence $(d_i',e_i',f_i')$ satisfies: 
	\begin{align*}
		\pi(e_j',f_j',d_i') &\text{ for all } i\neq j; \hspace{3pt} i,j<\omega;\\
		\neg\varphi(e_i',f_i',d_i') &\text{ for all } i<\omega.
	\end{align*}
	Take any $a \models p\! \upharpoonright_{d'_{<\omega}}$. Since $d_i' \equiv_A d_0'$  for all $i<\omega$ 
	and $p$ is $A$-invariant, we get that $d_i'\equiv_a d_0'$ for all $i<\omega$.
	Thus, the sequence $(d_i')_{i<\omega}$ satisfies the assumption of Lemma \ref{IP condition}, and so we get IP, a contradiction.
\end{proof}

We end this section with some comments on whether the large saturation condition in Theorem \ref{Est exists} is necessary or could be eliminated. 


Note that in the above proof, in order to extract indiscernibles from the sequence $(a_i,b_i,c_i)_{i<\lambda}$, we need to know that $\nu$ is at least  $\beth_{(2^{2^{|T| + |A|} +\lvert x \rvert + |T|+|A|})^+} = \beth_{(2^{2^{|T| + |A|} +\lvert x \rvert})^+}$. On the other hand, the proof of the claim requires that any number smaller than $\nu$ is smaller than the degree of saturation of $\C$. That is why the whole proof requires that $\C$ is at least $\beth_{(2^{2^{|T| + |A|} +\lvert x \rvert})^+}$-saturated. 
In the statement of the theorem, it is enough to assume that $\C$ is $\beth_{(2^{2^{|T| + |A|} +\lvert x \rvert})^+}$-saturated; we used a bigger  degree of saturation, which is notationally more concise.

Although our proof uses essentially the assumption on the degree of saturation, one could still try to transfer the existence of the finest relatively type-definable equivalence relation from big models to their elementary substructures. 

Let $\C\prec\C_1\prec \C'$ be such that $\C_1$ is $\C'$-small and at least as saturated 
as $\C$, and let $p_1(x)\in S(\C_1)$ be the unique $A$-invariant extension of $p(x)$.

While Corollary \ref{corollary: from C to C_1} allows us to transfer the existence of the finest relatively type-definable over a $\C$-small subset of $\C$ equivalence relation on $p(\C')$ with stable quotient to the finest relatively type-definable over a $\C_1$-small subset of $\C_1$ equivalence relation on $p_1(\C')$, in order to eliminate the specific saturation assumption in Theorem \ref{Est exists}, we would need to have a transfer going in the other direction. In Corollary \ref{corollary: from C_1 to C}, we proved such a transfer but only under the additional assumption that the finest relatively type-definable  over a $\C_1$-small subset of $\C_1$ equivalence relation $E$ on $p_1(\C')$ is relatively type-definable over a $\C$-small subset. Therefore, the specific saturation assumption could be eliminated if we could answer positively the following question.

\begin{question}
	In the context of Theorem \ref{Est exists}, is $E^{\textrm{st}}$ always relatively type-definable over $A$?
\end{question}

In the examples studied in the next section, this turns out to be true. Also, in the context of type-definable groups studied in \cite{MR3796277}, $G^{\textrm{st}}$ is type-definable over the parameters over which $G$ is type-definable.

\section{Examples}\label{section examples}
We present two examples where $E^{\textrm{st}}$ is computed explicitly, the second example is based on Section \ref{section: main example}.
 In fact, in both examples, we give full classifications of all relatively type-definable over $\C$-small subsets of $\C$ equivalence relations on $p(\C')$, for suitably chosen $p \in S(\C)$.

\subsection*{Example 1}
Let our language be $L:=\{R_r(x,y),f_s(x): r\in\mathbb{Q}^+, s\in \mathbb{Q}\}$ and $T$ be the theory of 
$(\mathbb{R},R_r,f_s)_{r\in\mathbb{Q}^+,s\in\mathbb{Q}}$, where $f_s(x):=x+s$ and $R_r(x,y)$ holds if and only if $0\leq y-x\leq r$.

We define the directed distance between two points as a function $$d:\C'\times\C'\to \mathbb{R}\sqcup \mathbb{Q}_+\sqcup\mathbb{Q}_-\sqcup \{\infty\}$$ 
(where $\mathbb{Q}_+$ and $\mathbb{Q}_-$ are disjoint copies of $\mathbb{Q}$ which are disjoint from $\mathbb{R}\sqcup \{\infty\}$) satisfying: 
\begin{align*}
	&d(x,y)=q\in\mathbb{Q} \iff y=f_q(x);\\
	&d(x,y)=r\in \mathbb{R}^{+}\setminus\mathbb{Q} \iff \forall s_1,s_2\in \mathbb{Q}^+ \text{ such that } s_1<r<s_2, \neg R_{s_1}(x,y)\wedge R_{s_2}(x,y);\\
	&d(x,y)=q_+ \in \mathbb{Q}_+ \iff y\ne f_q(x) \text{ is  infinitesimally close to }f_q(x)\text{ on the right;}\\
	&d(x,y)=q_- \in  \mathbb{Q}_- \iff y \ne f_q(x) \text{ is  infinitesimally close to }f_q(x)\text{ on the left;}\\
	&d(x,y)=\infty \iff \neg(R_s(x,y)\lor R_s(y,x)) \text{ for all } s\in \mathbb{Q}^+.
\end{align*}
We complete the definition of $d$ extending it symmetrically in the negative irrational case, i.e $d(y,x):=-d(x,y)$ whenever $d(x,y)\in \mathbb{R}^+ \setminus \mathbb{Q}$. 
This clearly gives us a well defined function $d$.
%
Addition on $\mathbb{R}$ is extended to  $\mathbb{R}\sqcup \mathbb{Q}_+\sqcup\mathbb{Q}_-\sqcup \{\infty\}$ in the natural way, in particular:
\begin{itemize}
	\item $q'+q_+:=(q'+q)_+ \text{ and }q'+q_-:=(q'+q)_-$ for any $q,q' \in \mathbb{Q}$;
	\item $r+\infty:=\infty \text{ for any } r\in \mathbb{R}\cup \mathbb{Q}_+\cup\mathbb{Q}_-\cup \{\infty\}.$
\end{itemize}
\begin{lema}\label{properties distance}
	Properties of the distance:
	\begin{enumerate}
		\item $d(a,f_q(b))=q+d(a,b)$ and $d(f_q(a),b)=-q+d(a,b)$;
		\item For any distinct real numbers $r_1,r_2$, if $d(a,b)=r_1$ and $d(a,c)=r_2$, then $d(b,c)=r_2-r_1$;
		\item For any irrational $r$, if $d(a,b)=r$ and $d(b,c)=0_{\pm}$, then $d(a,c)=r$;
		\item For any irrational $r$, if $d(a,b)=r=d(a,c)$, then $d(b,c)=0_\pm$.
	\end{enumerate}
\end{lema}

\begin{proof}
	$(1)$ follows from the definition of the distance.
	
	$(2)$ 
	Since the rational case is covered in $(1)$, we can assume that $r_1,r_2$ are irrationals. Consider the case $0<r_1<r_2$; other cases are similar. Let $q$ be any rational bigger than $r_2-r_1$.
	We can write $q$ as $q_2-q_1$, where $q_1,q_2$ are rationals such that $q_1<r_1<r_2<q_2$. 
	Since  $R_{q'}(a,b)$ and $\neg R_{q'}(a,c)$ hold for some $q'\in \mathbb{Q}^+$ (for any $r_1<q'<r_2$) and $R_{q_1}(a,b)$ does not hold, $ R_{q_2-q_1}(b,c)$ has to hold; otherwise $R_{q_2}(a,c)$ would not hold, contradicting $d(a,c)=r_2$. Hence, $d(b,c)\leq q$
	
	Let now $q$ be any positive rational smaller than $r_2-r_1$. 
	We can write $q$ as $q_2-q_1$, where $q_1,q_2$ are rationals such that $r_1<q_1<q_2<r_2$. Since $R_{q_1}(a,b)$ holds, $R_{q_2-q_1}(b,c)$ cannot hold; otherwise, $R_{q_2}(a,c)$ would hold, contradicting $d(a,c)=r_2$. Hence, $d(b,c)\geq q$.
	
	$(3)$ Consider the case $r>0$ and $d(b,c)=0_+$; the other cases are analogous.  Let $q$ be any rational bigger than $r$.
	We can write $q$ as $q_1+q_2$, where  $q_1,q_2$ are rationals, $q_1>r$, and $q_2>0$. Then, $R_{q_1}(a,b)$ and $R_{q_2}(b,c)$ hold, hence so does $R_{q_1+q_2}(a,c)$. 
	This implies that $d(a,c)\leq q$. Let now $q$ be any positive rational smaller than $r$. Then, $R_q(a,c)$ cannot hold; otherwise it would imply $R_q(a,b)$, a contradiction.
	
	$(4)$ 
	Consider the case $r>0$; the other case is similar. Consider any rationals $q_1$, $q_2$ satisfying $0<q_1<r<q_2$. Then, $R_{q_2-q_1}(f_{q_1}(a),b)\wedge R_{q_2-q_1}(f_{q_1}(a),c)$ holds, which imply $R_{q_2-q_1}(b,c)\vee  R_{q_2-q_1}(c,b)$. Since $q_2$ and $q_1$ were arbitrary, this means that $b$ and $c$ are infinitesimally close.
\end{proof}

It is clear that the distance determines the quantifier-free type of a pair $(a,b)$. Since our language only contains unary and binary symbols, 
the collection of distances between the elements of a given $n$-tuple determines its quantifier-free type.

\begin{prop}
	The theory $T$ has NIP and quantifier elimination.
\end{prop}
\begin{proof}
	$T$ has NIP, because it is a reduct of an o-minimal theory.
	
	We prove quantifier elimination using a back and forth argument. Let $\mathcal{M}$ and $\mathcal{N}$ be two $\aleph_0$-saturated models of $T$ and let $(a_1,\dots,a_n)$ and $(b_1,\dots,b_n)$ be tuples of elements of $\mathcal{M}$ and $\mathcal{N}$, respectively, satisfying the same quantifier free type. Choose a new element $a_{n+1}\in \mathcal{M}$. There are three cases:
	\begin{enumerate}
		\item $a_{n+1}$ is infinitely far from $a_1,\dots,a_n$;
		\item $a_{n+1}=f_q(a_i)$ for some $q\in \mathbb{Q}$ and $i=1,\dots,n$;
		\item $a_{n+1}$ is related (i.e., at finite distance) to some of the $a_i$'s but is not equal to $f_q(a_i)$ for any $q\in \mathbb{Q}$ and $i=1,\dots,n$.
	\end{enumerate}
	
	In the first two cases, by $\aleph_0$-saturation, we can clearly choose $b_{n+1}\in\mathcal{N}$ such that $(a_1,\dots,a_{n+1})$ and $(b_1,\dots,b_{n+1})$ have the same quantifier-free type. Now, let us tackle the third case. 
	
	In the third case, by removing the elements of the sequence $(a_1,\dots,a_n)$ which are at infinite distance from $a_{n+1}$ as well as the corresponding elements of the sequence $(b_1,\dots,b_n)$, we may assume that no $a_i$ is infinitely far from $a_{n+1}$. Note also that for each $i<n$ there is at most one $q_i\in \mathbb{Q}$ such that $f_{q_i}(a_i)$ is infinitesimally close to $a_{n+1}$. 
	Let $A$ be the set of all such $f_{q_i}(a_i)$'s. 
	
	
	First, consider the case when $A \ne \emptyset$. Then $A$ is a finite set totally ordered  by the relation $R_1(x,y)$ and all elements in $A$ are infinitesimally close to each other and to $a_{n+1}$. Let $B:=\{f_{q_i}(b_i) : f_{q_i}(a_i)\in A \}$. Note that all the elements in $B$ are infinitesimally close to each other and that the map sending $f_{q_i}(a_i)$ to $f_{q_i}(b_i)$ is an $R_1$-order isomorphism.
	Then, by density, there exists $b_{n+1}$ with the same $R_1$-relative position  to the elements in $B$ as $ a_{n+1}$ to the corresponding elements in $A$. 
	Hence, $d(b_{n+1},f_{q_i}(b_i))=d(a_{n+1},f_{q_i}(a_i))$ for each $f_{q_i}(a_i) \in A$, and, by Lemma \ref{properties distance}, this implies
	$$\qftp(b_1,\dots,b_n,b_{n+1})=\qftp(a_1,\dots,a_n, a_{n+1}).$$ 
	
	In the case when $A = \emptyset$, $d(a_i,a_{n+1})$ is irrational for every $i\leq n$. Pick $b_{n+1}$ so that $d(b_1,b_{n+1})=d(a_1,a_{n+1})$. Since $A=\emptyset$, by Lemma \ref{properties distance}(4), we get that $d(a_1,a_{n+1})\ne d(a_1,a_i)$ for all $1< i \leq n$. 
	Hence, Lemma  \ref{properties distance} implies that\\
	\begin{minipage}[t]{0.15\linewidth}
		\hfill
	\end{minipage}\hfill
	\begin{minipage}[t]{0.7\linewidth}
		\begin{center}
			\vspace{1pt}
			$\qftp(b_1,\dots,b_n,b_{n+1})=\qftp(a_1,\dots,a_n, a_{n+1}).$
		\end{center}
	\end{minipage}\hfill
	\begin{minipage}[t]{0.15\linewidth}
		\vspace{1pt}
		\hfill\qedhere
	\end{minipage}
\end{proof}

Let $p\in S_x(\C)$ be the 
invariant (over $\emptyset$) complete global type determined by $$ \bigwedge_{c\in\C}\bigwedge_{n\in\omega}\neg R_n(x,c)\wedge \neg R_n(c,x).$$
We denote by $E(x,y)$ the equivalence relation on $\C'$ defined by $$\bigwedge_{r\in\mathbb{Q}^+}R_r(x,y)\vee R_r(y,x)$$
and by $E\!\upharpoonright_{p}$ the equivalence relation on $p(\C')$ relatively defined by the same partial type.

\begin{lema}
	The hyperdefinable set $\C'/E(\C',\C')$ is stable.
\end{lema}
\begin{proof}
	By 
 Theorem \ref{equivteor}, it is enough to prove that for any $A \subseteq \C'$ with $|A| \leq \mathfrak{c}$ we have $|S_{\C'/E}(A)| \leq \mathfrak{c}$. 
	
	Clearly, the elements $c$ and $c'$ are in the same $E$-class if and only if $c=c'$ or $d(c,c')=0_\pm$. 
	%
	Note that whenever $d(c,a) =d(c',a)\neq \infty$, then $cEc'$. Therefore, specifying the distance $d(c,a) \ne \infty$ from $c$ to a given element $a \in A$ determines the class $[c]_E$. On the other hand, by q.e., the condition saying that $d(c,a)=\infty$ for all $a \in A$ determines $\tp(c/A)$. Therefore, $|S_{\C'/E}(A)| \leq \mathfrak{c}\times\mathfrak{c} +1= \mathfrak{c}$.
\end{proof}

\begin{prop}
	The only equivalence relations on $p(\C')$ relatively type-definable over a $\C$-small subset of $\C$ are equality, $E\!\upharpoonright_p$, and the total equivalence relation. 
\end{prop}

\begin{proof}
	Let $F(x,y)$ be any equivalence relation on $p(\C')$ relatively type-definable over a $\C$-small subset of $\C$. Let $S_n(x,y):=R_n(x,y)\vee R_n(y,x)$. There are two cases.
	
	Case 1: There are $ a,b\in p(\C')$ such that  $aFb$ and $\models\bigwedge_{n\in\mathbb{N}}\neg S_n(a,b)$. For any $c,d\in p(\C')$ we can find $e\in p(\C')$ such that $\models\bigwedge_{n\in\mathbb{N}}\neg S_n(c,e) \wedge \bigwedge_{n\in\mathbb{N}}\neg S_n(d,e)$. 
	Hence, by q.e., $$(d,e)\equiv_{\C}(a,b)\equiv_{\C}(c,e).$$ As $F$ is $\C$-invariant, we conclude that $cFd$. This implies that $F$ is the total relation.
	
	Case 2: For any $ a,b\in p(\C')$ with $aFb$ there exists $n\in \mathbb{N}$ such that $\models S_n(a,b)$. 
	
	First, we show that $aFb$ implies $aE\!\upharpoonright_pb$. Assume that it is not the case. Then there exists $m\in \mathbb{Q}^+$ such that $aFb$ and $\neg S_m(a,b)$. 
	On the other hand, $S_n(a,b)$ for some $n \in \mathbb{N}$.
	Since $a \equiv_\C b$, there is $\sigma\in\auto(\C'/\C)$ satisfying $\sigma(a)=b$. 
	Let $b_i:=\sigma^i(a)$ for $i<\omega$.
	Clearly, $$(a,b)\equiv_{\C}(b,b_2)\equiv_{\C}(b_2,b_3)\equiv_{\C}\cdots . $$ 
	We deduce that for all $k\in \mathbb{N}$, $aFb_k$ and $\models \neg S_{km}(a,b_k)$. Hence, by compactness, there exists $b'\in p(\C')$ such that $aFb'$ and $\models \neg S_n(a,b')$ for all $n\in \mathbb{N}$, contradicting the hypothesis of the second case.
	
	Finally, if $F$ is not equality, there exist elements $a\neq b \in p(\C')$ such that $aFb$, and so $aE\!\upharpoonright_pb$ by the last paragraph. Take any distinct $c,d\in p(\C')$ satisfying $cE\!\upharpoonright_{p}d$. Then, by q.e., either $(a,b)\equiv_{\C}(c,d)$ or $(a,b)\equiv_{\C}(d,c)$. Both cases imply $cFd$, which means that $F$ and $E\!\upharpoonright_p$ are the same equivalence relation.
\end{proof}
Note that $p(\C')$ is not stable since for any $a\in p(\C')$, we have that the set $a+\mu\subset p(\C')$ is totally ordered by $R_1(x,y)$. Hence, we obtain the following:
\begin{cor}
	The equivalence relation $E\!\upharpoonright_p$ is the finest equivalence relation on $p(\C')$ relatively type-definable over a $\C$-small set of parameters from $\C$ and  with stable quotient, 
	that is $E^{\textrm{st}} = E\!\upharpoonright_p$.
\end{cor}

\subsection*{Example 2}
This example is based on Section \ref{section: main example}. 
We work in the language $L:=\{+,-,1,R_r(x,y): r\in\mathbb{Q}^+\}$ and our theory $T$ is $\Th((\mathbb{R},+,-,1,R_r(x,y))_{r\in\mathbb{Q}^+})$, where $\mathbb{R}\models R_r(x,y)$ if and only if $0\leq y-x\leq r$. 

From Proposition \ref{proposition: NIP and unstable} and Proposition \ref{proposition: qe} we obtain the following:
\begin{fact}
	The theory $T$ has $NIP$ and quantifier elimination.
\end{fact}

Without loss of generality, for convenience we can assume that $\C'$ is a reduct of a monster model of $\Th(\mathbb{R},+,-,1,\leq)$. So it makes sense to use $\leq$.
Let $p\in S_x(\C)$ be the complete invariant (over $\emptyset$) global type determined by 
$$ \{ \neg R_r(x,c)\wedge \neg R_r(c,x) : c\in \C, r \in \mathbb{Q}^+ \}.$$

As in the previous example, let $S_r(x,y):= R_r(x,y) \vee R_r(y,x)$. We say that $x,y$ are {\em related} if $S_r(x,y)$ holds for some $r \in \mathbb{Q}^+$. 
We denote by $E(x,y)$ the equivalence relation on $p(\C')$ relatively defined by $$\bigwedge_{r\in\mathbb{Q}^+}S_r(x,y).$$
In other words, this is the relation on $p(\C')$ of lying in the same coset modulo the subgroup of all infinitesimals in $\C'$ which will be denoted by $\mu$.

Other possible relatively type-definable over a $\C$-small subset of $\C$ equivalence relations on $p(\C')$ are as follows. Take any $c \in \C$. Let $E_c$ be the equivalence relation on $p(\C')$ 
given by $xE_c y$ if and only if $x=y$ or $x+y=c$.
It is clear that this is an equivalence relation on $p(\C')$ relatively defined by a type over $c$. We also have the equivalence relation $E_c^\mu$ given by  $ x E_c^\mu y$ if and only if $xEy$ 
or $(x+y)Ec$, 
which is also relatively defined by a type over $c$.

For any non-empty $\C$-small set $A$ of positive infinitesimals in $\C$ we will consider the equivalence relation $E_A$ on $p(\C')$ given as 
$$\bigwedge_{a \in A} \bigwedge_{n \in \mathbb{N}^+} |x-y| \leq \frac{1}{n} a.$$
Note that this relation is relatively type-definable over $A$ on $p(\C')$  in the original language $L$ by the following condition
$$\bigwedge_{a \in A} \bigwedge_{n \in \mathbb{N}^+} R_1(n(x-y),a) \vee R_1(n(y-x),a).$$

One can also combine the above examples to produce one more class of equivalence relations on $p(\C')$. 
Take any $c \in \C$ and any non-empty $\C$-small set $A$ of positive infinitesimals in $\C$. Let $\mu_A$ be the infinitesimals in $\C'$ defined by  
$$\bigwedge_{a \in A} \bigwedge_{n \in \mathbb{N}^+} |x| \leq \frac{1}{n} a.$$ 
Then we have the equivalence relation $E_{A,c}$ on $p(\C')$ given by $x E_{A,c} y$ if and only if $xE_Ay$ or $(x+y) E_A c$,
which is clearly relatively defined on $p(\C')$ by a type over $Ac$.

\begin{teor}\label{theorem: classification}
	The only equivalence relations on $p(\C')$ relatively type-definable over a $\C$-small subset of $\C$ are: the total equivalence relation, equality, $E$, the relations of the form $E_c$ or $E_c^\mu$ (where $c \in \C$), and the relations of the form $E_A$  or $E_{A,c}$ for any non-empty $\C$-small set $A$ of positive infinitesimals in $\C$ and any $c \in \C$. 
\end{teor}

In the proof below, by a 
non-constant term $t(x,y)$ (in the language $L$) we mean an expression $nx +my +k$, where $m,n,k \in \mathbb{Z}$ and $m \ne 0$ or $n \ne 0$.

\begin{proof}
	Let $F(x,y)$ be an arbitrary equivalence relation on $p(\C')$ relatively type-definable over a $\C$-small subset of $\C$. 
	
	\begin{claimnum}
		Either $F$ is the total equivalence relation, or $F$ is finer than $E_c^\mu$ (i.e., $F \subseteq E_c^\mu$) for some $c \in \C$ .
	\end{claimnum}
	
	\begin{proof}[Proof of Claim] 
		We consider two cases.

		Case 1: There are $ a$, $b\in p(\C')$ such that $aFb$ and $\models \bigwedge_{n \in \mathbb{Q}^+} \bigwedge_{c\in\C}\neg S_n(t(a,b),c)$
		for all non-constant terms $t(x,y)$. Take any $a',b'\in p(\C')$. 
		By compactness and $|\C|^+$-saturation of $\C'$, we can find $d'\in p(\C')$ such that $\models \bigwedge_{n \in \mathbb{Q}^+} \bigwedge_{c\in\C}\neg S_n(t(a',d'),c)$ and $\models \bigwedge_{n \in \mathbb{Q}^+} \bigwedge_{c\in\C}\neg S_n(t(b',d'),c)$ for all non-constant terms $t(x,y)$. Then, by q.e. and  
  Remark \ref{R},
		$(a',d')\equiv_{\C}(a,b)\equiv_{\C}(b',d')$. Since $F$ is $\C$-invariant, we conclude that $a'Fb'$, hence $F$ is the total equivalence relation.
		
		Case 2: For any $ a$, $b\in p(\C')$ with $aFb$ there are $n \in \mathbb{Q}^+$, $c\in \C$, and a non-constant term $t(x,y)$ such that $\models S_n(t(a,b),c)$. Suppose that for every $c \in \C$, $F$ is not finer than  $E_c^\mu$. We will reach a contradiction, but this will require quite a bit of work.
		
		First, we claim that there are $a,b\in p(\C')$ such that 
		$$(*) \;\;\;\;\;\; aFb \;\;\textrm{and}\;\; \models \bigwedge_{q\in \mathbb{Q}^+ } \neg S_q(a,b) \;\; \textrm{and}\;\; a+b\in p(\C').
		$$ 
		
		Firstly, note that by (topological) compactness of the intervals $[-r,r]$, $r \in \mathbb{Q}^+$, we easily get that $a\not\in p(\C')$ if and only if $a\in c+\mu$ for some $c\in \C$.
		Assume that $(*)$ does not hold, that is, for any $aFb$ we have $aE_c^\mu b$ for some $c\in \C$ or $S_n(a,b)\wedge \neg S_m(a,b)$ for some $m,n\in \mathbb{Q}^+$. Since $F$ is not contained in any $E_c^\mu$, either we get a pair $(a,b) \in F$ such that 
		$S_m(a,b) \wedge \neg S_n(a,b)$ for some $m,n\in \mathbb{Q}^+$, or we get two pairs $(a,b), (a',b') \in F$ and elements $c,c' \in \C$ such that $c-c' \notin \mu$ and $a+b -c \in \mu$ and $a'+b'-c' \in \mu$. In this second case, applying an automorphism of $\C'$ over $\C$  mapping $a'$ to $a$, we may assume that $a'=a$, and so we get $F(b,b')$ and $b-b' \in c-c' + \mu$. Then $b+b' \in 2b' +c -c' + \mu$ is not related to any element of $\C$ (as $2b'$ is not related),  
		so $b+b' \in p(\C')$. Since we assumed that $(*)$ fails, we conclude that  $S_m(b,b') \wedge \neg S_n(b,b')$ for some $m,n\in \mathbb{Q}^+$. 
		In this way, the whole second case reduces to the first one, i.e. we have a pair $(a,b) \in F$ with  $S_m(a,b) \wedge \neg S_n(a,b)$ for some $m,n\in \mathbb{Q}^+$.
		
		Let $\sigma\in\auto(\C'/\C)$ be such that $\sigma(a)=b$; set $b_k:=\sigma^k(a)$. 
		We produced an infinite sequence 
		\begin{center}
			\begin{tikzpicture}
				\node at (0, 0)  {$a$};
				\node at (1, 0)  {$b_1$};
				\node at (2, 0)  {$b_2$};
				\node at (3, 0)  {$\cdots$};
				\draw[-stealth]   (0.1,-0.15) to[out=-60,in=-120](0.9,-0.15);
				\node at (0.5, -0.5)  {$\sigma$};
				\draw[-stealth]   (1.1,-0.15) to[out=-60,in=-120](1.9,-0.15);
				\node at (1.5, -0.5)  {$\sigma$};
				\draw[-stealth]   (2.1,-0.15) to[out=-60,in=-120](2.9,-0.15);
				\node at (2.5, -0.5)  {$\sigma$};
			\end{tikzpicture}
		\end{center}
		Then for all $k\in\mathbb{N}^+$, $aFb_k$ and $\models S_{km}(a,b_k)$ and $\models \neg S_{kn}(a,b_k)$. Since  $\models S_{km}(a,b_k)$ and $b_k$ is not related to anything in $\C$, we get $\bigwedge_{q\in \mathbb{Q}^+ } \bigwedge_{c \in \C} \neg S_q(a, -b_k + c)$, 
		that is $a+b_k \in p(\C')$.
		As we can use arbitrarily large $k$, by compactness (or rather $|\C|^+$-saturation of $\C'$), there exist $b$  such that $(a,b)$ satisfies $(*)$, a contradiction.
		
		We will show now that there is $b'\in p(\C')$ such that 
		$$(**)\;\;\;\;\;\; aFb' \;\; \textrm{with}\;\; a+b',a-b'\in p(\C').$$
		Namely, either $b':=b$ already satisfies it, or $a-b$ is related to some infinite $c\in \C$.  In the latter case, $a-b$ is related precisely to the elements from the set $c+\mathbb{R}+\mu$. 
		
		Let $\sigma\in\auto(\C'/\C)$ be such that $\sigma(a)=b$; set $b_k:=\sigma^k(a)$. 
		We have
		
		\begin{center}
			\begin{tikzpicture}
				\node at (0, 0)  {$a$};
				\node at (1, 0)  {$b_1$};
				\node at (2, 0)  {$b_2$};
				\node at (3, 0)  {$\cdots$};
				\draw[-stealth]   (0.1,-0.15) to[out=-60,in=-120](0.9,-0.15);
				\node at (0.5, -0.5)  {$\sigma$};
				\draw[-stealth]   (1.1,-0.15) to[out=-60,in=-120](1.9,-0.15);
				\node at (1.5, -0.5)  {$\sigma$};
				\draw[-stealth]   (2.1,-0.15) to[out=-60,in=-120](2.9,-0.15);
				\node at (2.5, -0.5)  {$\sigma$};
			\end{tikzpicture}
		\end{center}
		Then $aFb_k$, and one easily checks that $a-b_k$ is related precisely to the elements from the set $kc + \R+\mu$, and so $a+b_k$ is not related to anything in $\C$. Since $c$ is infinite, the sets $kc + \mathbb{R} + \mu$ are pairwise disjoint for different $k$'s, and so we find the desired $b'$ using compactness (or rather $|\C|^+$-saturation of $\C'$).
		
		Since we are working in a divisible group, 
		using Remark \ref{R}, 
  we can replace the terms $t(x,y)$ in the statement of Case 2 by expressions $t_q(x,y):=nx-my$, where $q=\frac{n}{m}$ is the reduced fraction of $q$
		(i.e. $\gcd(m,n)=1$ with $m>0$). In particular, note that no term of the form $t(x,y)=nx$ or $t(x,y)=my$ can occur in Case 2 hypothesis, since that would contradict $a,b\in p(\C')$.
		Notice that for each 
		$d \in p(\C')$  there exists at most one rational $q$ such that $$S_k(t_q(a,d),c)$$ holds for some $c\in \C$ and $k\in \mathbb{Q}^+$. 
		For if there existed $q \ne q'\in\mathbb{Q}$ (with reduced fractions $\frac{n}{m}$ and $\frac{n'}{m'}$, respectively), $k,k' \in \mathbb{Q}^+$, and $c,c' \in \C$ such that 
		$S_{k}(t_q(a,d),c)$ and $S_{k'}(t_{q'}(a,d),c')$, 
		this would imply $S_{n'k+nk'}((mn'-m'n)d,nc'-n'c)$, contradicting that $d \in p(\C')$ when $mn'-m'n\neq 0$. 

		We will show now that there exists $b''\in p(\C')$ such that 
		$$aFb''\;\; \textrm{and} \;\; \models \bigwedge_{q \in \mathbb{Q}} \bigwedge_{n \in \mathbb{Q}^+}\bigwedge_{c\in\C}\neg S_n(t_q(a,d),c),$$ 
		contradicting the assumption of Case 2.
		
		Namely, either $d:=b'$ does the job, or there are $q \in \mathbb{Q}$, $n \in \mathbb{Q}^+$, and $c \in \C$ such that  $S_n(t_q(a,b'),c)$. By the choice of $a$ and $b'$ satisfying $(**)$, we have that $q \notin\{-1,0,1\}$.
		
		Again, let $\sigma\in\auto(\C'/\C)$ be such that $\sigma(a)=b'$; set $b'_k:=\sigma^k(a)$. We have
		
		\begin{center}
			\begin{tikzpicture}
				\node at (0, 0)  {$a$};
				\node at (1, 0)  {$b_1'$};
				\node at (2, 0)  {$b_2'$};
				\node at (3, 0)  {$\cdots$};
				\draw[-stealth]   (0.1,-0.15) to[out=-60,in=-120](0.9,-0.15);
				\node at (0.5, -0.5)  {$\sigma$};
				\draw[-stealth]   (1.1,-0.15) to[out=-60,in=-120](1.9,-0.15);
				\node at (1.5, -0.5)  {$\sigma$};
				\draw[-stealth]   (2.1,-0.15) to[out=-60,in=-120](2.9,-0.15);
				\node at (2.5, -0.5)  {$\sigma$};
			\end{tikzpicture}
		\end{center}
		Then $aFb_k'$ for all $k \in \mathbb{N}^+$. On the other hand, applying powers of $\sigma$, we easily conclude that for every $k \in \mathbb{N}^+$, $t_{q^{k}}(a,b_k')$ is related to some element of $\C$. Hence, by an observation above, we get that for all rationals $r \ne q^k$, $t_r(a,b_k')$ is not related to anything in $\C$. Since $q \notin\{-1,0,1\}$, we know that $q,q^2,\dots$ are pairwise distinct.  So, by compactness, the desired $b''$ exists.
	\end{proof}
	
	\begin{claimnum}
		$F \cap E$ is either equality, or $ E$, or $E_A$ for some non-empty $\C$-small set $A$ of positive infinitesimals in $\C$.
	\end{claimnum}
	
	\begin{proof}[Proof of Claim]

		We may assume that $F \subseteq E $, and just work with $F$. Let $B$ be a $\C$-small $\dcl$-closed subset of $\C$ over which $F$ is relatively defined on $p(\C')$. Extending the notation from before the statement of Theorem \ref{theorem: classification}, for any $B' \subseteq B$ put $$E_{B'}: =\{(x,y) \in p(\C')^2: \bigwedge_{b \in B'^{+}} \bigwedge_{n \in \mathbb{N}^+} |y-x| \leq \frac{1}{n}b\},$$
		where 
		$B'^{+}:= \{ b \in B': 0< b \leq 1\}$. Let $A := \bigcup\{B' \subseteq B: F \subseteq E_{B'}\}$. Then 
		$$F \subseteq \bigcap \{E_{B'}: B' \subseteq B \;\; \textrm{such that}\;\; F \subseteq E_{B'}\} = E_A,$$
		and, as $F \subseteq E $, we have that $1 \in A$.
		
		We will show that either $F$ is equality, or $F= E_A$. This will clearly complete the proof of the claim (note that if $A$ does not contain any positive infinitesimals, then $E_A = E$).
		Suppose $F$ is not the equality. It remains to show that $F \supseteq E_A$.
		
		Case 1: $A=B$. Pick any distinct $\alpha ,\beta \in p(\C')$ such that $\alpha F\beta $. 
		Then $$\bigwedge_{a \in A^+} |\alpha-\beta| \leq a.$$ Consider any $\alpha',\beta'  \in p(\C')$ with $\alpha'E_A\beta'$. Then either $\alpha'=\beta'$ (and so $\alpha'F\beta'$), or $\bigwedge_{a \in A^+} 0<|\beta'-\alpha'|\leq a$. In the latter case, it remains to show that $\alpha \beta \equiv_A \alpha' \beta'$  or $\alpha \beta \equiv_A \beta'\alpha'$ 
		(as then $\alpha'F\beta'$, since $F$ is relatively type-definable over $A$). Without loss of generality, $\beta > \alpha$ and $\beta'>\alpha'$; equivalently, $R_1(\alpha,\beta)$ and $R_1(\alpha',\beta')$ both hold. Since $\alpha \equiv_{\C} \alpha'$, we can assume that $\alpha=\alpha'$. It suffices to show that $$\{0<t-\alpha \leq a: a \in A^+\}$$ determines a complete type over $\dcl(A,\alpha)$. 
		By o-minimality 
		of $(\mathbb{R},+,-,1,\leq)$, this boils down to showing that there is no $b \in \dcl^*(A,\alpha)$ with $\bigwedge_{a \in A^+} \alpha<b\leq\alpha+a$, where $\dcl^*$ is computed in the language $\{+,-,1,\leq\}$. If there was such a $b$, then, by q.e. for the theory of divisible ordered abelian groups,  
		it would be of the form $\gamma + q\alpha$ for some $\gamma \in A$ and $q \in \mathbb{Q}$, and we would have $\bigwedge_{a \in A^+} 0< \gamma +(q-1)\alpha \leq a$. If $q=1$, we get 
		$0<\gamma\leq \frac{1}{2}\gamma<\gamma$, a contradiction. If $q \ne 1$, we get that $\alpha$ is related to an element of $A$ which contradicts the fact that $\alpha \in p(\C')$.
		
		Case 2: $A \subsetneq B$. 
		Take any $b\in B\setminus A$. 
		Then, by maximality of $A$, $F\nsubseteq E_{A\cup \{b\}}$, so there is $(x,y)\in F$ such that $(x,y)\notin  E_{A\cup \{b\}}=E_A \cap E_b$; swapping $x$ and $y$ if necessary, we may assume that $y>x$. As $F\subseteq E_A$, we have that $(x,y)\notin E_b$. 
		In particular, this implies that $b \in B^+$ and $y-x > \frac{1}{n}b$ for some $n\in \mathbb{N}^+$. 
		Since $F\subseteq E_A$, we have that $\lvert y-x \rvert< \frac{1}{n}a$ for all $a\in A^+$ and $n\in \mathbb{N}^+$, concluding $\bigwedge_{a \in A^+} b<a$. 
		
		Let $\sigma \in\auto(\C'/\C)$ be such that $\sigma(x)=y$; set $y_k:=\sigma^{k}(x)$. We have
		
		\begin{center}
			\begin{tikzpicture}
				\node at (0, 0)  {$a$};
				\node at (1, 0)  {$y_1$};
				\node at (2, 0)  {$y_2$};
				\node at (3, 0)  {$\cdots$};
				\draw[-stealth]   (0.1,-0.15) to[out=-60,in=-120](0.9,-0.15);
				\node at (0.5, -0.5)  {$\sigma$};
				\draw[-stealth]   (1.1,-0.15) to[out=-60,in=-120](1.9,-0.15);
				\node at (1.5, -0.5)  {$\sigma$};
				\draw[-stealth]   (2.1,-0.15) to[out=-60,in=-120](2.9,-0.15);
				\node at (2.5, -0.5)  {$\sigma$};
			\end{tikzpicture}
		\end{center}
		We easily conclude that $F(x,y_k)$ and $y_k -x > \frac{k}{n}b$ for all $k$; in particular, $y_n -x>b$. By compactness (or rather $|\C|^{+}$-saturation of $\C'$), there exist $x',y' \in p(\C')$ such that $F(x',y')$ and:
		\begin{enumerate}
			\item $\bigwedge_{a \in A^+} 0<y'-x'<a$;
			\item $\bigwedge_{b \in B^+\setminus A^+} b<y'-x'$.
		\end{enumerate}
		
		We will check now that whenever $x'',y'' \in p(\C')$ satisfy (1) and (2), then $x'y' \equiv_B x''y''$.
		For that, without loss of generality, we can assume that $x'=x''$. It remains to show that the partial type 
		$$\pi(t/x'):= \{0<t-x'<a: a \in A^+\} \cup \{b<t-x': b \in B^+\setminus A^+ \}$$ 
		determines a complete type over $\dcl(B,x')$.  
		By o-minimality of $(\mathbb{R},+,-,1,\leq)$, this boils down to showing that there is no $c \in \dcl^*(B,x')$ realizing $\pi(t/x')$, where $\dcl^*$ is computed in the language $\{+,-,1,\leq\}$. If there was such a $c$, then, by q.e. for the theory of divisible ordered abelian groups,  it would be of the form $\beta + qx'$ for some $\beta \in B$ and $q \in \mathbb{Q}$, so 
		$$ \bigwedge_{a\in A^+} \bigwedge_{b \in B^+\setminus A^+}  b<\beta +(q-1)x'<a.$$
		If $q=1$, we get $\bigwedge_{a \in A^+} 0<\beta<a$, so $\beta \in B^+\setminus A^+$, concluding $\beta < \beta$, a contradiction. 
		If $q \ne 1$, as $1\in A$, we get that $x'$ is related to an element of $B$, which contradicts the fact that $x' \in p(\C')$.
		
		Finally, consider any $(\alpha,\beta) \in E_A$, say with $\beta >\alpha$ so $0<\beta-\alpha<\frac{1}{2}a$ for all $a\in A^+$. 
		Applying $\sigma \in \auto(\C'/\C)$ mapping $y'$ to $\alpha$, we obtain $\gamma:=\sigma(x')$ such that $\gamma F\alpha$ and $\bigwedge_{b \in B^+\setminus A^+} b<\alpha-\gamma$. Since $F\subseteq E_A$, we get  $b<\alpha -\gamma < \frac{1}{2}a$ for all $b\in B^+\setminus A^+$ and $a\in A^+$. Therefore, $b<\beta - \gamma < a$ for all $b\in B^+\setminus A^+$ and $a\in A^+$, 
		So, by the previous paragraph, $\gamma \alpha \equiv_B x'y' \equiv_B \gamma\beta$. As $(x',y') \in F$, we conclude that $(\alpha,\beta) \in F$, which completes the proof of the claim.
	\end{proof}
	
	By the above two claims, in order to prove the theorem, it remains to consider the case when $E\cap F\subsetneq F \subseteq E_{c_0}^\mu$ for some $c_0 \in \C$.
	By the second claim, we have the following two cases.
	
	Case 1: $E \cap F$ is the equality. We will show that then $F =E_c$ for some $c \in \C$. Consider any $a \in p(\C')$. Since $F \ne \;=$, there exists $b \ne a$ such that $aFb$. Since $F \subseteq E_{c_0}^\mu$ and $E \cap F$ is the equality, we get that such a $b$ is unique: if $b' \ne a$ also satisfies $aFb'$, then $b,b' \in -a+c_0 + \mu$, so $b-b' \in \mu$, hence $b=b'$ because $bFb'$. 
	This unique $b$ belongs to $\dcl(\C, a)$, so $g:=a+b-c_0 \in \dcl(\C,a) \cap \mu$.
	Since $a$ is not related to any element of $\C$ and $\dcl$ is given by ``terms" with rational coefficients (which follows from q.e. for $T$),  we get that $g \in \C$.
	Hence, $c:=a+b=c_0+g \in \C$. Applying automorphisms over $\C$, we get $F=E_c$.
	
	Case 2: $E \cap F=E$ or $E \cap F=E_A$ for some non-empty $\C$-small set $A$ of positive infinitesimals. Since $E=E_{\{1\}}$ (with the obvious extension of the definition of $E_A$), we can write $E \cap F=E_A$, where $A$ is either a non-empty $\C$-small set $A$ of positive infinitesimals or $A=\{1\}$. We will show that then $F =E_{A,c}$ for some $c \in \C$, where $E_{\{1\},c}:=E^\mu_c$. Extend the definition of $\mu_A$ via $\mu_{\{1\}}:=\mu$.
	
	Consider any $a \in p(\C')$. Since $E\cap F \ne F$ and $F \subseteq E_{c_0}^\mu$, there exists $b\in p(\C')$ such that $(a,b) \in F \setminus E$ and $a+b = c_0 +g$ for some $g \in \mu$. As  $E \cap F=E_A$, we get that $\sigma(g)-g \in \mu_A$ for every $\sigma \in \auto(\C'/\C a)$. 
	
	Since $\sigma(g)-g \in \mu_A$ for every $\sigma \in \auto(\C'/\C a)$, we conclude by o-minimality of $(\mathbb{R},+,-,1,\leq)$ that, for every $\alpha\in A$ and $\in \mathbb{N}^+$, there are $c_{\alpha,n}, d_{\alpha,n} \in \dcl^*(\C,a)$ such that $g-\frac{1}{n}\alpha<c_{\alpha,n}\leq g\leq d_{\alpha,n}<g+ \frac{1}{n}\alpha$, 
	where  $\dcl^*$ is the definable closure computed in the language $\{+,-,1,\leq\}$ (which coincides with $\dcl$ as both closures are given by ``terms" with rational coefficients). 
	Since $a$ is not related to any element of $\C$ and for every $\alpha\in A$ and $n\in \mathbb{N}^+$ the elements $c_{\alpha,n}, d_{\alpha,n}$ are related to zero, using that $\dcl^*$ is given by  ``terms'' with rationals coefficients, we conclude that $c_{\alpha,n}$ and $d_{\alpha,n}$ belong to $\C$ for every $\alpha\in A$ and $n\in \mathbb{N}^+$. Since $A$ is $\C$-small, the set of all $c_{\alpha,n}$ and $d_{\alpha,n}$ is $\C$-small, and hence there is $e \in \C$ with $g-\frac{1}{n}\alpha< e <g + \frac{1}{n}\alpha$ for all $\alpha \in A$ and $n \in \mathbb{N}^+$. Then, $g\in e+\mu_A$ with $e\in \C$, concluding $a+b\in c+\mu_A$, where $c=c_0+e\in \C$, so $aE_{A,c}b$.

	From the conclusion of the previous paragraph and the fact that $E \cap F=E_A$, we obtain $(a+\mu_A) \cup (-a+ c +\mu_A) \subseteq [a]_F$. By automorphisms over $\C$, the same is true for any other element of $p(\C')$ in place of $a$, so $E_{A,c} \subseteq F$. The opposite inclusion easily follows using the assumptions $F \subseteq E_{c_0}^\mu$ and $E \cap F=E_A$. Namely, using automorphisms over $\C$, it is enough to show that $[a]_F \subseteq [a]_{E_{A,c}}$. Consider any $b' \in [a]_F$. Since $(a,b) \in F \setminus E$ and  $F \subseteq E_{c_0}^\mu$, we have that $b' \in a + \mu$ or $b' \in b + \mu$. As $E \cap F=E_A$, we conclude that $b' \in a + \mu_A$ (and so $b'E_{A,c}a$) or $b' \in b+ \mu_A$ (and so $b'E_{A,c}b$ which together with $aE_{A,c}b$ implies $b'E_{A,c}a$).
\end{proof}

\begin{cor}
	The equivalence relation $E$ is the finest equivalence relation on $p(\C')$ relatively type-definable over a $\C$-small set of parameters of $\C$ and with stable quotient, that is $E^{\textrm{st}}=E$
\end{cor}
\begin{proof}
	The quotient $p(\C')/E$ is stable by Proposition \ref{proposition: stable quotient}. 
	Let $F$ be a relatively type-definable over a $\C$-small subset of $\C$ equivalence relation on $p(\C')$ strictly finer than $E$. 
	By Theorem \ref{theorem: classification}, $F=E_A$ for some non-empty $\C$-small set $A$ of positive infinitesimals in $\C$. 
	(There is also the case when $F$ is the equality, but then $p(\C')/F=p(\C')$ is clearly unstable.)
	
	Pick any $\alpha \in A$.
	It is easy to check that for every $a,d,e\in \C'$ and infinitesimal $c\in \C'$ bigger than all infinitesimals in $\C$ and such that $|d - a|\leq \frac{1}{2}\alpha$ and $|e -(a+c)|\leq\frac{1}{2}\alpha$, we have $d<e$, and so $\neg R_1(e,d)$. And note that ``$|x-y|\leq\frac{1}{2}\alpha$'' can be written as an $L_\alpha$-formula.

	
	Take any $a \in p(\C')$ and infinitesimal $c\in \C'$ bigger than all infinitesimals in $\C$.
	Using Ramsey's theorem and compactness, 
	we find a $\C$-indiscernible sequence $(a_i')_{i<\omega}$ having the same Erenfeucht-Mostowski type as the sequence $(a+kc)_{k<\omega}$.
	Then, the sequence $([a'_i]_F)_{i<\omega}$ is $\C$-indiscernible but not totally $\C$-indiscernible, since the formula $R_1(x,y)$ witnesses that $$ \tp([a'_i]_F,[a'_{i+1}]_F/ \C)\neq\tp([a'_{i+1}]_F, [a'_i]_F/ \C). $$ 
	Thus, $p(\C')/F$ is unstable.
\end{proof}

\chapter{$n$-dependent continuous theories and hyperdefinable sets}\label{Chapter 5}
\section{Generalized indiscernibles}\label{section: modelling property}
Let $\mL'$ be a first-order language and $\mL$ be a continuous logic language. Unless specified otherwise, $T$ is a complete continuous $\mL$-theory with $\C\models T$ a monster model (i.e. $\kappa$-saturated and strongly $\kappa$-homogeneous for a strong limit cardinal $> |T|$) and $\I,\mathcal{J}$ are $\mL'$-structures. 
As in Chapter \ref{Chapter 3}, we will use the formalism from \cite[Subsection 3.1]{hrushovski2021amenability} and \cite[Section 3]{hrushovski2021order} and adapt some results from \cite{MR2657678}.  In particular, by a {\em CL-formula over $A$} we mean a continuous function $\varphi \colon S_n(A)\to \mathbb{R}$.  If $\varphi$ is such a CL-formula, then for any $\bar b\in M^n$ (where $M \models T$) by $\varphi(\bar b)$ we mean $\varphi(\tp(\bar b/A))$.
In the sections concerning hyperdefinable sets, we allow the domain of a CL-formula to be an infinite Cartesian power of $\C$ (which is necessary to deal with $X/E$ in the case when $X\subseteq \C^\lambda$ where $\lambda$ is infinite).

In this section, we present natural adaptations of the concepts of generalized indiscernibles and the modeling property to continuous logic and give a characterization of the continuous modeling property in the form of a continuous logic counterpart of \cite[Theorem 2.10]{SCOW2021102891}.

The following idea first appeared in \cite[Definition VIII.2.4]{Shelah1982-SHECTA-5}.

\begin{defin}\label{defin: gen. indisc.}
    Let $\II=(a_i: i\in \I)$ be an $\I$-indexed sequence, and let $A\subset \C$ be a small set of parameters.  We say that $\II$ is an $\I$-\emph{indexed indiscernible sequence over} $A$ if for all $n\in\omega$ and all sequences $i_1,\dots,i_n,j_1,\dots,j_n$ from $\I$ we have that 
    $$\qftp(i_1,\dots,i_n)=\qftp(j_1,\dots,j_n) \implies \tp({a}_{i_1}, \dots, {a}_{i_n}/A)=\tp({a}_{j_1}, \dots, {a}_{j_n}/A).$$
\end{defin}

We will refer to $\I$-indexed indiscernible sequences as $\I$-indiscernibles.

Next, we adapt the definition of \emph{locally based on} given in \cite{10.1215/00294527-3132797}. The first reference to this concept can be found in \cite{MZiegler}.

\begin{defin}[Locally based on] 
Let $\II=(a_i:i\in \I)$ be an $\I$-indexed sequence. We say that a $\mathcal{J}$-indexed sequence $(b_j:j\in\mathcal{J})$ is \emph{locally based on $\II$} if for any finite set of $\mL$ formulas $\Delta$, any finite tuple $\overline{j}\subseteq \mathcal{J}$ and $\varepsilon>0$ there is $\overline{i}\subseteq \I$ such that:
\begin{enumerate}
    \item $\qftp(\overline{i})=\qftp(\overline{j})$ in the language $\mL'$.
    \item $\lvert \varphi(b_{\overline{j}})- \varphi(a_{\overline{i}}) \rvert\leq \varepsilon$ for all $\varphi\in \Delta$.
\end{enumerate}
\end{defin}

 The original definition presented in \cite[Definition 2.5]{10.1215/00294527-3132797} is the following:
 \begin{defin}[Classical definition of Locally based on]
Let $\II=(a_i:i\in \I)$ be an $\I$-indexed sequence. We say that a $\mathcal{J}$-indexed sequence $(b_j:j\in\mathcal{J})$ is \emph{locally based on $\II$} if for any finite set of $\mL$ formulas $\Delta$ and any finite tuple $\overline{j}\subseteq \mathcal{J}$ there is $\overline{i}\subseteq \I$ such that:
\begin{enumerate}
    \item $\qftp(\overline{i})=\qftp(\overline{j})$.
    \item $\tp^{\Delta}(b_{\overline{j}})= \tp^{\Delta}(a_{\overline{i}})$.
\end{enumerate}
\end{defin}
 
 Note that if we tried to use this stronger version of the property, it is easy to show that even for $\I=(\mathbb{N},<)$ we can find a sequence for which there are no indiscernible sequences locally based on it. Consider for example the theory $\Th([0,1],d)$ where $d$ is the distance predicate and the sequence $(1/n)_{n<\omega}$.

The next definition is then the natural continuous counterpart of \cite[Definition 2.17]{SCOW20121624}.

\begin{defin}[Continuous Modeling property]\label{definition: CMP}
    Given a continuous theory $T$, we say that \emph{$\I$-indexed indiscernibles have the continuous modeling property in $T$} if given any $\I$-indexed sequence $\II=(a_i:i\in \I)$ in a monster model $\C$ of $T$ there exists an $\I$-indiscernible sequence $(b_i:i\in\I)$ in $\C$ locally based on $\II$. We say that $\I$ has the \emph{continuous modeling property} if $\I$-indexed indiscernibles have the continuous modeling property in every continuous theory.
\end{defin}

If a first-order structure $\I$ has the continuous modeling property then it has the modeling property. First, note that given a classical first-order theory $T$, there always exists a continuous logic theory $T'$ such that any model of $T'$  is a model of $T$ with the discrete metric. We show that the following holds:

\begin{prop} \label{CMP implies MP}
    Let $T$ be a first-order theory and let $T'$ be its continuous logic counterpart. Then $\I$ has the continuous modeling property in $T'$ if and only if $\I$ has the modeling property in $T$.
\end{prop}

\begin{proof}
    Clearly, If $\I$ has the continuous modeling property in $T'$ then it has the modeling property in $T$ since classical formulas are a subset of the $\{0,1\}$-valued continuous logic formulas.

    Assume now that $\I$ has the modeling property in $T$. Let $\II=(a_i:i\in \I)$ be any sequence in $\C\models T$. Since $\I$ has the modeling property, there is an $\I$-indiscernible sequence $(b_i:i\in \I)$ locally based on $\II$ (in the classical sense). We show that the sequence $(b_i:i\in \I)$ is locally based on $\II$ in our continuous logic sense. Note that first-order formulas generate a subalgebra $$\mathcal{A}=\bigcup_{n<\omega} \{ u(\varphi_1,\dots, \varphi_n):  \varphi_1,\dots, \varphi_n \text{ f.o. formulas}; u \text{ CL connective}\}$$ which is a dense subalgebra of the set of all continuous logic formulas. Hence, for each continuous logic formula $f(x)$ and $\varepsilon>0$ there is $\varphi(x)\in \mathcal{A}$ such that $\lvert f(x)-\varphi(x)\rvert \leq \varepsilon/2$. Thus, for any tuples 
    $\overline{i},\overline{j}\subseteq \I$ 
    and tuples $a_{\overline{j}}, b_{\overline{i} }$ we have $$ \lvert f(a_{\overline{j}}) -f(b_{\overline{i}})\rvert \leq \lvert f(x)-\varphi(x)\rvert + \lvert \varphi(a_{\overline{j}}) -\varphi(b_{\overline{i}}) \rvert + \lvert f(x)-\varphi(x)\rvert \leq \lvert \varphi(a_{\overline{j}}) -\varphi(b_{\overline{i}}) \rvert +\varepsilon.$$
    Finally, note that by the definition of being locally based on (in the classical sense) for any $b_{\overline{i}}$ and finite $\Sigma \subset \mathcal{A}$, there is $\overline{j}$ with the same quantifier free type as $\overline{i}$ such that $ \varphi(b_{\overline{i}})=\varphi(a_{\overline{j}}) $ for every $\varphi\in \Sigma$. 
    Therefore, the sequence $(b_i:i\in \I)$ is locally based on $\II$ in the continuous sense.
\end{proof}

Next, we define two partial types that will be useful during this section. The first one is a generalization of the classical Ehrenfeucht-Mostowski type (EM-type for short). The second is a type whose realizations are exactly the $\I$-indiscernible sequences. They are based on \cite[Definitions 2.6 and 2.10]{SCOW20121624} respectively.

\begin{defin}
Let $\II=(a_i:i\in \I)$ be an $\I$-indexed sequence. The \emph{$EM$-type of $\II$} is the set of all conditions $\varphi(x_{i_1}\dots,x_{i_n})=0$ such that $\varphi(a_{j_1}\dots,a_{j_n})=0$ holds for every $j_1,\dots,j_n\in \I$ with the same quantifier free type as $i_1,\dots,i_n$. That is \begin{align*}
    \emtp(\II)(x_i: i\in \I)=\{& \varphi(x_{i_1},\dots,x_{i_n})=0 :  \varphi \in \mL, i_1,\dots,i_n\in \I  \\ &\text{ and for any } j_1,\dots,j_n\in \I  \text{ such that } \\ &\qftp(j_1,\dots,j_n)=\qftp(i_1,\dots,i_n),   \models\varphi(a_{j_1}\dots,a_{j_n})=0 \}.
\end{align*}
\end{defin}

\begin{defin}
    We define $\Ind(\I,\mL)$ as the following partial type:
    \begin{align*}
       \Ind(\I,\mL)(x_i:i\in \I):= \{& \varphi(x_{i_1},\dots,x_{i_n})= \varphi(x_{j_1},\dots,x_{j_n}):&\\
        & n<\omega, \overline{i}, \overline{j}\subseteq \I, \qftp(\overline{i})=\qftp(\overline{j}), \varphi(x_{i_1},\dots,x_{i_n})\in \mL\}.
    \end{align*}
\end{defin}

Finally, we define what it means for a partial type to be finitely satisfiable in a sequence. 

\begin{defin}[Finitely satisfiable] Let $\Gamma(x_i:i\in \I)$ be an $\mL$-type, and let $\II=(a_i:i\in \I)$ be an $\I$-indexed sequence. We say that $\Gamma$ is \emph{finitely satisfiable in $\II$} if for every finite $\Gamma_0\subseteq \Gamma^+$ and for every finite $A\subseteq \I$, there is $B\subseteq \I$, a bijection $f:A\to B$, and an enumeration $\overline{i}$ of $A$ such that:
$$ \qftp(\overline{i})=\qftp(f[\overline{i}])\text{ and } (a_{f(i)}:i\in A)\models \Gamma_0\upharpoonright \{x_i: i\in A\}. $$ 
Where $\Gamma^+:=\{ \varphi \leq 1/n: n<\omega ; (\varphi=0)\in \Gamma\}$.
\end{defin}

The following result gives a sufficient condition for the existence of an $\I$-indiscernible sequence locally based on $\II=(a_i:i\in \I)$.

\begin{lema}
\label{EMTP, based on and Ind(I,L)}
Let $\mathcal{J} \supseteq \I$ be $\mL'$-structures with the same age and let $\II=(a_i:i\in \I)$ be an $\I$-indexed sequence. 
\begin{enumerate}
    \item  A $\mathcal{J}$-indexed sequence $\mathbf{J}=(b_j:j\in \mathcal{J})$ is locally based on $\II$ if and only if $\emtp(\mathbf{J})\supseteq \emtp(\II)$.
    \item If $\Ind(\I,\mL)$ is finitely satisfiable in $\II$, then there is an $\I$-indexed indiscernible sequence $\Tilde{\II}:=(b_i:i\in \mathcal{I})$ locally based on $\II$.
\end{enumerate}
\end{lema}
\begin{proof}
    \begin{enumerate}
        \item Suppose $\mathbf{J}$ is locally based on $\II$. Fix $\varphi(x_{i_1},\dots,x_{i_n})=0\in \emtp_{\mL'}(\II)$ and let $\overline{i}=(i_1,\dots,i_n)$. If $ \varphi(x_{\overline{i}})=0$ is not in $\emtp(\mathbf{J})$, then $\varphi(b_{\overline{j}})\geq \varepsilon$ for some $\varepsilon>0$ and $\overline{j}\subseteq \mathcal{J}$ with the same quantifier free type as $\overline{i}$. By assumption, there is $\overline{i}'\subseteq \I$ satisfying the same quantifier free type as $\overline{j}$ and such that $\varphi(a_{\overline{i}'})\geq \varepsilon/2$, which contradicts $\varphi(x_{\overline{i}})\in \emtp(\II)$.

        Suppose now that $\emtp(\mathbf{J})\supseteq \emtp(\II)$. For a contradiction, assume that $\mathbf{J}=(b_j:j\in \mathcal{J})$ is not locally based on $\II$. That is, there is $\Delta \subseteq \mL$, $b_{\overline{j}}:=(b_{j_1},\dots,b_{j_n})$ from $\mathbf{J}$ and $\varepsilon>0$ such that there is no $\overline{i}\subseteq \I$ satisfying $\qftp(\overline{i})=\qftp(\overline{j})$ and $\lvert \varphi(b_{\overline{j}})- \varphi(a_{\overline{i}}) \rvert\leq \varepsilon$ for all $\varphi\in \Delta$. Let $\psi(x):=\max \{ \lvert  \varphi(x) - \varphi(b_{\overline{j}})\rvert : \varphi\in \Delta \}$, $\psi$ is a continuous logic formula and $\psi(b_{\overline{j}})=0$. By assumption, for any $\overline{i}\subseteq \I$ with the same quantifier free type as $\overline{j}$, $\psi(a_{\overline{i}})\geq \varepsilon$. Thus, $\psi(x)\geq \varepsilon \in \emtp(\II)$, which contradicts $\emtp(\mathbf{J})\supseteq \emtp(\II)$.
    
    \item Observe that if the type  $\Ind(\I,\mL)(x_i: i\in \I)$ 
     is finitely satisfiable in $\II$, then 
    $\Ind(\I,\mL) \cup \emtp(\II)$ is satisfiable. Let $\mathbf{J}\models \Ind(\I,\mL)\cup \emtp(\II)$. $\mathbf{J}$ is an $\I$-indiscernible sequence and is locally based on $\II$ by $(1)$.
    \end{enumerate}    
\end{proof}

We now prove the main result of this section. It is an extension of \cite[Theorem 2.10]{SCOW2021102891} to continuous logic. The reader can recall the definitions of the concepts appearing in Definition \ref{defin ERP}, Definition \ref{defin: arrow notation} and Definition \ref{defin: locally finite and age}.

\begin{teor}\label{CMP iff Ramsey} Let $\mL'$ be a first-order language 
and let $\I$ be an infinite locally finite $\mL'$-structure
. Then, the following are equivalent:
\begin{enumerate}
    \item $\age(\I)$ has ERP.
    \item $\I$-indiscernibles have the continuous modeling property.
\end{enumerate}
\end{teor}
\begin{proof} 
    $(1)\implies (2)$. Assume $\age(\I)$ has ERP and let $\II= (a_i)_{i\in \I}$ be any $\I$-indexed sequence. Our goal is to prove that there exists $\mathbf{J}=(b_i)_{i\in \I}$ locally based on $\II$. By Lemma \ref{EMTP, based on and Ind(I,L)}, it is enough to show that $\Ind(\I,\mL)$ is finitely satisfiable in $\II$.

    Let $\Gamma_0\subset \Ind(\I,\mL)^+$ be any finite subset. For some $K,M<\omega$  $$\Gamma_0=\{ \lvert \varphi(x_{\overline{i}_p})-\varphi(x_{\overline{j}_p})\rvert <\frac{1}{n_m}: \qftp(\overline{i}_p)=\qftp(\overline{j}_p), \varphi\in \Delta, p<K,m<M  \}.$$ $\Gamma_0$ involves finitely many formulas $\Delta:=\{ \varphi_0,\dots,\varphi_m\}$, finitely many tuples $\overline{i}_p$, $\overline{j}_p$ and finitely many rationals $\frac{1}{n_m}$. Without loss of generality, we may assume that the formulas  $\varphi\in \Delta$ (and their tuples of variables) are of the form $\varphi((x_g)_{g\in A})$, where $A\in \age(\I)$ is the $\mL'$-structure generated by the set of indices of the original tuple of variables. Let $B\in \age(\I)$ be the structure generated by all the 
    coordinates of the tuples $\overline{i}_p$, $\overline{j}_p$ involved in $\Gamma_0$. It is enough to prove the existence of a copy $B'$ of $B$ such that for any $\varphi((x_g)_{g\in A})\in \Delta$ and $A',A''\subseteq B'$ copies of $A$,  $$\lvert \varphi((a_g)_{g\in A'}) - \varphi((a_g)_{g\in A''})\rvert \leq \frac{1}{n} $$ for some $n<\omega$ such that $1/n$ is smaller than any rational involved in $\Gamma_0$.

    If $\Delta$ involves only one formula $\varphi((x_g)_{g\in A})$, we proceed in the following manner: linearly order the set of intervals $\{ [\frac{i}{n},\frac{i+1}{n}]: i<n \}$ and define an $n$-coloring of the copies $A'$ of $A$ by coloring each $A'$ with the first interval that contains $\varphi((a_g)_{g\in A'})$. Since $\age(\I)$ is Ramsey, we can find a copy $B'$ of $B$ homogeneous with respect to the coloring.
    Then, $(a_g)_{g\in B'}$ witness that $\Gamma_0$ is satisfied in $\II$. If $\Delta$ involves $k<\omega$ formulas $ \{ \varphi_i((x_g)_{g\in A_i}) : i<k\}$ and all the sets $A_i$ involved are isomorphic we can apply a similar trick, using as colors the hypercubes $$\{ [\frac{i_1}{n}, \frac{i_1+1}{n}]\times\cdots\times [\frac{i_k}{n}, \frac{i_k+1}{n}]: i_1,\dots,i_k<n \}.$$ 

    We claim that the latter is the only case we need to check. The proof is a standard argument in Ramsey theory which we sketch here for completeness
    
    Let $A_1\dots,A_m$ be structures in $\age(\I)$ and let $B\in \age(\I)$ embed every $A_i$ for $i<m$. Let $k_1,\dots,k_n$ be natural numbers and let $Z_n\in \age(\I)$ be such that $$Z_n\to (Z_{n-1})^{A_n}_{k_n}$$  for every $n<m$. We construct by induction a sequence of structures $Y_n\in \age(\I)$ for $0\leq n\leq m$.

    Case $n=0$: $Y_0=Z_m$

    Case $0<n<m$: By induction we have $Y_{n-1}\in \age(\I)$ isomorphic to $Z_{m-n+1}$. Color the copies of $A_{m-n+1}$ inside $Y_{n-1}$ with $k_{n-1}$ colors. By definition of $Z_{m-n+1}$, there is a copy $Y_n$ of $Z_{m-n}$ inside $Y_{n-1}$ such that all of the copies of $A_{m-n+1}$ contained in $Y_n$ have the same color.

Note that since $Y_n\subseteq Y_{n-1}$ for $0\leq n \leq m$ and all copies of $A_{m-n+1}$ inside $Y_n$ are of the same color, we have that $Y_n$ is homogeneous for copies of $A_j$ for all $m-n+1\leq j \leq m$. Therefore, $Y_m$ is homogeneous for all copies of $A_1,\dots,A_m$ and so it is the $B'$ we were looking for in the proof.

    $(2)\implies(1)$. 
    Let $A\subseteq B \in \age(\I)$ be arbitrary finite substructures of $\I$ and let  $\chi$ be a $k$-coloring of the embeddings of $A$ into $\I$. We expand $\I$ by adding a predicate $R_i$ for each fiber of the coloring. Let us denote this expanded structure by $\I'$ and this new language by $\mL$. Let $T$ be the $\mL$-theory of $\I'$. Since $\I$ has the modeling property in $T$, there is an $\I$-indiscernible sequence $(b_i)_{i\in \I}$ locally based on $(i)_{i\in \I}$. Using the definition of locally based on for $\Delta:=\{ R_1,\dots,R_k \}$ we can find and embedding $f$ from $B$ into $\I$ such that \begin{align*}
    \qftp_{\mL'}(B)&=\qftp_{\mL'}(f[B])\\ &\text{ and}\\
    \tp^{\Delta}((b_g)_{g\in B})&=\tp^{\Delta}( (f(g))_{g\in B}).
    \end{align*}
    This implies that $\chi\upharpoonright_{f\circ \binom{B}{A}}$ is constant.
\end{proof}

In light of the previous theorem we will not make a distinction between continuous or classical modeling property from now on.

\section{Characterizing n-dependence through collapse of indiscernibles}\label{section:n-dep}
Let $T$ be a complete continuous $\mL$-theory with $\C\models T$ a monster model (i.e. $\kappa$-saturated and strongly $\kappa$-homogeneous for a strong limit cardinal $\kappa > |T|$). 

 In this section, we study $n$-dependent continuous formulae and give an analogous result to \cite[Theorem 5.4]{Chernikov2014OnN}, or proof gives an alternative method of achieving this result in the classical context. 

The next few paragraphs contain basic facts about hypergraphs taken almost verbatim from \cite{Chernikov2014OnN}.

We work with three families of languages 
\begin{align*}
    \mL^n_{op}&=\{<, P_0(x),\dots, P_{n-1}(x)\},\\
    \mL^n_{og}&=\{<, R(x_0,\dots,x_{n-1})\},\\
    \mL^n_{opg}&=\{<, R(x_0,\dots,x_{n-1}), P_0(x),\dots, P_{n-1}(x)\}.
\end{align*}
  When $n<\omega$ is clear we will simply omit it. We consider the Ramsey classes of ordered $n$-uniform hypergraphs and ordered $n$-partite $n$-uniform hypergraphs.

An $\mL^n_{og}$-structure $(M,<,R)$ is an ordered $n$-uniform hypergraph if 
\begin{itemize}
    \item $(M,<)\models LO$
    \item $R(a_0,\dots,a_{n-1})$ implies that $a_0,\dots,a_{n-1}$ are different,
    \item the relation $R$ is symmetric.
\end{itemize}

An $\mL^n_{opg}$-structure $(M,<,R,P_0,\dots,P_{n-1})$ is an ordered $n$-partite $n$-uniform hypergraph if
\begin{itemize}
    \item $M$ is the disjoint union $P_0\sqcup \dots \sqcup P_{n-1}$ such that if $R(a_0,\dots,a_{n-1})$ then $P_i\cap \{a_0,\dots,a_{n-1}\}$ is a singleton for every $i<n$,
    \item the relation $R$ is symmetric,
    \item $M$ is linearly ordered by $<$ and $P_0<\cdots<P_{n-1}$.
\end{itemize}

The following fact was proven in \cite{NESETRIL1977289}, \cite{NESETRIL1983183} and independently in \cite{1a7d4f5c-0dcd-37ec-b761-69bf07cad14f} for the case of nonpartite hypergraphs and in \cite{Chernikov2014OnN} for the case of partite hypergraphs.

\begin{fact} \label{fact: ordered hypergraphs are ramsey}
Let $K$ be the set of all finite ordered $n$-partite $n$-uniform hypergraphs and $\Tilde{K}$ be the set of all finite ordered $n$-uniform hypergraphs. The classes $K$ and $\Tilde{K}$ have the embedding Ramsey property.
\end{fact}

We will denote by $G_{n,p}$ the Fraïssé limit of $K$ and by $G_n$ the Fraïssé limit of $\Tilde{K}$.

\begin{remark}
    The theories of $G_n$ and $G_{n,p}$ can be axiomatized in the following way:
    \begin{enumerate}
        \item[1.] $(M, <, R)\models \Th(G_n)$ if and only if
        \begin{itemize}
            \item $(M,<)\models DLO$,
            \item $(M, <, R)$ is an ordered $n$-uniform hypergraph,
            \item For every finite disjoint sets $A_0,A_1\subset M^{n-1}$ such that $A_0$ consists of tuples with pairwise distinct coordinates and $b_0<b_1\in M$, there is $b\in M$ such that $b_0<b<b_1$ and $R(b,a_{i,1},\dots,a_{i,n-1})$ holds for every $(a_{0,1},\dots,a_{0,n-1})\in A_0$ and $\neg R(b,a_{1,1},\dots,a_{1,n-1})$ holds for every $(a_{1,1},\dots,a_{1,n-1})\in A_1$.
        \end{itemize}
        \item[2.] $(M, <, R, P_0,\dots, P_n)\models \Th(G_{n,p})$ if and only if
        \begin{itemize}
            \item For every $i<n$, $P_i(M)\models DLO$,
            \item $(M, <, R, P_0,\dots, P_n)$ is an ordered $n$-partite $n$-uniform hypergraph,
            \item for every $j<n$, finite disjoint sets $A_0,A_1\subset \prod_{i\neq j} P_i(M)$ and $b_0<b_1\in P_j(M)$ there is $b\in P_j(M)$ such that $b_0<b<b_1$ and $R(b,a_{i,1},\dots,a_{i,n-1})$ holds for every $(a_{0,1},\dots,a_{0,n-1})\in A_0$ and $\neg R(b,a_{1,1},\dots,a_{1,n-1})$ holds for every $(a_{1,1},\dots,a_{1,n-1})\in A_1$.
        \end{itemize}
    \end{enumerate}
\end{remark}

Next, we define the $n$-independence property for continuous formulas. An equivalent definition was first formulated in \cite{Chernikov2020HypergraphRA} using the $VC_n$ dimension.

\begin{defin}[$n$-independent formula]
     We say that a formula $f(x,y_0,\dots,y_{n-1})$ \emph{ has the $n$-independence property}, $IP_n$ for short, if there exist $r<s\in\mathbb{R}$ and a sequence $(a_{0,i},\dots, a_{n-1,i})_{i< \omega}$ such that for every finite $w\subseteq \omega^n$ there exists $b_w$ such that 
     \begin{align*}
         f(b_w, a_{0,i_0},\dots, a_{n-1,i_{n-1}}  )\leq r &\iff (i_0,\dots, i_{n-1})\in w\\
         &and\\
         f(b_w, a_{0,i_0},\dots, a_{n-1,i_{n-1}}  )\geq s &\iff (i_0,\dots, i_{n-1})\notin w.
     \end{align*} We say that the $\mL$-theory $T$ is \emph{$n$-dependent}, or $NIP_n$, if no $\mL$-formula has $IP_n$. 
\end{defin}

The following remark is a collection of basic properties of $n$-dependent formulas.

\begin{remark} \label{naming parameters and dummy variables}
    \begin{enumerate}
        \item Naming parameters preserves $n$-dependence. If the $\mL(A)$-formula $f(x,y_0,\dots,y_{n-1},A)$ has $IP_n$ witnessed by $(a_{0,i},\dots, a_{n-1,i})_{i< \omega}$, then the $\mL$-formula $g(x,z_0,\dots,z_{n-1})$ has $IP_n$ witnessed by $(a_{0,i}A,\dots, a_{n-1,i}A)_{i< \omega}$ where $z_i=y_iw$ and $g(x,z_0,\dots,z_{n-1})=f(x,y_0,\dots,y_{n-1},w)$.
        \item Adding dummy variables preserves $n$-dependence. Namely, let $x\subset w$ and $y_i\subset z_i$ for all $i<n$. If $g(x,z_0,\dots,z_{n-1}):=f(x,y_0,\dots,y_{n-1})$ has $IP_n$, then so does $f(x,y_0,\dots,y_{n-1})$.
        \item Every $n$-dependent theory is $(n+1)$-dependent.
    \end{enumerate}
\end{remark}

Next, we define what it means for a continuous logic formula to encode a partite and a nonpartite hypergraph.

\begin{defin}[Encoding partite hypergraphs]
    We say that a formula $f(x_0,\dots,x_{n-1})$ \emph{encodes an $n$-partite $n$-uniform hypergraph} $(G,R,P_0,\dots,P_{n-1})$ if there is a $G$-indexed sequence $(a_g)_{g\in G}$ and $r<s\in\mathbb{R}$ satisfying 
    \begin{align*}
    f( a_{g_0},\dots, a_{g_{n-1}}  )\leq r &\iff  R(g_0,\dots, g_{n-1})\\
    &and\\
    f(a_{g_0},\dots, a_{g_{n-1}}  )\geq s &\iff \neg R(g_0,\dots, g_{n-1})
    \end{align*}
    for all $g_0,\dots,g_{n-1}\in P_0\times\cdots\times P_{n-1}$.
    We say that a formula $f(x_0,\dots,x_{n-1})$ \emph{ encodes $n$-partite $n$-uniform hypergraphs} if there exist $r<s\in\mathbb{R}$ such that $f(x_0,\dots,x_{n-1})$ encodes every finite $n$-partite $n$-uniform hypergraph using the same $r$ and $s$.
\end{defin}

\begin{defin}[Encoding hypergraphs] \label{def: encoding nonpartite}
    We say that a formula $f(x_0,\dots,x_{n-1})$ \emph{encodes an $n$-uniform hypergraph $(G,R)$} if there is a $G$-indexed sequence $(a_g)_{g\in G}$  and $r<s\in\mathbb{R}$ satisfying 
    \begin{align*}
        f( a_{g_0},\dots, a_{g_{n-1}}  )\leq r &\iff  R(g_0,\dots, g_{n-1})\\
        &and\\
        f(a_{g_0},\dots, a_{g_{n-1}}  )\geq s &\iff \neg R(g_0,\dots, g_{n-1})
    \end{align*}
    for all 
    $g_0,\dots,g_{n-1}\in G$.
    We say that a formula $f(x_0,\dots,x_{n-1})$ \emph{encodes $n$-uniform hypergraphs} if there exist $r<s\in\mathbb{R}$ such that $f(x_0,\dots,x_{n-1})$ encodes every finite $n$-uniform hypergraph using the same $r$ and $s$.
\end{defin}

Note that if a formula encodes $n$-uniform hypergraphs, then it encodes $n$-partite $n$-uniform hypergraphs. 

It is no surprise that the $n$-independence property and encoding $(n+1)$-partite $(n+1)$-uniform hypergraphs are equivalent also in our continuous context. The following result was first shown in \cite[Proposition 10.2]{Chernikov2020HypergraphRA}, we include a proof using the terminology defined in this section. 

\begin{prop} \label{IP_n iff encoding partite} Let $f(x,y_0,\dots,y_{n-1})$ be a formula. Then, the following are equivalent:
\begin{enumerate}
    \item $f$ has $IP_n$.
    \item $f$ encodes $(n+1)$-partite $(n+1)$-uniform hypergraphs.
    \item $f$ encodes $G_{n+1,p}$ as a partite hypergraph.
    \item $f$ encodes $G_{n+1,p}$ as a partite hypergraph witnessed by a $G_{n+1,p}$-indiscernible sequence.
\end{enumerate}
\end{prop}
\begin{proof}
    $(1)\implies(2)$. Let $r<s\in \mathbb{R}$, $( a_{0,i},\dots, a_{n-1,i} )_{i<\omega}$ and $(b_w)_{w\subseteq\omega^n}$ witness that $f(x,y_0,\dots,y_{n-1})$ has $IP_n$. Let a finite $n+1$-uniform $n+1$-partite hypergraph $G$ be given. Without loss of generality, we may assume that $\lvert P_0(G)\rvert=\cdots=\lvert P_n(G)\rvert=k$. For every $g\in P_0(G)$ consider the set $w_g:=\{(g_1,\dots,g_n): G\models R(g,g_1,\dots,g_n)\}$. By identifying $P_m(G)$ with $\{ (m,i): i<k\}$ and the definition of $IP_n$, we can find $b_{w_g}$ satisfying $$ f(b_{w_g}, a_{g_1},\dots, a_{g_n}  )\leq r \iff (g_1,\dots, g_{n})\in w_g$$ and $$f(b_{w_g}, a_{g_1},\dots, a_{g_n}  )\geq s \iff (g_1,\dots, g_{n})\notin w_g.$$ Then, $(a_g)_{g\in G}$ witnesses that $f$ encodes $G$, where $a_g:=b_{w_g}$ for every $g\in P_0(G)$.
    
    $(2)\implies(3)$. Follows from compactness.
    
    $(3)\implies(4)$. Let $\II=(a_g)_{g\in G_{n+1,p}}$ witness that $f(x,y_0,\dots,y_{n-1})$ encodes $G_{n+1,p}$ as a partite hypergraph. By Theorem \ref{CMP iff Ramsey} and Fact \ref{fact: ordered hypergraphs are ramsey}, there exists a $G_{n+1,p}$-indiscernible sequence $(b_g)_{g\in G_{n+1,p}}$ locally based on $\II$. It is easy to see that $(b_g)_{g\in G_{n+1,p}}$ also witnesses that $f(x,y_0,\dots,y_{n-1})$ encodes $G_{n+1,p}$ as a partite hypergraph.
    
    $(4)\implies(1)$. Let $(a_g)_{g\in G_{n+1,p}}$ witness that $f(x,y_0,\dots,y_{n-1})$ encodes $G_{n+1,p}$ as a partite hypergraph. We write $$G_{n+1,p}=\{ (j,m): j\leq n; m<\omega \}.$$ Since $G_{n+1,p}$ is the random partite hypergraph, using compactness we conclude that for any  finite $w\subseteq \omega^n$ we can find $b_{w}$ such that
    \begin{align*}
        f(b_{w}, a_{1,i_1},\dots, a_{n,i_{n}}  )\leq r \iff (i_1,\dots, i_{n})\in w; \\
        f(b_{w}, a_{1,i_1},\dots, a_{n,i_{n}}  )\geq s \iff (i_1,\dots, i_{n})\notin w.
    \end{align*}
\end{proof}

As in \cite[Corollary 5.3]{Chernikov2014OnN}, from the fact that any permutation of the parts of the partition of $G_{n+1,p}$ is induced by an automorphism of $G_{n+1.p}$ treated as a pure hypergraph, we obtain the following as an easy corollary:

\begin{cor}\label{Corolary: preservation under permutation of variables}
    Let $f(x,y_0,\dots,y_{n-1})$ be a formula and $(w,z_0,\dots,z_{n-1} )$ be any permutation of $(x,y_0,\dots,y_{n-1})$. Then $g(w,z_0,\dots,z_{n-1} ):= f(x,y_0,\dots,y_{n-1})$ is $n$-dependent if and only if $f(x,y_0,\dots,y_{n-1})$ is $n$-dependent.
\end{cor}

A more involved proof which gives a bound on the $VC_n$ dimension can be found in \cite[Proposition 10.6]{Chernikov2020HypergraphRA}.

We cannot guarantee that an $n$-independent formula will encode $(n+1)$-uniform hypergraphs. However, we show that given a continuous theory $T$, $T$ has $IP_n$ if and only if there exists a formula encoding $(n+1)$-uniform hypergraphs. This generalizes the result in \cite[Lemma 2.2]{LASKOWSKI2003263} and allows us to give an alternative proof of \cite[Theorem 5.4]{Chernikov2014OnN} that avoids the mistake mentioned in the introduction. In private communications Chernikov, Palacín and Takeuchi provided a quick fix to the problem the original proof, it can be found after Counterexample \ref{counterexample}. In the proof of the next result, we will write $f(y_0,\dots,y_{n-1},x)$ instead of $f(x,y_0,\dots,y_{n-1})$ for convenience.

\begin{prop}\label{IP_n iff coding nonpartite} Let $T$ be a continuous logic theory. The following are equivalent:
\begin{enumerate}
    \item T has $IP_n$.
    \item There is a continuous logic formula encoding $(n+1)$-uniform hypergraphs.
    \item There is a continuous logic formula encoding $G_{n+1}$ as a hypergraph.
    \item There is a continuous logic formula encoding $G_{n+1}$ as a hypergraph witnessed by a $G_{n+1}$-indiscernible sequence.
\end{enumerate}
\end{prop}
    \begin{proof}
            $(1)\implies(2)$. Let $f(y_0,\dots,y_{n-1},x)$ be a formula with $IP_n$. We show that the symmetric formula $$ \psi(y^0_0y^1_0\cdots y^{n-1}_0x_0,\dots,y^0_ny^1_n\cdots y^{n-1}_nx_n)=\min_{\sigma\in Sym(n)} \{ f(y^0_{\sigma(0)},\dots ,y^{n-1}_{\sigma({n-1})},x_{\sigma(n)}) \}$$ encodes every finite $(n+1)$-uniform hypergraph. By Proposition \ref{IP_n iff encoding partite}, $f$ encodes $G_{n+1,p}$ as a partite hypergraph, which is witnessed by a $G_{n+1,p}$-indiscernible sequence $(a_g)_{g\in G_{n+1,p}}$ and some $r<s\in \mathbb{R}$. We enumerate the elements of $G_{n+1,p}$ as $$\{ g^i_m: i\leq n ; m<\omega \},$$ where the superscript indicates which part of the partition the element belongs to.

            Let an $(n+1)$-uniform finite hypergraph $\mathcal{H}=(H,R_H)$ be given, we write $H:=\{ h_i: i<k \}$ for some $k<\omega$. We construct $\Tilde{\mathcal{H}}=(\Tilde{H},R_{\Tilde{H}})$ an isomorphic copy of $\mathcal{H}$ consisting of elements $\Tilde{h}_i$ for $i<k$ of the form $( g^0_i,\dots,g^{n-1}_i,g^{n}_{c(i)} )$ and show that the formula $\psi$ encodes $\Tilde{\mathcal{H}}$ (and hence encodes $\mathcal{H}$), where the relation $R_{\Tilde{H}}$ and the function $c:\omega\to\omega$ are to be defined.

            We start by defining the function $c$. For every $i<\omega$, let $c(i)$  be the smallest $m<\omega$ such that for any $(j_0,\dots,j_{n-1})\in [k]^n$ $$ R_{G_{n+1,p}}(g^0_{j_0},\dots, g^{n-1}_{j_{n-1}}, g^n_m) \iff j_0<\cdots < j_{n-1}<i \wedge R_H(h_{j_0},\dots,h_{j_{n-1}},h_i).$$ Note that the existence of such $m$ is guaranteed by the fact that $G_{n+1,p}$ is the random partite hypergraph.

            The relation $R_{\Tilde{H}}$ is defined in the following manner: for $i_0<\cdots<i_n<k$ we set $$R_{\Tilde{H}}(\Tilde{h}_{i_0},\dots,\Tilde{h}_{i_n}) \iff R_{G_{n+1,p}}( g^0_{i_0}, g^1_{i_1},\dots,g^{n-1}_{i_{n-1}},g^n_{c(i_n)} ),$$ the rest of the cases are defined by symmetry of $R_{\Tilde{H}}$ and by declaring $\neg R_{\Tilde{H}}(\Tilde{h}_{i_0},\dots,\Tilde{h}_{i_n})$ whenever $i_{m_1}=i_{m_2}$ for some $m_1\neq m_2$. Note that by construction, we have $(H,R_{H})\cong(\Tilde{H},R_{\Tilde{H}})$.

            \begin{claim}
                The elements $b_{\Tilde{h}_i}:=(a_{g^0_{i}},\dots, a_{g^{n-1}_i},a_{ g^n_{c(i)} })$ for $i<k$ witness that $\psi$ encodes $\Tilde{\mathcal{H}}$ with the above $r<s$.
            \end{claim}
            \begin{proof}[Proof of claim]
                To ease the notation, for each $\sigma \in Sym(n)$ we write $$f_\sigma:=f(y^0_{\sigma(0)},\dots ,y^{n-1}_{\sigma({n-1})},x_{\sigma(n)}).$$ Our goal is the following: \begin{align*}
                    \psi( b_{\Tilde{h}_{i_0}},\dots, b_{\Tilde{h}_{i_n}})\leq r &\iff R_{\Tilde{H}}(\Tilde{h}_{i_0},\dots,\Tilde{h}_{i_n})\\
                    \psi( b_{\Tilde{h}_{i_0}},\dots, b_{\Tilde{h}_{i_n}})\geq s &\iff \neg R_{\Tilde{H}}(\Tilde{h}_{i_0},\dots,\Tilde{h}_{i_n})
                \end{align*}

                First, note that if $i_{m_1}=i_{m_2}$ for some $m_1,m_2<n$ then we have $\neg R_{\Tilde{H}}(\Tilde{h}_{i_0},\dots,\Tilde{h}_{i_n})$ by definition of $R_{\Tilde{H}}$ and $\psi( b_{\Tilde{h}_{i_0}},\dots b_{\Tilde{h}_{i_n}})\geq s$ by construction of the function $c$. Hence, we only need to prove the equivalences above in the case where all the $i$'s are pairwise distinct. We prove the first equivalence; the second one is easily deduced from it.
                
                Assume that $R_{\Tilde{H}}(\Tilde{h}_{i_0},\dots,\Tilde{h}_{i_n})$ holds. By symmetry, without loss of generality we may assume that $i_0<\cdots<i_n$. Then, by the definition of $R_{\Tilde{H}}$, this implies that $R_{ G_{n+1,p} }( g^0_{i_0},\dots, g^{n-1}_{i_{n-1}}, g^n_{c(i_n)}  )$ holds. Since the formula $f$ encodes $G_{n+1,p}$, this is equivalent to $f( a_{g^0_{i_0}},\dots, a_{g^{n-1}_{i_{n-1}}},a_{ g^n_{c(i_n)} } )\leq r$ . Hence, $\psi( b_{\Tilde{h}_{i_0}},\dots b_{\Tilde{h}_{i_n}})\leq r$.
                
                Assume $\psi( b_{\Tilde{h}_{i_0}},\dots b_{\Tilde{h}_{i_n}})\leq r$, again by symmetry, we may assume without loss of generality that $i_0<\cdots<i_n$. This implies that for some $\sigma\in Sym( n )$ we have $f_\sigma\leq r$. However, by construction of the function $c$, the only possibility is that $f_{Id}\leq r$, that is, $f( a_{g^0_{i_0}},\dots, a_{g^{n-1}_{i_{n-1}}},a_{ g^n_{c(i_n)} } )\leq r$.
                Since the formula $f$ encodes $G_{n+1,p}$, this is equivalent to $R_{ G_{n+1,p} }( g^0_{i_0},\dots, g^{n-1}_{i_{n-1}}, g^n_{c(i_n)}  )$, which implies $R_{\Tilde{H}}(\Tilde{h}_{i_0},\dots,\Tilde{h}_{i_n})$ by definition of the relation $R_{\Tilde{H}}$.
            \end{proof}

            $(2)\implies(3)$. Follows from compactness.
            
            $(3)\implies(4)$. Let $\II=(a_g)_{g\in G_{n+1}}$ witness that $\psi(x_0,x_1,\dots,x_n)$ encodes $G_{n+1}$ as a hypergraph. By Theorem \ref{CMP iff Ramsey} and Fact \ref{fact: ordered hypergraphs are ramsey}, there exists a $G_{n+1}$-indiscernible sequence $(b_g)_{g\in G_{n+1}}$ locally based on $\II$. It is easy to see that $(b_g)_{g\in G_{n+1}}$ also witnesses that $\psi(x_0,x_1,\dots,x_n)$ encodes $G_{n+1}$ as a hypergraph.
            
            $(4)\implies(1)$. If a formula encodes $G_{n+1}$ as a hypergraph, then it also encodes $G_{n+1,p}$ as a partite hypergraph, which implies that the formula has $IP_n$ by Proposition \ref{IP_n iff encoding partite}.
    \end{proof}

To prove the main theorem of this section we need the next two facts about hypergraphs.

Let $(G_*,\mL_{o*},\mL_{g*})$ be either $(G_n,\{ <\}, \{<, R\})$ or $(G_{n,p},\mL^n_{op},\mL^n_{opg})$. 

Let $V\subset G_*$ be a finite set and $g_0,\dots,g_{n-1}, g'_0,\dots,g'_{n-1}\in G_*\setminus V$ such that $$R(g_0,\dots,g_{n-1})\notiff R(g'_0,\dots,g'_{n-1}).$$
By $W=g_0\dots g_{n-1}V$ we mean the set $\{ g_0,\dots,g_{n-1} \}\cup V$ with the inherited structure from $G_*$. Let $\mL_*$ be $\mL_{o*}$ or $\mL_{g*}$, and let $W=g_0\dots g_{n-1}V$, $W'=g'_0\dots g'_{n-1}V$, we write $W\cong_{\mL_*}W'$ if the map acting as identity on $V$ and sending $g_i$ to $g'_i$ for $i<n$ is an $\mL_*$ isomorphism. 

\begin{defin} Let $V\subset G_*$ be a finite set and $g_0,\dots,g_{n-1}, g'_0,\dots,g'_{n-1}\in G_*\setminus V$ be such that $$R(g_0,\dots,g_{n-1})\wedge \neg R(g'_0,\dots,g'_{n-1}).$$
We say that $W=g_0\dots g_{n-1}V$ is $V$-adjacent to $W'=g'_0\dots g'_{n-1}V$ if \begin{itemize}
    \item $W\cong_{L_{o*}} W'$,
    \item for every $k\leq n$, $\overline{v}\in V$ with $\lvert \overline{v}\rvert=k$ and $i_0,\dots,i_{n-k-1}<n$ $$ R(g_{i_0},\dots g_{i_{n-k-1}},\overline{v}) \iff  R(g'_{i_0},\dots g'_{i_{n-k-1}},\overline{v}).$$
\end{itemize}
$W$ is said to be adjacent to $W'$ if there is $V\subset W \cap W'$ such that $W$ is $V$-adjacent to $W'$. 
\end{defin}

The proofs of the next two results can be found in \cite[Lemma 5.6 and Lemma 5.7]{Chernikov2014OnN}, respectively.
\begin{fact} \label{ Sequence of adjacent graphs}
    Let $W,W'\subset G_*$ be subsets such that $W\cong_{L_{o*}} W'$. Then there is a sequence $W=W_0,W_1,\dots,W_k$ such that $W_{i+1}$ is adjacent to $W_i$ for every $i<k$ and $W_k\cong_{L_{g*}}W'$
\end{fact}
\begin{fact} \label{Isomorphic copy of random partite hypergraph}
    Let $V\subset G_*$ be a finite set and $g_0<\dots<g_{n-1} \in G_* \setminus V$ with $R(g_0,\dots,g_{n-1})$. Then there are infinite sets $X_0<\dots<X_{n-1}\subseteq G_*$ such that
    \begin{itemize}
        \item $(G';<;R)\cong (G_{n,p};<;R)$ where $G'=X_0\dots X_{n-1}$ (i.e. each $X_i$ correspond to the part $P_i$ of the partition),
        \item for any $g'_i\in X_i$ ($i<n$), either $W\cong_{\mL_{opg}} W'$ or $W$ is $V$-adjacent to $W'$, where $W=g_0\dots g_{n-1}V$ and $W'=g'_0\dots g'_{n-1}V$.
    \end{itemize}
\end{fact}

We are now ready to prove the main theorem of the section.

\begin{teor} \label{n-dep and collapsing}
    Let $T$ be a complete continuous logic theory. The following are equivalent:
    \begin{enumerate}
        \item $T$ is $n$-dependent.
        \item Every $G_{n+1,p}$-indiscernible is $\mL_{op}$-indiscernible.
        \item Every $G_{n+1}$-indiscernible is order indiscernible.
    \end{enumerate}
\end{teor}
    \begin{proof}
            $(1)\implies(2)$. Let $(a_g)_{g\in G_{n+1,p}}$ be a $G_{n+1,p}$-indiscernible sequence which is not $\mL_{op}$-indiscernible. Then there are $\mL_{op}$-isomorphic $W,W'\subset G_{n+1}$ subsets of size $m$, a formula $f(x_0,\dots,x_{m-1})$ and $r<s\in \mathbb{R}$ such that $f( (a_g)_{g\in W}) \leq r$ and $f( (a_g)_{g\in W'})\geq s$ (where the elements $a_g$ are substituted for the variables $x_0,\dots,x_{m-1}$ according to the ordering on $W$ and $W'$). Without loss of generality, by Fact \ref{ Sequence of adjacent graphs}, we may assume that $W$ is $V$-adjacent to $W'$ for some subset $V$ such that $W=g_0\dots g_n V$, $W'=g'_0\dots g'_n V$, $R(g_0,\dots g_n)$ and $\neg R(g'_0,\dots g'_n)$. 

            Now we apply Fact \ref{Isomorphic copy of random partite hypergraph} to $V$ and $g_0,\dots,g_n$. This yields $G'\subseteq G_{n+1,p}$ such that for every $(h_0,\dots,h_n)\in \prod_{i\leq n} P_i(G')$ $$ R(h_{0},\dots,h_{n}) \iff h_0\dots h_n V\cong_{\mL_{opg}}W $$ and $$ \neg R(h_0,\dots,h_n) \iff h_0\dots h_nV\cong_{\mL_{opg}}W'.$$ Recall that the sequence $(a_g)_{g\in G_{n+1,p}}$ is $G_{n+1,p}$-indiscernible and let $f'$ be the formula defined by permuting the variables $(x_0,\dots,x_{m-1})$ of $f$ in such a way that the first $n+1$-variables are the ones corresponding to $h_0,\dots,h_n$ according to the ordering on the set $\{h_0,\dots,h_n\}\cup V$. Then,
            $$ f'(a_{h_0},\dots, a_{h_n},A)\leq r \iff R(h_{0},\dots,h_{n})$$
            and
            $$f'(a_{h_0},\dots, a_{h_n},A)\geq s \iff \neg R(h_{0},\dots,h_{n}),$$ where $A=(a_g)_{g\in V}$. Since $G'$ is isomorphic to $G_{n+1,p}$, by Proposition \ref{IP_n iff encoding partite}, the formula $f'(x,y_{0},\dots, y_{n-1},A)$ has $IP_n$ and hence, by Remark \ref{naming parameters and dummy variables} there is a continuous logic formula $g(x,z_{0},\dots, z_{n-1})$ which has $IP_n$.
            
            $(1)\implies(3)$. The proof is exactly as the proof of $(1)\implies (2)$.
            
            $(2)\implies(1)$. It follows from Proposition \ref{IP_n iff encoding partite}. If the formula $f$ encodes $G_{n+1,p}$ as a partite hypergraph witnessed by a $G_{n+1,p}$-indiscernible sequence $(a_g)_{g\in G_{n+1,p}}$, then $(a_g)_{g\in G_{n+1,p}}$ cannot be $\mL_{op}$-indiscernible.
            
            $(3)\implies(1)$ Follows from Proposition \ref{IP_n iff coding nonpartite}. If $T$ has $IP_n$, there is a continuous logic formula encoding $G_{n+1}$ as a hypergraph witnessed by a $G_{n+1}$-indiscernible sequence $(a_g)_{g\in G_{n+1}}$. Then $(a_g)_{g\in G_{n+1}}$ cannot be order-indiscernible.
    \end{proof}

We know explain the error in the proof of \cite[Theorem 5.4 $(3)\implies (2)$]{Chernikov2014OnN}

Without loss of generality we write $$G_{n+1,p}=\{ g^i_q: i<n; q\in \mathbb{Q}\},$$ where $g^i_q \in P_i(G_{n+1,p})$ and $g^i_q<g^i_p$ for all $q<p\in \mathbb{Q}$. We define the ordered $(n+1)$-uniform hypergraph $G^*_{n+1}$ as follows:
\begin{itemize}
    \item $G^*_{n+1}=\{ h_q: h_q=(g^0_q,\dots, g^n_q), q\in \mathbb{Q} \}$,
    \item $ R_{G^*_{n+1}}( h_{q_0},\dots, h_{q_n} ) \iff R_{G_{n+1,p}}( g^0_{q_0},\dots, g^n_{q_n} ) $ for $q_0<\dots<q_n$,
    \item $h_q<h_p \iff q<p$.  
\end{itemize}
Clearly, $G^*_{n+1}$ embeds every finite ordered $(n+1)$-uniform hypergraph.

Let $(a_g)_{g\in G_{n+1,p}}$ be a $G_{n+1,p}$-indiscernible sequence which is not $\mL_{op}$-indiscernible. For $h_q\in G^*_{n+1}$, let $b_{h_q}=( a_{g^0_q},\dots, a_{g^n_q} )$ and consider the $G^*_{n+1}$-indexed sequence $(b_h)_{h\in G^*_{n+1}}$. The following claim is made in the proof:
\begin{claim}
    Whenever $X\equiv_{<_{G^*_{n+1}},R_{G^*_{n+1}}}Y \subseteq G^*_{n+1}$, we have $$ \tp( (b_h)_{h\in X} )=\tp((b_h)_{h\in Y}) .$$
\end{claim}
However, this claim is not true as shown by the following counterexample. 
\begin{counterexample}\label{counterexample}
    Consider the theory $T=\Th(G_{n+1,p})$ and the sequence $(g)_{g\in G_{n+1,p}}$ (This sequence corresponds to the sequence $(a_g)_{g\in G_{n+1,p}}$ of the claim, which is clearly $G_{n+1,p}$-indiscernible but not order indiscernible). Then, for $h_q=(g^0_q,\dots, g^n_q)$ we have $b_{h_q}:=(g^0_q,\dots,g^n_q)$. Let $X:=\{h_{q_0},h_{q_1} \}$ for some $q_0<q_1\in \mathbb{Q}$, by randomness, there is $q_0<\Tilde{q}\in \mathbb{Q}$ and $h_{\Tilde{q}}=(g^0_{\Tilde{q}},\dots,g^n_{\Tilde{q}})$ such that $$ R_{ G_{n+1,p}}(g^0_{q_0},\dots, ,g^{n-1}_{q_0}, g^n_{q_1})\iff \neg  R_{ G_{n+1,p}}(g^0_{q_0},g^1_{q_0},\dots,g^{n-1}_{q_0}, g^n_{\Tilde{q}}).$$
Thus, $h_{q_0},h_{q_1}\equiv_{<_{G^*_{n+1}},R_{G^*_{n+1}}}  h_{q_0},h_{\Tilde{q}}$ since: \begin{itemize}
    \item $\Th(G^*_{n+1})$ has quantifier elimination,
    \item $h_{q_0}<h_{q_1}$ and $h_{q_0}<h_{\Tilde{q}}$,
    \item $R_{G^*_{n+1}}$ has arity $n+1$.
\end{itemize}

We also have $\tp(b_{h_{q_0}},b_{h_{q_1}})\neq \tp(b_{h_{q_0}},b_{h_{\Tilde{q}}})$ since, by our choice of $\Tilde{q}$, $$ R_{ G_{n+1,p}}(g^0_{q_0},\dots, ,g^{n-1}_{q_0}, g^n_{q_1})\iff \neg  R_{ G_{n+1,p}}(g^0_{q_0},g^1_{q_0},\dots,g^{n-1}_{q_0}, g^n_{\Tilde{q}}).$$
\end{counterexample}

The following quick correction is due to Artem Chernikov, Daniel Palacín and Kota Takeuchi. If we substitute the claim above in the original proof by the following: Let $A\subset G_{n+1,p}$ be any finite substructure. Without loss of generality, we may assume that if $g^i_q,g^j_p\in A$ and $i<j$ then $q<p$. Let $A^*=\{ h_q: g^i_q\in A \text{ for some } i\}$ and let $\varphi^*((x^0_q,\dots,x^{n-1}_q)_{h_q\in A^*}):=\varphi((x^i_q)_{g^i_q\in A})$ for each formula $\varphi((x^i_q)_{g^i_q\in A})$. 
\begin{claim}
  Whenever $A^*\equiv_{<_{G^*_{n+1}},R_{G^*_{n+1}}}X \subseteq G^*_{n+1}$, we have $$ \tp_{\varphi^*}( (b_h)_{h\in A^*} )=\tp_{\varphi^*}((b_h)_{h\in X}) .$$
\end{claim}
Then, the proof goes through. Our counterexample shows that one has to work with a restricted family of formulas $\varphi^*$.


The formula $\psi$ that we constructed in the proof of Proposition \ref{IP_n iff coding nonpartite} will be important throughout the rest of the chapter.

\begin{notation}\label{notation psi_f}
     Given a continuous logic formula $f(y_0,\dots,y_{n-1},x)$ we denote the symmetric formula constructed in the proof of Proposition \ref{IP_n iff coding nonpartite} as $\psi_f$. Namely: $$ \psi_f(y^0_0y^1_0\cdots y^{n-1}_0x_0,\dots,y^0_ny^1_n\cdots y^{n-1}_nx_n)=\min_{\sigma\in Sym(n)} \{ f(y^0_{\sigma(0)},\dots ,y^{n-1}_{\sigma({n-1})},x_{\sigma(n)}) \}.$$
\end{notation}

It turns out that $IP_n$ of the formulas $f$ and $\psi_f$ are equivalent. The implication $(2)\implies (1)$ in the next result can also be deduced from \cite[Proposition 10.4 and Proposition 10.6]{Chernikov2020HypergraphRA} since the set of connectives $\{\neg, \frac{1}{2}, \dot{-}\}$ is full (we can approximate every continuous formula uniformly by formulas constructed using only that set of connectives). However, we provide a direct proof with ideas that will be useful in the next section.
\begin{lema} \label{IP_n f and psi}
    Let $f(y_0,\dots,y_{n-1},x)$ be a continuous logic formula. Then, the following are equivalent:
    \begin{enumerate}
        \item $f$ has $IP_n$.
        \item $\psi_f$ has $IP_n$.
    \end{enumerate}
\end{lema}
\begin{proof}
        $(1)\implies (2)$. Follows from the proof of $(1)\implies (2)$ of Proposition \ref{IP_n iff coding nonpartite} and Proposition \ref{IP_n iff encoding partite}.
        
        $(2)\implies (1)$. First, recall that if a formula is $n$-dependent and we add dummy variables, then the new formula is still $n$-dependent and that $n$-dependence is preserved under permutations of variables (by Remark \ref{naming parameters and dummy variables} and Corollary \ref{Corolary: preservation under permutation of variables}). 
    
    We consider the formula $$f_\sigma(y^0_0y^1_0\cdots y^{n-1}_0x_0,\dots,y^0_ny^1_n\cdots y^{n-1}_nx_n)=f(y^0_{\sigma(0)},\dots ,y^{n-1}_{\sigma({n-1})},x_{\sigma(n)}).$$We have $\psi_f=\min_{\sigma\in Sym(n)} \{f_\sigma\}$.
    Let $(b_g)_{g\in G_{n+1,p}}$ be a witness that $\psi_f$ encodes $G_{n+1,p}$ as a partite hypergraph with some $r<s$, which exists by Proposition \ref{IP_n iff encoding partite} and the assumption that $\psi_f$ has $IP_n$. Fix a linear ordering of $Sym(n)$. First, note that if $\psi_f(b_{g_0},\dots,b_{g_n})\leq r$, then there at least one $\sigma\in Sym(n)$ such that $f_\sigma(b_{g_0},\dots,b_{g_n})\leq r$. Note also that since $\psi_f$ encodes $G_{n+1,p}$, there exists an edge $(g_0,\dots,g_n)$ in $G_{n+1,p}$ if and only if $\psi_f(b_{g_0},\dots,b_{g_n})\leq r$. We color the edge $(g_0,\dots,g_n)$ according to the $\sigma$ for which $f_\sigma(b_{g_0},\dots,b_{g_n})\leq r$ holds. Let $H$ be any finite ordered $(n+1)$-partite $(n+1)$-uniform hypergraph. Since $\age(G_{n+1,p})$ has ERP, we can find a monochromatic isomorphic copy of $H$ inside $G_{n+1,p}$. This implies that $f_\sigma$ encodes $H$ as a partite hypergraph with the same $r<s$ as above (where $\sigma$ is the color of the edges of this monochromatic copy of $H$). Since each finite $(n+1)$-partite $(n+1)$-uniform hypergraph is encoded by some $f_\sigma$, by compactness there is $f_\sigma$ encoding $G_{n+1,p}$ as a partite hypergraph. Thus, $f_\sigma$ has $IP_n$ which, by the previous paragraph, implies that $f(y_0,\dots,y_{n-1},x)$ has $IP_n$.
\end{proof}

\section{$n$-dependence for hyperdefinable sets}\label{section: hyperdef n-dep}

Let $T$ be a complete, first-order theory, and $\C \models T$ a monster model (i.e. $\kappa$-saturated and strongly $\kappa$-homogeneous for a strong limit cardinal $> |T|$). Let $E$ be a $\emptyset$-type-definable equivalence relation on a $\emptyset$-type-definable subset $X$ of $\C^\lambda$ (or a product of sorts), where $\lambda< \kappa$.


We recall the family $\mathcal{F}_{X/E}$ defined in Section \ref{section: characterizations of stability}. $\mathcal{F}_{X/E}$ is the family  of all functions $f : X \times \mathfrak{C}^m \to \mathbb{R}$ which factor through $X/E\times\mathfrak{C}^m$ and can be extended to a continuous logic formula $\C^{\lambda} \times \C^m \to \mathbb{R}$ over $\emptyset$ (i.e. factors through a continuous function $S_{\lambda+m}(\emptyset))$, where $m$ ranges over $\omega$. 


Let $A \subset \C$ (be small). Recall that the complete types over $A$ of elements of $X/E$ can be defined as the $\aut(\C/A)$-orbits on $X/E$,  or the preimages of these orbits under the quotient map, or the partial types defining these preimages. The space of all such types is denoted by $S_{X/E}(A)$.

In this section we apply the results obtained in continuous logic to the context of hyperdefinable sets to obtain a counterpart to Theorem \ref{n-dep and collapsing}. 

 In Proposition \ref{proposition: F_X/E separates points},  
 we showed that the family of functions $\mathcal{F}_{X/E}$ separates points in $S_{X/E\times \C^m}(\emptyset)$. Namely,
\begin{fact}\label{diff type diff f value}
For any $a_1=a'_1/E$, $a_2=a'_2/E$ in $X/E$ and $b_1,b_2\in \mathfrak{C}^m$ 
		$$\tp(a_1,b_1)\neq \tp(a_2,b_2)\iff (\exists f\in \mathcal{F}_{X/E})(f(a'_1,b_1)\neq f(a'_2,b_2)) $$
\end{fact}
This allows us to work with elements of $X/E$ as real elements if we restrict ourselves to functions from the family $\mathcal{F}_{X/E}$. Hence, we introduce the following notation:

\begin{notation} Let $\Delta$ be a set of (continuous) formulas in variables $(x_i)_{i<\lambda}$ all from the same product of sorts. We say that a sequence $(a_i)_{i\in \I}$ of elements from the appropriate product of sorts is $\I$-indiscernible with respect to $\Delta$ if for any tuples $i_1,\dots,i_n,j_1,\dots,j_n\in \I$ we have that 
    $$\qftp(i_1,\dots,i_n)=\qftp(j_1,\dots,j_n) \implies \tp^\Delta(a_{i_1}, \dots, a_{i_n})=\tp^\Delta(a_{j_1}, \dots, a_{j_n}),$$
    where the tuples $a_i$ are substituted for the variables of the formulas from $\Delta$.
\end{notation}

We define generalised indiscernible sequences of hyperimaginaries exactly as we did in Definition \ref{defin: gen. indisc.}.
\begin{defin}
    Let $\II=(a_i: i\in \I)$ be an $\I$-indexed sequence of hyperimaginaries (maybe of different sorts), and let $A\subset \C$ be a small set of parameters.  We say that $\II$ is an $\I$-\emph{indexed indiscernible sequence over} $A$ if for all $n\in\omega$ and all sequences $i_1,\dots,i_n,j_1,\dots,j_n$ from $\I$ we have that 
    $$\qftp(i_1,\dots,i_n)=\qftp(j_1,\dots,j_n) \implies \tp({a}_{i_1}, \dots, {a}_{i_n}/A)=\tp({a}_{j_1}, \dots, {a}_{j_n}/A).$$
\end{defin}

By Fact \ref{diff type diff f value}, a sequence of hyperimaginaries $(a_i/E)_{i\in \I}$ is $\I$-indiscernible if and only if the sequence $(a_i)_{i\in \I}$ is $\I$-indiscernible with respect to the family of functions $f: X^n\to \R$ that factor through $(X/E)^n$ and can be extended to a continuous formula $f:\C^{n\lambda}\to \R$ over $\emptyset$ where $n$ ranges over $\omega$.

Next, we define the $n$-independence property for hyperdefinable sets. 

\begin{defin}\label{Definition: hyperim IP_n}
A hyperdefinable set $X/E$ has the \emph{$n$-independence property}, $IP_n$ for short, if for some $m<\omega$ there exist two complete types $p,q\in S_{X/E\times \C^m}(\emptyset)$ with $p\neq q$ and a sequence $(a_{0,i},\dots, a_{n-1,i})_{i< \omega}$ such that for every finite $w\subset \omega^n$ there exists $b_w\in X/E$ satisfying $$\tp(b_w, a_{0,i_0},\dots, a_{n-1,i_{n-1}}  )=p \iff (i_0,\dots, i_{n-1})\in w$$ $$\tp(b_w, a_{0,i_0},\dots, a_{n-1,i_{n-1}}  )=q \iff (i_0,\dots, i_{n-1})\notin w. $$ 
\end{defin}

Note that $1$-dependent hyperdefinable sets are exactly the hyperdefinable sets with $NIP$ (see Definition \ref{nip hyperdef}) by Lemma \ref{lemma: equivalences NIP}.

We can easily modify our definition of $n$-independence to suit functions in $\mathcal{F}_{X/E}$.

\begin{defin}\label{defin: n-dep function/E}
     We say that $f(x,y_0,\dots,y_{n-1})\in \mathcal{F}_{X/E}$ \emph{ has the $n$-independence property}, $IP_n$ for short, if there exist $r<s\in\mathbb{R}$ and a sequence $(a_{0,i},\dots, a_{n-1,i})_{i< \omega}$ from anywhere such that for every finite $w\subseteq \omega^n$ there exists $b_w\in X$ satisfying 
     \begin{align*}
         f(b_w, a_{0,i_0},\dots, a_{n-1,i_{n-1}}  )\leq r &\iff (i_0,\dots, i_{n-1})\in w\\
         &and\\
         f(b_w, a_{0,i_0},\dots, a_{n-1,i_{n-1}}  )\geq s &\iff (i_0,\dots, i_{n-1})\notin w.
     \end{align*} 
\end{defin}


Similarly, we can define what it means for a function in $\mathcal{F}_{X/E}$ to encode ($n$-partite) $n$-uniform hypergraphs.
\begin{defin}
    We say that  $f(x,y_1,\dots,y_{n-1})\in \mathcal{F}_{X/E}$ \emph{encodes a $n$-partite $n$-uniform hypergraph} $(G,R,P_0,\dots,P_{n-1})$ if there are a $G$-indexed sequence $(a_g)_{g\in G}$ with $a_g\in X$ for every $g\in P_0(G)$ and $r<s\in\mathbb{R}$ satisfying 
    \begin{align*}
    f( a_{g_0},\dots, a_{g_{n-1}}  )\leq r &\iff  R(g_0,\dots, g_{n-1})\\
    &and\\
    f(a_{g_0},\dots, a_{g_{n-1}}  )\geq s &\iff \neg R(g_0,\dots, g_{n-1})
    \end{align*}
    for all $g_0,\dots,g_{n-1}\in P_0\times\cdots\times P_{n-1}$.
    We say that $f(x_0,\dots,x_{n-1})\in \mathcal{F}_{X/E}$ \emph{ encodes $n$-partite $n$-uniform hypergraphs} if there exist $r<s\in\mathbb{R}$ such that $f(x,y_1,\dots,y_{n-1})$ encodes every finite $n$-partite $n$-uniform hypergraph using the same $r$ and $s$.
\end{defin}

The proof of the following fact is exactly as in the case of a general continuous formula.

\begin{prop} \label{Hyperim: IP_n iff encoding partite} Let $f(x,y_0,\dots,y_{n-1})\in \mathcal{F}_{X/E}$. Then, the following are equivalent:
\begin{enumerate}
    \item $f$ has $IP_n$.
    \item $f$ encodes $(n+1)$-partite $(n+1)$-uniform hypergraphs.
    \item $f$ encodes $G_{n+1,p}$ as a partite hypergraph.
    \item $f$ encodes $G_{n+1,p}$ as a partite hypergraph witnessed by a $G_{n+1,p}$-indiscernible sequence.
\end{enumerate}
\end{prop}

The following lemma allows us to understand $IP_n$ of a hyperdefinable set $X/E$ through the family of functions $\mathcal{F}_{X/E}$.

\begin{lema}\label{IP_n hyperdef. set iif formulas}
    $X/E$ has $IP_n$ if and only if some $f(x,y_0,\dots,y_{n-1})\in \mathcal{F}_{X/E}$ has $IP_n$.
\end{lema}
\begin{proof}
    Assume that $X/E$ has $IP_n$. Take witnesses $p$ and $q$ from Definition \ref{Definition: hyperim IP_n}. Then, by Fact \ref{diff type diff f value}, there exists $f(x,y_0,\dots,y_{n-1})\in\mathcal{F}_{X/E}$ and $r<s$ such that $f(x,y_0,\dots,y_{n-1})\leq r \in p $ and $f(x,y_0,\dots,y_{n-1})\geq s \in q$. The elements witnessing $IP_n$ for $X/E$ also  witness that $f(x,y_0,\dots,y_{n-1})$ has $IP_n$.

 Assume now that some $f(x,y_0,\dots,y_{n-1})\in\mathcal{F}_{X/E}$ has $IP_n$. By Proposition \ref{Hyperim: IP_n iff encoding partite}, the function $f$ encodes $G_{n+1,p}$ as a partite hypergraph witnessed by a $G_{n+1,p}$-indiscernible sequence $(a_g)_{g\in G_{n+1,p}}$ and some $r<s\in \R$. By enumerating each part $P_i(G_{n+1,p})$ of the partition, we write $$G_{n+1,p}=\{ (j,m): j\leq n; m<\omega \},$$ where the first element of the tuple indicates the part of the partition and the second one the index in the enumeration. Then, since $G_{n+1,p}$ is the random partite hypergraph, for any finite disjoints $$w_0,w_1\subseteq \prod_{i=1}^n P_i(G_{n+1,p}),$$ we can find $g\in P_0(G_{n+1,p})$ such that there is an edge containing $g$ and any element of $w_0$ and there is no edge containing $g$ and any element of $w_1$. This implies that for any  finite disjoint $s_0,s_1\subseteq \omega^n$ we can find $b_{s_0,s_1}=a_g\in X$ for some $g\in P_0(G_{n+1,p})$ such that
    \begin{align*}
        (i_1,\dots, i_{n})\in s_0\implies f(b_{s_0,s_1}, a_{1,i_1},\dots, a_{n,i_{n}}  )\leq r,\\
       (i_1,\dots, i_{n})\in s_1 \implies f(b_{s_0,s_1}, a_{1,i_1},\dots, a_{n,i_{n}}  )\geq s .
    \end{align*}
    Moreover, by $G_{n+1,p}$-indiscernibility, there exist two distinct complete types $p,q\in S_{X/E\times \C^m}(\emptyset)$ such that 
    \begin{align*}
        (i_1,\dots, i_{n})\in s_0 \implies \tp(b_{s_0,s_1}/E, a_{1,i_1},\dots, a_{n,i_{n}}  )=p; \\
        (i_1,\dots, i_{n})\in s_1 \implies \tp(b_{s_0,s_1}/E, a_{1,i_1},\dots, a_{n,i_{n}}  )=q.
    \end{align*}

    By compactness, it follows that $X/E$ has $IP_n$.
\end{proof}

 \begin{notation}\label{notation PSI}
     For each $f(x,y_0,\dots, y_n)\in \mathcal{F}_{X/E}$, let $f'(y_0,\dots, y_n,x):=f(x,y_0,\dots, y_n)$. We denote by $\Psi^{n+1}_{ \mathcal{F}_{X/E} }$ the set containing all functions $$ \psi_{f'}(y^0_0y^1_0\cdots y^{n-1}_0x_0,\dots,y^0_ny^1_n\cdots y^{n-1}_nx_n)=\min_{\sigma\in Sym(n)} \{ f'(y^0_{\sigma(0)},\dots ,y^{n-1}_{\sigma({n-1})},x_{\sigma(n)}) \}$$ for $f(x,y_0,\dots, y_n)\in \mathcal{F}_{X/E}$ (here $\psi_{f'}$ is constructed as in Notation \ref{notation psi_f}). $\Psi_{ \mathcal{F}_{X/E} }$ is the union of all $\Psi^{n+1}_{ \mathcal{F}_{X/E} }$ for $n<\omega$.
 \end{notation}

 The next definition is the natural counterpart of Definition \ref{def: encoding nonpartite} for functions of the family $\Psi_{ \mathcal{F}_{X/E} }$.


\begin{defin}
    We say that $\psi(x_0,\dots,x_{n-1})\in \Psi^{n}_{ \mathcal{F}_{X/E} } $ \emph{encodes an $n$-uniform hypergraph $(G,R)$} if there is a $G$-indexed sequence $(a_g)_{g\in G}$ in $\C^m\times X$ (for some fixed $m<\omega$)  and $r<s\in\mathbb{R}$ satisfying 
    \begin{align*}
        \psi( a_{g_0},\dots, a_{g_{n-1}}  )\leq r &\iff  R(g_0,\dots, g_{n-1})\\
        &and\\
        \psi(a_{g_0},\dots, a_{g_{n-1}}  )\geq s &\iff \neg R(g_0,\dots, g_{n-1})
    \end{align*}
    for all 
    $g_0,\dots,g_{n-1}\in G$.
    We say that $\psi(x_0,\dots,x_{n-1})$ \emph{encodes $n$-uniform hypergraphs} if there exist $r<s\in\mathbb{R}$ such that $\psi(x_0,\dots,x_{n-1})$ encodes every finite $n$-uniform hypergraph using the same $r$ and $s$.
\end{defin}

As in the general continuous case, we have an equivalence between $IP_n$ of the hyperdefinable set $X/E$ and the existence of some function coding $(n+1)$-uniform hypergraphs.

\begin{prop}\label{Hyperim: IP_n iff coding nonpartite} 
The following are equivalent:
\begin{enumerate}
    \item $X/E$ has $IP_n$.
    \item There is a function in $\Psi_{\mathcal{F}_{X/E}}$ encoding $(n+1)$-uniform hypergraphs.
    \item There is a function in $\Psi_{\mathcal{F}_{X/E}}$ encoding $G_{n+1}$ as a hypergraph.
    \item There is a function in $\Psi_{\mathcal{F}_{X/E}}$ encoding $G_{n+1}$ as a hypergraph witnessed by a $G_{n+1}$-indiscernible sequence.
\end{enumerate}
\end{prop}

\begin{proof}
    $(1)\implies (2)$. By Lemma \ref{IP_n hyperdef. set iif formulas} and Proposition \ref{Hyperim: IP_n iff coding nonpartite}, there exists a function $f(x,y_0,\dots,y_{n-1})\in\mathcal{F}_{X/E}$ encoding $G_{n+1,p}$ as a partite hypergraph. Consider the function $f'(y_0,\dots,y_{n-1},x):=f(x,y_0,\dots,y_{n-1})$, following the proof of Proposition \ref{IP_n iff coding nonpartite}, we see that $\psi_{f'}\in \Psi_{\mathcal{F}_{X/E}}$ encodes $G_{n+1}$. 
    
    $(2)\implies (3)$ follows by compactness.
    
    $(3) \implies (4)$. Follows from the proof of $(3)\implies (4)$ of Proposition \ref{IP_n iff coding nonpartite}.
    

    $(4)\implies (1)$ We slightly modify the proof of Lemma \ref{IP_n f and psi}. Let $(b_g)_{g\in G_{n+1}}$ be a witness that $\psi_f$ encodes $G_{n+1}$ as a hypergraph. Note that for every $g\in G_{n+1}$, $b_g=(a^0_g,\dots, a^n_g)$ with $a^n_g\in X$. Fix a linear ordering of $Sym(n)$ and color the edges of $G_{n+1}$ according to the first $\sigma$ such that $f(a^0_{g_{\sigma(0)}},\dots,a^n_{g_{\sigma(n)}})\leq r$ whenever $\psi_f(b_{g_0},\dots,b_{g_n})\leq r$. Let $H=\{ h^i_m: i\leq n, m\leq k \}$ be any finite ordered $(n+1)$-partite $(n+1)$-uniform hypergraph. Since $\age(G_{n+1})$ has ERP, we can find a monochromatic isomorphic copy of $H$ inside $G_{n+1}$ (as a non partite hypergraph). This implies that the function $f(y_0,\dots,y_{n-1},x)$ encodes $H$ as a partite hypergraph, witnessed by the elements $\{ a^i_{h^{\sigma(i)}_m}: i\leq n, m\leq k \},$ where $\sigma$ is the color of the monochromatic copy of $H$. By compactness, $f'(x,y_0,\dots,y_{n-1}):=f(y_0,\dots,y_{n-1},x)$ encodes $G_{n+1,p}$ as a partite hypergraph. Therefore, since the function $f'(x,y_0,\dots,y_{n-1})$ is in $\mathcal{F}_{X/E}$, by Proposition \ref{Hyperim: IP_n iff encoding partite}, $X/E$ has $IP_n$.
\end{proof}

We finish the section with a characterization of $n$-dependent hyperdefinable sets analogous to the one in Theorem \ref{n-dep and collapsing}. Recall the definition of the family $\Psi^{n+1}_{\mathcal{F}_{X/E}}$ from Notation \ref{notation PSI}.

\begin{teor}
The following are equivalent:
    \begin{enumerate}
        \item $X/E$ is $n$-dependent.
        \item Every $G_{n+1,p}$-indiscernible $(a_g)_{g\in G_{n+1,p}}$ where for every $g\in P_0(G_{n+1,p})$ we have $a_g\in X/E$, is $\mL_{op}$-indiscernible.
        \item For every $m\in \mathbb{N}$, every $G_{n+1}$-indiscernible with respect to $\Psi^{n+1}_{\mathcal{F}_{X/E}}$ sequence of elements of $\C^m\times X$  is order indiscernible with respect to $ \Psi^{n+1}_{\mathcal{F}_{X/E}}$.
    \end{enumerate}
\end{teor}

\begin{proof}
            $ (1)\implies (2)$. Let $(a_g)_{g\in G_{n+1,p}}$ be a $G_{n+1,p}$-indiscernible sequence, where for every $g\in P_0(G_{n+1,p})$ we have $a_g\in X/E$, which is not $\mL_{op}$-indiscernible and let $(a'_g)_{g\in G_{n+1,p}}$ be a sequence of representatives (note that we are only choosing representatives for $a_g$ where $g\in P_0(G_{n+1,p})$). Then, by Lemma \ref{diff type diff f value}, there are $\mL_{op}$-isomorphic $W,W'\subset G_{n+1}$ substructures of size $m$, a function $f(x_0,\dots,x_{m-1})$ and $r<s\in \mathbb{R}$ such that $f( (a'_g)_{g\in W}) \leq r$ and $f( (a'_g)_{g\in W'})\geq s$ (where the elements $a'_g$ are substituted for the variables $x_0,\dots,x_{m-1}$ according to the ordering on $W$ and $W'$). Without loss of generality, by Fact \ref{ Sequence of adjacent graphs}, we may assume that $W$ is $V$-adjacent to $W'$ for some subset $V$ such that $W=g_0\dots g_n V$, $W'=g'_0\dots g'_n V$, $R(g_0,\dots g_n)$ and $\neg R(g'_0,\dots g'_n)$. 
             As in the proof of $(1) \implies (2)$ of Theorem \ref{n-dep and collapsing}, we apply Fact \ref{Isomorphic copy of random partite hypergraph} to $V$ and $g_0,\dots,g_n$. This yields $G'\subseteq G_{n+1,p}$ such that for every $(h_0,\dots,h_n)\in \prod_{i\leq n} P_i(G')$ $$ R(h_{0},\dots,h_{n}) \iff h_0\dots h_n V\cong_{\mL_{opg}}W $$ and $$ \neg R(h_0,\dots,h_n) \iff h_0\dots h_nV\cong_{\mL_{opg}}W'.$$ Recall that the sequence $(a_g)_{g\in G_{n+1,p}}$ is $G_{n+1,p}$-indiscernible and let $f'$ be the formula defined by permuting the variables $(x_0,\dots,x_{m-1})$ of $f$ in such a way that the first $n+1$-variables are the ones corresponding to $h_0,\dots,h_n$ according to the ordering on the set $\{h_0,\dots,h_n\}\cup V$. Then,
            $$ f'(a_{h_0},\dots, a_{h_n},A)\leq r \iff R(h_{0},\dots,h_{n})$$
            and
            $$f'(a_{h_0},\dots, a_{h_n},A)\geq s \iff \neg R(h_{0},\dots,h_{n}),$$ where $A=(a_g)_{g\in V}$. Since $G'$ is isomorphic to $G_{n+1,p}$, by Proposition \ref{IP_n iff encoding partite}, the formula $f'(x,y_{0},\dots, y_{n-1},A)$ has $IP_n$, where $A=(a'_g)_{g\in V}$. The tuple $A$ is contained in some product (with repetition) of $X$ and $\C$. Thus, when performing the change of variables done in Remark \ref{naming parameters and dummy variables} (1) we might end with a function $g(x,z_1,\dots,z_n)$ with infinite tuples of variables, each of which corresponds to tuples of elements from $\C$ and $X$.  However, since this new function $g(x,z_1\dots,z_n)$ is the uniform limit of functions from the family $\mathcal{F}_{X/E}$, we might find a suitable formula $g'\in \mathcal{F}_{X/E}$ witnessing $IP_n$. 
            
            
            $(1)\implies (3)$. 
            Let $(a_g)_{g\in G_{n+1}}$ be a $G_{n+1}$-indiscernible with respect to $\Psi^{n+1}_{\mathcal{F}_{X/E}}$ sequence which is not order indiscernible with respect to $\Psi^{n+1}_{\mathcal{F}_{X/E}}$. Then there are $W,W'\subset G_{n+1}$ subsets of size $n+1$, a function $\psi_f(x_0,\dots,x_{n})$ and $r<s\in \mathbb{R}$ such that $\psi_f( (a_g)_{g\in W}) \leq r$ and $\psi_f( (a_g)_{g\in W'})\geq s$. By Fact \ref{ Sequence of adjacent graphs} and the fact that $G_{n+1}$ is self-complementary, 
            we may assume $W=g_0\dots g_n$, $W'=g'_0\dots g'_n$, $R(g_0,\dots g_n)$ and $\neg R(g'_0,\dots g'_n)$.

            By $G_{n+1}$-indiscernibility of $(a_g)_{g\in G_{n+1}}$ with respect to  $\Psi^{n+1}_{\mathcal{F}_{X/E}}$ and by symmetry of the relation $R$ and $\psi_f$, this implies that 
            $$ \psi_f(a_{h_0},\dots, a_{h_n})\leq r \iff R(h_{0},\dots,h_{n})$$
            and
            $$\psi_f(a_{h_0},\dots, a_{h_n})\geq s \iff \neg R(h_{0},\dots,h_{n}).$$ By Proposition \ref{Hyperim: IP_n iff coding nonpartite}, the set $X/E$ has $IP_n$.

            $(2)\implies(1)$ It follows from Proposition \ref{Hyperim: IP_n iff encoding partite}. If the function $f\in \mathcal{F}_{X/E}$ encodes $G_{n+1,p}$ witnessed by a $G_{n+1,p}$-indiscernible sequence $(a_g)_{g\in G_{n+1,p}}$, then $(a_g)_{g\in G_{n+1,p}}$ cannot be $\mL_{op}$-indiscernible.
            
            $(3)\implies(1)$ It follows from Proposition \ref{Hyperim: IP_n iff coding nonpartite}. If $T$ has $IP_n$, there is a function $\psi_f\in \Psi^{n+1}_{\mathcal{F}_{X/E}}$ encoding $G_{n+1}$ witnessed by a $G_{n+1}$-indiscernible sequence $(a_g)_{g\in G_{n+1}}$. Then, $(a_g)_{g\in G_{n+1}}$ cannot be order-indiscernible with respect to $ \Psi^{n+1}_{\mathcal{F}_{X/E}}$.
\end{proof}
Note that the results of this section easily generalise to deal with $n$-dependence of imaginary sorts in continuous logic. We finish the chapter with an example illustrating that, in general, the theorem above is optimal. Namely, for an $n$-dependent hyperdefinable set $X/E$ and $n'>n$ there might be a $G_{n+1}$-indiscernible sequence of elements of $\C^m\times X$ for some $m<\omega$ which are not order indiscernible with respect to $\Psi^{n'+1}_{\mathcal{F}_{X/E}}$ or with respect to more general families of functions from $\mathcal{F}_{(\C^m \times X/E)^{n+1}}$.
\begin{example}\label{example optimal}
    Let $\mathcal{N}$ be a monster model of a NIP theory and $\mathcal{R}$ a monster model of the theory of random ordered graphs. We consider the structure $\mathcal{N}\sqcup \mathcal{R}$ i.e. the structure with disjoint sorts for $\mathcal{N}$ and $\mathcal{R}$ with no interaction between the sorts. Let $X=\mathcal{N}$ and $E$ be the equality relation. Clearly, $X/E$ has NIP.

    \begin{claim}
        Let $n\geq 1$. The sequence $(a_g)_{g\in G_2}:=(\overbrace{g,\dots,g}^m,n_0)_{g\in G_2}$, where $n_0$ is a fixed element of $\mathcal{N}$,  is $G_2$-indiscernible but it is not order indiscernible with respect to the family of formulas $f(x_0,\dots,x_n)$ where each $x_i$ is a tuple $(x_i^0,\dots,x_i^m)$ of variables of length $m+1$ whose last coordinate corresponds to $X$.
    \end{claim}

    \begin{proof}[Proof of the first claim]
    Let $g_0'<g_0<g_1<\dots<g_n\in G_2$ be such that $R(g_0,g_1)$ and $\neg R(g_0',g_1)$ and let $f(x_0,\dots,x_n):= R(x_0^0,x_1^0)$. Clearly, the tuples $(g_0,g_1,\dots,g_n)$ and $(g'_0,g_1,\dots,g_n)$ have the same order type but $f(a_{g_0},\dots,a_{g_n})\neq f(a_{g'_0},\dots,a_{g_n}) $.
    \end{proof}

    \begin{claim}
        For $n'>1$, the sequence $(a_g)_{g\in G_2}:=(\overbrace{g,\dots,g}^{n'},n_0)_{g\in G_2}$, where $n_0$ is a fixed element of $\mathcal{N}$, is $G_2$-indiscernible but it is not order indiscernible with respect to the family of functions $\Psi^{n'+1}_{\mathcal{F}_{X/E}}$.
    \end{claim}

    \begin{proof}[Proof of the second claim]
        We show it for $n'=2$. 
        Let $f(y,z,x):=R(y,z)$. Then the formula $$\psi_f(y_0z_0x_0,y_1z_1x_1,y_2z_2y_2):=\min_{\sigma\in Sym(2)} f(y_{\sigma(0)},z_{\sigma(1)},x_{\sigma(2)})$$ belongs to $\Psi^{3}_{\mathcal{F}_{X/E}}$. However, taking  $g_0'<g_0<g_1<g_2$ such that \begin{itemize}
            \item $R(g_0,g_1)$, $R(g_0,g_2)$, $R(g_1,g_2)$
            \item $R(g'_0,g_2)$ and  $\neg R(g'_0,g_1)$
        \end{itemize}
        we have that the tuples $(g_0,g_1,g_2)$ and $(g_0',g_1,g_2)$ are order-isomorphic and $$\psi_f(g_0g_0n_0,g_1g_1n_0,g_2g_2n_0)\neq \psi_f(g'_0g'_0n_0,g_1g_1n_0,g_2g_2n_0).$$
    \end{proof}
\end{example}

\chapter{Topological dynamics}\label{Chapter 6}
We present the framework for this chapter. Let $T$ be a complete first-order theory 
	of infinite models in a language $\mL$. Let $\C\prec\C'$ be models of  $T$ which are sufficiently saturated and strongly homogeneous. The precise degrees of saturation needed for particular sections or results will be given in the relevant places. Recall that we say that a set is $\C$-small if its cardinality is $\C$-small i.e. smaller than the saturation degree of $\C$. $X$ will denote an $\emptyset$-type-definable subset of $\C^\lambda$ (or a product of $\lambda$ sorts).

Unless specified otherwise, we denote by $r$ the restriction map $r \colon S_X(\C') \to S_X(\C)$.

\section{Preliminaries}\label{section: prelim top dyn}

In this section we introduce the necessary machinery of definability patterns language and the definability patterns structure on $S_X(\C)$ that will be used throughout the rest of the chapter. The results here are based on Krzysztof Krupi\'nski's lecture on topological dynamics in model theory, which is an alternative approach to Hrushovski's ``infinitary core" inspired by Simon's seminar notes on \cite{hrushovski2022definability}. We will assume that $\C$ is at least $\aleph_0$-saturated and strongly $\aleph_0$-homogeneous.

Recall the definition of content and of the order $\leq^c$ from Definition \ref{defin: content} and \ref{defin: content order}.

\begin{defin}[Infinitary definability patterns structure on $S_X(\C)$]
For any $n$-tuple of formulas $\overline{\varphi}$ consisting of $\varphi_1(x,y),\dots,\varphi_n(x,y)\in \mL$ and $q(y)\in S_y(\emptyset)$, we define $R_{\overline{\varphi},q}$ on $S_X(\C)^n$ by $$R_{\overline{\varphi},q}(\overline{p}) \iff (\varphi_1(x,y),\dots,\varphi_n(x,y),q)\notin c(\overline{p}), $$ where $\overline{p}=(p_1,\dots,p_n)$,
    i.e.,  there is no  $b\models q$  such that  $\varphi_1(x,b)\in p_1\wedge\cdots\wedge \varphi_n(x,b)\in p_n$. The \emph{infinitary definability patterns structure on} $S_X(\C)$ consists of all such relations $R_{\overline{\varphi},q}$.  We denote by $\End(S_X(\C))$ the semigroup of endomorphisms of $S_X(\C)$ with the infinitary definable patterns structure.
\end{defin}

During this section, we also consider the structure on $S_X(\C)$ given by the flow $(\aut(\C),S_X(\C)).$ Recall that we denote by $E(S_X(\C))$ the Ellis semigroup of this flow (see Definition \ref{defin: ellis semigroup}).

\begin{lema}\label{S_X is homogeneous}
    We have the following:
    \begin{itemize}
        \item $\End(S_X(\C))=E(S_X(\C))$,
        \item $S_X(\C)$ is homogeneous in the sense that any partial morphism between substructures of $S_X(\C)$ (i.e., any structure preserving map $f:A\to B$ where $A,B\subseteq S_X(\C)$) extends to an endomorphism of $S_X(\C)$.
    \end{itemize}
\end{lema}
\begin{proof}
    $\End(S_X(\C))\supseteq E(S_X(\C))$ follows from the right to left implication of Fact \ref{content and ellis semigroup}.

    The inclusion $\End(S_X(\C))\subseteq E(S_X(\C))$ and homogeneity follow from the left to right implication of Fact \ref{content and ellis semigroup} and compactness of $E(S_X(\C))$.
\end{proof}
\begin{prop}\label{group isomorphism delta}
    Let $\mathcal{M}\unlhd E(S_X(\C))$ be a minimal left ideal and $u\in\mathcal{J}(\mathcal{M})$. Let $\overline{\mathcal{J}}:=\Image(u)\subseteq S_X(\C)$. Then the map $$\delta: u\mathcal{M}\to \aut(\overline{\mathcal{J}}) $$ given by $\delta(\eta):=\eta\!\!\upharpoonright_{\overline{\mathcal{J}}}$ is a group isomorphism, where $\aut(\overline{\mathcal{J}})$ is the group of automorphisms of $\overline{\mathcal{J}}$ in the infinitary patterns language.
\end{prop}
\begin{proof}
    Since $u\mathcal{M}$ is a group, we easily get that $\delta$ takes values in $\Sym(\overline{\mathcal{J}})$. The fact that $\delta$ takes values in $\aut(\overline{\mathcal{J}})$ follows from Lemma \ref{S_X is homogeneous}. Clearly, the map $\delta$ is a homomorphism. Injectivity follows from the fact that $u\mathcal{M}u=u\mathcal{M}$ and surjectivity follows by homogeneity of $S_X(\C)$.
\end{proof}

\begin{prop}\label{isomorphic image idempotent}
    For any minimal left ideals $\mathcal{M}, \mathcal{N}$ of $E(S_X(\C))$ and idempotents $u\in \mathcal{M}$ and $v\in \mathcal{N}$, $\Image(u)\cong\Image(v)$ as the infinitary definability patterns structure.
\end{prop}
    \begin{proof}
        By the Ellis theorem (see Fact \ref{Ellis theorem}), there is an idempotent $u'\in \mathcal{M}$ such that $vu'=u'$ and $u'v=v$. Then, $\Image(u')=\Image(v)$, so we can assume $\mathcal{M}=\mathcal{N}$ without loss of generality. Then, $uv=u$ and $vu=v$, and so the maps\begin{align*}
            u\!\!\upharpoonright_{\Image(v)}: &\Image(v)\to \Image (u)\\
            &\text{ and }\\
            v\!\!\upharpoonright_{\Image(u)}: &\Image(u)\to \Image (v)
        \end{align*} are mutual inverses. Hence, $\Image(v)\cong \Image (u)$ by Lemma \ref{S_X is homogeneous}.
    \end{proof}

By Proposition \ref{isomorphic image idempotent}, up to isomorphism, both $\overline{\mathcal{J}}=\Image(u)$ and $\aut(\overline{\mathcal{J}})$ do not depend on the choice of the minimal left ideal $\mathcal{M}$ and idempotent $u\in\mathcal{M}$.

\begin{defin}[ipp-topology]
    The \emph{ipp-topology} on $\aut(\overline{\mathcal{J}})$ is given by the subbasis of closed sets consisting of $$\{f\in \aut(\overline{\mathcal{J}}): R_{\overline{\varphi},r}(f(p_1),\dots, f(p_m), p_{m+1},\dots, p_n ) \} $$ for any $\varphi_1(x,y),\dots,\varphi_n(x,y)\in \mL$, $r\in S_y(\emptyset)$ and $p_1,\dots,p_m,p_{m+1},\dots,p_n\in \overline{\mathcal{J}} $.
\end{defin}

The proof of the following fact can be found in Appendix \ref{AppendixB}.

\begin{fact}\label{delta is homeomorphism}
    The map  $$\delta: u\mathcal{M}\to \aut(\overline{\mathcal{J}}) $$ from Proposition \ref{group isomorphism delta}  is a homeomorphism when $u\mathcal{M}$ is equipped with the $\tau$-topology and $\aut(\overline{\mathcal{J}})$ with the ipp-topology.
\end{fact}

\begin{defin}
    A subset $Q\subseteq S_X(\C)$ is \emph{ip-minimal} (from infinitary patterns minimal) if any morphism $f:Q\to S_X(\C)$ is an isomorphism onto $\Image(f)$.
\end{defin}

\begin{prop}\label{equivalences ip minimal}
    Let $\overline{p}$ be an enumeration of $S_X(\C)$ and $\overline{q}=\eta \overline{p}$ (coordinate-wise) for some $\eta\in E(S_X(\C))$. Then the following are equivalent:
    \begin{enumerate}
        \item $\overline{q}$ is $\leq^c$ minimal in $E(S_X(\C))\overline{p}$, where $\overline{q}'\leq^c \overline{q}''$ means $$(q'_{i_1},\dots q'_{i_n})\leq^c (q''_{i_1},\dots q''_{i_n})$$ for any finite sets of indices $i_1<\dots<i_n$, or equivalently, $$ \overline{q}'\leq^c \overline{q}''\iff \overline{q}'\in E(S_X(\C))\overline{q}''  .$$
        \item The coordinates of $\overline{q}$ form an ip-minimal subset $Q$.
        \item There is $\mathcal{M}\unlhd E(S_X(\C))$ a minimal left ideal and $\eta_0\in\mathcal{M}$ such that $\overline{q}=\eta_0\overline{p}$.
        \item $\eta$ belongs to a minimal left ideal of $E(S_X(\C))$.
    \end{enumerate}
\end{prop}
\begin{proof}
    $(1)\implies (2)$. Take any $f:Q\to S_X(\C)$, Lemma \ref{S_X is homogeneous} implies that $f$ can be extended to $\tilde{f}:S_X(\C)\to S_X(\C)$ and such $\tilde{f}$ can be seen as $\eta\in E(S_X(\C))$. Hence, we have $\eta\overline{q}\leq^c \overline{q}$. By minimality of $\overline{q}$, there is $\eta'\in E(S_X(\C))$ such that $\eta'\eta \overline{q}=\overline{q}$. Thus, $f$ is an isomorphism to its image.

    $(2)\implies (1)$. Take any $\overline{q}'\in E(S_X(\C))\overline{p}$ such that $\overline{q}'\leq^c \overline{q}$, that is, $\overline{q}'=\eta \overline{q}$ for some $\eta\in E(S_X(\C))$. Such $\eta$ can be seen as morphism $f:Q\to S_X(\C)$, hence it is an isomorphism to its image and has an inverse $f': \Image(f)\to Q$. By Lemma \ref{S_X is homogeneous}, $f'$ can be extended to $\eta'\in E(S_X(\C))$ satisfying $\eta'\overline{q}'=\overline{q}$. This implies that $\overline{q}\leq^c\overline{q}'$.

    $(1)\implies (4)$ Consider any $\eta' \in E(S_X(\C))$. Then, $\eta'\overline{q}\leq^c \overline{q}$ and $\eta'\overline{q}=\eta'\eta\overline{p}$. By $(1)$, $\overline{q}\leq^c\eta'\overline{q}$, so there is $\eta''\in E(S_X(\C))$ such that $\eta'' \eta' \overline{q}=\overline{q}$, that is $\eta''\eta'\eta \overline{p}=\eta\overline{p}$. Since $\overline{p}$ is an enumeration of $S_X(\C)$, we get $\eta''\eta'\eta=\eta$ as elements of $E(S_X(\C))$. Hence, Fact \ref{Ellis theorem}  implies that $E(S_X(\C))\eta$ is a minimal left ideal.

    $(4)\implies (3)$. Trivial.

    $(3)\implies (1)$. Let $\overline{q}=\eta_0\overline{p}$, where $\eta_0$ belongs to a minimal left ideal $\mathcal{M}$. Consider any $\overline{q}'\in E(S_X(\C))$ such that $\overline{q}'\leq^c\overline{q}$. Then, $$\overline{q}'=\eta'\overline{q}=\eta'\eta_0\overline{p}$$ for some $\eta'\in E(S_X(\C))$. Since $E(S_X(\C))\eta_0$ is a minimal left ideal, there is $\eta''\in E(S_X(\C))$ such that $\eta''\eta'\eta_0=\eta_0$. Thus $\eta''\overline{q}'=\eta''\eta'\eta_0\overline{p}=\eta_0\overline{p}=\overline{q}$. Therefore, $\overline{q}\leq^c\overline{q}'$.
\end{proof}

\begin{cor} \label{ip-minimal exists}
    There exists an ip-minimal $Q\subseteq S_X(\C)$ with a morphism $$g:S_X(\C)\to Q.$$
\end{cor}
\begin{proof}
    Let $\overline{p}$ be an enumeration of $S_X(\C)$, $\mathcal{M}\unlhd E(S_X(\C)))$ a minimal left ideal and $\eta \in \mathcal{M}$. Then, $Q:=\eta[S_X(\C)]$ and $g:=\eta$ satisfy the requirements by Proposition \ref{equivalences ip minimal}.
\end{proof}

The next remark follows from Proposition \ref{equivalences ip minimal}, since for any element $\eta$ in a minimal left ideal $\mathcal{N}\unlhd E(S_X(\C))$ we can find an idempotent $v\in \eta \mathcal{N}$ which satisfies $\Image(v)=\Image(\eta)$.
\begin{remark}\label{Remark Q is the core}
    If the conditions of Proposition \ref{equivalences ip minimal} hold, then the ip-minimal set $Q$ of Proposition \ref{equivalences ip minimal} $(2)$ satisfies $Q\cong \overline{\mathcal{J}}$.
\end{remark}

\begin{lema}
    A subset $Q\subseteq S_X(\C)$ is ip-minimal if and only if every finite $Q_0\subseteq Q$ is ip-minimal. In particular, the union of a chain of ip-minimal subsets of $S_X(\C)$ is ip-minimal.
\end{lema}
\begin{proof}
    It follows by Lemma \ref{S_X is homogeneous}, in particular by homogeneity of $S_X(\C)$.
\end{proof}

\begin{remark}
    By the previous lemma and Zorn's lemma, there exists some ip-minimal $I_\C\subseteq S_X(\C)$ maximal with respect to the inclusion.
\end{remark}

\begin{lema}\label{morphism are surjective}
    Let $f:I_\C\to K$ be a morphism, where $K$ is an ip-minimal set. Then $f$ is surjective and is therefore an isomorphism. 
\end{lema}
\begin{proof}
    Let $I':=f[I_\C]\subseteq K$. By ip-minimality of $I_\C$, the map $f:I_\C\to I'$ is an isomorphism in the infinitary patterns language. Let $g:I'\to I_\C$ be the inverse of $f$. By Lemma \ref{S_X is homogeneous}, there exists $\overline{g}\in \End(S_X(\C))$ extending $g$. Let $K':=\overline{g}[K]$. Since $K$ is ip-minimal, $\overline{g}\!\!\upharpoonright_K: K\to K'$ is an isomorphism (in the infinitary patterns language) and $K'$ is also ip-minimal. By maximality of $I_\C$, we have $K'=I_\C$. Therefore, by injectivity of $\overline{g}\!\!\upharpoonright_K$, $I'=K$. Thus, $f:I_\C\to K$ is onto.
\end{proof}
\begin{cor}\label{unique ip-minimal}
    There is a unique (up to isomorphism) ip-minimal subset $K\subseteq S_X(\C)$ with a morphism $S_X(\C)\to K$. It is the ip-minimal set $I_\C$ described above. Moreover, there is a retraction $S_X(\C)\to I_\C$. 
\end{cor}
\begin{proof}
    Let $g: S_X(\C)\to K$ be a morphism and $K$ an ip-minimal subset (which exists by Corollary \ref{ip-minimal exists}). By Lemma \ref{morphism are surjective}, the map $g\!\!\upharpoonright_{I_\C}:I_\C\to K$ is an isomorphism, which proves the uniqueness. 

    For the moreover part, take a morphism $g: S_X(\C)\to I_\C$ (which exists by  the first part). Then $g\!\!\upharpoonright_{I_\C}: I_\C \to I_\C$ is an isomorphism, and so  $f=:(g\!\!\upharpoonright_{I_\C})^{-1}\circ g$ is a retraction from $S_X(\C)$ to $I_\C$. 
\end{proof}

By Corollary \ref{unique ip-minimal} and Remark \ref{Remark Q is the core}, we can assume that $I_\C=\overline{\mathcal{J}}$.

\begin{lema}\label{there is a section}
    Let $\C'\succ \C$. Then the restriction $$r:S_X(\C')\to S_X(\C)$$ is a morphism  of infinitary definability patterns structures and has a section $$s:S_X(\C)\to S_X(\C')$$ which is a morphism of infinitary definability patterns structures.
\end{lema}
\begin{proof}
    Let $\overline{p}=(p_i)_{i<\mu}$ be an enumeration of $S_X(\C)$ and $(a_i)_{i<\mu}$ be a realization of $\overline{p}$ and let the type $\tp ((a'_i)_{i<\mu}/\C')=:p'$ be a strong heir extension of $p$ (see Definition \ref{strong heirs}) in the language $\mL$. Note that in particular, for any $n<\omega$ and $i_1,\dots i_n<\mu$ we have $c(p_{i_1},\dots,p_n)=c(p'_{i_1},\dots,p'_n)$, where $p_i'=\tp (a'_i/\C')$. Hence, the map $s: S_X(\C)\to S_X(\C')$ given by $s(p_i)=\tp(a_i'/\C')$ is a morphism between the infinitary definability patterns structures.
\end{proof}
\begin{teor}\label{absoluteness of the core}
    Up to isomorphism in the definability patterns language, $\overline{\mathcal{J}}$ does not depend on the choice of the $\aleph_0$-saturated, strongly $\aleph_0$-homogeneous model for which it is computed.
\end{teor}
\begin{proof}
    It is enough to show that for $\aleph_0$-saturated, strongly $\aleph_0$-homogeneous models  $\C\prec \C'$, $I_\C\cong I_{\C'}$, where $I_{\C'}$ is defined for $\C'$ the same way as $I_\C$ was defined for $\C$. We have the following maps:
    \begin{itemize}
        \item The restriction morphism $r: S_X(\C') \to S_X(\C)$;
        \item Section $s:S_X(\C)\to S_X(\C') $ given by Lemma \ref{there is a section};
        \item The retraction $f_\C: S_X(\C)\to I_\C$ given by Corollary \ref{unique ip-minimal};
        \item The retraction $f_{\C'}: S_X(\C')\to I_{\C'}$ given by Corollary \ref{unique ip-minimal}.
    \end{itemize}
    Then, the maps $h_1:=f_\C\circ ( r\!\!\upharpoonright_{I_{\C'}} ): I_{\C'}\to I_\C$ and $h_2:=f_{\C'}\circ ( s\!\!\upharpoonright_{I_\C} ): I_{\C}\to I_{\C'}$ are morphisms. Hence, the maps $h_2\circ h_1:I_{\C'}\to I_{\C'} $ and $h_1 \circ h_2: I_{\C}\to I_{\C}$ are isomorphisms by Lemma \ref{morphism are surjective}.  The first thing implies that $h_1$ is an isomorphism onto its image and the second one that $\Image(h_1)=I_\C$. Therefore, $h_1$ is an isomorphism.
\end{proof}

One can show that $\overline{\mathcal{J}}$ is precisely Hrushovski's infinitary core (but localized to $X$) considered in \cite[Appendix A]{hrushovski2022definability}); however, we will not use this approach in this chapter.

    \begin{cor}\label{absoluteness of the ellis group}
        The Ellis group of the flow $(\aut(\C),S_X(\C))$ does not depend (as a topological group) on the choice of $\C$ as long as it is $\aleph_0$-saturated and strongly $\aleph_0$-homogeneous.
    \end{cor}
    \begin{proof}
    It follows from Proposition \ref{group isomorphism delta}, Fact \ref{delta is homeomorphism} and Theorem \ref{absoluteness of the core}.
\end{proof}

We finish the section by introducing several equivalence relations which we will study through the chapter.
    
Let $\fwap\subset S_X(\C)\times S_X(\C)$ be the finest closed $\aut(\C)$-invariant equivalence relation on $S_X(\C)$ such that the flow $( \aut(\C), S_X(\C)/\fwap )$ is WAP, and let $\fwapp\subset S_X(\C')\times S_X(\C')$ be the finest closed $\aut(\C')$-invariant equivalence relation on $S_X(\C')$ such that the flow $( \aut(\C'), S_X(\C')/\fwapp )$ is WAP. These equivalence relations always exist due to general topological dynamics reasons. Namely, for a general $G$-flow $(G,X)$, there exists a correspondence between closed $G$-invariant equivalence relations on $X$ and closed unital $G$-subalgebras of $C(X)$, stablished by the following fact:

\begin{fact}
    Let $(G,X)$ be a flow and and $\mathcal{A}$ be a closed unital $G$-subalgebra of $C(X)$. Define an equivalence relation on $X$ by $x\sim y$ if $f(x)=f(y)$ for all $f\in \mathcal{A}$, and let $Y=X/\sim$. Then:
    \begin{itemize}
        \item Every $f\in\mathcal{A}$ factors uniquely thought the quotient map $\pi:X\to Y$ as $f=f_Y\circ \pi$ and he quotient topology is the minimal topology under which every such $f_Y$ is continuous.
        \item The quotient topology is compact and Hausdorff, and $\sim$ is a closed $G$-invariant equivalence relation.
    \end{itemize}

    Conversely, let $(G,Y)$ be a flow and $\pi: X\to Y$ a topological quotient map and a $G$-flow epimorphism. Then $Y$ can be obtained as a quotient flow of $X$ using the closed $G$-subalgebra $\mathcal{A}=\{ f\circ \pi: f\in C(Y) \}$
\end{fact}
\begin{proof}
    The fact that every $f\in\mathcal{A}$ factors uniquely through the quotient map follows by construction of $\sim$. Let $\mathcal{T}_1$ be the quotient topology on $Y$ and $\mathcal{T}_2$ the minimal topology on which every $f_Y$ is continuous. We show that $\mathcal{T}_1$ refines $\mathcal{T}_2$. Let $V\subseteq \R$ be an open set and $f\in\mathcal{A}$, the set $U=f_Y^{-1}(V)$ is an open set of $\mathcal{T}_2$. Since $\pi^{-1}(V)=f^{-1}(V)\subseteq X$ is open, $U\in\mathcal{T}_1$. Note that $\mathcal{T}_1$ is a compact topology (it is the quotient topology on a compact set) and that $\mathcal{T}_2$ is Hausdorff (any distinct $y_1,y_2\in Y$ can be separated by some $f_Y$), and so $\mathcal{T}_1$=$\mathcal{T}_2$. Finally, since the quotient topology is Hausdorff, the equivalence relation $\sim$ is closed by \cite[Fact 2.7]{rzepecki2018bounded}, and clearly $G$-invariant.

    For the converse, the space $Y'=X/\sim$ constructed in the above manner can be identified with $Y$. As before, the original topology on $Y$ refines the quotient topology, so they must coincide.  Note also that $\mathcal{A}=\{ f\circ \pi: f\in C(Y) \}$ is a closed unital $G$-subalgebra of $C(X)$.
\end{proof}

By Fact \ref{wap closed unital}, the set of all WAP functions of $C(S_X(\C))$ (denoted $WAP(S_X(\C)$) is a closed unital $\aut(\C)$-subalgebra. Then, the equivalence relation $\fwap$ on $S_X(\C)$ given by $$p\fwap q \iff \forall f\in \textrm{WAP}(S_X(\C))[f(p)=f(q)]$$ is closed and $\aut(\C)$-invariant,  and by \ref{separating points: WAP}, the flow $(\aut(\C),S_X(\C)/\fwap)$ is WAP. The same is true for $\fwapp$. It is easy to see that they are the finest such equivalence relations, since the quotient by any strictly finer equivalence relation on $S_X(\C)$ will need a non-WAP function to separate points.

Similarly, let $\fta\subset S_X(\C)\times S_X(\C)$ be the finest closed $\aut(\C)$-invariant equivalence relation on $S_X(\C)$ such that the flow $( \aut(\C), S_X(\C)/\fta )$ is tame, and let $\ftaa\subset S_X(\C')\times S_X(\C')$ be the finest closed $\aut(\C')$-invariant equivalence relation on $S_X(\C')$ such that the flow $( \aut(\C'), S_X(\C')/\ftaa )$ is tame.  As in the WAP case, these equivalence relations exist because of general topological dynamics reasons, this time using Facts \ref{tame closed unital} and \ref{separating points: tame} instead.


By Remark \ref{finest eq exist for fixed param}, one easily gets

\begin{remark}\label{remark: pist and piNIP}
Given an $\emptyset$-type-definable $X\subseteq \C^\lambda$, there exists a type $\pi^{\textrm{st}}(x, y)$ over the empty set with $|x|=\lambda$ which for every $(\aleph_0 + \lambda)^+$-saturated
model $\C$ defines the finest $\emptyset$-type-definable equivalence relation
on $X$ with stable quotient; we will denote this relation by $E^{\textrm{st}}_\emptyset$. Similarly, there exists a type $\pi^{\textrm{NIP}}(x, y)$ over the empty set with $|x|=\lambda$ which for every
$(\aleph_0 +\lambda)^+$-saturated model $\C$ defines the finest $\emptyset$-type-definable equivalence relation
on $X$ with NIP quotient; we will denote this relation by $E^{\textrm{NIP}}_\emptyset$.
\end{remark}
Even if $\C$ is not sufficiently saturated,
then by $E^{\textrm{st}}_\emptyset$ we could mean  $\pi^{\textrm{st}}(X(\C),X(\C))$. However, we do not need to talk about
it, as we will work with the equivalence relation $\tilde{E}^{\textrm{st}}_\emptyset$ on $S_X(\C)$ which is defined by
$$p \tilde{E}^{\textrm{st}}_\emptyset q \iff (\exists a \models p, b \models q)( \pi^{\textrm{st}}(a,b)),$$
where $a,b$ are taken in a big monster model.
Similarly for NIP, we are interested in the relation $\tilde{E}^{\textrm{NIP}}_\emptyset$ on $S_X(\C)$ defined by
$$p \tilde{E}^{\textrm{NIP}}_\emptyset q \iff (\exists a \models p, b \models q)( \pi^{\textrm{NIP}}(a,b)).$$
By  $E'^{\textrm{st}}_\emptyset$, $E'^{\textrm{NIP}}_\emptyset$, $\tilde{E}'^{\textrm{st}}_\emptyset$, and $\tilde{E}'^{\textrm{NIP}}_\emptyset$ we denote the relations defined as above but working with $\C'$ in place of $\C$ (where $\C'$ is another model of $T$).

\section{Ellis groups of compatible quotients are isomorphic}\label{section: compatible quotients}

We introduce some conditions guaranteeing that, for a closed $\aut(\C)$-invariant equivalence relation $F$ on $S_X(\C)$, the Ellis group of the quotient flow $(\aut(\C),S_X(\C)/F)$ is independent of the choice of $\C$ as long as it is $\aleph_0$-saturated and strongly $\aleph_0$-homogeneous. Hence, for this section we will assume that $\C$ satisfies only those saturation assumptions; and similarly for $\C'$.

\begin{defin}
    Let $F'$ be a closed, $\aut(\C')$-invariant equivalence relation defined on $S_X(\C')$, and $F$ a closed, $\aut(\C)$-invariant equivalence relation defined on $S_X(\C)$. We say that $F'$ and $F$ are compatible if $r[F']=F$ (i.e., $\{ (r(a),r(b)): (a,b)\in F' \}=F)$, where $r: S_X(\C')\to S_X(\C)$ is the restriction map.
\end{defin}

\begin{teor}\label{compatible implies absolute}
    If $F'$ and $F$ are compatible equivalence relations respectively on $S_X(\C')$ and $S_X(\C)$, then the Ellis group of the flow $(\aut(\C'),S_X(\C')/F')$ is topologically isomorphic to the Ellis group of the flow $(\aut(\C),S_X(\C)/F)$.
\end{teor}
\begin{proof}
   Let $(p_i)_{i<\mu}$ be an enumeration of $S_X(\C)$ and $(a_i)_{i<\mu}$ be a sequence of realizations. Consider the type $p:=\tp((a_i)_{i<\mu}/\C)$ and let $\tp ((a'_i)_{i<\mu}/\C')=:p'$ be a strong heir extension of $p$.  For each $i<\mu$ we denote $\tp(a_i'/\C')$ by $p_i'$.

    Let $s: S_X(\C)\to S_X(\C')$ be the function given by $s(p_i)=p_i'$. The function $s$ is a section of the restriction map $r$ and since $p'$ is a strong heir extension of $p$, $s$ is an isomorphism to its image in the infinitary patterns language.

    Choose and idempotent $u \in \mathcal{M}$, where $\mathcal{M}$ is a minimal left ideal of $E(S_X(\C))$. Then $I_{\C}:=\Image(u)$ is ip-minimal by Proposition \ref{equivalences ip minimal}. Then, by Proposition \ref{group isomorphism delta}, there is an isomorphism $\delta$ from $u\mathcal{M}$ to $\aut(I_{\C})$. 
    The fact that $s \!\! \upharpoonright_{I_{\C}} \colon I_{\C} \to s[I_{\C}]$ is an isomorphism follows from the definition of strong heirs and infinitary definability patterns structures. 
    On the other hand, by Corollary \ref{unique ip-minimal}, let $I_{\C'}$ be the unique up to isomorphism ip-minimal subset of $S_X(\C')$ for which there is a morphism from $S_X(\C')$ to $I_{\C'}$. By Theorem \ref{absoluteness of the core}, we have $I_{\C'} \cong I_{\C}$ and therefore $I_{\C'} \cong s[I_{\C}]$. Hence, $s[I_{\C}]$ is ip-minimal in $S_X(\C')$.  
    By Lemma \ref{morphism are surjective}, the morphism $\eta:=s \circ u \circ r \colon S_X(\C') \to s[I_{\C}]$ is surjective. Thus, by Lemma \ref{S_X is homogeneous}, $\eta \in E(S_X(\C'))$. Using Proposition \ref{equivalences ip minimal}, we conclude that $\eta$ is in some minimal left ideal $\mathcal{M'}$ of $E(S_X(\C'))$. Finally, taking an idempotent $u' \in \eta\mathcal{M}'$, we get $\Image(u')=\Image(\eta)= s[I_{\C}]$.
    

    The map $\pi: S_X(\C')\to S_X(\C')/F' $ induces the following functions:

      \begin{minipage}[c]{0.8\textwidth}
              \hfill
        \begin{tikzpicture}\label{diagram pi}
        \matrix (m) [matrix of math nodes,row sep=3em,column sep=4em,minimum width=2em]
  {
     \Tilde{\pi}:u'\mathcal{M}' & \Tilde{\pi}(u')\Tilde{\pi}(\mathcal{M}') \\
     \Tilde{\Tilde{\pi}}:\aut(s[I_\C]) & \Tilde{\pi}(u')\Tilde{\pi}(\mathcal{M}')\!\!\upharpoonright_{\pi[s[I_\C]]} \subseteq Sym(\pi[s[I_\C]]) \\};
  \path[-stealth]
    (m-1-1) edge [->>] node [left] {$\cong$} (m-2-1)
            edge [->>] 
            (m-1-2)
    (m-2-1.east|-m-2-2) edge [->>] 
            (m-2-2)
    (m-1-2) edge [->>]  node [right] {$\cong$} (m-2-2);
    \end{tikzpicture}
  \end{minipage}
  \begin{minipage}[c]{0.1\textwidth}
    \hfill (1) 
  \end{minipage}

    where:
    \begin{itemize}
        \item $\aut(s[I_\C])$ is the group of automorphisms in the infinitary definability patterns language.
        \item The map $\Tilde{\pi}$ is given by Fact \ref{epimorphism semigroup to group}.
        \item The isomorphism on the left is given by Proposition \ref{group isomorphism delta} and the fact that $s[I_\C]=\Image(u')$.
        \item The isomorphism on the right is given by Fact \ref{restrictio to image is monomorphism}.
        \item The map $\Tilde{\Tilde{\pi}}:\aut(s[I_\C]) \to Sym(\pi[s[I_\C]])$ is given by $\Tilde{\Tilde{\pi}}(\sigma)(\pi(x)):=\pi(\sigma(x))$ (which is the composition of the inverse of the left side isomorphism, the map $\tilde{\pi}$ and the right side isomorphism).
    \end{itemize}

    Similarly, $\rho: S_X(\C)\to S_X(\C)/F $ induces the following functions:
    
      \begin{minipage}[c]{0.8\textwidth}
              \hfill
        \begin{tikzpicture} \label{diagram rho}
        \matrix (m) [matrix of math nodes,row sep=3em,column sep=4em,minimum width=2em]
  {
     \Tilde{\rho}:u\mathcal{M} & \Tilde{\rho}(u)\Tilde{\rho}(\mathcal{M}) \\
     \Tilde{\Tilde{\rho}}:\aut(I_\C) & \Tilde{\rho}(u)\Tilde{\rho}(\mathcal{M})\!\!\upharpoonright_{\rho[I_\C]}\subseteq Sym(\rho[I_\C]) \\};
  \path[-stealth]
    (m-1-1) edge [->>] node [left] {$\cong$} (m-2-1)
            edge [->>] 
            (m-1-2)
    (m-2-1.east|-m-2-2) edge [->>] 
            (m-2-2)
    (m-1-2) edge [->>]  node [right] {$\cong$} (m-2-2);
    \end{tikzpicture}
  \end{minipage}
  \begin{minipage}[c]{0.1\textwidth}
    \hfill (2) 
  \end{minipage}
  
 where:
 \begin{itemize}
     \item  $\aut(I_\C)$ is the group of automorphisms in the infinitary definability patterns language.
     \item The map $\Tilde{\rho}$ is given by Fact \ref{epimorphism semigroup to group}.
    \item The isomorphism on the left is given by Proposition \ref{group isomorphism delta}.
    \item The isomorphism on the right is given by Fact \ref{restrictio to image is monomorphism}.
    \item The map $\Tilde{\Tilde{\rho}}:\aut(I_\C) \to Sym(\rho[I_\C])$ is given by $\Tilde{\Tilde{\rho}}(\sigma)(\rho(x)):=\rho(\sigma(x))$ (which is the composition of the inverse of the left side isomorphism, the map $\tilde{\rho}$ and the right side isomorphism).
 \end{itemize}

 Since the map $r:s[I_\C]\to I_\C$ is an isomorphism in the infinitary patterns language, it induces an isomorphism $$\overline{r}: \aut(s[I_\C])\to \aut(I_\C)$$ given by $\overline{r}(\sigma)(r(p)):=r(\sigma(p))$. Our goal is then to prove that there exists an isomorphism $f$ such that the diagram below commutes:
   \begin{center}
        \begin{tikzpicture}
        \matrix (m) [matrix of math nodes,row sep=3em,column sep=4em,minimum width=2em]
  {
     \aut(s[I_\C]) & \Tilde{\pi}(u')\Tilde{\pi}(\mathcal{M}')\!\!\upharpoonright_{\pi[s[I_\C]]} \\
     \aut(I_\C) & \Tilde{\rho}(u')\Tilde{\rho}(\mathcal{M}')\!\!\upharpoonright_{\rho[I_\C]} \\};
  \path[-stealth]
    (m-1-1) edge [->>] node [left] {$\cong$} node [right] {$\overline{r}$} (m-2-1)
            edge [->>] node [above] {$\Tilde{\Tilde{\pi}}$} 
            (m-1-2)
    (m-2-1.east|-m-2-2) edge [->>]  node [above] {$\Tilde{\Tilde{\rho}}$}
            (m-2-2)
    (m-1-2) edge [->>] [dashed] node [right] {$\exists f$ } 
    (m-2-2);
    \end{tikzpicture}
\end{center}

We first show that a function $f$ such that the above diagram commutes exists. It is enough to show $\ker(\Tilde{\Tilde{\pi}})\subseteq \ker(\Tilde{\Tilde{\rho}}\circ \overline{r})$. Note that $\sigma\in \ker(\Tilde{\Tilde{\pi}})$ if and only if for every $p\in s[I_\C]$ we have that $\sigma(p)F' p$. Take an arbitrary $\sigma\in \ker(\Tilde{\Tilde{\pi}})$. By compatibility of $F'$ and F, $\sigma(p)F' p$ implies $ r(\sigma(p))F r(p)$. Hence, for every $p\in s[I_\C]$ we have $\overline{r}(\sigma)(r(p))F r(p)$. Therefore, $\overline{r}(\sigma)$ is in $\ker(\Tilde{\Tilde{\rho}})$ and $\sigma$ is in $\ker(\Tilde{\Tilde{\rho}}\circ \overline{r})$.

To see that $f$ is an isomorphism, it is enough to show that $\ker(\Tilde{\Tilde{\pi}})\supseteq \ker(\Tilde{\Tilde{\rho}}\circ \overline{r}).$ Take an arbitrary $\sigma\in \ker(\Tilde{\Tilde{\rho}}\circ \overline{r})$. Then, for every $p\in s[I_\C]$ we have that $r(\sigma(p))F r(p)$. \begin{claim}
    $r(\sigma(p))F r(p)$ implies $\sigma(p)F' p$.
\end{claim}
\begin{proof}[Proof of claim]
    By compatibility of $F'$ and $F$, there are $s_1,s_2\in S_X(\C')$ such that $r(s_1)=r(\sigma(p))$, $r(s_2)=r(p)$ and $s_1F' s_2$. At the same time, since $p,\sigma(p)\in s[I_\C]$, there are $i,j<\mu$ such that \begin{align*}
       \sigma(p)=p_i' &\wedge p=p_j';\\
        r(\sigma(p))=p_i &\wedge r(p)=p_j.        
    \end{align*}
    Hence, $c(s_1,s_2)\supseteq c(r(s_1),r(s_2))=c(r(\sigma(p)),p)=c(\sigma(p),p)$. Thus, there is $\eta\in E(S_X(\C'))$ such that $$\eta(s_1,s_2)=(\sigma(p),p).$$
    Therefore, since $F'$ is $\aut(\C')$-invariant and closed, $\sigma(p)F' p$.
\end{proof}
By the claim and the above description of $\ker(\Tilde{\Tilde{\pi}})$, we get that $\sigma\in \ker(\Tilde{\Tilde{\pi}})$.

Moreover, $f$ is a homeomorphism. To see this, note the following:
\begin{itemize}
    \item $\tilde{\tilde{\pi}}$ and $\tilde{\tilde{\rho}}$ are topological quotient maps (with the typologies on the right hand sides of the diagram induced by the $\tau$-topologies via the right vertical isomorphisms)
    \item $\bar r$ is a topological isomorphism.
\end{itemize}
The fact that  $\tilde{\tilde{\pi}}$ and $\tilde{\tilde{\rho}}$ are topological quotient maps follows from the fact that in their corresponding diagrams \hyperref[diagram pi]{(1)} and \hyperref[diagram rho]{(2)}, the upper horizontal maps are topological quotient maps by Fact \ref{epimorphism semigroup to group} and the left vertical maps are topological isomorphisms by Fact \ref{delta is homeomorphism}. The fact that $\bar r$ is a homeomorphism follows trivially by the definition of the ipp-topology and the fact that $\bar r$ is induced by an isomorphism of infinitary definability patterns structures.
\end{proof}

\section{Applications to tame, stable, NIP and WAP context}\label{section: nip stable tame wap}

In this section, we present several equivalence relations that are compatible, allowing us to use Theorem \ref{compatible implies absolute} to obtain absoluteness of their Ellis groups. For the notation used in this section, see Remark \ref{remark: pist and piNIP} and the comments following it. 

The following holds without any saturation assumptions on $\C$:
\begin{prop}
The equivalence relations $\tilde{E'}^{\textrm{st}}_\emptyset$ and $\tilde{E}^{\textrm{st}}_\emptyset$  are compatible.
\end{prop}

\begin{proof}
Let $\pi(x,y)$ be a partial type over $\emptyset$ (closed under conjunction) defining $E^{\textrm{st}}_\emptyset$. The same type defines $E'^{\textrm{st}}_\emptyset$. Recall also that $\tilde{E}^{\textrm{st}}_\emptyset$ is the equivalence relation on $S_X(\C)$ given by $$p \tilde{E}^{\textrm{st}}_\emptyset q \iff (\exists a \models p,  b \models q) (\pi(a,b)),$$ and $\tilde{E'}^{\textrm{st}}_\emptyset$ is the equivalence relation on $S_X(\C')$ given by $$p' \tilde{E'}^{\textrm{st}}_\emptyset q' \iff (\exists a' \models p',  b' \models q') (\pi(a',b')).$$

The goal is to prove that $r[\tilde{E'}^{\textrm{st}}_\emptyset] = \tilde{E}^{\textrm{st}}_\emptyset$.

$(\subseteq)$ Consider any $p',q' \in S_X(\C')$ with $p' \tilde{E'}^{\textrm{st}}_\emptyset  q'$. Then there are $a' \models p'$ and $b' \models q'$ such that $\pi(a',b')$. Hence, $a' \models r(p')$ and $b' \models r(b')$, and we get $r(p') \tilde{E}^{\textrm{st}}_\emptyset r(q')$.

$(\supseteq)$ Consider any $p,q \in S_X(\C)$ with $p \tilde{E}^{\textrm{st}}_\emptyset q$. The goal is to find some extensions $p',q' \in S_X(\C')$ of $p$ and $q$ respectively, satisfying $p' \tilde{E'}^{\textrm{st}}_\emptyset q'$. 

Take $a \models p$ and $b \models q$ such that $\pi(a,b)$. Let $\tp(a'b'/\C')$ be an heir extension of $\tp(ab/\C)$. We claim that $p':=\tp(a'/\C')$ and $q':=\tp(b'/\C)$ do the job. If not, then, by compactness, there are a formula $\varphi(x,y) \in \pi(x,y)$ and formulas $\psi_1(x,c') \in p'$ and $\psi_2(x,c') \in q'$ for which there are no $a''$ and $b''$ such that $\psi_1(a'',c') \wedge \psi_2(b'',c') \wedge \varphi(a'',b'')$ (here $\psi_i(x,x')$ is a formula without parameters and $c'$ is a tuple from $\C'$). Then 
$$\psi_1(x,c') \wedge \psi_2(y,c') \wedge \neg (\exists z)(\exists t)(\psi_1(z,c') \wedge \psi_2(t,c') \wedge \varphi(z,t)) \in \tp(a'b'/\C').$$
Since $\tp(a'b'/\C')$ is an heir extension of $\tp(ab/\C)$, there is $c \in \C$ such that 
$$\psi_1(x,c) \wedge \psi_2(y,c) \wedge \neg (\exists z)(\exists t)(\psi_1(z,c) \wedge \psi_2(t,c) \wedge \varphi(z,t)) \in \tp(ab/\C).$$
Taking $z:=a$ and $t:=b$, we get a contradiction with the fact that $\pi(a,b)$.
\end{proof}

From the previous proposition and Theorem \ref{compatible implies absolute}, we get the following immediate corollary (note the saturation assumption).

\begin{cor}
    The Ellis group of $S_X(\C)/\tilde{E}^{\textrm{st}}_\emptyset$ (treated as a topological group with the $\tau$-topology) does not depend on the choice of $\C$ as long as $\C$ is at least $\aleph_0$-saturated and strongly $\aleph_0$-homogeneous.
\end{cor}

Similarly, the following holds:

\begin{prop}
The equivalence relations $\tilde{E'}^{\textrm{NIP}}_\emptyset$ and $\tilde{E}^{\textrm{NIP}}_\emptyset$  are compatible.
\end{prop}

\begin{cor}
    The Ellis group of $S_X(\C)/\tilde{E}^{\textrm{NIP}}_\emptyset$ (treated as a topological group with the $\tau$-topology) does not depend on the choice of $\C$ as long as $\C$ is at least $\aleph_0$-saturated and strongly $\aleph_0$-homogeneous.
\end{cor}

Next, we will show that the same is true for the equivalence relations $\fwapp$ and $\fwap$ described in the preliminaries of this chapter. However, this time we will need $\C$ and $\C'$ to be at least $(\aleph_0+\lambda)^+$-saturated and strongly $(\aleph_0+\lambda)^+$-homogeneous (where lambda is such that $X\subseteq \C^\lambda$). In particular, if $\lambda$ is finite, then $\aleph_1$-saturated and strongly $\aleph_1$-homogeneous is enough. 
\begin{remark}
    For the proof of the next theorem, without loss of generality we will assume that $\C$ is $\C'$-small. This is because if the theorem holds under those assumptions, we can take a monster model $\C'' \succ \C'$ in which both $\C$ and $\C'$ are small, and apply the result to the pairs $\C'' \succ \C'$ and $\C'' \succ \C$. Namely, for $r_1\colon S_X(\C') \to S_X(\C)$, $r_2 \colon S_X(\C'') \to S_X(\C')$, and $r_3 \colon S_X(\C'') \to S_X(\C)$ being the restriction maps, the theorem yields $r_2[F_{\textrm{WAP}}'']=\fwapp$ and $r_3[F_{\textrm{WAP}}'']=\fwap$. As $r_1[r_2[F_{\textrm{WAP}}'']]=r_3[F_{\textrm{WAP}}'']$, we conclude that $r_1[\fwapp] =\fwap$.
\end{remark}

\begin{teor}\label{fwapp and fwap}
The equivalence relations $\fwapp$ and $\fwap$ are compatible as long as $\C$ and $\C'$ are at least $(\aleph_0+\lambda)^+$-saturated and strongly $(\aleph_0+\lambda)^+$-homogeneous.
\end{teor}

\begin{proof}
    We first prove that $\fwap\subseteq r[\fwapp]$. It suffices to show that $ r[\fwapp]$ is a closed $\aut(\C)$-invariant equivalence relation on $S_X(\C)$ with WAP quotient.
    \begin{itemize}
        \item Closedness is clear.
        \item Equivalence relation: Take $a=r(\alpha)$, $b=r(\beta)=r(\beta')$ and $c=r(\gamma)$ where $\alpha \fwapp \beta$ and $\beta' \fwapp \gamma$. Let $(p_i)_{i<\mu}$ be an enumeration of $S_X(\C)$ and $(a_i)_{i<\mu}$ be a sequence of realizations. Consider the type $p:=\tp((a_i)_{i<\mu}/\C)$ and let $\tp ((a'_i)_{i<\mu}/\C')=:p'$ be a strong heir extension of $p$ (see Definition \ref{strong heirs}). For each $i<\mu$ we denote $\tp(a_i'/\C')$ by $p_i'$. Obviously, $a=p_{i_1}$, $b=p_{i_2}$ and $c=p_{i_3}$ for some $i_1,i_2,i_3<\mu$.
        \begin{claim}
            $p'_{i_1}\fwapp p'_{i_2} \fwapp p'_{i_3}$.
        \end{claim}
        \begin{proof}[Proof of claim]
            Clearly, $r(p'_i)=p_i$ for all $i<\mu$. Since $p'$ is a strong heir of $p$ we have 
            \begin{align*}
                c(p'_{i_1},p'_{i_2})&=c(p_{i_1},p_{i_2})\subseteq c(\alpha,\beta)\\
                c(p'_{i_2},p'_{i_3})&=c(p_{i_2},p_{i_3})\subseteq c(\beta',\gamma).
            \end{align*}
            Thus, by Fact \ref{content and ellis semigroup}, \begin{align*}
                \exists \eta_1 \in E(S_X(\C'))&[ \eta_1 (\alpha, \beta)=(p'_{i_1},p'_{i_2})]\\
                \exists \eta_2 \in E(S_X(\C'))&[ \eta_2 (\beta', \gamma)=(p'_{i_2},p'_{i_3})].
            \end{align*}
            Since the relation $\fwapp$ is $\aut(\C')$-invariant, closed, and $ \alpha \fwapp \beta \wedge \beta'\fwapp \gamma $, we conclude that $p'_{i_1}\fwapp p'_{i_2} \fwapp p'_{i_3}$.
        \end{proof}
        By the claim, $p'_{1_1}\fwapp p'_{i_3}$, so $a=r(p'_{i_1}) r[\fwapp] r(p'_{i_3})=c$.
        \item $\aut(\C)$-invariant: Take an arbitrary $\sigma\in \aut(\C)$ and extend it to $\sigma'\in \aut(\C')$. Consider any $a,b\in S_X(\C)$ such that $a r[\fwapp] b$ and let $\alpha$, $\beta \in S_X(\C')$ be such that $r(\alpha)=a$, $r(\beta)=b$ and $\alpha \fwapp \beta$. Then, $$ \sigma(a)=\sigma(r(\alpha))=r(\sigma'(\alpha))r[\fwapp] r(\sigma'(\beta))=\sigma(r(\beta))=\sigma(b). $$
        \item $S_X(\C)/r[\fwapp]$ is WAP: Assume for a contradiction that there is a function $f \in C ( S_X(\C)/r[\fwapp] )$ which is not WAP. That is, there is a net $(\sigma_i)_{i\in \I}\subseteq \aut( \C)$ such that the functions $\sigma_if$ converge pointwise to some function $g\notin C ( S_X(\C)/r[\fwapp] )$. Note that $r^{-1}[r[\fwapp]]\supseteq \fwapp$ is a closed $\aut(\C'/\{ \C \})$-invariant equivalence relation on $S_X(\C')$ and $r$ induces a homeomorphism $$\Tilde{r}: \bslant{S_X(\C')}{r^{-1}[r[\fwapp]]}\to \bslant{S_X(\C)}{r[\fwapp]} $$ satisfying $$ \Tilde{r}\left(\sigma'\left(\bslant{p}{r^{-1}[r[\fwapp]]}\right) \right)=\sigma'\!\!\upharpoonright_\C\left(\bslant{r(p)}{r[\fwapp]}\right) $$ for all $\sigma'\in\aut(\C'/\{ \C\})$. In particular, we have a homomorphism
        $$ \left( \aut(\C'/\{ \C \}), \bslant{ S_X(\C')} {r^{-1}[r[\fwapp]]} \right) \to \left(\aut(\C), \bslant {S_X(\C)}{r[\fwapp]} \right) $$ given by $$ \sigma'\left( \bslant{p}{r^{-1}[r[\fwapp]]} \right)=\sigma'\!\!\upharpoonright_\C\left( \bslant{r(p)}{r[\fwapp]} \right). $$

        For every $i\in \I$ choose an extension $\sigma'_i\in \aut(\C'/\{ \C \})$ of $\sigma_i$. The above homeomorphism together with $f$ induce a function $f'\in C ( S( \C')/ r^{-1}[r[\fwapp]] )$ given by $$f'\left( \bslant{p}{r^{-1}[r[\fwapp]]} \right)=f\left( \bslant{r(p)}{r[\fwapp]} \right). $$ By construction, the net $(\sigma'_if')_{i\in I}$ converges pointwise to a function $$g'\notin C ( S( \C')/ r^{-1}[r[\fwapp]] ).$$ Hence, the flow $$ \left( \aut(\C'/\{ \C \}), \bslant{ S_X(\C')} {r^{-1}[r[\fwapp]]} \right)$$  is not WAP, which implies that $(\aut(\C'), S_X(\C')/\fwapp) $ is not WAP (because WAP is preserved under decreasing the acting group and under taking quotients of flows), a contradiction.
    \end{itemize}

    Now, we prove $\fwap\supseteq r[\fwapp]$. This is equivalent to $r^{-1}[\fwap]\supseteq \fwapp$. Note that $r^{-1}[\fwap]$ is a closed equivalence relation on $S_X(\C')$ but it might not be $\aut(\C')$-invariant. To solve it, we consider the equivalence relation $$ F:= \bigcap_{\sigma \in \aut(\C')} \sigma(r^{-1}[\fwap] ). $$ Then, it is enough to show that $F\supseteq \fwapp$ which is equivalent to the flow $$(\aut(\C'),  S_X(\C')/F)$$ being WAP.

    Assume for a contradiction that there is $f\in C(S_X(\C')/F)$ which is not WAP. Let $\overline{f}:S_X(\C')\to \R$ be given by $\overline{f}=f\circ \pi_F$, where $\pi_F:S_X(\C')\to S_X(\C')/F$ is the quotient map. Then, $\overline{f}\in C( S_X(\C'))$ and it is not a WAP function. By Fact \ref{double limit}, there are $(\sigma_n)_{n<\omega}\subseteq \aut(\C')$ and $(c'_m )_{m<\omega}\subset X(\C')$ such that \begin{equation}\label{dlp f}
        \lim_n\lim_m \sigma_n\overline{f} (\tp(c'_m/\C')) \neq \lim_m\lim_n \sigma_n\overline{f} (\tp(c'_m/\C'))
    \end{equation} where both limits exist. Note that here we are using that the types over $\C'$ of the elements of $X(\C')$ form a dense subset of $S_X(\C')$, which uses that $X$ lives on $\C'$-small (even $\C$-small) tuples.

    Consider the set $N:=\{ c'_m: m<\omega \}\cup \{ \sigma_n^{-1}(c'_m): m,n<\omega \}.$ By $(\aleph_0+\lambda)^+$-saturation of $\C$, we may assume that $N\subset \C$ (just find $\sigma\in\aut(\C)$ such that $\sigma[N]\subset \C$ and replace $N$ by $\sigma[N]$, $c'_m$ by $\sigma(c'_m)$ and $\sigma_n$ by $\sigma\circ \sigma_n$). Note that it might happen that none of the restrictions $\sigma_n\!\!\upharpoonright_\C$ is an automorphism of $\C$. However, for every $n<\omega$, by strong $(\aleph_0+\lambda)^+$-homogeneity of $\C$ (and strong $\lvert \C\rvert^+$-homogeneity of $\C')$, we can replace $\sigma_n$ by $\sigma'_n\in \aut(\C')$ so that $\sigma_n'^{-1}$ extends  $\sigma_n^{-1}\!\!\upharpoonright_{\{c'_m: m<\omega \}}$ and the restriction of $\sigma'_n$ to $\C$ is in $\aut(\C)$, so we may assume that, for every $n<\omega$, the restriction $\sigma_n\!\!\upharpoonright_\C$ is an element of $\aut(\C)$.

    Let $\mathcal{H}:= \{ p\in S_X(\C): p \text{ is invariant over } N \}$; we enumerate $\mathcal{H}= (p_i)_{i<\mu}$ and choose $a_i\models p_i$ for every $i<\mu$. Consider the type $p=\tp((a_i)_{i<\mu}/\C)$  and let $\tp ((a'_i)_{i<\mu}/\C')=:p'$ be a strong heir extension of $p$ in the language $\mL_N$ (it exists by $\aleph_0$-saturation of $\C$ in the language $\mL_N$ which follows from $\aleph_1$-saturation of $\C$ in the language $\mL$). We denote $\tp(a_i'/\C')$ by $p_i'$ for each $i<\mu$. Since $\tp(a_i/\C)$ is $N$-invariant, $p_i'$ is the unique $N$-invariant extension of $p_i$ to $\C'$. Hence, the set $\mathcal{H}':=\{p_i'\}_{i<\mu}$ is precisely the set of all types in $S_X(\C')$ invariant over $N$, so it is closed in $S_X(\C')$. 

    Now, we define $h:\mathcal{H} \to \R$ by $h(p_i):=\overline{f}(p_i')$. The function $h$ belongs to $C(\mathcal{H})$ since for each closed interval $I\subseteq \R$ we have that $h^{-1}[I]=r[\overline{f}^{-1}[I] \cap \mathcal{H}']$ is a closed subset of $S_X(\C)$. Note that for each $c'\in N$ and $i<\mu$, if $p_i=\tp(c'/\C)$ then $p_i'=\tp(c'/\C')$ and so \begin{equation}\label{6.2}
        \sigma_n(\overline{f}( \tp( c'_m/\C') ) = \overline{f}(\tp(\sigma_n^{-1}(c'_m)/\C'))=h(\tp(\sigma_n^{-1}(c'_m)/\C))= \sigma_n(h( \tp( c'_m/\C) ).
    \end{equation} Using \ref{dlp f} and the fact above we deduce:
    \begin{equation}\label{6.3}
        \lim_n\lim_m \sigma_n h (\tp(c'_m/\C)) \neq \lim_m\lim_n \sigma_n h (\tp(c'_m/\C))
    \end{equation} where both limits exist.

    \begin{claim}
        For $p_i,p_j \in\mathcal{H}$, if $p_i \fwap p_j$ then $h(p_i)=h(p_j)$.
    \end{claim}
    \begin{proof}[Proof of claim]
        We show that $p_i' F p_j'$. Choose an arbitrary $\sigma'\in\aut(\C')$. We have that $$ c(p_i,p_j)=c(p_i',p_j')=c(\sigma'(p_i'),\sigma'(p_j'))\supseteq c(r(\sigma'(p_i')),r(\sigma'(p_j'))) .$$ Hence, there is $\eta\in E(S_X(\C))$ such that $\eta(p_i,p_j)=(r(\sigma'(p_i')),r(\sigma'(p_j')))$, which implies that $ r(\sigma'(p_i')) \fwap r(\sigma'(p_j')) $ (because $\fwap$ is $\aut(\C)$-invariant and closed). We then have $\sigma'(p_i') r^{-1}[\fwap]\sigma'(p_j')$, and since $\sigma'$ was arbitrary, we conclude that $p_i' F p_j'$. Therefore, since $\overline{f}=f\circ\pi_F$, we obtain that $\overline{f}(p'_i)=\overline{f}(p'_j)$, so $h(p_i)=h(p_j)$.
    \end{proof}

    Clearly, $\mathcal{H}/\fwap$ is a closed subset of $S_X(\C)/\fwap$ and, by the claim, $h=g\circ \rho$ for some $g\in C(\mathcal{H}/\fwap )$ and $\rho: \mathcal{H} \to \mathcal{H}/\fwap$ the quotient map. Tietze's extension theorem yields a function $\overline{g}\in C(S_X(\C)/\fwap)$ extending $g$. By construction, this function satisfies \begin{equation}\label{6.4}
        \lim_n\lim_m \sigma_n\overline{g} \left(\bslant{\tp(c'_m/\C)}{\fwap}\right) \neq \lim_m\lim_n \sigma_n\overline{g} \left(\bslant{\tp(c'_m/\C)}{\fwap}\right)
    \end{equation} where both limits exist, which by Fact \ref{double limit} contradicts the fact that $S_X(\C)/\fwap$ is a WAP flow.
\end{proof}

From the previous theorem and Theorem \ref{compatible implies absolute}, the following corollary follows immediately.

\begin{cor}
    The Ellis group of $S_X(\C)/\fwap$ (treated as a topological group with the $\tau$-topology) does not depend on the choice of $\C$ as long as $\C$ is at least $(\aleph_0+\lambda)^+$-saturated and strongly $(\aleph_0+\lambda)^+$-homogeneous.
\end{cor}

Similarly, and under the same saturation assumptions, for the equivalence relations $\ftaa$ and $\fta$ we have the following:

\begin{teor}\label{ftaa and fta}
The equivalence relations $\ftaa$ and $\fta$ are compatible as long as $\C$ and $\C'$ are at least $(\aleph_0+\lambda)^+$-saturated and strongly $(\aleph_0+\lambda)^+$-homogeneous.
\end{teor}

\begin{proof}
    We first prove that $\fta\subseteq r[\ftaa]$. It suffices to show that $ r[\ftaa]$ is a closed $\aut(\C)$-invariant equivalence relation on $S_X(\C)$ with tame quotient. The fact that $r[\ftaa]$ is a closed $\aut(\C)$-invariant equivalence relation on $S_X(\C)$ follows by the same arguments as in the WAP context (see Theorem \ref{fwapp and fwap}).
    
    To show that $S_X(\C)/r[\ftaa]$ is tame, suppose for a contradiction that there is a function $f \in C ( S_X(\C)/r[\ftaa] )$ which is not tame. That is, there is a sequence $(\sigma_i)_{i<\omega}\subset\aut(\C)$ such that $(\sigma_i f)_{i<\omega}$ is an independent sequence. 
    Then we apply the corresponding part of the proof of Theorem \ref{fwapp and fwap}, replacing ``WAP" by ``tame" and noticing that by construction the sequence $(\sigma'_if')_{i<\omega}$ is independent, which leads to a contradiction.

    Now we prove $\fta\supseteq r[\ftaa]$. This is equivalent to $r^{-1}[\fta]\supseteq \ftaa$. We define a closed, $\aut(\C')$-invariant equivalence relation $$ F:= \bigcap_{\sigma \in \aut(\C')} \sigma(r^{-1}[\fta] ). $$ Then, it is enough to show that $F\supseteq \ftaa$ which is equivalent to the flow $$(\aut(\C'),  S_X(\C')/F)$$ being tame.

    Assume for a contradiction that there is $f\in C(S_X(\C')/F)$ which is not tame. Let $\overline{f}:S_X(\C')\to \R$ be given by $\overline{f}=f\circ \pi_F$. Then, $\overline{f}\in C( S_X(\C'))$ and it is not a tame function. That is, there exist $r<s\in \R$, $(\sigma_i)_{i<\omega}\subset\aut(\C')$ and $\{c'_{P,M}: P,M\subset_{fin} \omega \text{ disjoint} \}\subset X(\C')$ such that for any finite disjoint $P,M\subset \omega$ \begin{equation}\label{indep f}
        \models \{ \sigma_i\overline{f}(\tp(c'_{P,M}/\C'))<r: i\in P\} \cup \{\sigma_i\overline{f}(\tp(c'_{P,M}/\C'))>s: i\in M \}.
    \end{equation}
    The fact that we can choose $c'_{P,M}\in X(\C')$ follows from the fact that the types over $\C'$ of elements of $X(\C')$ form a dense subset of $S_X(\C')$ and the second part of Defnition \ref{defin: indep functons}.

    Consider the set $$N:=\{c'_{P,M}: P,M\subset_{fin} \omega \text{ disjoint} \}\cup \{\sigma_i^{-1}(c'_{P,M}): i<\omega, P,M\subset_{fin} \omega \text{ disjoint} \}.$$ By $(\aleph_0+\lambda)^+$-saturation of $\C$, applying an automorphism of $\C'$, we may assume that $N\subset \C$. Note that it might happen that none of the restrictions $\sigma_n\!\!\upharpoonright_\C$ is an automorphism of $\C$. However, for every $n<\omega$, by strong $(\aleph_0+\lambda)^+$-homogeneity of $\C$ (and strong $\lvert \C \rvert^{+}$-homogeneity of $\C'$), we can replace $\sigma_n$ by $\sigma'_n\in \aut(\C')$ so that $\sigma_n'^{-1}$ extends $\sigma_n^{-1}\!\!\upharpoonright_{\{c'_{P,M}: P,M\subset_{fin} \omega \text{ disjoint}\}}$ and the restriction of $\sigma'_n$ to $\C$ is in $\aut(\C)$, so we may assume that, for every $n<\omega$, $\sigma_n\!\!\upharpoonright_\C$ is an element of $\aut(\C)$.

    Let $\mathcal{H}:= \{ p\in S_X(\C): p \text{ is invariant over } N \}$, we enumerate $\mathcal{H}=( p_i)_{i<\mu}$ and choose $a_i\models p_i$ for every $i<\mu$.  Consider the type $p=\tp((a_i)_{i<\mu}/\C)$ and let $\tp ((a'_i)_{i<\mu}/\C')=:p'$ be a strong heir extension of $p$ in the language $\mL_N$, we denote $\tp(a_i'/\C')$ by $p_i'$ for each $i<\mu$. Since $\tp(a_i/\C)$ is $N$-invariant, $p_i'$ is the unique $N$-invariant extension of $p_i$ to $\C'$. Hence, the set $\mathcal{H}':=(p_i')_{i<\mu}$ is closed in $S_X(\C')$. 

    Then we apply the corresponding part of the proof of Theorem \ref{fwapp and fwap}, where the formulas \ref{6.2}, \ref{6.3} and \ref{6.4} are replaced by 
    $$ \sigma_n(\overline{f}( \tp( c'_{P,M}/\C') ) = \overline{f}(\tp(\sigma_n^{-1}(c'_{P,M})/\C'))=h(\tp(\sigma_n^{-1}(c'_{P,M})/\C))= \sigma_n(h( \tp( c'_{P,M}/\C) ),$$
    $$\models \{ \sigma_ih(\tp(c'_{P,M}/\C))<r: i\in P\} \cup \{\sigma_ih(\tp(c'_{P,M}/\C))>s: i\in M \}$$ and $$\models \{ \sigma_i\overline{g}(\tp(c'_{P,M}/\C))<r: i\in P\} \cup \{\sigma_i\overline{g}(\tp(c'_{P,M}/\C))>s: i\in M \},$$ respectively.
\end{proof}

Again, from the previous theorem and Theorem \ref{compatible implies absolute}, the following corollary follows immediately.

\begin{cor}
    The Ellis group of $S_X(\C)/\fta$ (treated as a topological group with the $\tau$-topology) does not depend on the choice of $\C$ as long as $\C$ is at least $(\aleph_0+\lambda)^+$-saturated and strongly $(\aleph_0+\lambda)^+$-homogeneous.
\end{cor}

\section{Stable vs WAP, and NIP vs tame}

In this last section, we will study the relation between the equivalence relations $\fta$ and $\fwap$ and the finest $\emptyset$-type-definable equivalence relations on $X$ with NIP and stable quotients, respectively. Assume that $\C$ is $(\aleph_0+\lambda)$-saturated and strongly $(\aleph_0+\lambda)$-homogeneous (where $X\subseteq \C^\lambda)$.

Let $E$ be a $\emptyset$-type-definable equivalence relation on $X$. By \ref{2.4} we know that the quotient $X/E$ is stable if and only if every $f\in \mathcal{F}_{X/E}$ is stable. The connection between stable theories and WAP flows is well known (see \cite{ben2016weakly}). This connection is still true for the hyperdefinable set $X/E$.

\begin{prop}\label{wap iff stable formula}
    Let $f(x,y)\in \mathcal{F}_{X/E}$, and let $b\in \C^{\lvert y \rvert}$. We denote by $f_b$ the function $f_b:S_{X/E}(\C)\to \R$ given by $f_b(p)=f(a,b)$ for any $a/E\models p$. Then the following are equivalent:
    \begin{enumerate}
        \item $f(x,y)$ is stable.
        \item For all $b\in \C^{\lvert y \rvert}$ the function $f_b$ is WAP.
    \end{enumerate}
\end{prop}
\begin{proof}
    $(1)\implies (2)$ If the function $f_b$ is not WAP, by Fact \ref{double limit}, there is a sequence $(a_n)_{n<\omega}\subset X$ and a sequence of automorphisms $(\sigma_m)_{m<\omega}\subset \aut(\C)$ such that $$\lim_m\lim_n f(a_n,\sigma_m(b))\neq \lim_n\lim_m f(a_n,\sigma_m(b)) .$$ Note that we are using that the realized types are dense in $S_X(\C)$.
    Assume, without loss of generality that for some $r<s\in \R$ we have $$\lim_m\lim_n f(a_n,\sigma_m(b))>s \text{ and } \lim_n\lim_m f(a_n,\sigma_m(b))<r.$$
    Let us denote $\sigma_m(b)$ by $b_m$. It is clear that we can choose a subsequence $(a_i',b_i')_{i<\omega}$ from $(a_n,b_n)_{n<\omega}$ such that $f(a'_i,b'_j)>s$ whenever $i>j$ and $f(a'_i,b'_j)<r$ whenever $i<j$. That is, the sequence $(a_i',b_i')_{i<\omega}$ witnesses unstability of the formula $f(x,y)$. 

    $(2)\implies (1)$ If $f(x,y)$ is unstable, we can find an indiscernible sequence $(a_i,b_i)_{i<\omega}$ with $a_i\in X$ and $b_i\in \C^{\lvert y \rvert}$ such that $f(a_i,b_j)\neq f(a_j,b_i)$ for some (all) $i<j$, By indiscernibility, for each $i<\omega$ there is $\sigma_i \in \aut(\C)$ such that $\sigma_i(b_0)=b_i$ and there exists $r<s\in \R$ such that $$f(a_i,b_j)=r<s=f(a_j,b_i)$$ for all $i<j$. Hence $f_{b_0}$ is not a WAP function.
\end{proof}

\begin{cor}\label{cor: stable iff wap}
    The flow $(\aut(\C), S_{X/E}(\C))$ is WAP if and only if $X/E$ is stable.
\end{cor}
\begin{proof}
    $(\Longrightarrow)$ By assumption, for any $f(x,y)\in\mathcal{F}_{X/E}$ and $b\in\C^{\lvert y\rvert}$, $f_b$ is WAP, so $f(x,y)$ is stable by Proposition \ref{wap iff stable formula}. Hence, $X/E$ is stable by Corollary \ref{2.4}.
    
    $(\Longleftarrow)$ By assumption and Corollary \ref{2.4}, every function $f(x,y)\in\mathcal{F}_{X/E}$ is stable, so for any $b\in \C^{\lvert y \rvert}$ the function $f_b$ is WAP by Proposition \ref{wap iff stable formula}. Thus, since by Proposition \ref{proposition: F_X/E separates points} the family of functions $\{ f_b: f\in \mathcal{F}_{X/E},b\in \C^{\lvert y \rvert}  \}$ separates points in $S_{X/E}(\C)$, we conclude that $(\aut(\C), S_{X/E}(\C))$ is WAP by Fact \ref{separating points: WAP}.
\end{proof}


Similarly, by Lemma \ref{IP_n hyperdef. set iif formulas} we know that the quotient $X/E$ has NIP if and only if every $f\in \mathcal{F}_{X/E}$ has NIP (see Definition \ref{defin: n-dep function/E} for $n=1$). The connection between NIP theories and tame flows is well known, it was first noticed independently in \cite{definably_Amenable_nip}, \cite{ibarlucia2016dynamical} and \cite{Khanaki2020-KHASTN} and further developed in \cite{Krupinski2018GaloisGA}. This connection is still true for the hyperdefinable set $X/E$.

\begin{prop}\label{tame iff nip formula}
    Let $f(x,y)\in \mathcal{F}_{X/E}$, and let $b\in \C^{\lvert y \rvert}$. We denote by $f_b$ the function $f_b:S_{X/E}(\C)\to \R$ given by $f_b(p)=f(a,b)$ for any $a/E\models p$. Then the following are equivalent:
    \begin{enumerate}
        \item $f(x,y)$ has NIP.
        \item For all $b\in \C^{\lvert y \rvert}$ the function $f_b$ is tame.
    \end{enumerate}
\end{prop}
\begin{proof}
    $(1)\implies (2)$. If $f_b$ is not tame for some $b\in \C^{\lvert y \rvert}$, then there is a sequence $(\sigma_i)_{i<\omega}\subset \aut(\C)$ such that the sequence of functions $(f(x,\sigma_i(b)))_{i<\omega}$ is independent on $X$, so $f$ has IP by compactness. 

    $(2)\implies (1)$. If $f(x,y)$ has $IP$, then we can find an element $a\in X$, an indiscernible sequence $(b_i)_{i<\omega}\subset \C^{\lvert y \rvert}$ and $r<s\in \mathbb{R}$ such that $f(a,b_i)<r$ if and only if $i$ is even and $f(a,b_i)>s$ if and only if $i$ is odd (see Lemma \ref{even odd nip}). By indiscernibility, for each $i<\omega$ there is $\sigma_i \in \aut(\C)$ such that $\sigma_i(b_0)=b_i$. Hence, $f_{b_0}$ is not a tame function because the sequence $(f(x,b_i))_{i<\omega}$ is independent.
\end{proof}

\begin{cor}\label{cor: nip iff tame}
    The flow $(\aut(\C), S_{X/E}(\C))$ is tame if and only if $X/E$ is NIP.
\end{cor}
\begin{proof}
    $(\Longrightarrow)$ By assumption, for any $f(x,y)\in\mathcal{F}_{X/E}$ and $b\in\C^{\lvert y\rvert}$, $f_b$ is tame, so $f(x,y)$ has NIP by Proposition \ref{tame iff nip formula}. Hence, $X/E$ has NIP by Lemma \ref{IP_n hyperdef. set iif formulas}.
    
    $(\Longleftarrow)$ By assumption and Lemma \ref{IP_n hyperdef. set iif formulas}, every function $f(x,y)\in\mathcal{F}_{X/E}$ has NIP, so for any $b\in \C^{\lvert y \rvert}$ the function $f_b$ is tame by Proposition \ref{tame iff nip formula}. Thus, since by Proposition \ref{proposition: F_X/E separates points} the family of functions $\{ f_b: f\in \mathcal{F}_{X/E},b\in \C^{\lvert y \rvert}  \}$ separates points in $S_{X/E}(\C)$, we conclude that $(\aut(\C), S_{X/E}(\C))$ is tame by Fact \ref{separating points: tame}.
\end{proof}

Below we will use the notation introduced in Remark \ref{remark: pist and piNIP} and the comments following it. Note that, while a $\emptyset$-type-definable equivalence relation on $X$ induces a closed $\aut(\C)$-invariant equivalence relation on $S_X(\C)$, the converse is not true. Namely, not every closed $\aut(\C)$-invariant equivalence relation on $S_X(\C)$ is created this way. In the case of $\Tilde{E}^{\textrm{WAP}}_{\emptyset}$ and $\Tilde{E}^{\textrm{NIP}}_{\emptyset}$, by Corollaries \ref{cor: stable iff wap} and \ref{cor: nip iff tame}, we get that $S_X(\C)/\Tilde{E}^{\textrm{st}}_{\emptyset}$ is a WAP flow and $S_X(\C)/\Tilde{E}^{\textrm{NIP}}_{\emptyset}$ is a tame flow. However, the next proposition shows that $\Tilde{E}^{\textrm{st}}_{\emptyset}$ is may not be equal to $\fwap$ and $\Tilde{E}^{\textrm{NIP}}_{\emptyset}$  may not be equal to $\fta$.

\begin{prop}\label{fwap strictly finer}
    Assume that $X$ is an $\emptyset$-type-definable subset of $\C^\lambda$, and $\C$ is $(\aleph_0+\lambda)$-saturated and strongly $(\aleph_0+\lambda)$-homogeneous.
    \begin{enumerate}
        \item If $X$ is unstable, then $\fwap\subsetneq \Tilde{E}^{\textrm{st}}_{\emptyset}$.
        \item If $X$ has IP, then $\fta\subsetneq \Tilde{E}^{\textrm{NIP}}_{\emptyset}$.
    \end{enumerate}
\end{prop}
\begin{proof}
    The inclusions follow from the above observations that $S_X(\C)/\Tilde{E}^{\textrm{st}}_{\emptyset}$ is WAP and $S_X(\C)/\Tilde{E}^{\textrm{NIP}}_{\emptyset}$ is tame. It remains to show that $\fwap \neq \Tilde{E}^{\textrm{st}}_{\emptyset} $ and $\fta\neq \Tilde{E}^{\textrm{NIP}}_{\emptyset}$. We will prove the first thing; the proof of the second one is analogous.

Since $X$ is unstable, $E^{\textrm{st}}_\emptyset \ne  \;\,=$, so  $\tilde{E}^{\textrm{st}}_\emptyset$ glues some types in $S_X(\C)$ which are realized in $\C$.

On the other hand, define a closed equivalence relation $E$ on $S_X(\C)$ by $$pEq \iff p_= = q_=,$$ where $p_=$ and $q_=$ denote the restrictions of $p$ and $q$ to the empty language (so we allow only the equality relation). Let $S^=_{X}(\C)$ be the collection of all global types in the empty language of the elements from $X(\C')$, where $\C' \succ \C$ is a monster model of the original theory in which $\C$ is small. Then $S^=_{X}(\C) = \{\tp(a/\C)_=: a \in X\} \cup \{\textrm{the unique non-realized type}\}$ is a closed subset of $S^=_{\C}(\C)$ invariant under $\aut(\C)$. We also see that $S_X(\C)/E \cong S^=_{X}(\C)$ as $\aut(\C)$-flows.
As the theory of $\C$ in the empty language is stable, by Corollary \ref{cor: stable iff wap}, we get that $(\Sym(\C), S^=_{\C}(\C))$ is WAP. Hence, since WAP  is closed under decreasing the acting group and under taking subflows, $(\aut(\C), S^=_{X}(\C))$ is also WAP, and so is $(\aut(\C),S_X(\C)/E)$. Therefore, $\fwap \subseteq E$. Thus, since $E$ does not glue any realized types in $S_X(\C)$, neither does $\fwap$.

By the conclusions of the last two paragraphs, we conclude that $\fwap \ne \tilde{E}^{\textrm{st}}_\emptyset$.
\end{proof}

Although $\fwap$ and $\fta$ are almost always strictly finer than $\Tilde{E}^{\textrm{st}}_\emptyset$ and $\Tilde{E}^{\textrm{NIP}}_\emptyset$, the following question and its analog for the tame case remain open:

\begin{question}
    Are the Ellis groups of the flows $$(\aut(\C), S_X(\C)/\fwap)$$ and $$(\aut(\C), S_X(\C)/\Tilde{E}^{\textrm{st}}_\emptyset)$$ isomorphic?
\end{question}

Proposition \ref{fwap strictly finer} justifies our interest in $\fwap$ and $\fta$, because it suggests that the quotients by these equivalence relations should capture more information about the theory in question than the quotients by $\Tilde{E}^{\textrm{st}}_\emptyset$ and $\Tilde{E}^{\textrm{NIP}}_\emptyset$ while maintaining similar good properties (having in mind that WAP is a dynamical version of stability and tameness a dynamical version of NIP.

\appendix
\chapter{Products of stable and NIP hyperdefinable sets}
We prove that the properties of stability and NIP for hyperdefinable sets are preserved under (possibly infinite) Cartesian products and taking type-definable subsets. 

\begin{remark} \label{finest eq exist for fixed param}
    The above is enough to guarantee that for a fixed set of parameters $A$, an arbitrary intersection of $A$-type-definable equivalence relations $(E_i)_{i<\mu}$ with stable (respectively NIP) quotient on a type-definable set $X$ is an equivalence relation with stable (respectively NIP) quotient space. Moreover, the finest $A$-type-definable equivalence relation on $X$ with stable (respectively NIP) quotient always exists.
\end{remark}
\begin{proof}[Proof of the Remark]
     The hyperdefinable set $$\bslant{X}{\bigcap_{i<\mu}E_i}$$ can be naturally identified with a type-definable subset of $$\prod_{i<\mu}X/E_i.$$
     
     For the moreover part, consider the $A$-type-definable equivalence relation on $X$ defined as the intersection of all $A$-type-definable equivalence relations on $X$ with stable quotient.
\end{proof}

Clearly, taking type-definable subsets preserves both stability and NIP. Moreover, it is enough to have preservation of stability (respectively NIP) under products of two hyperdefinable sets since this case easily implies the general case.

For the remainder of the appendix, we consider two hyperdefinable sets $X/E$ and $Y/F$ where $X,Y\subset \C^\lambda$.

First, we prove it for stability. This was first stated in \cite[Remark 1.4]{MR3796277}, the proof we present comes from Anand Pillay.

\begin{prop}\label{stable preserved product}
    Let $X/E$ and $Y/F$ by stable hyperdefinable sets. Then, $X/E\times Y/F$ is a stable hyperdefinable set.
\end{prop}
\begin{proof}
    Suppose the conclusion does not hold. Let $(a_i,b_i,c_i)_{i<\omega}$ be an indiscernible sequence witnessing unstability of $X/E\times Y/F$. That is, for all $i,j<\omega$ $a_i\in X/E$, $b_i\in Y/F$, and $\tp(a_i,b_i,c_j)\neq \tp(a_j,b_j,c_i)$ for all $i\neq j$. Without loss of generality we may extend $\omega$ to a sufficiently big (to be able to extract indiscernibles) dense linear order $\I$ without endpoints.
    \begin{claim}
        $(c_j)_{j\neq 0}$ is indiscernible over $a_0$.
    \end{claim}
    \begin{proof}
        Note that it is enough to show that for any $i_1<\dots< i_n<0<i'<j_1< \dots < j_m$ $$c_{i_1},\dots c_{i_n},c_{j_1},\dots,c_{j_m} \equiv_{a_0}c_{i_1},\dots c_{i_{n-1}},c_{i'},c_{j_1},\dots,c_{j_m}.$$ Note also that the sequence $(a_i,c_i)_{i\in (i_{n-1},j_1)}$ is indiscernible over $$K:=\{c_{i_k}: k=1,\dots,n-1 \}\cup \{c_{j_k}: k=1,\dots,m \}.$$
        Suppose that the conclusion does not hold, this implies that for some setting as above, $$(c_{i_n},a_0)\not\equiv_K (c_{i'},a_0).$$ However, by indiscernibility of the original sequence, $(c_{i_n},a_0)\equiv_K (c_{0},a_{i'})$. Therefore, the sequence $(a_i,c_iK)_{i\in (i_{n-1},j_1)}$ contradicts the stability of $X/E$.
    \end{proof}
    \begin{claim}
        $\tp(a_0,b_0,c_j)$ is constant for $j>0$, $\tp(a_0,b_0,c_j)$ is constant for $j<0$, and $\tp(a_0,b_0,c_1)\neq \tp(a_0,b_0,c_{-1})$.
    \end{claim}
    \begin{proof}
        The fact that it is constant follows from the indiscernibility of the original sequence $(a_i,b_i,c_i)_{i\in \I}$. Moreover, we have $$\tp(a_0,b_0,c_{-1})\neq \tp(a_{-1},b_{-1},c_0)=\tp(a_0,b_0,c_{1}).$$ 
    \end{proof}

    From the claims it follows that for each $k\in\I$ and $i_1<\dots< i_n<k<j_1< \dots < j_n$ all distinct from $0$ there exists $b'_k\in Y/F$ such that \begin{align*}
    b'_kc_{i_1}\equiv_{a_0}\cdots\equiv_{a_0} b'_kc_{i_n}\equiv_{a_0}b_0c_{-1}\\
    b'_kc_{j_1}\equiv_{a_0}\cdots\equiv_{a_0} b'_kc_{j_n}\equiv_{a_0}b_0c_{1}.
\end{align*}
Thus, by compactness and extracting indiscernibles (see Fact \ref{extracting indiscernibles}), there is a sequence $(b''_i,c'_i)_{i<\omega}$ which is indiscernible over $a_0$ and $$\tp(b''_i,c'_j/a_0)\neq tp(b''_j,c'_i/a_0)$$ for all $i<j$, contradicting stability of $Y/F$.
\end{proof}

The next characterization of functions of the family $\mathcal{F}_{X/E}$ with NIP easily follows from Proposition \ref{Hyperim: IP_n iff encoding partite} and compactness.

\begin{lema}\label{even odd nip}
    For any $f(x,y)\in  \mathcal{F}_{X/E}$ the following are equivalent:
    \begin{enumerate}
        \item $f$ has IP.
        \item There exists an indiscernible sequence $(b_i)_{i<\omega}$, $a\in X$ and $r<s\in \R$ such that $$ f(a,b_i)<r\iff i \text{ is even} $$ $$f(a,b_i)>s\iff i \text{ is odd.}$$
    \end{enumerate}
\end{lema}

\begin{prop}\label{NIP formulas have limit}
    For every $f(x,y)\in\mathcal{F}_{X/E}$ and for every (infinite) linear order  $\I$ without maximal element, $f(x,y)$ has NIP if and only if for every indiscernible sequence $(b_i)_{i\in\I}$ and $a\in X$ there is $L\in\Image(f)\subseteq [r_1,r_2]$ (for some $r_1,r_2\in \R$) such that for every $\varepsilon>0$ there is $\I_0\subset \I$ an end segment satisfying $$ \lvert f(a,b_i)-L\rvert \leq \varepsilon$$ for all $i\in\I_0$ (i.e., $(f(a,b_i))_{i\in\I_0}$ converges to $L$).
\end{prop}
\begin{proof}
    
$(\Leftarrow)$ : Suppose $f(x,y)$ has IP. Let $(b_{i})_{i<\omega}$, $a\in X$, and $r< s \in[r_1,r_2]$ be such that \begin{align*}
    f(a,b_i)<r\iff i \text{ is even},\\
    f(a,b_i)>s\iff i \text{ is odd},
\end{align*}
(which exists by Lemma \ref{even odd nip}).

By compactness, we may extend the indiscernible sequence $(b_{i})_{i<\omega}$ to a new indiscernible sequence $(b'_{i})_{i\in \I}$ such that for any $i\in I$ if $f(a,b'_i)<r$ then $f(a,b'_{i+1})>s$ and vice versa.
By assumption, there is some $L$ to which $(f(a, b'_i))_{i\in \I}$ converges. 
Let $0<\epsilon<\frac{s-r}{2}$. So there is an end segment $\I_0\subseteq \I$ such that for all $i\in \I_0$, $\left|f\left(a, b'_i\right)-L\right| \leq \epsilon$. Then, $$\left|f\left(a, b'_{i+1}\right)-L\right| \geq\left|f\left(a, b'_{i+1}\right)-f\left(a, b'_i\right)\right|-\left|f\left(a, b'_i\right)-L\right| \geq(s-r)-\epsilon \geq \frac{s-r}{2}>\epsilon.$$ Which is a contradiction.

$(\Rightarrow)$ : Let $(b_i: i \in\I)$ be an indiscernible sequence and $a\in X$, and suppose the conclusion does not hold for $(b_i: i \in\I)$ and $a$. That is, for every $L$ there is some $\epsilon>0$ such that for every  end segment $\I_{0} \subseteq \I$, there is $i \in \I_{0}$ such that $\left|f\left(a, b_i\right)-L\right|>\epsilon$.

Since $\left\{f\left(a, b_i\right) \mid i\in \I \right\} \subset[r_1,r_2]$ and is infinite, it must have some accumulation point $L_{0}$. That is, for any $\epsilon >0$, for cofinally many $i\in \I$, we have $\left|f\left(a, b_i\right)-L_{0}\right| \leq \epsilon$.

Since $\left(f\left(a, b_i\right)\right)_{i\in \I}$ does not converge, there is $\epsilon>0$ such that for every end segment $\I_0\subseteq \I$, there is $j\in \I_0$ such that $\left|f\left(a, b_j\right)-L_{0}\right|>\epsilon$ and since \(L_{0}\) is an accumulation point, there are cofinally many $i\in \I_0$ for which we have \(\mid f\left(a, b_i\right)-\) $L_{0} \left\lvert\, \leq \frac{\epsilon}{2}\right.$

Note that there must be either cofinally many $j\in \I$ such that $f\left(a, b_j\right)>L_{0}+\epsilon$ or cofinally many such that $f\left(a, b_j\right)<L_{0}-\epsilon$. We prove the result for the former case, the latter is analogous. Let $r=L_{0}+\frac{\epsilon}{2}$ and $s=L_{0}+\epsilon$. 

We now construct an indiscernible sequence $(c_i)_{i<\omega}$ which, together with $a$, will witness that $f(x,y)$ has IP.
Let $c_{0}=b_{i}$ for some $b_{i}$ such that $f(a,b_i) \leq r$. This is possible since there are cofinally many \(b_{i}\) within \(\frac{\epsilon}{2}\) of \(L_{0}\). Let \(c_{1}=a_{j}\) with $j>i$ be such that $f(a,c_{1}) \geq s$. Similarly, this is possible since there are cofinally many \(j\in \I\) with \(b_{j}\) such that $ f\left(a, b_j\right)-L_{0}>\epsilon$. Iterating this process infinitely many times, we get a subsequence $(c_i)_{i<\omega}$ of $(b_i)_{i\in \I}$ which is indiscernible, 
$f\left(a,c_{i}\right) \leq r$ if and only if \(i\) is even, and $f\left(a,c_{i}\right) \geq s$ if and only if \(i\) is odd. Thus, this sequence is as required (by Lemma \ref{even odd nip}).
\end{proof}

Using the previous results for functions of the family $\mathcal{F}_{X/E}$, we prove the following:

\begin{prop}\label{type stabilizes NIP}
    Let $X/E$ be a hyperdefinable set with $X\subseteq \C^\lambda$. If $X/E$ has NIP, for any indiscernible sequence $(b_i)_{i<(\lvert T\rvert +\lambda)^+}$ (of tuples from $\C$ of length at most $\lambda$) and any $a/E\in X/E$ there exists $\alpha < (\lvert T\rvert +\lambda)^+$ such that $(b_i)_{\alpha<i <(\lvert T\rvert +\lambda)^+}$ is indiscernible over $a/E$.
\end{prop}
\begin{proof}
    Assume the conclusion does not hold. Then, by Proposition \ref{proposition: F_X/E separates points}, for every $\alpha <(\lvert T\rvert +\lambda)^+$ we can find a function $f_\alpha(x,y_1,\dots,y_{k(\alpha)})\in \mathcal{F}_{X/E\times\C^{k(\alpha)\lambda}}$, two tuples of indices indices $\alpha<i_1<\dots<i_{k(\alpha)}<(\lvert T\rvert +\lambda)^+$ and $\alpha<j_1<\dots<j_{k(\alpha)}<(\lvert T\rvert +\lambda)^+$, and $r_\alpha<s_\alpha\in \mathbb{Q}$ satisfying $$\models f_\alpha(a,b_{i_1},\dots,b_{i_{k(\alpha)}})\leq r_\alpha \wedge f_\alpha(a,b_{j_1},\dots,b_{j_{k(\alpha)}})\geq s_\alpha.$$
    Moreover, the functions $f_\alpha$ can be chosen from a dense subset of $\mathcal{F}_{X/E }$ of cardinality $\lvert T \rvert +\lambda$ (see the proof of Corollary \ref{corollary: 2.4} for a detailed justification). Therefore, there is some function $f(x,y_1,\dots,y_k)$ and $r<s\in \mathbb{Q}$ such that $f_\alpha=f$, $r_\alpha=r$ and $s_\alpha=s$ for cofinally many values of $\alpha$. Then, we can construct inductively a sequence $\II=(i_1^l,\dots,i_k^l)_{l<\omega}$ such that $i_1^l<\dots<i_k^l<i_1^{l+1}$ for all $l<\omega$ and \begin{itemize}
        \item $\models f(a,b_{i^l_1},\dots,b_{i^l_{k}})<r$ if and only if $l$ is even,
        \item $\models f(a,b_{i^l_1},\dots,b_{i^l_{k}})>s$ if and only if $l$ is odd.
    \end{itemize}
    As the sequence $(b_{i_1^l},\dots, b_{i_k^l})_{l<\omega}$ is indiscernible, this implies that the function $f(x,y_1,\dots,y_k)$ has IP. By Lemma \ref{IP_n hyperdef. set iif formulas}, this is a contradiction with the assumption that $X/E$ has NIP.
\end{proof}
\begin{cor}\label{nip preserved product}
    Let $X/E$ and $Y/F$ be hyperdefinable sets with $NIP$. Then $X/E\times Y/F$ has NIP.
\end{cor}
\begin{proof}
    Let $(c_i)_{i<(\lvert T\rvert +\lambda)^+}$ be an arbitrary indiscernible sequence (of tuples from $\C$) and $(a,/E,b/F)$ an arbitrary pair from $X/E\times Y/F$. By applying Proposition \ref{type stabilizes NIP} twice, we get that there is $\alpha < (\lvert T\rvert +\lambda)^+$ such that $(c_i)_{\alpha <i<(\lvert T\rvert +\lambda)^+}$ is indiscernible over $(a/E,b/F)$.  Thus, the type of $(a/E,b/F,c_i)_{\alpha <i<(\lvert T\rvert +\lambda)^+}$ is constant. Since $a/E$, $b/F$ and the sequence $(c_i)_{i<(\lvert T\rvert +\lambda)^+}$ were arbitrary, Proposition \ref{NIP formulas have limit} implies that every $f(x,y,z)\in \mathcal{F}_{X/E\times Y/F}$ has NIP. Therefore, $X/E\times Y/F$ has NIP by Lemma \ref{IP_n hyperdef. set iif formulas}.
\end{proof}
\chapter{A proof of Fact \ref{delta is homeomorphism}}\label{AppendixB}
To prove Fact \ref{delta is homeomorphism}, we need an auxiliary lemma. This result originally appeared in the first arXiv version of \cite{Krupiski2019RamseyTA}. We include a proof here for completeness.

\begin{lema}\label{tau closure equivalence}
    For any flow $(G,X)$, and $A\subseteq u\mathcal{M}$, the $\tau$-closure $\cl_\tau(A)$ can be described as the set of all limits contained in $u\mathcal{M}$ of nets $\eta_i a_i$ such that $\eta_i\in\mathcal{M}$, $a_i\in A$ and $\lim_i \eta_i =u$
\end{lema}
\begin{proof}
    Consider $a \in \cl_{\tau}(A)$. Then, by the definition of the $\tau$-topology, there are nets $(g_i)_i \subseteq G$ and $(a_i)_i \subseteq A$ such that $\lim_i g_i = u$ and $\lim_i g_i a_i = a$. Note that $u a_i = a_i$, as $a_i \in A \subseteq u\mathcal{M}$. Put $\eta_i := g_i u \in \mathcal{M}$ for all $i$. By left continuity, we have that $\lim_i \eta_i = \lim_i g_i u = (\lim_i g_i) u = u u = u$. Furthermore, $\lim_i \eta_i a_i = \lim_i g_i u a_i = \lim_i g_i a_i = a$.

Conversely, consider any $a \in u\mathcal{M}$ for which there are nets $(\eta_i)_i \subseteq \mathcal{M}$ and $(a_i)_i \subseteq A$ such that $\lim_i \eta_i = u$ and $\lim_i \eta_i a_i = a$. Since each $\eta_i$ can be approximated by elements of $G$ and the semigroup operation is left continuous, one can find a subnet $(a_j')_j$ of $(a_i)_i$ and a net $(g_j)_j \subseteq G$ such that $\lim_j g_j = u$ and $\lim_j g_j a_j' = a$, which means that $a \in \cl_{\tau}(A)$.
\end{proof}

\begin{proof}[\textbf{Proof of Fact \ref{delta is homeomorphism}}]
We first show that $\delta$ is continuous. Consider an arbitrary closed set $F_{\bar{\varphi}, \bar{p},\bar{q}, r}$ and take $\eta \in \cl_\tau(\delta^{-1}[F_{\bar{\varphi}, \bar{p},\bar{q}, r}])$. Choose some nets $(\sigma_i)$ in $\aut(\C)$ and $(\eta_i)$ in $\delta^{-1}[F_{\bar{\varphi}, \bar{p},\bar{q}, r}]$ such that $\lim \sigma_i = u$ and $\lim \sigma_i\eta_i= \eta$. 

Suppose for a contradiction that $ \eta \notin \delta^{-1}[F_{\bar{\varphi}, \bar{p},\bar{q}, r}]$, i.e. $\neg R_{\bar{\varphi}, r}(\eta(\bar{p}),q)$. Then there is $b \models r(y)$ such that $$\varphi_1(x,y) \in \eta(p_1), \dots, \varphi_m(x,b) \in \eta(p_m)$$ and $$\varphi_{m+1}(x,b) \in q_1, \dots, \varphi_{m+n}(x,b) \in q_{n}.$$ Since $\lim \sigma_i = u$ and $\lim \sigma_i\eta_i= \eta$, there exists a natural number $N$ such that for all $i>N$ we have $\neg R_{\bar{\varphi}, r}(\eta_i(\bar{p}),\bar{q})$, a contradiction with the choice of $\eta_i$.

Next, we show the continuity of $\delta^{-1}$. Consider an arbitrary $\eta \in u\mathcal{M}$ and a basis of open neighborhoods of $\delta(\eta)$ consisting of all open neighborhoods of $\delta(\eta)$ of the form $U_{\bar{\varphi}, \bar{p},\bar{q}, r}$, where $U_{\bar{\varphi}, \bar{p},\bar{q}, r}$ is the complement of $F_{\bar{\varphi}, \bar{p},\bar{q}, r}$. Let $I$ consist of all tuples $(\bar{\varphi}, \bar{p},\bar{q}, r, b)$, where the tuples $(\bar{\varphi}, \bar{p},\bar{q}, r)$ have the form described above and $b$ is any tuple realizing $r$ and such that
$$\varphi_1(x,b) \in \delta(\eta)(p_1),\dots, \varphi_m(x,b) \in \delta(\eta)(p_m),
\varphi_{m+1}(x,b) \in q_1,\dots, \varphi_{m+n}(x,b) \in q_n.$$

Order $I$ by:
$(\bar{\varphi}(x,y), \bar{p},\bar{q}, r(y),b) \leq (\bar{\varphi}'(x,z), \bar{p}',\bar{q}', r'(z),b')$ if $y \subseteq z$, $r' |_y=r$, $b'|_y=b$, $m:=|\bar{p}| \leq |\bar{p}'|=:m'$ and $n:=|\bar{q}| \leq |\bar {q}'|=:n'$, and there exists an injection $$\sigma \colon \{1,\dots,m+n\} \to \{1,\dots,m'+n'\}$$ such that $\sigma[\{1,\dots,m\}] \subseteq \{1,\dots,m'\}$, $\sigma[\{m+1,\dots,m+n\}] \subseteq \{m'+1,\dots,m'+n'\}$ and:
\begin{enumerate}
    \item $\varphi_j = \varphi'_{\sigma(j)}$  for all $j \leq m+n$,
    \item $p_j = p'_{\sigma(j)}$ for all $j \leq m$,
    \item $q_{m+j}=q'_{\sigma(m+j)}$ for all $j \leq n$.
\end{enumerate}
It is easy to check that $(I, \leq)$ is a directed set. For $i =(\bar{\varphi}, \bar{p},\bar{q}, r, b) \in I$ by $U_i$ we mean $U_{\bar{\varphi}, \bar{p},\bar{q}, r}$, clearly $i \leq i'$ implies $U_i' \subseteq U_i$.

It suffices to show that for any net $(\eta_k)_{k \in K}$ in $u\mathcal{M}$ such that $\lim_k \delta(\eta_k) = \delta(\eta)$ we have $\tau-\lim_k \eta_k =\eta$. For that it is enough to show that for any subnet $(\rho_j)_{j \in J}$ of $(\eta_k)_{k \in K}$ we have that $\eta$ is in the $\tau$-closure of $(\rho_j)_{j \in J}$.

Since $\lim_j \delta(\rho_j) = \delta(\eta)$ for every $i \in I$ there exists $j_i \in J$ such that $\delta(\rho_{j_i}) \in U_i$. Writing $i=(\bar{\varphi}, \bar{p},\bar{q}, r,b)$, we have $b_i \models r$ such that 
\begin{equation}\tag{$\star$}
\begin{split}
    \varphi_1(x,b_i) \in \delta(\rho_{j_i})(p_1), \dots,  \varphi_m(x,b_i) \in \delta(\rho_{j_i})(p_m), \\
\varphi_{m+1}(x,b_i) \in q_1, \dots , \varphi_{m+n}(x,b_i) \in q_n.
\end{split}
\end{equation}


On the other hand, by the definition of $I$, we have 
\begin{equation}\tag{$\star\star$}
    \begin{split}
        \varphi_1(x,b) \in \delta(\eta)(p_1),\dots, \varphi_m(x,b) \in \delta(\eta)(p_m),\\
        \varphi_{m+1}(x,b) \in q_1,\dots, \varphi_{m+n}(x,b) \in q_n.
    \end{split}
\end{equation}

For each $i\in I$, choose $\sigma_i \in \aut(\C)$ such that $\sigma(b_i)=b$. 
Then, by $(\star)$ we have $\varphi_{m+1}(x,b) \in \sigma_i(q_1), \dots , \varphi_{m+n}(x,b) \in \sigma_i(q_n)$.

Since for varying $i \in I$ the tuples $(q_1,\dots,q_n)$ range over all finite tuples from $\bar{\mathcal{J}}$ and the tuples $(\varphi_{m+1}(x,b),\dots, \varphi_{m+n}(x,b))$ over all possible tuples of formulas belonging, respectively, to $q_1,\dots,q_n$, we see that $\lim_i \sigma_i |_{\bar{\mathcal{J}}} = Id_{\bar{\mathcal{J}}}$. So $\lim_i \pi_{\sigma_i} \circ u = u$ and clearly $\pi_{\sigma_i} \circ u \in \mathcal{M}$. 

Also by $(\star)$, 
$\varphi_1(x,b) \in \sigma_i (\rho_{j_i} (p_1)), \dots,  \varphi_m(x,b) \in \sigma_i (\rho_{j_i}(p_m))$. Again, since for varying $i \in I$ the tuples $(p_1,\dots,p_m)$ range over all finite tuples from $\bar{\mathcal{J}}$ and the tuples $(\varphi_1(x,b),\dots, \varphi_m(x,b))$ over all possible tuples of formulas belonging, respectively, to $\eta(p_1),\dots,\eta(p_m)$, we see that $\lim_i (\pi_{\sigma_i} \circ \rho_{j_i}) |_{\bar{\mathcal{J}}} = \eta |_{\bar{\mathcal{J}}}$.
Since $\rho_{j_i} = \rho_{j_i}u=u\rho_{j_i}$ and $\eta=\eta u$, we get $\lim_i \sigma_i u \rho_{j_i} = \eta$.

Hence, by the previous two paragraphs and Lemma \ref{tau closure equivalence}, we conclude that $\eta$ is in the $\tau$-closure of $(\rho_j)_{j \in J}$ as required.
    
\end{proof}
\singlespacing
\printbibliography

\end{document}